\documentclass[reqno,12pt]{amsart}

\usepackage[colorlinks = true,linkcolor = blue,urlcolor  = blue,citecolor = blue,anchorcolor = blue,linktocpage=true,
]{hyperref}
\usepackage{url}

\usepackage[rgb]{xcolor} 
\definecolor{mygreen}{rgb}{0,0.7,0.3}
\definecolor{myblue}{rgb}{0,0.50,1.20}
\definecolor{orange}{rgb}{2.55,1.65,0}
\definecolor{myorange}{rgb}{1,0.5,0.1}

\definecolor{fillred}{rgb}{1,0.9,0.9}
\definecolor{fillgreen}{rgb}{0.9,1,0.9}

\usepackage{amsthm,amssymb,amscd}
\usepackage{cleveref}
\usepackage{mathrsfs}
\usepackage{bbm}
\usepackage
{graphicx}
\usepackage{epic,eepic}
\usepackage{tikz-cd}
\tikzcdset{scale cd/.style={every label/.append style={scale=#1},
    cells={nodes={scale=#1}}}}
\usepackage{stmaryrd}

\usepackage{here}
\usepackage{bm}
\usepackage[all]{xy}

\usepackage{tikz}
\usetikzlibrary{calc,decorations.markings}
\usetikzlibrary{arrows.meta}
\usetikzlibrary{positioning}
\usetikzlibrary{cd, intersections, calc, decorations.pathmorphing, arrows, decorations.pathreplacing}

\usepackage{version}

\excludeversion{NB}
\excludeversion{NB2}
 \numberwithin{equation}{section}

\usepackage{thmtools}
\usepackage{cleveref}

\declaretheorem[
  style=plain,
  name=Theorem,
  numberwithin=section
]{thm}

\declaretheorem[
  style=plain,
  name=Proposition,
  sibling=thm
]{prop}

\declaretheorem[
  style=plain,
  name=Corollary,
  sibling=thm
]{cor}

\declaretheorem[
  style=plain,
  name=Lemma,
  sibling=thm
]{lem}

\declaretheorem[
  style=plain,
  name=Theorem
]{introthm}

\declaretheorem[
  style=definition,
  name=Definition,
  sibling=thm
]{dfn}

\declaretheorem[
  style=definition,
  name=Example,
  sibling=thm
]{ex}

\declaretheorem[
  style=definition,
  name=Claim,
  sibling=thm
]{claim}

\declaretheorem[
  style=definition,
  name=Notation,
  sibling=thm
]{conv}

\declaretheorem[
  style=remark,
  name=Remark,
  sibling=thm
]{rem}

\crefname{thm}{Theorem}{Theorems}
\crefname{prop}{Proposition}{Propositions}
\crefname{cor}{Corollary}{Corollaries}
\crefname{lem}{Lemma}{Lemmas}

\crefname{introthm}{Theorem}{Theorems}
\crefname{introcor}{Corollary}{Corollaries}
\crefname{introconj}{Conjecture}{Conjectures}

\crefname{dfn}{Definition}{Definitions}
\crefname{ex}{Example}{Examples}
\crefname{claim}{Claim}{Claims}
\crefname{conj}{Conjecture}{Conjectures}
\crefname{conv}{Notation}{Notations}

\crefname{rem}{Remark}{Remarks}
\crefname{prob}{Problem}{Problems}

\crefname{figure}{Figure}{Figures}
\crefname{section}{Section}{Sections}
\crefname{subsection}{Section}{Sections}
\crefname{appendix}{Appendix}{Appendices}

\def\C{{\mathbb C}}
\def\Q{{\mathbb Q}}
\def\Z{{\mathbb Z}}
\def\ve{{\varepsilon}}

\newcommand\R{{\mathbb{R}}}

\newcommand{\std}{\mathrm{std}}
\newcommand{\Hom}{\mathop{\mathrm{Hom}}\nolimits}

\newcommand{\uf}{\mathrm{uf}}
\newcommand{\fr}{\mathrm{fr}}

\newcommand{\mrm}{\mathrm{m}}

\newcommand{\ext}{\mathrm{e}}

\newcommand{\GS}{\mathrm{GS}}
\newcommand{\bbs}{\mathbbmtt{s}}
\newcommand{\ols}{\overline{s}}
\newcommand{\olS}{\overline{S}}
\newcommand{\Inv}{\mathrm{Inv}}

\newcommand\vw{{\overline{w_0}}}

\newcommand{\lieg}{\mathfrak{g}}
\newcommand{\lieh}{\mathfrak{h}}

\newcommand{\Sy}{\mathfrak{S}}
\newcommand{\op}{\mathrm{op}}

\newcommand{\weight}{\mathop{\mathrm{wt}}}

\newcommand{\cR}{{\mathcal{R}}}

\newcommand{\tcR}{{\underline{\mathcal{R}}}}
\newcommand{\inv}{\mathrm{inv}}
\newcommand{\wtmu}{\widetilde{\sfm}}
\newcommand{\GKdim}{\mathop{\mathrm{GKdim}}}

\newcommand\A{{\mathcal{A} }}

\newcommand{\cF}{\mathcal{F}}

\newcommand{\cO}{\mathcal{O}}
\newcommand{\cT}{\mathcal{T}}
\newcommand{\bcT}{\overline{\mathcal{T}}}
\newcommand{\cD}{\mathcal{D}}
\newcommand{\cE}{\mathcal{E}}
\newcommand{\cM}{\mathcal{M}}
\newcommand{\bA}{\mathbf{A}}
\newcommand\bs{{\boldsymbol{s}}}
\newcommand\ba{{\boldsymbol{a}}}
\newcommand\bn{{\boldsymbol{n}}}
\newcommand\bi{{\mathbf{i}}}
\newcommand\bg{{\boldsymbol{g}}}
\newcommand{\bD}{\boldsymbol{\Delta}}
\newcommand{\bJ}{\mathbf{J}}

\newcommand{\bbD}{\boldsymbol{D}}

\newcommand{\sfS}{\mathsf{S}}
\newcommand{\sfT}{\mathsf{T}}

\newcommand{\sfm}{\mathsf{m}}
\newcommand{\sfA}{\mathsf{A}}

\newcommand{\sfI}{\mathsf{I}}
\newcommand{\sfJ}{\mathsf{J}}
\newcommand{\sfRev}{\mathsf{Rev}}
\newcommand{\sfL}{\mathsf{L}}
\newcommand{\sfR}{\mathsf{R}}
\newcommand{\qq}{q}
\newcommand{\qqd}{q^{\frac{1}{2d}}}

\newcommand{\Uq}{\mathcal{U}_{q}}

\newcommand{\bt}{\mathop{\widetilde{\otimes}}}
\newcommand{\ut}{\mathop{\underline{\otimes}}}
\newcommand{\Deg}{\mathrm{Deg}}

\newcommand{\tbbs}{\widetilde{\bbs}}
\newcommand{\uDelta}{\mathop{\underline{\Delta}}}
\newcommand{\flA}{\mathsf{A}}
\newcommand{\sA}{\mathscr{A}}

\newcommand{\bSigma}{{\boldsymbol{\Sigma}}}
\newcommand{\vtr}[2]{#1\! \vartriangleright\! #2}
\newcommand{\btr}[2]{#1\! \blacktriangleright\! #2}
\newcommand{\isigma}{\imath_{4,\qq}^{\mathrm{sh}}}
\newcommand{\Phisigma}{\Phi_{3}^{\ext, \mathrm{sh}}}

\DeclareRobustCommand{\triangle}{%
  \mathchoice
    {\triangleaux{1}}
    {\triangleaux{1}}
    {\triangleaux{0.7}}
    {\triangleaux{0.5}}
}
\newcommand{\triangleaux}[1]{
\begin{tikzpicture}[scale=#1]
\draw(-30:0.15) -- (90:0.15) -- (210:0.15) -- cycle;
\end{tikzpicture}
}

\DeclareRobustCommand{\square}{%
  \mathchoice
    {\squareaux{1}}
    {\squareaux{1}}
    {\squareaux{0.7}}
    {\squareaux{0.5}}
}
\newcommand{\squareaux}[1]{
\begin{tikzpicture}[scale=#1]
\draw(0,0) -- (0,0.25) -- (0.25, 0.25) -- (0.25, 0) -- cycle;
\end{tikzpicture}
}

\DeclareRobustCommand{\dsquare}{%
  \mathchoice
    {\dsquareaux{1}}
    {\dsquareaux{1}}
    {\dsquareaux{0.7}}
    {\dsquareaux{0.5}}
}
\newcommand{\dsquareaux}[1]{
\begin{tikzpicture}[scale=#1]
\draw(0,0) -- (0,0.25) -- (0.25, 0.25) -- (0.25, 0) -- cycle;
\draw(0, 0.07) -- (0.25, 0.07);
\draw(0, 0.18) -- (0.25, 0.18);
\end{tikzpicture}
}

\DeclareMathOperator{\wt}{wt}
\DeclareMathOperator{\ad}{ad}
\DeclareMathOperator{\id}{id}
\DeclareMathOperator{\Ima}{Im}

\newcommand\Conf{{\mathrm{Conf}}}%

\DeclareMathOperator{\sgn}{sgn}

\newcommand{\dprod}{\mathop{\overrightarrow{\prod}}\limits}

\usepackage[backend=biber,
  style=alphabetic, 
  giveninits=true, 
  doi=false, 
  url=false, 
  maxbibnames=99, 
  maxcitenames=99,
  maxalphanames=99,
  sorting=nyt
]{biblatex}

\AtEveryBibitem{
  \clearfield{issn}
  \clearfield{isbn}
}

\DeclareFieldFormat[article]{title}{\mkbibemph{#1}}

\DeclareFieldFormat[inproceedings]{title}{\mkbibemph{#1}}

\DeclareFieldFormat[incollection]{title}{\mkbibemph{#1}}

\DeclareFieldFormat{booktitle}{#1}

\DeclareFieldFormat{journaltitle}{#1}

 \DeclareFieldFormat{title}{\mkbibemph{#1}\addcomma\space}

\DeclareBibliographyDriver{book}{%
  \printnames{author}%
  \setunit{\addcomma\space}%
  \mkbibemph{\printfield{title}}%
  \setunit{\addcomma\space}%
  \printfield{series}%
  \setunit{\space}%
  \printfield{volume}%
  \setunit{\addperiod\space}%
  \printlist{publisher}%
  \setunit{\addcomma\space}%
  \printfield{year}%
  \setunit{\addcomma\space}%
  \printfield{pages}%
  \finentry
}

\renewbibmacro*{volume+number+eid}{%
  \mkbibbold{\printfield{volume}}%
  \setunit{\addcomma\space}%
  \printtext{no.\space\printfield{number}}%
}

\DeclareFieldFormat{pages}{#1}

\DeclareBibliographyDriver{article}{%
  \printnames{author}%
  \setunit{\addcomma\space}%
  \printfield[title]{title}%
  \setunit{\addcomma\space}%
  \printfield{journaltitle}%
  \setunit{\addspace}%
  \usebibmacro{volume+number+eid}%
  \setunit{\addspace}%
  \printtext{\mkbibparens{\printfield{year}}}%
  \setunit{\addcomma\space}%
  \printfield{pages}%
  \finentry
}

\tikzset{->-/.style 2 args={
	postaction={decorate},
	decoration={markings, mark=at position #1 with {\arrow[thick, #2]{>}}} 
    },
    ->-/.default={0.5}{}
}
\tikzset{-<-/.style 2 args={
	postaction={decorate},
	decoration={markings, mark=at position #1 with {\arrow[thick, #2]{<}}} 
    },
    -<-/.default={0.5}{}
}

\tikzset{wdist/.style 2 args={
    blue,->,shorten >=0.9em,shorten <=0.9em,transform canvas={xshift=#1 cm,yshift=#2 cm}
    },
    wdist/.default={0}{-0.2}
}
\tikzset{wdistt/.style 2 args={
    blue,<->,shorten >=0.9em,shorten <=0.9em,transform canvas={xshift=#1 cm,yshift=#2 cm}
    },
    wdistt/.default={0}{-0.2}
}
\tikzset{hdist/.style 2 args={
    mygreen!40,{Triangle[width=6pt,length=6pt]}-, line width=2pt,shorten >=0.9em,shorten <=0.9em,transform canvas={xshift=#1 cm,yshift=#2 cm}
    },
    hdist/.default={0}{-0.6}
}

\newcommand{\bline}[3]{
    \path (#1)++(0,-#3) coordinate(m1);
    \path (#2)++(0,-#3) coordinate(m2);
    \filldraw[gray!30] (m1) -- (#1) -- (#2) -- (m2) --cycle;
    \draw[thick] (#1) -- (#2);
}
\newcommand{\tline}[3]{
    \path (#1)++(0,#3) coordinate(m1);
    \path (#2)++(0,#3) coordinate(m2);
    \filldraw[gray!30] (m1) -- (#1) -- (#2) -- (m2) --cycle;
    \draw[thick] (#1) -- (#2);
}

\newcommand\dnode[2]{
\filldraw[draw=#2,fill=#2!15,even odd rule](#1) circle(1.7pt) circle(3pt)
}
\newcommand\tnode[2]{
\draw[#2](#1) circle(3pt); \filldraw[draw=#2,fill=#2](#1) circle(1.7pt)
}
\newcommand{\fnode}[2]{
\draw(#1) node[draw=#2,scale=0.6]{\pgfuseplotmark{square*}}
}
\newcommand{\fdnode}[2]{
\draw(#1) node[draw=#2,fill=#2!15,scale=0.7]{\pgfuseplotmark{square*}} node[draw=#2,fill=white,scale=0.35]{\pgfuseplotmark{square*}}
}
\newcommand{\ftnode}[2]{
\draw(#1) node[draw=#2,scale=0.7]{\pgfuseplotmark{square*}} node[draw=#2,fill=#2,scale=0.35]{\pgfuseplotmark{square*}}
}

\tikzset{
  qarrow/.style={->,shorten >=2pt,shorten <=2pt,>=latex
      },
}
\tikzset{
  head/.style={shorten >=4pt
      },
}
\tikzset{
  tail/.style={shorten <=4pt
      },
}
\newcommand{\qarrow}[2]{\draw[qarrow](#1) -- (#2)}
\newcommand{\qdarrow}[2]{\draw[qarrow,dashed](#1) -- (#2)}
\newcommand{\qsarrow}[2]{\draw[qarrow,head,tail](#1) -- (#2)}

\newcommand{\divpt}[3]{
\pgfmathsetmacro{\n}{#3}
\foreach \i in {0,...,\n} \draw ($(#1)!\i/\n!(#2)$) coordinate(x\i); 
}

\newcommand{\quiverAv}[4]{
    \pgfmathsetmacro{\r}{#4}
    \pgfmathsetmacro{\n}{\r+1}
    \divpt{#2}{#1}{\n}
    \foreach \i in {1,...,\r} \draw(x\i) coordinate (a\i); 
    \divpt{#3}{#1}{\n}
    \foreach \i in {1,...,\r} \draw(x\i) coordinate (b\i); 
    \divpt{#2}{#3}{\n}
    \foreach \i in {1,...,\r} \draw(x\i) coordinate (V\i 0); 
    \foreach \j in {1,...,\r} {   
        \pgfmathsetmacro{\k}{\r+1-\j} 
        \divpt{a\j}{b\j}{\k}; 
        \foreach \h in {0,...,\k} \draw(x\h) coordinate (V\h\j) ; 
        }
    \foreach \j in {1,...,\r} {   
        \pgfmathsetmacro{\k}{\r+1-\j}
        \pgfmathsetmacro{\h}{\k-1}
        \foreach \p in {0,...,\h} { 
            \pgfmathtruncatemacro{\q}{\p+1}
            \draw[qarrow](V\q\j)--(V\p\j);
            }
        \foreach \p in {0,...,\h} { 
            \pgfmathtruncatemacro{\q}{\p+1}
            \pgfmathtruncatemacro{\i}{\j-1}
            \ifnum\j<2 \draw[qarrow,myblue] (V\p\j)--(V\q\i); 
            \else
            \draw[qarrow](V\p\j)--(V\q\i);
            \fi
            }
        }    
    \foreach \j in {1,...,\r} {   
        \pgfmathsetmacro{\k}{\r+1-\j}
        \pgfmathsetmacro{\h}{\k-1}
        \pgfmathsetmacro{\i}{\j-1}
        \foreach \p in {1,...,\k} { 
            \ifnum \j<2 
            \draw[qarrow,myblue](V\p\i)--(V\p\j);
            \else
            \draw[qarrow](V\p\i)--(V\p\j);
            \fi
            }
        }
}

\newcommand{\vertexA}[4]{
    \pgfmathsetmacro{\r}{#4}
    \pgfmathsetmacro{\n}{\r+1}
    \divpt{#2}{#1}{\n}
    \foreach \i in {1,...,\r} \draw(x\i) coordinate (a\i); 
    \divpt{#3}{#1}{\n}
    \foreach \i in {1,...,\r} \draw(x\i) coordinate (b\i); 
    \divpt{#2}{#3}{\n}
    \foreach \i in {1,...,\r} \draw(x\i) coordinate (V\i 0); 
    \foreach \j in {1,...,\r} {   
        \pgfmathsetmacro{\k}{\r+1-\j} 
        \divpt{a\j}{b\j}{\k}; 
        \foreach \h in {0,...,\k} \draw(x\h) coordinate (V\h\j) ; 
        }
}

\newcommand{\quiverCv}[4]{
    \path($(#1)!0.2!(#2)$) coordinate(x1L); 
    \path($(#1)!0.2!(#3)$) coordinate(x1R); 
    \pgfmathsetmacro{\r}{#4}
    \pgfmathsetmacro{\n}{\r+1}
    \divpt{#2}{x1L}{\n}
    \foreach \i in {1,...,\r} \draw(x\i) coordinate (a\i); 
    \divpt{#3}{x1R}{\n}
    \foreach \i in {1,...,\r} \draw(x\i) coordinate (b\i); 
    \divpt{#2}{#3}{\n}
    \foreach \i in {1,...,\r} \draw(x\i) coordinate (V\i 0) ; 
    \foreach \i in {1,...,\r} \draw(V\i 0);
    \pgfmathtruncatemacro{\s}{\r-1}
    \foreach \j in {1,...,\s} {   
        \divpt{a\j}{b\j}{\r}; 
        \foreach \h in {0,...,\r} \draw(x\h)
        coordinate (V\h\j); 
        }
    \divpt{a\r}{b\r}{\r}; 
        \foreach \h in {0,...,\r} \draw(x\h) coordinate (V\h\r); 
    \foreach \j in {1,...,\s} { 
        \foreach \p in {0,...,\s} { 
            \pgfmathtruncatemacro{\q}{\p+1}
            \draw[qarrow](V\q\j)--(V\p\j);
            }
        }
        \foreach \p in {0,...,\s} { 
            \pgfmathtruncatemacro{\q}{\p+1}
            \draw[qarrow,head,tail](V\q\r)--(V\p\r);
            }
    \pgfmathsetmacro{\u}{\r-2}
     \ifnum\r>2 
    \foreach \j in {1,...,\u} { 
        \foreach \p in {1,...,\s} { 
            \pgfmathtruncatemacro{\k}{\j+1}
            \draw[qarrow](V\p\j)--(V\p\k);
            }
        }
    \foreach \j in {1,...,\u} { 
        \foreach \p in {1,...,\r} { 
            \pgfmathtruncatemacro{\k}{\j+1}
            \pgfmathtruncatemacro{\q}{\p-1}
            \draw[qarrow](V\q\k)--(V\p\j);
            }
        }
    \else
    \fi 
        \foreach \p in {1,...,\r} { 
            \pgfmathtruncatemacro{\q}{\p-1}
            \draw[qarrow,tail](V\q\r)--(V\p\s);
            }
        \foreach \p in {1,...,\s} { 
            \draw[qarrow,head](V\p\s)--(V\p\r);
            }
    \foreach \p in {0,...,\u} {
        \pgfmathtruncatemacro{\q}{\p+1}
        \draw[qarrow,myblue](V\p 1) -- (V\q 0);
        \draw[qarrow,myblue](V\q 0) -- (V\q 1);
    }
    \draw[qarrow,head,tail,myblue] (V\s \r) -- (V\r 0);
    \draw[qarrow,head,tail,myblue] (V\r 0) -- (V\r \r);
}

\newcommand{\quiverCvMiddle}[4]{
    \pgfmathsetmacro{\r}{#4}
    \pgfmathsetmacro{\n}{\r+1}
    \divpt{#2}{#1}{\n}
    \foreach \i in {1,...,\r} \draw(x\i) coordinate (a\i); 
    \divpt{#3}{#1}{\n}
    \foreach \i in {1,...,\r} \draw(x\i) coordinate (b\i); 
    \divpt{#2}{#3}{\n}
    \foreach \i in {1,...,\r} \draw(x\i) coordinate (V\i 0) ; 
    \foreach \i in {1,...,\r} \draw(V\i 0);
    \pgfmathtruncatemacro{\s}{\r-1}
    \foreach \j in {1,...,\s} {   
        \divpt{a\j}{b\j}{\r}; 
        \foreach \h in {0,...,\r} \draw(x\h)
        coordinate (V\h\j); 
        }
    \divpt{a\r}{b\r}{\r}; 
        \foreach \h in {0,...,\r} \draw(x\h) coordinate (V\h\r); 
    \foreach \j in {1,...,\s} { 
        \foreach \p in {0,...,\s} { 
            \pgfmathtruncatemacro{\q}{\p+1}
            \draw[qarrow](V\q\j)--(V\p\j);
            }
        }
        \foreach \p in {0,...,\s} { 
            \pgfmathtruncatemacro{\q}{\p+1}
            \draw[qarrow,head,tail](V\q\r)--(V\p\r);
            }
    \pgfmathsetmacro{\u}{\r-2}
     \ifnum\r>2 
    \foreach \j in {1,...,\u} { 
        \foreach \p in {1,...,\s} { 
            \pgfmathtruncatemacro{\k}{\j+1}
            \draw[qarrow](V\p\j)--(V\p\k);
            }
        }
    \foreach \j in {1,...,\u} { 
        \foreach \p in {1,...,\r} { 
            \pgfmathtruncatemacro{\k}{\j+1}
            \pgfmathtruncatemacro{\q}{\p-1}
            \draw[qarrow](V\q\k)--(V\p\j);
            }
        }
    \else
    \fi 
        \foreach \p in {1,...,\r} { 
            \pgfmathtruncatemacro{\q}{\p-1}
            \draw[qarrow,tail](V\q\r)--(V\p\s);
            }
        \foreach \p in {1,...,\s} { 
            \draw[qarrow,head](V\p\s)--(V\p\r);
            }
    \foreach \p in {0,...,\u} {
        \pgfmathtruncatemacro{\q}{\p+1}
        \draw[qarrow,myblue](V\p 1) -- (V\q 0);
        \draw[qarrow,myblue](V\q 0) -- (V\q 1);
    }
    \draw[qarrow,head,tail,myblue] (V\s \r) -- (V\r 0);
    \draw[qarrow,head,tail,myblue] (V\r 0) -- (V\r \r);
}

\newcommand{\vertexC}[4]{
    \pgfmathsetmacro{\r}{#4}
    \pgfmathsetmacro{\n}{\r+1}
    \divpt{#2}{#1}{\n}
    \foreach \i in {1,...,\r} \draw(x\i) coordinate (a\i); 
    \divpt{#3}{#1}{\n}
    \foreach \i in {1,...,\r} \draw(x\i) coordinate (b\i); 
    \divpt{#2}{#3}{\n}
    \foreach \i in {1,...,\r} \draw(x\i) coordinate (V\i 0) ; 
    \pgfmathtruncatemacro{\s}{\r-1}
    \foreach \j in {1,...,\s} {   
        \divpt{a\j}{b\j}{\r}; 
        \foreach \h in {0,...,\r} \draw(x\h) coordinate (V\h\j) ; 
        }
    \divpt{a\r}{b\r}{\r}; 
        \foreach \h in {0,...,\r} \draw(x\h) coordinate (V\h\r); 
    
}

\pgfdeclarelayer{bg}    
\pgfsetlayers{bg,main}  

\begin{document}
\title[Quantum configuration spaces of decorated flags]
{Algebraic study of quantum configuration spaces of decorated flags}

\author[Tsukasa Ishibashi]{Tsukasa Ishibashi}
\address{(Tsukasa Ishibashi) Mathematical Institute, Tohoku University, 6-3 Aoba, Aramaki, Aoba-ku, Sendai, Miyagi 980-8578, Japan.}
\email{tsukasa.ishibashi.a6@tohoku.ac.jp}

\author[Hironori Oya]{Hironori Oya}
\address{(Hironori Oya) Department of Mathematics, Institute of Science Tokyo, 2-12-1 Ookayama, Meguro-ku, Tokyo, 152-8551, Japan}
\email{hoya@math.titech.ac.jp}

\date{\today}

\begin{abstract}
Let $G$ be a connected, simply connected complex simple algebraic group and $\sA_G=G/U^+$ its base affine space, whose elements are called decorated flags. We introduce \emph{the quantum configuration space of decorated flags} $\mathcal{O}_q(\mathrm{Conf}_K \mathscr{A}_G)$ and initiate its algebraic study, based on the representation theory of quantized enveloping algebras. 
Our algebra $\mathcal{O}_q(\mathrm{Conf}_K \mathscr{A}_G)$ gives a quantum analogue of the configuration space $\mathrm{Conf}_K \mathscr{A}_G$ of $K$ decorated flags, which provides local building blocks for the Fock--Goncharov moduli space $\mathscr{A}_{G,\boldsymbol{\Sigma}}$ \cite{FG06} of decorated twisted $G$-local systems on a marked surface $\boldsymbol{\Sigma}$. 

We establish basic algebraic properties of $\mathcal{O}_q(\mathrm{Conf}_K \mathscr{A}_G)$ such as the domain property, quantum normalization of representatives, the quantum cyclic shifts, the quantum Wilson lines, whose classical counterparts have been fundamental in the study of $\mathscr{A}_{G,\boldsymbol{\Sigma}}$. Moreover, we construct 
quantum seeds for $\mathcal{O}_q(\mathrm{Conf}_4 \mathscr{A}_G)$ by transporting the Berenstein--Zelevinsky quantum cluster structure on $\mathcal{O}_q(G)$ \cite{BZ} via quantum Wilson lines, and prove that $\mathcal{O}_q(\mathrm{Conf}_4 \mathscr{A}_G)$ coincides with the corresponding quantum cluster algebra and its upper counterpart after the localization at frozen variables. 
The exchange matrices for our quantum seeds agree with the Goncharov--Shen exchange matrices \cite{GS:Quantum}. 
We also show that quantum seeds for $\mathcal{O}_q(\mathrm{Conf}_4 \mathscr{A}_G)$ restrict to those for $\mathcal{O}_q(\mathrm{Conf}_3 \mathscr{A}_G)$. 

Finally, we construct quantum seeds for $\mathcal{O}_q(\mathrm{Conf}_K \mathscr{A}_G)$, $K\geq 5$, and prove that the corresponding quantum cluster algebras contain the quantum configuration space $\mathcal{O}_q(\mathrm{Conf}_K \mathscr{A}_G)$.
\end{abstract}

\maketitle

\setcounter{tocdepth}{1}
\tableofcontents

\section{Introduction}
Let $G$ be a connected, simply connected complex simple algebraic group. The moduli space $\mathrm{Loc}_{G,\Sigma}$ of $G$-local systems (or equivalently, flat $G$-connections) on a surface $\Sigma$ has been studied from various geometric viewpoints, especially motivated by its appearance as the classical phase space of gauge theories, such as 2-dimensional Yang--Mills theory \cite{AB,Witten91} and 3-dimensional Chern--Simons theory \cite{Witten89}. The symplectic/Poisson structures on $\mathrm{Loc}_{G,\Sigma}$ are further investigated by Goldman \cite{Goldman84,Goldman86} and Fock--Rosly \cite{FR99}, where the latter work establishes the intimate connection between the Atiyah--Bott--Goldman Poisson structure and the classical $r$-matrix of the Lie group. 
There are several approaches to the mathematical construction of its quantization from various viewpoints, including the geometric quantization \cite{Hitchin,APW}, skein quantization \cite{Turaev91,Bullock97,PS00,Sikora01}, quantum moduli algebra \cite{AGS1,AGS2,AGS3,BR1,BR2}, quantization via factorization homology \cite{BBJ18,JLSS}, among others. These constructions reveal, in different ways, a close relationship between the quantization of $\mathrm{Loc}_{G,\Sigma}$ and quantum groups, including the quantized enveloping algebra and the quantum coordinate ring of $G$.

A quantization approach based on explicit coordinate systems have lead to the quantum cluster theory \cite{Kashaev98,CF99,BZ,FG:Quantum,GS:Quantum}. In particular, Fock and Goncharov \cite{FG06} introduced an extension $\sA_{G,\bSigma}$ of $\mathrm{Loc}_{G,\Sigma}$, which parameterize the twisted $G$-local systems on a `marked' surface $\bSigma$, together with additional data of \emph{decorations} at marked points of $\bSigma$. 
It has been established that the moduli space $\sA_{G,\bSigma}$ admits a natural cluster $K_2$-structure for any connected, simply connected complex simple algebraic group $G$ in the sequence of works
\cite{FG06,Zickert,Le19,GS:Quantum}, and moreover if $\bSigma$ has no interior marked points (punctures), the quantum cluster theory provides a quantized algebra $\cO_q^{\mathrm{cl}}(\sA_{G,\bSigma})$ of functions as a quantum upper cluster algebra. See \cite{IOS,IOpuncture} for the comparison of $\cO(\sA_{G,\bSigma})$ and the upper cluster algebra. 

While the quantum cluster framework is plausible for its explicit computability, positivity and other features, it is often non-trivial to relate it to the classical geometry. In particular, its connection to the quantum groups is not apparent at all from the definition. Our aim is to rather introduce an algebraic model $\cO_q(\sA_{G,\bSigma})$ of quantized algebra of functions based on the representation theory of quantum groups, and then relate it to $\cO_q^{\mathrm{cl}}(\sA_{G,\bSigma})$. In this paper, we initiate such a study for the case where $\bSigma=\bbD_K$ is the disk with $K$-marked points on the boundary. In this case, $\sA_{G,\bbD_K}$ is isomorphic to the configuration space $\Conf_K \sA_G$ of $K$ decorated flags. 

\subsection{Quantum configuration spaces}
Fix an opposite pair $B^\pm\subset G$ of Borel subgroups, let $H:=B^+\cap B^-$ be the associated Cartan subgroup, and $U^\pm:=[B^\pm,B^\pm]$ the maximal unipotent subgroups. The elements of the homogeneous $G$-space $\sA_G:=G/U^+$ are called \emph{decorated flags}. 

For $K \in \Z_{\geq 2}$, the configuration space of $K$ decorated flags is defined to be the quotient stack 
\[
\Conf_K \sA_G := [(\overbrace{\sA_G \times \dots \times \sA_G}^{K\text{ times}})/G],
\]
where the action of $G$ is the diagonal left action.\footnote{We do not need the details about the geometry of stacks, as we discuss them only for motivating the algebraic constructions in this paper. In particular, we only focus on the algebras of global functions on them. We just remark that the algebra of global functions on a quotient stack $[X/G]$ is given by $\cO([X/G])=\cO(X)^G$, and that any morphism of stacks $f: \mathcal{X} \to \mathcal{Y}$ induces a morphism of algebras $f^\ast:\cO(\mathcal{Y}) \to \cO(\mathcal{X})$. For more details, see \cite{IO:Wilson, IOS} and the references therein.} To motivate our constructions, let us describe the algebra of functions $\cO(\Conf_K \sA_G)$ on 
$\Conf_K \sA_G$. By the algebraic Peter--Weyl theorem, we have an isomorphism of $G\times G$-modules 
\begin{align}\label{eq:O(G)_PW}
    \cO(G) \cong \bigoplus_{\lambda \in P_+} V(\lambda)^\ast \boxtimes V(\lambda),
\end{align}
where $P_+$ denotes the set of integral dominant weights and $V(\lambda)$ is the irreducible  highest weight $G$-module with highest weight $\lambda \in P_+$. 
It restricts to an isomorphism of $G$-modules 
\begin{align}\label{eq:O(A)_PW}
    \cO(\sA_G)=\cO(G)^{1 \times U^+} \cong \bigoplus_{\lambda \in P_+} V(\lambda)^\ast.
\end{align}
We thus get an isomorphism
\begin{align}\label{eq:classical_Conf_k}
    \cO(\Conf_K \sA_G)= (\cO(\sA_G)^{\otimes K})^G \cong \bigoplus_{\lambda_1,\dots,\lambda_K \in P_+} (V(\lambda_1)^\ast \otimes \dots \otimes V(\lambda_K)^\ast)^G.
\end{align}
The quantum analogue of \eqref{eq:O(G)_PW} is obtained by upgrading $V(\lambda)$ into an irreducible integrable highest weight module of highest weight $\lambda\in P_+$ over the quantized enveloping algebra $\Uq$ (over $\Bbbk:=\Q(\qqd)$; see \cref{subsec:QEA}), which is called \emph{the quantized coordinate ring} $\cO_{\qq}(G)$ of $G$. 
We then introduce $\cO_{\qq}(\sA_G)$ as a quantum analogue of \eqref{eq:O(A)_PW} (see \eqref{eq:def_AG}). In order to consider the quantum analogue of \eqref{eq:classical_Conf_k}, we use \emph{the braided tensor product} (see \cite{Majid} and \cref{subsec:braided_product} below) $\A_1 \bt \A_2$ defined for locally finite $\Uq$-module algebras $\A_1$, $\A_2$. It turns out that $\cO_{\qq}(\sA_G)^{\bt K}$ is an appropriate quantum analogue of $\cO(\sA_G)^{\otimes K}$. We then define $\cO_{\qq}(\Conf_K \sA_G)$ as 
\[
\cO_{\qq}(\Conf_K \sA_G):=\left(\cO_{\qq}(\sA_G)^{\bt K}\right)^\inv,
\]
where $\inv$ stands for a $\Uq$-invariant part. Since we consider the braided tensor product, it is a subalgebra of $\cO_q(\sA_G)^{\bt K}$, which is an Ore domain (\cref{prop:qcsgrading}). 
We call $\cO_{\qq}(\Conf_K \sA_G)$ \emph{the quantum configuration space of $K$ decorated flags} (\cref{def:quantum_config_space}). We remark that the same algebra (up to conventional differences) also appears in the categorical work of Jordan--Le--Schrader--Shapiro \cite{JLSS}, denoted by $\cF_\qq(\Conf_K^{\mathrm{fr}})^{U_\qq(\mathfrak{g})}$ in their notation, whose quantum cluster structure is described for $G=SL_2$ and $PGL_2$.
The main theme of this paper is the algebraic study of $\cO_\qq(\Conf_K \sA_G)$ for general $G$. 

\subsection{Normalized representative}
Consider a $G$-orbit $[\flA_1,\dots,\flA_K] \in \Conf_K\sA_G$ such that the pair $(\flA_1,\flA_K)$ is \emph{generic} in the sense that $[\flA_1,\flA_K]=[h.[U^+], \vw.[U^+]]$ for some $h \in H$. Here $\vw$ is a certain lift of the longest element $w_0$ of the Weyl group $W$ in $G$. 
Then there exists a unique representative of the form $([U^+],\flA'_2,\dots,\flA'_{K-1}, h^{-1}\vw.[U^+])$ with a unique $h \in H$ in the orbit $[\flA_1,\dots,\flA_K]$. We call it \emph{the normalized representative} with respect to the position $(1, K)$. 
This leads to an isomorphism of algebraic varieties
\begin{align}\label{eq:normalized_config}
    H \times \sA_G^{K-2} \cong \Conf_K^{\{K\}} \sA_G, \quad (h,\flA'_2,\dots,\flA'_{K-1}) \mapsto [[U^+],\flA'_2,\dots,\flA'_{K-1}, h^{-1}\vw.[U^+]],
\end{align}
where $\Conf_K^{\{K\}} \sA_G \subset \Conf_K \sA_G$ denotes the Zariski open subspace (which is representable by an algebraic variety) consisting of $G$-orbits $[\flA_1,\dots,\flA_K]$ such that the pair $(\flA_1,\flA_K)$ is generic. See \cref{fig:normalized_config} for an illustration.

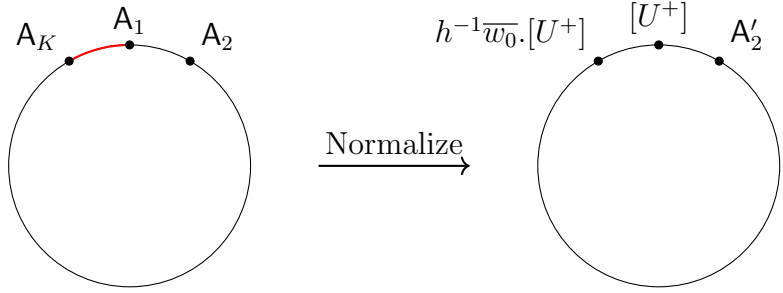
\begin{figure}[ht]
    \centering
\begin{tikzpicture}
\draw(0,0) circle(1.6cm);
\draw[red,thick] (90:1.6) arc(90:120:1.6);
\foreach \i in {90,60,120} \filldraw(\i:1.6) circle(1.5pt);
\node[above] at (90:1.6) {$\flA_1$};
\node[above right] at (60:1.6) {$\flA_2$};
\node[above left] at (120:1.6) {$\flA_K$};
\draw[thick,->] (2.5,0) --node[midway,above]{Normalize} ++(2,0);
\begin{scope}[xshift=7cm]
\draw(0,0) circle(1.6cm);
\foreach \i in {90,60,120} \filldraw(\i:1.6) circle(1.5pt);
\node[above] at (90:1.6) {$[U^+]$};
\node[above right] at (60:1.6) {$\flA'_2$};
\node[above left] at (120:1.6) {$h^{-1}\vw.[U^+]$};
\end{scope}
\end{tikzpicture}
    \caption{Taking the normalized representative with respect to the position $(1,K)$.}
    \label{fig:normalized_config}
\end{figure}

We remark that a similar normalization also works for the moduli space $\sA_{G,\boldsymbol{\Sigma}}$ by choosing any ideal arc and imposing the genericity condition for the associated pair of decorations, which is a fundamental feature of this moduli space that allows us to cut-and-paste the geometric data along ideal arcs. This leads to the construction of cluster charts via ideal triangulation \cite{FG06,Zickert,Le19,GS:Quantum}, and the definition of Wilson lines \cite{IO:Wilson,IOS}.

We will establish a quantum analogue of the normalized representative in the following generalized form. 
\begin{introthm}[{\cref{t:invstr}}]\label{introthm:normrep}
Let $\A$ be a locally finite $\Uq$-module algebra, and let $\A^{\ext}:=\A\otimes \Bbbk[P]$ be its extension by the group algebra $\Bbbk[P]=\bigoplus_{\lambda\in P}\Bbbk e(\lambda)$ of the weight lattice $P$, endowed with certain multiplication (see before \cref{t:invstr} for its precise definition). Then there is an injective algebra homomorphism
\[
\Phi_{\A}^{\ext}\colon
\left(\cO_\qq(\sA_G)\bt\A\bt\cO_\qq(\sA_G)\right)^{\inv}
\to \A^{\ext}.
\]
Moreover, for each $\lambda\in P_+$ there is $D_\lambda\in \left(\cO_\qq(\sA_G)\bt\A\bt\cO_\qq(\sA_G)\right)^{\inv}$ such that $\Phi_{\A}^{\ext}(D_\lambda)=e(\lambda)$. The set $\cD:= \{\qq^aD_{\lambda}\mid a\in \frac{1}{2d}\Z, \lambda\in P_+\}$ is an Ore set of $(\cO_\qq(\sA_G)\bt \A\bt \cO_\qq(\sA_G))^{\inv}$, and $\Phi_{\A}^{\ext}$ extends to an algebra isomorphism
\begin{align}\label{eq:intro_normalization}
\left(\cO_\qq(\sA_G)\bt\A\bt\cO_\qq(\sA_G)\right)^{\inv}[\cD^{-1}]
\xrightarrow{\sim}\A^{\ext}.
\end{align}
\end{introthm}

In particular when $\A=\cO_{\qq}(\sA_G)^{\bt (K-2)}$, $\Phi_{\A}^{\ext}$ provides an isomorphism
\begin{align*}
    \cO_{\qq}(\Conf_K \sA_G)[\cD^{-1}]\xrightarrow{\sim} (\cO_q(\sA_G)^{\bt K-2})^{\ext},
\end{align*}
which is a quantum analogue of \eqref{eq:normalized_config}. 

\subsection{Quantum cyclic shift}
We have an isomorphism of stacks
\begin{align}\label{eq:cyclic_shift}
    \mathcal{S}_K: \Conf_K \sA_G \xrightarrow{\sim} \Conf_K \sA_G, \quad [\flA_1,\flA_2,\dots,\flA_K] \mapsto [\flA_2,\dots,\flA_K,s_G.\flA_1],
\end{align}
called \emph{the twisted cyclic shift},
where $s_G:=\overline{w}_0^2 \in Z(G)$. This is a fundamental symmetry of $\Conf_K \sA_G$. For example, the cluster structure on $\Conf_K \sA_G$ is introduced so that it is invariant under $\mathcal{S}_K$ \cite{FG06,GS:DT,GS:Quantum}. In particular, it is lifted to an automorphism on $\cO_q^{\mathrm{cl}}(\Conf_K \sA_G)$ as a composite of quantum cluster transformations. 
We will introduce a quantum analogue of $\mathcal{S}_K$ on our quantum configuration space in the following generalized form.
\begin{introthm}[\cref{t:cyclic}]\label{introthm:cyclic}
Let $\A$ be a locally finite $\Uq$-module algebra. Define a $\Bbbk$-linear map
\begin{align*}
    \sigma_{\A}\colon (\cO_\qq(\sA_G)\bt \A\bt \cO_\qq(\sA_G))^{\inv}\to (\cO_\qq(\sA_G)\bt \cO_\qq(\sA_G)\bt \A)^{\inv}
\end{align*}
by 
\begin{align*}
    \sum_{i}\phi_{i}\otimes a_i\otimes \phi'_{i}\mapsto 
     \sum_{i}(-1)^{\langle 2\rho^{\vee}, \wt \phi'_{i}\rangle}\qq^{(2\rho, \wt \phi'_{i})}\phi'_{i}\otimes \phi_{i}\otimes a_i.
\end{align*}
See \cref{s:QEA} for the definition of $\rho, \rho^{\vee}$ and $\wt$. Then $\sigma_{\A}$ is an isomorphism of $\Bbbk$-algebras. 
\end{introthm}

If we set $\A:=\cO_{\qq}(\sA_G)^{\bt (K-2)}$, we obtain a $\Bbbk$-algebra isomorphism 
\begin{align*}
\sigma_K:=\sigma_{\A}\colon \cO_\qq(\Conf_K\sA_G)\xrightarrow{\sim}\cO_\qq(\Conf_K\sA_G). 
\end{align*}
It is called \emph{a quantum cyclic shift} on $\cO_\qq(\Conf_K\sA_G)$. For $\lambda\in P_+$, we write the element $D_{\lambda}\in \cO_\qq(\Conf_K\sA_G)$ in \cref{introthm:normrep} as $D_{K;\lambda}$, and set 
\begin{align*}
    D_{i;\lambda}:=\sigma_K^{i}(D_{K;\lambda})
\end{align*}
for $i=1,\dots, K-1$. Then, for $I=\{i_1,\dots, i_k\}\subset [1,K]$,  
\[
\cD_I:= \left\{\qq^a D_{i_1;\lambda_1}\cdots D_{i_k;\lambda_k}\ \middle|\  a\in \frac{1}{2d}\Z, \lambda_1,\dots, \lambda_k\in P_+\right\}
\]
forms an Ore set of $\cO_\qq(\Conf_K\sA_G)$, and we set 
\begin{align*}
&\cO_\qq(\Conf_K^I\sA_G):=\cO_\qq(\Conf_K\sA_G)[\cD_I^{-1}],\\
&\cO_\qq(\Conf_K^\times \sA_G):=\cO_\qq(\Conf_K\sA_G)[\cD_{[1,K]}^{-1}].
\end{align*}
We regard $\cO_\qq(\Conf_K^I\sA_G)$ as a quantum analogue of the subspace $\Conf_K^I\sA_G\subset \Conf_K\sA_G$ consisting of $G$-orbits $[\flA_1,\dots,\flA_K]$ such that the pairs  $(\flA_{i_l},\flA_{i_{l+1}})$, $l=1,\dots, k$ ($\flA_{K+1}:=\flA_1$) are generic. 
\subsection{Disjoint embeddings of quantum configuration spaces}
An order-preserving injective map $p\colon[1,K]\to[1,L]$ induces a $\Bbbk$-algebra embedding
\[
\iota_p^{\inv}\colon \cO_\qq(\Conf_K\sA_G)\longrightarrow \cO_\qq(\Conf_L\sA_G).
\]
We prove that if the polygons determined by order-preserving injective maps $p\colon[1,K]\to[1,L]$ and $p'\colon[1,K]\to[1,L]$ have disjoint interiors (see \cref{fig:intro_polygon_embedding}), then the images of $\iota_p^{\inv}$ and $\iota_{p'}^{\inv}$ $\qq$-commute. It is a consequence of the following theorem. 
\begin{figure}[ht]
    \centering
\begin{tikzpicture}[scale=0.78,every node/.style={transform shape}]
\def\R{2}
\draw[thick] (0,0) circle(\R);
\foreach \i in {1,2,3,4,5,6,7}{
    \fill(141.42-51.42*\i:\R) circle(1.5pt) coordinate(A\i);
    \node at (141.42-51.42*\i:\R+0.3) {\scriptsize \i};
}
\filldraw[fill=myblue!15,draw=blue] (A1) -- (A2) -- (A3) -- (A4) --cycle;
\node at ($(A1)!0.6!(A3)$) {$p$};
\filldraw[fill=myblue!15,draw=blue] (A1) -- (A4) -- (A5) -- (A6) -- (A7) --cycle;
\node at ($(A1)!0.6!(A5)$) {$p'$};

\begin{scope}[xshift=6cm]
\draw[thick] (0,0) circle(\R);
\foreach \i in {1,2,3,4,5,6,7}{
    \fill(141.42-51.42*\i:\R) circle(1.5pt) coordinate(A\i);
    \node at (141.42-51.42*\i:\R+0.3) {\scriptsize \i};
}
\filldraw[fill=myblue!15,draw=blue] (A1) -- (A2) -- (A3) -- (A4) --cycle;
\node at ($(A1)!0.6!(A3)$) {$p$};
\filldraw[fill=myblue!15,draw=blue] (A5) -- (A6) -- (A7) --cycle;
\node[left] at ($(A7)!0.5!(A5)$) {$p'$};
\end{scope}

\begin{scope}[xshift=12cm]
\draw[thick] (0,0) circle(\R);
\foreach \i in {1,2,3,4,5,6,7}{
    \fill(141.42-51.42*\i:\R) circle(1.5pt) coordinate(A\i);
    \node at (141.42-51.42*\i:\R+0.3) {\scriptsize \i};
}
\filldraw[fill=myblue!15,draw=blue] (A1) -- (A2) -- (A3) -- (A4) --cycle;
\node at ($(A1)!0.6!(A3)$) {$p$};
\filldraw[fill=myblue!15,draw=blue] (A4) -- (A5) -- (A6) -- (A7) --cycle;
\node at ($(A7)!0.5!(A5)$) {$p'$};
\end{scope}
\end{tikzpicture}
    \caption{The cases when the polygons determined by $p$ and $p'$ have disjoint interiors.}
    \label{fig:intro_polygon_embedding}
\end{figure}
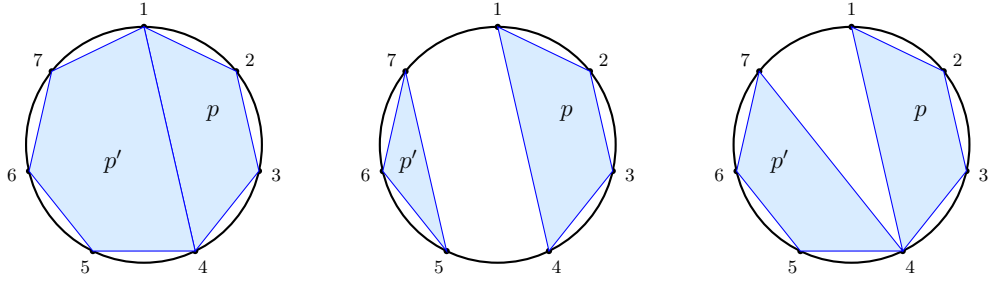
\begin{introthm}[{\cref{thm:edge}}]\label{introthm:disjoint}
Let $K, K'\in \Z_{\geq 2}$. Define $p\colon [1,K]\to[1,K+K'-2]$ and $p'\colon [1,K']\to [1,K+K'-2]$ by 
  \begin{align*}
      &p(i)=i\ \text{for}\ i\in [1, K],\\ 
      &p'(1)=1,\ p'(j)=K+j-2\ \text{for}\ j\in [2,K'].
  \end{align*}
 Then, for $\phi\in \cO^{\inv}(\lambda_1,\dots, \lambda_K)$ and $\phi'\in \cO^{\inv}(\lambda'_1,\dots, \lambda'_{K'})$, we have 
  \[
  \iota_{p}^{\inv}(\phi)\iota_{p'}^{\inv}(\phi')=\qq^{(\lambda_1, \lambda'_1)-(\lambda_K, \lambda'_2)}\iota_{p'}^{\inv}(\phi')\iota_p^{\inv}(\phi)
  \]
  in $\cO_\qq(\Conf_{K+K'-2}\sA_G)$.
\end{introthm}
This theorem is used when we ``amalgamate'' the quantum seeds for $\cO_\qq(\Conf_{3}\sA_G)$ to construct a quantum seed for $\cO_\qq(\Conf_{K}\sA_G)$. 

\subsection{Quantum Wilson lines}
For any marked surface $\bSigma$ and a homotopy class $[c]$ of an arc running between boundary intervals of $\bSigma$, we can associate a morphism
\begin{align*}
    g_{[c]}: \sA^\times_{G,\bSigma} \to G,
\end{align*}
using the holonomy of $G$-local system along $c$ and the normalizations at the boundary intervals where $c$ starts and ends. Here, $\sA^\times_{G,\bSigma} \subset \sA_{G,\bSigma}$ denotes the subspace such that the pair of decorations along each boundary interval is generic. See \cite{IOS} for details. 
When we consider an arc $c$ connecting boundary intervals of $\bSigma$ which do not share their endpoints, $g_{[c]}$ can be computed by restricting twisted $G$-local systems to a band neighborhood $B_c$ of $c$. This can be illustrated as follows.             
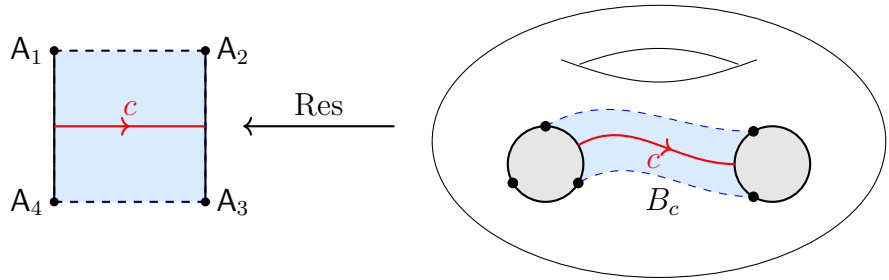
\[
\begin{tikzpicture}
\begin{scope}[rotate=90]
\draw[thick](-1,-1) -- (1,-1);
\draw[thick] (1,1) -- (-1,1);
\draw[dashed] (-1,-1) -- (-1,1);
\draw[dashed] (1,-1) -- (1,1);
\foreach \i in {(-1,-1),(1,-1),(1,1),(-1,1)} \filldraw \i circle(1.5pt);
\node[left] at (-1,1){$\flA_4$};
\node[left] at (1,1){$\flA_1$};
\node[right] at (1,-1){$\flA_2$};
\node[right] at (-1,-1){$\flA_3$};
\draw[red,thick,->-] (0,1)--node[midway,above]{$c$} (0,-1) ;
\end{scope}
\draw[thick,<-](1.5,0) --node[midway,above]{Res} ++(2,0);
\begin{pgfonlayer}{bg} 
\filldraw[fill=myblue!15,draw=blue,dashed,thick] (-1,-1) -- (1,-1) -- (1,1) -- (-1,1) --cycle;
\end{pgfonlayer} 
\begin{scope}[xshift=7cm]
\draw(0,-0.2) ellipse(3cm and 1.8cm);
\draw[shorten >=-15pt,shorten <=-15pt] (-0.8,0.7) to[bend right=20]
node[inner sep=0,pos=-0.15](A){} node[inner sep=0,pos=1.15](B){}
(0.8,0.7);
\draw(A) to[bend left=20] node[inner sep=0,pos=0.5](C){} (B);
\filldraw[thick,fill=gray!20] (-1.5,-0.5) circle(.5cm);
\filldraw[thick,fill=gray!20] (1.5,-0.5) circle(.5cm);
\foreach \i in {-30,90,210}
\fill(-1.5,-0.5)++(\i:0.5) circle(2pt);
\draw(-1.5,-0.5)++(90:0.5);
\draw(-1.5,-0.5)++(-30:0.5);
\foreach \i in {120,240}
\fill(1.5,-0.5)++(\i:0.5) circle(2pt);
\draw(1.5,-0.5)++(120:0.5);
\draw(1.5,-0.5)++(240:0.5);
\draw[red,thick,->-={0.6}{}] (-1.5,-0.5)++(30:0.5) to[out=30,in=180] node[midway,below]{$c$} ($(1.5,-0.5)+(180:0.5)$);
\begin{pgfonlayer}{bg}
    \filldraw[fill=myblue!15,draw=blue,dashed] (-1.5,-0.5)++(90:0.5) to[out=30,in=180] ($(1.5,-0.5)+(120:0.5)$) arc[start angle=120, end angle=240, radius=0.5] to[out=180,in=30] node[midway,below]{$B_c$} ($(-1.5,-0.5)+(-30:0.5)$) arc[start angle=-30, end angle=60, radius=0.5];
\end{pgfonlayer}
\end{scope}
\end{tikzpicture}
\]
Then the subspace $\Conf_4^{\{2,4\}}\sA_G\subset \Conf_4\sA_G$ naturally appears. 
We have an injective morphism 
\begin{align}
G\times H\times H \to \Conf_4\sA_G,\ (g, h, h')\mapsto [[U^+], gh. [U^+], g\overline{w_0}. [U^+], (h')^{-1}\overline{w_0}.[U^+]],\label{eq:classical_Wilson}
\end{align}
which induces an isomorphism of algebraic varieties 
\begin{align}
\imath_{4}\colon G\times H\times H\xrightarrow{\sim} \Conf_4^{\{2, 4\}}\sA_G.\label{eq:classical_Wilson_isom}
\end{align}
Then, the composite morphism 
\[
    \sA^\times_{G,\bSigma}\xrightarrow{\text{Res}}\Conf_4^{\{2, 4\}}\sA_G \xrightarrow{\imath_{4}^{-1}} G\times H\times H\xrightarrow{\text{projection}} G
\]
gives the Wilson line $g_{[c]}$ along $c$. In this paper, we construct a quantum analogue of \eqref{eq:classical_Wilson} and \eqref{eq:classical_Wilson_isom} as follows.

\begin{introthm}[\cref{thm:qW}]\label{introthm:qW}
There is an injective algebra homomorphism 
    \begin{align*}
    \cO_\qq(\Conf_4\sA_G)\to {}^{\ext'}\cO_\qq(G)^{\ext},
    \end{align*}
    which extends to an isomorphism
    \[
    \imath_{4,\qq}\colon\cO_\qq(\Conf_4^{\{2,4\}}\sA_G)\xrightarrow{\sim}{}^{\ext'}\cO_\qq(G)^{\ext}.
    \]
    Here ${}^{\ext'}\cO_\qq(G)^{\ext}:=\Bbbk[P]\otimes \A\otimes \Bbbk[P]=\bigoplus_{\lambda, \mu\in P}\Bbbk e'(\lambda)\otimes \cO_\qq(G)\otimes e(\mu)$ is an extension of $\cO_\qq(G)$ by the two copies of $\Bbbk[P]$, endowed with certain multiplication (see before \cref{t:Gext} for its precise definition). 
\end{introthm}

The \emph{quantum Wilson line} is the composite $\cO_\qq(G)\hookrightarrow {}^{\ext'}\cO_\qq(G)^{\ext} \xrightarrow{\imath_{4,\qq}^{-1}} \cO_\qq(\Conf_4^{\{2,4\}}\sA_G)$, where the first map is the inclusion.

\subsection{Quantum cluster structures for $\cO_\qq(\Conf_3\sA_G)$ and $\cO_\qq(\Conf_4\sA_G)$}
The cluster structures \cite{CAI,FG06} on $\Conf_K\sA_G$ are extensively studied by many authors, for example, \cite{FG06,Fei,Le19,Magee,GHKK,GS:DT,GS:Quantum,KellerLiu}. From a geometric viewpoint, constructing a cluster structure amounts to finding a collection of compatible toric charts that realizes the positive structure of the variety. We consider their quantum analogues in our terminology. We use the quantum Wilson line to establish quantum cluster structures \cite{BZ} on the quantum configuration spaces.

Let $\bbs$ be a double reduced word of $(w_0,w_0)\in W\times W$. The quantized coordinate ring $\cO_\qq(G)$ carries the Berenstein--Zelevinsky quantum seed $\bi(\bbs)$ associated with $\bbs$. 
By adding the frozen variables to $\bi(\bbs)$ (corresponding to the extension of $\cO_\qq(G)$ by the two copies of $\Bbbk[P]$) and applying a certain strict similarity transform  (\cref{def:coeffmod}), we obtain a quantum seed $\bi_{\square}(\bbs)$ for $\cO_\qq(\Conf_4\sA_G)$ via $\isigma:=\imath_{4,\qq}\circ\sigma_4^{-1}$. 
\begin{introthm}[{\cref{thm:cluster_square}, \cref{thm:GS_quantum}}]\label{introthm:cluster_square}
For an arbitrary double reduced word $\bbs$ of $(w_0,w_0)$, we can construct a quantum seed $\bi_{\square}(\bbs)$ in the skew field of fractions $\cF(\cO_\qq(\Conf_4^{\times}\sA_G))$ of $\cO_\qq(\Conf_4^{\times}\sA_G)$ such that \begin{align*}
A_\qq(\bi_{\square}(\bbs);J_{\square}(\bbs)^{\fr})
=U_\qq(\bi_{\square}(\bbs);J_{\square}(\bbs)^{\fr})
=\cO_\qq(\Conf_4^{\times}\sA_G).
\end{align*}
Here $A_\qq(\bi_{\square}(\bbs);J_{\square}(\bbs)^{\fr})$ denotes the quantum cluster algebra associated with $\bi_{\square}(\bbs)$ whose invertible frozen variables are indexed by $J_{\square}(\bbs)^{\fr}$, and $U_\qq(\bi_{\square}(\bbs);J_{\square}(\bbs)^{\fr})$ is its upper counterpart.
Moreover, the exchange matrix of $\bi_{\square}(\bbs)$ coincides with that of the Goncharov--Shen seed $\bi^\GS(\bbs)$. If $\bbs$ and $\bbs'$ are double reduced words of $(w_0, w_0)$, then $\bi_{\square}(\bbs)$ and $\bi_{\square}(\bbs')$ are mutation equivalent.
\end{introthm}
Let us make one technical remark. 
In the construction of $\bi_{\square}(\bbs)$, we transport the quantum seed for ${}^{\ext'}\cO_\qq(G)^{\ext}$ via $\isigma=\imath_{4,\qq}\circ\sigma_4^{-1}$ rather than $\imath_{4,\qq}$. At the classical level, $\isigma$ corresponds to the left picture, while $\imath_{4,\qq}$ corresponds to the right one. 

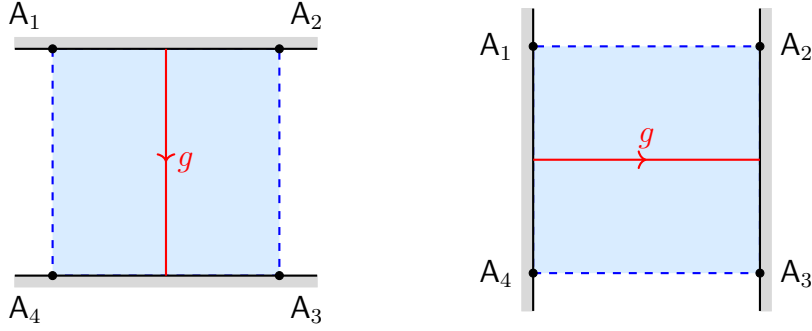
\[
\begin{tikzpicture}
\path (3,3) coordinate(A);
\path (0,3) coordinate(B);
\path (0,0) coordinate(C);
\path (3,0) coordinate(D);
\bline{-0.5,0}{3.5,0}{0.15};
\tline{-0.5,3}{3.5,3}{0.15};
\draw[red,->-,thick] (1.5,3) --node[midway,right]{$g$} (1.5,0);
\filldraw(A) circle(1.5pt) node[above=0.3em,anchor=south west]{$\flA_2$}; 
\filldraw(B) circle(1.5pt) node[above=0.3em,anchor=south east]{$\flA_1$};
\filldraw(D) circle(1.5pt) node[below=0.3em,anchor=north west]{$\flA_3$};
\filldraw(C) circle(1.5pt) node[below=0.3em,anchor=north east]{$\flA_4$};
\begin{pgfonlayer}{bg}  
\filldraw[fill=myblue!15,draw=blue,dashed,thick] (B) -- (C) -- (D) -- (A) --cycle;
\end{pgfonlayer}
\end{tikzpicture}
\qquad \qquad 
\raisebox{8pt}{
\begin{tikzpicture}[rotate=90]
\path (3,3) coordinate(A);
\path (0,3) coordinate(B);
\path (0,0) coordinate(C);
\path (3,0) coordinate(D);
\bline{-0.5,0}{3.5,0}{0.15};
\tline{-0.5,3}{3.5,3}{0.15};
\draw[red,->-,thick] (1.5,3) --node[midway,above]{$g$} (1.5,0);
\filldraw(A) circle(1.5pt) node[left=0.3em,anchor=east]{$\flA_1$}; 
\filldraw(B) circle(1.5pt) node[left=0.3em,anchor=east]{$\flA_4$};
\filldraw(D) circle(1.5pt) node[right=0.3em,anchor=west]{$\flA_2$};
\filldraw(C) circle(1.5pt) node[right=0.3em,anchor=west]{$\flA_3$};
\begin{pgfonlayer}{bg}  
\filldraw[fill=myblue!15,draw=blue,dashed,thick] (B) -- (C) -- (D) -- (A) --cycle;
\end{pgfonlayer}
\end{tikzpicture}}
\]
We adopt the former convention since it is more compatible with the amalgamation of cluster structures when constructing a quantum seed for $\cO_\qq(\Conf_K\sA_G)$. 
For the study of quantum cluster structure on $\cO_\qq(\Conf_K\sA_G)$, we need a refinement of \cref{introthm:cluster_square} in which only some of the frozen variables are inverted.
\begin{introthm}[{\cref{thm:cluster_square_cpt}}]\label{introthm:cluster_square_cpt}
For an arbitrary double reduced word $\bbs$ of $(w_0,w_0)$, 
\begin{align*}
A_\qq(\bi_{\square}(\bbs);S_{-\infty}\sqcup S_{\infty})
=U_\qq(\bi_{\square}(\bbs);S_{-\infty}\sqcup S_{\infty})
=\cO_\qq(\Conf_4^{\{1,3\}}\sA_G).
\end{align*}
Here $S_{-\infty}\sqcup S_{\infty}$ is a specific proper subset of $J_{\square}(\bbs)^{\fr}$. 
\end{introthm}
The key input for the proof of \cref{introthm:cluster_square_cpt} is the equalities 
\[
A_\qq(\bi(\bbs);\varnothing)=U_\qq(\bi(\bbs);\varnothing)=\cO_\qq(G), 
\]
recently shown in \cite[Theorem B]{QY}, \cite[Theorem A]{OQY}, and \cite[Corollary B]{Dey}. 

Next, we define a quantum seed for $\cO_\qq(\Conf_3^{\times}\sA_G)$ by restricting the quantum seed for $\cO_\qq(\Conf_4^{\times}\sA_G)$. Let $\bs$ be a reduced word of $w_0$. We can complete it to a double reduced word $\bbs=\bs\ast\overline{\bs'}$ of $(w_0,w_0)$ by adjoining another reduced word $\overline{\bs'}$ of $w_0$, written in barred letters. By selecting the quantum cluster variables of $\bi_{\square}(\bbs)$ whose second tensor component is equal to $1$, we obtain a ``subseed'' of $\bi_{\square}(\bbs)$. It induces a quantum seed for $\cO_\qq(\Conf_3^{\times}\sA_G)$ that does not depend on the choice of $\overline{\bs'}$.

\begin{introthm}[{\cref{thm:cluster_triangle}}]\label{introthm:cluster_triangle}
For an arbitrary reduced word $\bs$ of $w_0$, we can construct a quantum seed $\bi_{\triangle}(\bs)$ in the skew field of fractions $\cF(\cO_\qq(\Conf_3^{\times}\sA_G))$ of $\cO_\qq(\Conf_3^{\times}\sA_G)$, and 
\begin{align*}
A_\qq(\bi_{\triangle}(\bs);J_{\triangle}(\bs)^{\fr})
=U_\qq(\bi_{\triangle}(\bs);J_{\triangle}(\bs)^{\fr})
=\cO_\qq(\Conf_3^{\times}\sA_G).
\end{align*}
Moreover, if $\bs$ and $\bs'$ are reduced words of $w_0$, then $\bi_{\triangle}(\bs)$ and $\bi_{\triangle}(\bs')$ are mutation equivalent.
\end{introthm}

Conversely, we can show that the quantum seed for $\cO_\qq(\Conf_4\sA_G)$ can be obtained by ``amalgamating'' two quantum seeds for $\cO_\qq(\Conf_3\sA_G)$.

\begin{introthm}[{Informal: \cref{thm:triangle_square_1,thm:triangle_square_2} for precise statements}]\label{introthm:triangle_square}
Let $\bs_1, \bs_2$ be reduced words of $w_0$. Then the quantum seed $\bi_{\square}(\bs_1\ast(\overline{\bs}_2^\ast)^\op)$ (resp.~$\bi_{\square}((\overline{\bs}_2^\ast)^\op\ast \bs_1)$) is an amalgamation of $\bi_{\triangle}(\bs_1)$ and $\bi_{\triangle}(\bs_2)$ (resp.~$\bi_{\triangle}(\bs_2)$ and $\bi_{\triangle}(\bs_1)$) in this order. Here, $\op$ stands for the operation of reversing the order of a word, and the superscript $*$ denotes the Dynkin involution (see \cref{sec:Lie_theory}).
\end{introthm}
This theorem can easily be seen at the level of the exchange matrices of the quantum seeds: see \cref{itrofig:triangle_square_1,introfig:triangle_square_2}, where they are represented in terms of quivers (cf.~\cref{sec:quiver}). We  formulate this amalgamation by also taking the quantum cluster variables into account, which constitutes the nontrivial part of \cref{introthm:triangle_square}.

\begin{figure}[ht]
    \centering
\begin{tikzpicture}
\def\R{3}
\coordinate(A) at (0,0);
\coordinate(B) at (1.2*\R,0);
\coordinate(C) at (0.6*\R,\R);
\node[scale=0.7] at (0.6*\R,0.9*\R) {$\bullet$};
\draw[thick,right hook->] (1.2*\R,0.5*\R) --node[midway,above]{$\iota_{43}^{1\bullet}$ or $\iota_{\bullet 3}^{12}$} ++(2,0);
\draw[thick] (B) -- (C);
\draw[thick] (A) -- (B) -- (C) --cycle;

\vertexA{C}{B}{A}{3};
\foreach \i in {1,2,3}{
    \fnode{a\i}{black};
    \fnode{b\i}{black};
    \fnode{V\i 0}{myblue};
    }
\foreach \i in {11,21,12}
\draw(V\i) circle(2pt);
\qarrow{a3}{b3};
\qarrow{a2}{V12}; \qarrow{V12}{b2};
\qarrow{a1}{V11}; \qarrow{V11}{V21}; \qarrow{V21}{b1};
\qarrow{b2}{V21};
\qarrow{b3}{V12}; 
\qarrow{V12}{V11};
\qarrow{V21}{V12};
\qarrow{V12}{a3};
\qarrow{V11}{a2};
{\color{myblue}
\qarrow{V10}{a1};
\qarrow{V11}{V10};
\qarrow{V20}{V11};
\qarrow{V21}{V20};
\qarrow{V30}{V21};
\qarrow{b1}{V30};
}
\begin{scope}[xshift=2.2*\R cm]
\coordinate(A) at (0,0);
\coordinate(B) at (\R,0);
\coordinate(C) at (0,\R);
\coordinate(D) at (\R,\R);
\node[below left] at ($(A)$) {$4$};
\node[below right] at ($(B)$) {$3$};
\node[above left] at ($(C)$) {$1$};
\node[above right] at ($(D)$) {$2$};
\node[scale=0.7] at (0.05*\R,0.86*\R) {$\bullet$};
\node[scale=0.7] at (0.95*\R,0.14*\R) {$\bullet$};

\draw[thick,dashed] (B) -- (C);
\draw[thick] (A) -- (B) -- (D) -- (C) --cycle;

\vertexA{C}{B}{A}{3};
\foreach \i in {1,2,3}{
    \filldraw[fill=white](a\i) circle(2pt);
    \fnode{b\i}{black};
    \fnode{V\i 0}{myblue};
    }
\foreach \i in {11,21,12}
\draw(V\i) circle(2pt);
\qarrow{a3}{b3};
\qarrow{a2}{V12}; \qarrow{V12}{b2};
\qarrow{a1}{V11}; \qarrow{V11}{V21}; \qarrow{V21}{b1};
\qarrow{b2}{V21};
\qarrow{b3}{V12}; 
\qarrow{V12}{V11};
\qarrow{V21}{V12};
\qarrow{V12}{a3};
\qarrow{V11}{a2};
{\color{myblue}
\qarrow{V10}{a1};
\qarrow{V11}{V10};
\qarrow{V20}{V11};
\qarrow{V21}{V20};
\qarrow{V30}{V21};
\qarrow{b1}{V30};
}

\vertexA{B}{C}{D}{3};
\foreach \i in {1,2,3}{
    \fnode{b\i}{black};
    \fnode{V\i 0}{myblue};
    }
\foreach \i in {11,21,12}
\draw(V\i) circle(2pt);
\qarrow{a3}{b3};
\qarrow{a2}{V12}; \qarrow{V12}{b2};
\qarrow{a1}{V11}; \qarrow{V11}{V21}; \qarrow{V21}{b1};
\qarrow{b2}{V21};
\qarrow{b3}{V12}; 
\qarrow{V12}{V11};
\qarrow{V21}{V12};
\qarrow{V12}{a3};
\qarrow{V11}{a2};
{\color{myblue}
\qarrow{V10}{a1};
\qarrow{V11}{V10};
\qarrow{V20}{V11};
\qarrow{V21}{V20};
\qarrow{V30}{V21};
\qarrow{b1}{V30};
}
\end{scope}
\end{tikzpicture}
    \caption{Type $A_3$, $\bs=(1,2,3,1,2,1)$, $\bbs=(1,2,3,1,2,1,\overline{3},\overline{2},\overline{3},\overline{1},\overline{2},\overline{3})$. The quantum seed $\bi_{\square}(\bbs)$ can be obtained as an amalgamation of two copies of $\bi_{\triangle}(\bs)$.}
    \label{itrofig:triangle_square_1}
\end{figure}
\begin{figure}[ht]
    \centering
\begin{tikzpicture}
\def\R{3}
\coordinate(A) at (0,0);
\coordinate(B) at (1.2*\R,0);
\coordinate(C) at (0.6*\R,\R);
\node[scale=0.7] at (0.6*\R,0.9*\R) {$\bullet$};
\draw[thick,right hook->] (1.2*\R,0.5*\R) --node[midway,above]{$\iota_{4\bullet}^{12}$ or $\iota_{43}^{\bullet 2}$} ++(2,0);
\draw[thick] (B) -- (C);
\draw[thick] (A) -- (B) -- (C) --cycle;

\vertexA{C}{B}{A}{3};
\foreach \i in {1,2,3}{
    \fnode{a\i}{black};
    \fnode{b\i}{black};
    \fnode{V\i 0}{myblue};
    }
\foreach \i in {11,21,12}
\draw(V\i) circle(2pt);
\qarrow{a3}{b3};
\qarrow{a2}{V12}; \qarrow{V12}{b2};
\qarrow{a1}{V11}; \qarrow{V11}{V21}; \qarrow{V21}{b1};
\qarrow{b2}{V21};
\qarrow{b3}{V12}; 
\qarrow{V12}{V11};
\qarrow{V21}{V12};
\qarrow{V12}{a3};
\qarrow{V11}{a2};
{\color{myblue}
\qarrow{V10}{a1};
\qarrow{V11}{V10};
\qarrow{V20}{V11};
\qarrow{V21}{V20};
\qarrow{V30}{V21};
\qarrow{b1}{V30};
}
\begin{scope}[xshift=2.2*\R cm]
\coordinate(A) at (0,0);
\coordinate(B) at (\R,0);
\coordinate(C) at (0,\R);
\coordinate(D) at (\R,\R);
\node[below left] at ($(A)$) {$4$};
\node[below right] at ($(B)$) {$3$};
\node[above left] at ($(C)$) {$1$};
\node[above right] at ($(D)$) {$2$};
\node[scale=0.7] at (0.05*\R,0.14*\R) {$\bullet$};
\node[scale=0.7] at (0.95*\R,0.86*\R) {$\bullet$};

\draw[thick,dashed] (D) -- (A);
\draw[thick] (A) -- (B) -- (D) -- (C) --cycle;

\vertexA{D}{B}{A}{3};
\foreach \i in {1,2,3}{
    \filldraw[fill=white](b\i) circle(2pt);
    \fnode{a\i}{black};
    \fnode{V\i 0}{myblue};
    }
\foreach \i in {11,21,12}
\draw(V\i) circle(2pt);
\qarrow{a3}{b3};
\qarrow{a2}{V12}; \qarrow{V12}{b2};
\qarrow{a1}{V11}; \qarrow{V11}{V21}; \qarrow{V21}{b1};
\qarrow{b2}{V21};
\qarrow{b3}{V12}; 
\qarrow{V12}{V11};
\qarrow{V21}{V12};
\qarrow{V12}{a3};
\qarrow{V11}{a2};
{\color{myblue}
\qarrow{V10}{a1};
\qarrow{V11}{V10};
\qarrow{V20}{V11};
\qarrow{V21}{V20};
\qarrow{V30}{V21};
\qarrow{b1}{V30};
}

\vertexA{A}{C}{D}{3};
\foreach \i in {1,2,3}{
    \fnode{a\i}{black};
    \filldraw[fill=white](b\i) circle(2pt);
    \fnode{V\i 0}{myblue};
    }
\foreach \i in {11,21,12}
\draw(V\i) circle(2pt);
\qarrow{a3}{b3};
\qarrow{a2}{V12}; \qarrow{V12}{b2};
\qarrow{a1}{V11}; \qarrow{V11}{V21}; \qarrow{V21}{b1};
\qarrow{b2}{V21};
\qarrow{b3}{V12}; 
\qarrow{V12}{V11};
\qarrow{V21}{V12};
\qarrow{V12}{a3};
\qarrow{V11}{a2};
{\color{myblue}
\qarrow{V10}{a1};
\qarrow{V11}{V10};
\qarrow{V20}{V11};
\qarrow{V21}{V20};
\qarrow{V30}{V21};
\qarrow{b1}{V30};
}
\end{scope}
\end{tikzpicture}
    \caption{Type $A_3$, $\bs=(1,2,3,1,2,1)$, $\bbs=(\overline{3},\overline{2},\overline{3},\overline{1},\overline{2},\overline{3}, 1,2,3,1,2,1)$. The quantum seed $\bi_{\square}(\bbs)$ can be obtained as an amalgamation of two copies of $\bi_{\triangle}(\bs)$.}
    \label{introfig:triangle_square_2}
\end{figure}
\subsection{Quantum cluster structures for $\cO_\qq(\Conf_{K+2}\sA_G)$, $K\geq 3$}
The preceding results suggest a local-to-global construction of quantum seeds for $\cO_\qq(\Conf_{K+2}\sA_G)$, $K\geq 3$. The basic idea is that every cluster variable should be supported on a triangle, whereas every exchange relation should be visible in a quadrilateral.  Choose reduced words $\bs_1,\dots,\bs_K$ of $w_0$, and let 
\begin{align*}
    \tbbs&:=\bs_1\ast (\overline{\bs}_2^{\ast})^\op\ast \bs_3\ast (\overline{\bs}_4^\ast)^\op\ast \cdots\\
    &=
\begin{cases}
\bs_1\ast (\overline{\bs}_2^\ast)^\op\ast \bs_3\ast (\overline{\bs}_4^\ast)^\op\ast \cdots \ast (\overline{\bs}_K^\ast)^\op&\text{when $K$ is even},\\
\bs_1\ast (\overline{\bs}_2^\ast)^\op\ast \bs_3\ast (\overline{\bs}_4^\ast)^\op\ast \cdots \ast \bs_K&\text{when $K$ is odd}.
\end{cases}
\end{align*}
Using the triangle embeddings and their amalgamations, we construct a family of mutually $\qq$-commuting elements
\[
\left(A_{\bbD_{K+2},\tbbs;(s,m)}\right)_{(s,m)\in J_{\bbD_{K+2}}(\tbbs)}
\subset \cO_\qq(\Conf_{K+2}^{\times}\sA_G),
\]
and a $\Z$-valued $J_{\bbD_{K+2}}(\tbbs)^\uf\times J_{\bbD_{K+2}}(\tbbs)$-matrix $\cE_{\bbD_{K+2}}(\tbbs)$. This construction can be illustrated as follows. 
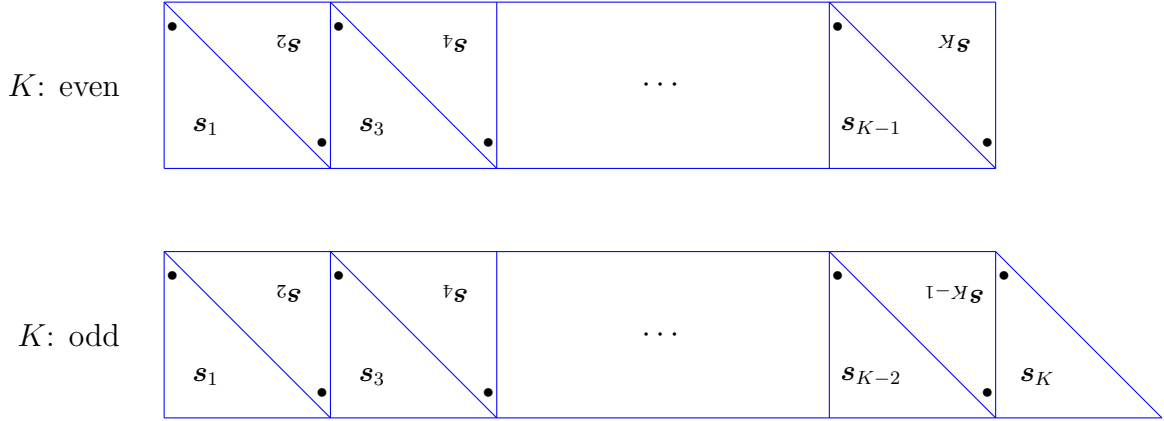
\[
\begin{tikzpicture}[scale=2.2]
\node[anchor=east] at (-0.2,0.5){$K$: even};
{\color{blue}
  \draw (0,1) -- (5,1);
  \draw (0,0) -- (5,0);
  \draw (0,0) -- (0,1);
  \draw (5,0) -- (5,1);
  \draw (1,0) -- (1,1);
  \draw (2,0) -- (2,1);
  \draw (4,0) -- (4,1);
  \draw (0,1) -- (1,0);
  \draw (1,1) -- (2,0);
  \draw (4,1) -- (5,0);
}
\node at (3,0.5) {$\cdots$};
\foreach \i in {0,1,4} \node[scale=0.7] at (0.05+\i,0.85){$\bullet$};
\foreach \i in {1,2,5} \node[scale=0.7] at (-0.05+\i,0.15){$\bullet$};
\node[scale=0.9] at (0.25,0.25){$\bs_1$};
\node[scale=0.9] at (1.25,0.25){$\bs_3$};
\node[scale=0.9] at (4.25,0.25){$\bs_{K-1}$};
\node[scale=0.9,rotate=180] at (0.75,0.75){$\bs_2$};
\node[scale=0.9,rotate=180] at (1.75,0.75){$\bs_4$};
\node[scale=0.9,rotate=180] at (4.75,0.75){$\bs_K$};

\begin{scope}[yshift=-1.5cm]
\node[anchor=east] at (-0.2,0.5){$K$: odd};
{\color{blue}
  \draw (0,1) -- (5,1);
  \draw (0,0) -- (6,0);
  \draw (0,0) -- (0,1);
  \draw (5,0) -- (5,1);
  \draw (1,0) -- (1,1);
  \draw (2,0) -- (2,1);
  \draw (4,0) -- (4,1);
  \draw (0,1) -- (1,0);
  \draw (1,1) -- (2,0);
  \draw (4,1) -- (5,0);
  \draw (5,1) -- (6,0);
}
\node at (3,0.5) {$\cdots$};
\foreach \i in {0,1,4,5} \node[scale=0.7] at (0.05+\i,0.85){$\bullet$};
\foreach \i in {1,2,5} \node[scale=0.7] at (-0.05+\i,0.15){$\bullet$};
\node[scale=0.9] at (0.25,0.25){$\bs_1$};
\node[scale=0.9] at (1.25,0.25){$\bs_3$};
\node[scale=0.9] at (4.25,0.25){$\bs_{K-2}$};
\node[scale=0.9] at (5.25,0.25){$\bs_K$};
\node[scale=0.9,rotate=180] at (0.75,0.75){$\bs_2$};
\node[scale=0.9,rotate=180] at (1.75,0.75){$\bs_4$};
\node[scale=0.9,rotate=180] at (4.75,0.75){$\bs_{K-1}$};
\end{scope}
\end{tikzpicture}
\]

We note that \cref{introthm:triangle_square} ensures that the quantum cluster variables on the diagonals are well-defined, and \cref{introthm:disjoint} controls the commutation relation of variables supported on different pieces. 
Since the variables $A_{\bbD_{K+2},\tbbs;(s,m)}$, $(s,m)\in J_{\bbD_{K+2}}(\tbbs)$,  $\qq$-commute, we have an algebra homomorphism 
\[
M_{\tbbs}\colon
\cT(\Lambda_{\bbD_{K+2}}(\tbbs))
\to
\cF(\cO_\qq(\Conf_{K+2}^{\times}\sA_G)).
\]
from the quantum torus $\cT(\Lambda_{\bbD_{K+2}}(\tbbs))$ (see \cref{ssec:Qtorus}) to the skew field of fractions $\cF(\cO_\qq(\Conf_{K+2}^{\times}\sA_G))$ of $\cO_\qq(\Conf_{K+2}^{\times}\sA_G)$. 
The main theorem is the following. 
\begin{introthm}[{\cref{thm:cluster_polygon}}]\label{introthm:cluster_polygon}
The collection
\[
\bi_{\bbD_{K+2}}(\tbbs)
=\left(\Lambda_{\bbD_{K+2}}(\tbbs),
\cE_{\bbD_{K+2}}(\tbbs),
\left(A_{\bbD_{K+2},\tbbs;(s,m)}\right)_{(s,m)\in J_{\bbD_{K+2}}(\tbbs)}\right)
\]
forms a quantum seed in $\cF(\cO_\qq(\Conf_{K+2}^{\times}\sA_G))$, and
\begin{align*}\label{eq:intro_cluster_polygon}
\cO_\qq(\Conf_{K+2}^{\times}\sA_G)
\subset
A_\qq(\bi_{\bbD_{K+2}}(\tbbs);J_{\bbD_{K+2}}(\tbbs)^{\fr}).
\end{align*}
\end{introthm}
We remark that \cref{introthm:cluster_square_cpt}, a refined version of \cref{introthm:cluster_square}, is essential in the proof of \cref{introthm:cluster_polygon}, since $\cO_\qq(\Conf_4^{\{1,3\}}\sA_G)$, rather than $\cO_\qq(\Conf_4^{\times}\sA_G)$, appears as a local piece of $\cO_\qq(\Conf_{K+2}^{\times}\sA_G)$.

\subsection{Future directions}
In this paper, we introduce the quantum configuration spaces of decorated flags and initiate their algebraic study. It seems worthwhile to investigate their precise relationships with other quantum algebras. Here, we highlight two directions in particular. See \cref{fig:future_direction}. 
\begin{itemize}
    \item[(A)] \textbf{Relation to the skein algebra:} In an ongoing work of the authors with Wataru Yuasa, based on the Reshetikhin--Turaev functor, we will introduce a \emph{clasped $G$-skein algebra} $\mathscr{S}_{G,\bSigma}$ a skein-theoretic quantization of the moduli space $\sA_{G,\bSigma}$ for any complex semisimple Lie group $G$ and a marked surface $\bSigma$ without punctures. It generalizes Muller's skein algebra \cite{Muller} for $G=SL_2$, clasped skein algebras \cite{IYsl3,IYsp4} for $G=SL_3,Sp_4$,
    and would be closely related to the projected stated skein algebra for $G=SL_n$ whose quantum cluster structure for the case $\bSigma=\boldsymbol{D}_K$ is established by Cao, Huang and Wang \cite{HW,CHW}.
    We will prove that for $\bSigma=\boldsymbol{D}_K$, our skein algebra $\mathscr{S}_{G,\boldsymbol{D}_K}$ is isomorphic to the quantum configuration space $\cO_\qq(\Conf_K \sA_G)$. 
    This relation would provides, a (conjectural for $K \geq 5$) quantum cluster structure on $\mathscr{S}_{G,\bSigma}$ after an appropriate localization, as well as graphical intuition for the elements of the quantum configuration space. 
    \item[(B)] \textbf{Relation to the quantum moduli algebra:} The \emph{quantum moduli algebra} has been introduced as a combinatorial quantization of the Chern--Simons theory associated with any punctured surface and extensively studied by Alekseev, Grosse, Schomerus \cite{AGS1,AGS2,AGS3}, Buffenoir and Roche \cite{BR1,BR2}. The root-of-unity case is studied by Baseilhac and Roche \cite{BaseilhacRoche1,BaseilhacRoche2}. The quantum moduli algebra is known to quantize $\mathrm{Loc}_{G,\Sigma}$, and Faitg and Baseilhac--Faitg--Roche \cite{Faitg,BFR} recently established its isomorphism with the $\mathfrak{g}$-skein algebra defined using the Reshetikhin--Turaev functor. We expect that our construction could be regarded as a ``clasped'' version of the quantum moduli algebra, after an appropriate generalization to marked surfaces. This relation would provides a quantum cluster structure on the quantum moduli algebra, and would also clarify the relationship between the cluster Poisson structure and the Fock--Rosly description of the Atiyah--Bott--Goldman Poisson structure at the classical limit.
\end{itemize}

\begin{figure}[ht]
    \centering
\begin{tikzpicture}
\node[draw,rectangle,rounded corners=5pt](A) at (0,0) {Quantum configuration space};
\node[draw,rectangle,rounded corners=5pt](B) at (-5.5,-3) {Clasped skein algebra};
\node[draw,rectangle,rounded corners=5pt](C) at (5.5,-3) {Clasped quantum moduli algebra};
\draw(A) --node[left=0.3em]{(A)} (B);
\draw(A) --node[right=0.3em]{(B)} (C);
\draw(B) --node[below]{Clasped version of \cite{BFR}} (C);
\end{tikzpicture}
    \caption{Expected relations among three quantization frameworks.}
    \label{fig:future_direction}
\end{figure}
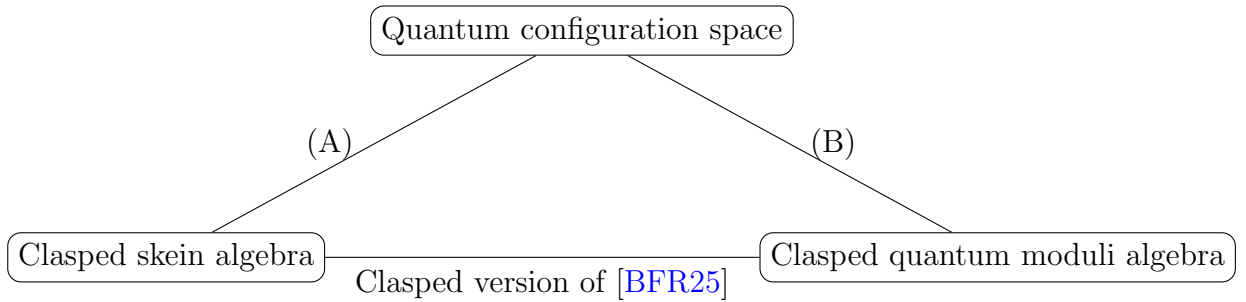

Other possible directions include the following. As shown in \cref{introthm:triangle_square} and \cref{introthm:cluster_polygon}, $\cO_\qq(\Conf_{K+2}\sA_G)$ can be viewed as an ``amalgamation'' of copies of $\cO_\qq(\Conf_{3}\sA_G)$. This is reminiscent of the structure of the quantum Grothendieck rings of monoidal categories of finite-dimensional modules over quantum loop algebras, introduced in \cite{Nakajima,VV,Her:qt}. Indeed, they can be viewed as an ``amalgamation'' of infinitely many copies of the quantum coordinate ring of $U^+$. This structure was revealed in \cite{HL:QGro,FHOO1} and has recently been generalized in the form of quantum virtual Grothendieck rings \cite{JLO} and bosonic extensions of quantum groups \cite{OP,KKOP:bosonic}; see also \cite{Contu}. It would be interesting to investigate the precise relationship between these algebras and the quantum configuration spaces.

\subsection{General notation}
\begin{itemize}
    \item[(1)] For integers $a, b\in \Z$ with $a\leq b$, we set $[a, b]:=\{m\in \Z\mid a\leq m\leq b\}$. 
    \item[(2)] For a vector space $V$, its dual space is denoted by $V^{\ast}$. For $v\in V$ and $\xi\in V^{\ast}$, the pairing of $v$ and $\xi$ is denoted by $\langle \xi, v\rangle$ or $\langle v, \xi\rangle$. 
    \item[(3)] In this paper, a module $M$ over an associative algebra $R$ always stands for a left module. The action of $R$ on $M$ is denoted by $r.m$ for $r\in R$ and $m\in M$. 
    \item[(4)]  \emph{An Ore set} $\mathcal{S}$ of an associative algebra $R$ stands for a left and right Ore set consisting of non-zero divisors. Denote by $R[\mathcal{S}^{-1}]$ the algebra of fractions with respect to the Ore set $\mathcal{S}$. In this case, $R$ is naturally regarded as a subalgebra of $R[\mathcal{S}^{-1}]$. See \cite[Chapter 6]{GW}.
    \item[(5)] An associative algebra $R$ is called \emph{an Ore domain} if $R$ is a domain and $R\setminus \{0\}$ forms an Ore set of $R$. In this case, $R[(R\setminus\{0\})^{-1}]$ is called \emph{the skew field of fractions of $R$}, and it is denoted by $\cF(R)$. 
    \item[(6)]Let $R, S$ be associative algebras, and $\mathcal{S}$ an Ore set of $R$. If there exists an algebra homomorphism $f:R\to S$ such that $f(s)$ is invertible in $S$ for all $s\in \mathcal{S}$, we can extend $f$ uniquely to the algebra homomorphism $R[\mathcal{S}^{-1}]\to S$. This extended homomorphism will be also denoted by $f$ in this paper. 
    \item[(7)] Let $R$ be an associative algebra. For elements $r_j\in R$ indexed by a totally ordered finite set $J=\{j_{0}<j_1<\cdots <j_k\}$, write
    \[
    \dprod_{j\in J}r_j:=r_{j_{0}}r_{j_{1}}\cdots r_{j_{k}}\in R.
    \]
\end{itemize}

\subsection*{Acknowledgements}
The authors are grateful to Fan Qin, Linhui Shen, Milen Yakimov, and Wataru Yuasa for numerous helpful discussions and comments. 
T. I. was supported by JSPS KAKENHI Grant Numbers~JP20K22304 and JP24K16914.
H. O. was supported by JSPS KAKENHI Grant-in-Aid for Early-Career Scientists, Grant Number JP23K12950.
\subsection*{Declaration on the Use of Generative AI}
The authors used a generative AI (ChatGPT, provided by OpenAI) for English-language editing to improve the readability of this manuscript. 
After the initial version of this paper was posted on arXiv, Conjecture 9.1 of that version was proved with assistance from GPT-6 Astra Pro model in ChatGPT. Specifically, this assistance concerned the proof that $\cO_\qq(\sA_G)^{\bt K}$ is a domain (\cref{prop:domain}) and the proof of the injectivity of $M_{\tbbs}$ in the proof of \cref{thm:cluster_polygon}. In particular, for the latter, the assistance concerned the application of \cite[Lemma 2.2 (b)]{LeYu} and the argument for estimating the Gelfand--Kirillov dimension of $\cO_\qq(\Conf_{K+2}^{\times}\sA_G)[M_{\tbbs}(\cM)^{-1}]$. The proofs were reconstructed and finalized by the authors. The authors take full responsibility for the mathematical arguments, references, and all other content of this paper.

\section{Quantized enveloping algebras and quantum coordinate rings}\label{s:QEA}
In this section, we review the basics of quantized enveloping algebras. 

\subsection{Complex simple Lie algebras}\label{sec:Lie_theory}
Let $\lieg$ be a complex finite dimensional simple Lie algebra associated with a Cartan matrix $C=(c_{st})_{s, t\in S}$.  Then $\lieg$ is the complex Lie algebra generated by $\{e_s, f_s, \alpha^{\vee}_s\mid s\in S\}$ subject to the following relations: 
\begin{itemize}
	\item[(i)] $[\alpha^{\vee}_s, \alpha^{\vee}_t]=0$,
	\item[(ii)] $[\alpha^{\vee}_s, e_{t}]=c_{st}e_{t}$, $[\alpha^{\vee}_s, f_{t}]=-c_{st}f_{t}$, 
	\item[(iii)] $[e_{s},f_{t}]=\delta_{st}\alpha^{\vee}_s$, 
	\item[(iv)] ${(\ad e_{s})^{1-c_{st}}(e_{t})=0}$ and ${(\mathrm{ad} f_{s})^{1-c_{st}}(f_{t})=0}$ for $s\neq t$. 
    
    Here, $(\ad x)(y):=[x, y]$ for $x, y\in \lieg$.
\end{itemize} 
Let $\lieh:=\sum_{s\in S}\mathbb{C}\alpha_s^{\vee}$ be a Cartan subalgebra of $\lieg$, and set 
\begin{align*}
&P:=\{\mu\in \lieh^{\ast}\mid \langle\alpha_s^{\vee}, \mu\rangle\in \mathbb{Z}\text{ for all } s\in S\},\\
&P_+:=\{\lambda\in P\mid \langle\alpha_s^{\vee}, \lambda \rangle\geq 0\text{ for all } s\in S\}, 
\end{align*}
called the weight lattice and the set of integral dominant weights, respectively. Denote by $\varpi_s\in \lieh^{\ast}$ (resp.~$\varpi_s^{\vee}\in \lieh$ ) the fundamental weight (resp.~the fundamental coweight) corresponding to $s\in S$, and set $\rho:=\sum_{s\in S}\varpi_s$ (resp.~$\rho^{\vee}:=\sum_{s\in S}\varpi_s^{\vee}$). Denote by $\Pi=\{\alpha_s\}_{s\in S}$, $\Pi^{\vee}=\{\alpha_s^{\vee}\}_{s\in S}$, $\Phi$, $\Phi^{\vee}$, $\Phi_+$, and $\Phi^{\vee}_+$ the set of simple roots, simple coroots, roots, coroots, positive roots, and positive coroots, respectively. 
We have 
\begin{align*}
    C_{st}=\langle \alpha_s^\vee,\alpha_t \rangle \in \Z, \quad \alpha_t = \sum_{u \in S}C_{ut}\varpi_u.
\end{align*}
We regard $\Pi, \Phi, \Phi_+$ as subsets of $P$. Set $Q:=\sum_{s\in S}\Z \alpha_{s}$, $Q_+:=\sum_{s\in S}\Z_{\geq 0} \alpha_{s}$. There is a partial ordering on $P$ defined by $\mu\leq \lambda$ if and only if $\lambda-\mu\in Q_+$. We write $\mu <\lambda$ if $\mu\leq \lambda$ and $\mu\neq \lambda$. 

The root space decomposition of $\lieg$ is written as 
\begin{align*}
	\lieg=\lieh\oplus \bigoplus_{\beta\in \Phi} \lieg_{\beta},\ \lieg_{\beta}:=\{x\in \lieg\mid [h, x]=\langle h, \beta\rangle x \text{ for }h\in \lieh\}. 
\end{align*}
Consider a $\Z$-bilinear form $(-,-)\colon P\times P\to \Q$ satisfying 
\begin{itemize}
	\item $(\alpha_s, \alpha_s)\in 2\Z_{>0}$ for all $s\in S$, and $\min\{(\alpha_s, \alpha_s)\mid s\in S\}=2$, 
	\item $\displaystyle \langle \alpha_s^{\vee}, \mu\rangle =2\frac{(\alpha_s, \mu)}{(\alpha_s, \alpha_s)}$ for all $s\in S$ and $\mu\in P$.  
\end{itemize}
(See \cite[Chapter 2]{Kac:book}.) 

Let $W$ be the Weyl group of $\lieg$, which is generated by the simple reflections $\{r_s \mid s\in S\}$. The Weyl group $W$ acts on $\lieh$ (resp.~$\lieh^{\ast}$) by 
\[
r_s(h)=h-\langle\alpha_s, h\rangle\alpha_s^{\vee}\quad  (\text{resp.~}r_s(\mu)=\mu-\langle\alpha_s^{\vee}, \mu\rangle\alpha_s)
\]
for $h\in \lieh$ (resp.~$\mu\in \lieh^{\ast}$) and $s\in S$. Then, the pairing between $\lieh$ and $\lieh^{\ast}$, and the bilinear form $(-,-)$ on $P$ are $W$-invariant. For $w\in W$, set
\[
\ell(w):=\min \{\ell\in \Z_{\geq 0} \mid w=r_{s_1}\cdots r_{s_\ell}\}. 
\]
This is called the length of $w$. It is known that $W$ has a unique element of maximum length, which is denoted by $w_0$. For $h\in \lieh$ (resp.~$\mu\in \lieh^{\ast}$), set 
\[
h^{\ast}:=-w_0(h)\quad (\text{resp.~}\mu^{\ast}:=-w_0(\mu).)
\]
Note that $(h^{\ast})^{\ast}=h$ and $(\mu^{\ast})^{\ast}=\mu$ for all $h\in \lieh$ and $\mu\in\lieh^{\ast}$. For $s\in S$, define $s^{\ast}\in S$ by the condition $\alpha_{s^{\ast}}=\alpha_s^{\ast}$. Then 
\[
\varpi_{s^{\ast}}=\varpi_{s}^{\ast},\ \alpha_{s^{\ast}}^{\vee}=(\alpha_s^{\vee})^{\ast},\ \varpi_{s^{\ast}}^{\vee}=(\varpi_s^{\vee})^{\ast},\ \rho^{\ast}=\rho,\ (\rho^{\vee})^{\ast}=\rho^{\vee}.
\]
For $w\in W$, set 
\[
S(w):=\{(s_1,\dots, s_{\ell(w)})\in S^{\ell(w)}\mid w=r_{s_1}\cdots r_{s_{\ell(w)}}\}.
\]
(Notice that $S(e)$ is the empty set $\varnothing$.) An element of $S(w)$ is called a reduced word of $w$.

\subsection{Quantized enveloping algebras}\label{subsec:QEA}
Write $d:=\det C\in \Z_{> 0}$. Let $\qqd$ be an indeterminate. Set 
	\begin{align*}
	    &\qq:=(\qqd)^{2d},\\
	    &\qq_{s}:=\qq^{\frac{(\alpha_{s},\alpha_{s})}{2}}=\qq^{(\rho, \alpha_s)}\text{ for }s\in S,\\
		&[n]_\qq:=\frac{\qq^{n}-\qq^{-n}}{\qq-\qq^{-1}}\ \text{ for\ }n\in\mathbb{Z},\\
		&{\displaystyle \left[\begin{array}{c}
				n\\
				k
			\end{array}\right]_\qq:=\begin{cases}
				{\displaystyle \frac{[n]_\qq[n-1]_\qq\cdots[n-k+1]_\qq}{[k]_\qq[k-1]_\qq\cdots[1]_\qq}} & \text{if\ }n\in\mathbb{Z},k\in\mathbb{Z}_{>0},\\
				1 & \text{if\ }n\in\mathbb{Z},k=0,
		\end{cases}}\\
		&[n]_\qq!:=[n]_\qq[n-1]_\qq\cdots[1]_\qq\text{\ for\ }n\in\mathbb{Z}_{>0},\ [0]_\qq!:=1.
	\end{align*}
	For a rational function $R\in\Q(\qq)$, we define $R_{\qq_s}$ as
	the rational function obtained from $R$ by substituting $\qq$ by $\qq_{s}$
	($s\in S$). In the following, we set $\Bbbk:=\Q(\qqd)$ and simply write $\otimes_{\Bbbk}$ as $\otimes$.

\begin{dfn}\label{d:QEA} The quantized enveloping algebra $\Uq:=\Uq(\lieg)$ is the unital associative $\Bbbk$-algebra defined by the
	generators 
	\[
	E_{s},F_{s}, K_s, K_s^{-1}\;(s\in S),
	\]
	and the relations (i)--(iv) below: 
	\begin{enumerate}
		\item[(i)] $K_sK_s^{-1}=1=K_s^{-1}K_s,\;K_sK_t=K_tK_s$ for $s, t\in S$,
		\item[(ii)] $K_sE_{t}=\qq_s^{c_{st}}E_{t}K_s,\;K_sF_{t}=\qq_s^{-c_{st}}F_{t}K_s$ for $s, t\in S$,
		\item[(iii)] ${\displaystyle \left[E_{s},F_{t}\right]=\delta_{st}\frac{K_{s}-K_{s}^{-1}}{\qq_{s}-\qq_{s}^{-1}}}$ for $s, t\in S$,
		\item[(iv)] ${\displaystyle \sum_{k=0}^{1-c_{st}}(-1)^{k}\left[\begin{array}{c}
				1-c_{st}\\
				k
			\end{array}\right]_{\qq_s}X_{s}^{k}X_{t}X_{s}^{1-c_{st}-k}=0}$ for $s, t\in S$ with $s\neq t$, and $X=E,F$. 
	\end{enumerate}
	The $\Bbbk$-subalgebra of $\Uq(\lieg)$ generated by $\{E_{s}\}_{s\in S}$
	(resp.~$\{F_{s}\}_{s\in S}$, $\{K_{s}^{\pm 1}\}_{s\in S}$, $\{E_{s}, K_s^{\pm 1}\}_{s\in S}$,
	$\{F_{s}, K_s^{\pm 1}\}_{s\in S}$) is denoted by $\Uq^{+}$
	(resp.~$\Uq^{-}$, $\Uq^{0}$, $\Uq^{\geq0}$, $\Uq^{\leq0}$).
	
	For $\alpha\in Q$, write 
    \[
    (\Uq)_{\alpha}:=\{X\in\Uq\mid K_sXK_s^{-1}=\qq_s^{\langle \alpha_s^{\vee},\alpha\rangle}X\;\text{for all}\;s\in S\}.
    \]
	The elements of $(\Uq)_{\alpha}$ are said to be homogeneous. For
	a homogeneous element $x\in(\Uq)_{\alpha}$, we set $\wt x=\alpha$.
	For any subset $L\subset\Uq$ and $\alpha\in Q$, we set $L_{\alpha}:=L\cap(\Uq)_{\alpha}$.
	
	For $\alpha=\sum_{s\in S}m_s\alpha_s\in Q$, write 
	\[
	K_{\alpha}:=\prod_{s\in S} K_s^{m_s}.
	\]
	Then  
	\[
	K_{\alpha}X=\qq^{(\alpha, \wt x)}XK_{\alpha}.
	\]
    for a homogeneous element $X\in \Uq$.
    
	The quantized enveloping algebra $\Uq$ is a Hopf algebra endowed with the coproduct $\Delta\colon\Uq\to\Uq\otimes\Uq$, the counit 	$\varepsilon\colon\Uq\to\Bbbk$ and the antipode $\sfS\colon\Uq\to\Uq$ 
	satisfying  
	\begin{align*}
		\Delta\left(E_{s}\right) & =E_{s}\otimes 1+K_{s}\otimes E_{s}, & \varepsilon\left(E_{s}\right) & =0, & \sfS\left(E_{s}\right) & =
		-K_{s}^{-1}E_{s},\\
		\Delta\left(F_{s}\right) & =F_{s}\otimes K_{s}^{-1}+1\otimes F_{s}, & \varepsilon\left(F_{s}\right) & =0, & 
		\sfS\left(F_{s}\right) & =-F_{s}K_{s},\\
		\Delta\left(K_s\right) & =K_s\otimes K_s, & \varepsilon\left(K_s\right) & =1, & \sfS\left(K_s\right) & =K_s^{-1}.
	\end{align*}
\end{dfn}
In this paper, we also use the following opposite coproduct 
\[
\uDelta:= \mathsf{P}\circ \Delta\colon \Uq\to\Uq\otimes\Uq
\]
where $\mathsf{P}\colon \Uq\otimes \Uq\to \Uq\otimes \Uq$ is the $\Bbbk$-algebra involution given by $X_1\otimes X_2\mapsto X_2\otimes X_1$.

\subsection{Integrable modules and universal $R$-matrices}

For a $\Bbbk$-vector space $V$, write $V^{\ast}:=\Hom_{\Bbbk}(V, \Bbbk)$, and when $V$ is a $\Uq$-module, its dual space $V^{\ast}$ is also considered as a  $\Uq$-module by 
\[
\langle X.\xi, v\rangle:= \langle \xi, \sfS(X).v\rangle
\]
for $\xi\in V^{\ast}, X\in\Uq$ and $v\in V$. For a $\Uq$-module $V$, write
    \begin{align*}
    V^{\inv}
    &:=\{v\in V\mid  X.v=\varepsilon(X)v\text{ for all }X\in \Uq\}.
    \end{align*}

For $\Uq$-modules $U$ and $V$, their tensor product $U\otimes V$ is regarded as a $\Uq$-module through the coproduct $\Delta$. 
Moreover, we can equip the vector space $U\otimes V$ with another $\Uq$-module structure by the opposite coproduct $\uDelta$. In this case, we write this space as $U\ut V$.

For finite dimensional $\Uq$-modules $U$ and $V$, we have an isomorphism of $\Uq$-modules
\begin{align}
    \Upsilon_{U, V}\colon U^{\ast}\ut V^{\ast} \xrightarrow{\sim} (U\otimes V)^{\ast},\ \xi\otimes \zeta\mapsto (u\otimes v\mapsto \langle \xi, u\rangle\langle \zeta, v\rangle), \label{eq:dualtensor}  
\end{align}
and the homomorphism of $\Uq$-modules
\begin{align}
    V^{\ast}\otimes V\to\Bbbk,\ \xi\otimes v\mapsto \langle \xi, v\rangle.\label{eq:dualeval}
\end{align}
Here $\Bbbk$ is regarded as a trivial $\Uq$-module. 

    Let $V$ be a $\Uq^0$-module. For $\mu \in P$, we set 
    \begin{align*}
	V_{\mu}:= \{u \in V \mid K_s.u=\qq_s^{\langle \alpha_s^{\vee}, \mu \rangle}u\ \text{\ for\ all\ } s \in S\},
    \end{align*}
    and 
    \begin{align*}
    V_{\neq \mu}&:=\bigoplus_{\nu\in P, \nu\neq \mu}V_{\nu},&
    V_{> \mu}&:=\bigoplus_{\nu\in P, \nu> \mu}V_{\nu},&
    V_{< \mu}&:=\bigoplus_{\nu\in P, \nu< \mu}V_{\nu},
    \end{align*}    
	The space $V_{\mu}$ is called the weight space of $V$ of weight $\mu$. For $u\in V_{\mu}$, we write $\weight u:=\mu$. When $V$ is a module over $\Uq$ (resp.~$\Uq^{\geq 0}$, $\Uq^{\leq 0}$), we always consider its weight spaces by the action of $\Uq^{0}$ induced from that of $\Uq$ (resp.~$\Uq^{\geq 0}$, $\Uq^{\leq 0}$). 
    A $\Uq$-module $V$ is called a \emph{type one} module if it is a direct sum of its weight spaces. In the following, all $\Uq$-modules are assumed to be type one. 
    
    A $\Uq$-module $V$ is said to be \emph{integrable} if $E_s$ and $F_s$ act locally nilpotently on $V$ for all $s\in S$. 

	For $\lambda\in P_{+}$, denote by $V(\lambda)$ the irreducible (integrable) highest weight $\Uq$-module generated by a highest weight vector $v_{\lambda}$ of weight $\lambda$.  

	For $w\in W$, define $v_{w\lambda}\in V(\lambda)_{w\lambda}$ by
	\begin{align*}
		v_{w\lambda}=F_{s_{1}}^{(\langle \alpha_{s_{1}}^{\vee}, r_{s_{2}}\cdots r_{s_{\ell}}\lambda\rangle)}\cdots F_{s_{\ell-1}}^{(\langle \alpha_{s_{\ell-1}}^{\vee}, r_{s_{\ell}}\lambda\rangle)}F_{s_{\ell}}^{(\langle \alpha_{s_{\ell}}^{\vee}, \lambda\rangle)}.v_{\lambda}
	\end{align*}
	for $(s_{1},\dots,s_{\ell})\in S(w)$. Here $F_s^{(k)} := F_s^k/[k]_{\qq_s}!$ for $s\in S, k \in \Z_{\geq 0}$.
    It is known that $v_{w\lambda}$ depends only on $w\lambda$, and does not depend on the choice of $(s_{1},\dots,s_{\ell})\in S(w)$ and $w\in W$ \cite[Proposition 39.3.7]{Lusztig:Intro}. Note that $V(\lambda)_{w\lambda}=\Bbbk v_{w\lambda}$.
    
    For $\lambda\in P_+$ and $w\in W$, define $\xi_{w\lambda}\in V(\lambda)^\ast$ by $\langle \xi_{w\lambda}, v_{w\lambda}\rangle=1$ and $\xi_{w\lambda}|_{V(\lambda)_{\neq w\lambda}}=0$. Remark that $\wt \xi_{w\lambda}=-w\lambda$ in our convention, and $(V(\lambda)^\ast)_{-w\lambda}=\Bbbk \xi_{w\lambda}$.  

Next, we fix our convention on $R$-matrices. 
\begin{prop}[{see~\cite[Chapter 6]{Jantzen:Quantum}}]\label{p:formU} 
	There uniquely exists a $\Bbbk$-bilinear pairing $(\ ,\ )_{D}\colon \Uq^{\leq 0}\times\Uq^{\geq 0}\to\Bbbk$
	such that 
	\begin{itemize}
		\item[(i)] $(\Delta(X),Y_{2}\otimes Y_{1})_{D}=(X,Y_{1}Y_{2})_{D}$ for $X\in \Uq^{\leq 0}, Y_{1},Y_{2}\in \Uq^{\geq 0}$, 
		\item[(ii)] $(X_{1}\otimes X_{2},\Delta(Y))_{D}=(X_{1}X_{2},Y)_{D}$ for $X_{1}, X_{2}\in\Uq^{\leq 0},Y\in\Uq^{\geq 0}$, 
		\item[(iii)] $(F_{s}, K_{\alpha})_{D}=(K_{\alpha},E_{s})_{D}=0$ for $s\in S, \alpha\in Q$, 
		\item[(iv)] $(K_{\alpha}, K_{\beta})_{D}=\qq^{-(\alpha, \beta)}$ for $\alpha, \beta\in Q$, 
		\item[(v)] ${\displaystyle (F_{s},E_{t})_{D}=-\frac{\delta_{st}}{\qq_{s}-\qq_{s}^{-1}}}$
		for $s, t\in S$ ,
	\end{itemize}
	here the $\Bbbk$-bilinear map $(\ ,\ )_{D}\colon (\Uq^{\leq 0}\otimes \Uq^{\leq 0})\times (\Uq^{\geq 0}\otimes\Uq^{\geq 0})\to\Bbbk$
	is defined by 
	\[
	(X_{1}\otimes X_{2}, Y_{1}\otimes Y_{2})_{D}=(X_{1}, Y_{1})_{D}(X_{2}, Y_{2})_{D}
	\]
	for $X_{1}, X_{2}\in\Uq^{\leq 0}, Y_{1},Y_{2}\in\Uq^{\geq0}$. 
\end{prop}

The bilinear pairing $(\ ,\ )_{D}$ is called the \emph{Drinfeld pairing}. It has the following
properties. 
\begin{enumerate}
	\item For $\alpha,\beta\in Q_{+}$, ${(\ ,\ )_{D}}\mid_{(\Uq^{\leq 0})_{-\alpha}\times(\Uq^{\geq 0})_{\beta}}=0$
	unless $\alpha=\beta$. 
	\item For $\alpha\in Q_{+}$, $\left.(\ ,\ )_{D}\right|_{(\Uq^{-})_{-\alpha}\times(\Uq^{+})_{\alpha}}$
	is non-degenerate.
\end{enumerate}

Let $\alpha\in Q_+$. Take a basis $\{B_{\alpha, 1},\dots, B_{\alpha, k_{\alpha}}\}$ of $(\Uq^-)_{-\alpha}$. Then, by the property of Drinfeld pairing, we have a basis $\{B_{\alpha, 1}^+,\dots, B_{\alpha, k_{\alpha}}^+\}$ of $(\Uq^+)_{\alpha}$ satisfying $(B_{\alpha, i}, B_{\alpha, j}^+)_D=\delta_{ij}$. Set 
\begin{align*}
&\Theta_{\alpha}:=\sum_{i=1}^{k_{\alpha}}B_{\alpha, i}\otimes B_{\alpha, i}^+\in \Uq^-\otimes \Uq^+.
\end{align*}
It is easy to show that $\Theta_{\alpha}$ does not depend on the choice of $\{B_{\alpha, 1},\dots, B_{\alpha, k_{\alpha}}\}$ and that $\Theta_{0}=1\otimes 1$. 
For $N\in\mathbb{Z}_{>0}$, set 
\begin{align*}
&\mathcal{H}_N:=\Uq^{\geq 0}\left(\sum_{\alpha\in Q_+; \langle \rho^{\vee}, \alpha\rangle \geq N} (\Uq^-)_{-\alpha}\right)\otimes \Uq+\Uq\otimes \Uq^{\leq 0}\left(\sum_{\alpha\in Q_+; \langle \rho^{\vee}, \alpha\rangle \geq N} (\Uq^+)_{\alpha}\right),\\
&\Uq\otimes_{\mathrm{comp}} \Uq:=\lim_{\underset{N}{\longleftarrow}} (\Uq\otimes \Uq)/\mathcal{H}_N.
\end{align*}
Then $\Uq\otimes_{\mathrm{comp}} \Uq$ has the $\Bbbk$-algebra structure induced from that of $\Uq\otimes \Uq$, and there exists an embedding $\Uq\otimes \Uq\hookrightarrow \Uq\otimes_{\mathrm{comp}} \Uq$ in an obvious way. See \cite[Chapter 4]{Lusztig:Intro}. We set 
\begin{align*}
&\Theta:=\sum_{\alpha\in Q_+}\Theta_{\alpha}\in \Uq\otimes_{\mathrm{comp}} \Uq.
\end{align*}
It is called the \emph{quasi-$R$-matrix}.

For $\Uq$-modules $U$ and $V$, define a $\Bbbk$-linear isomorphism $\Pi:=\Pi_{U,  V}\colon U\otimes V\to U\otimes V $ by 
\[
u\otimes v\mapsto \qq^{-(\weight u, \weight v)}u\otimes v
\]
for weight vectors $u\in U$ and $v\in V$.

\begin{prop}
Let $U$ and $V$ be finite dimensional $\Uq$-modules. Then $\Theta$ defines a well-defined $\Bbbk$-linear automorphism on $V\otimes U$, and we have 
\begin{align*}
\Delta(X)\circ \Theta\circ \Pi=\Theta\circ \Pi\circ \uDelta(X)
\end{align*}
for all $X\in \Uq$ on $V\otimes U$. Therefore, if we define $P_{U, V}\colon U\otimes V\to V\otimes U$ as $u\otimes v\mapsto v\otimes u$, then 
\[
\cR_{U, V}:= \Theta\circ \Pi\circ P_{U, V}\colon U\otimes V\to V\otimes U
\]
is an isomorphism of $\Uq$-modules.
\end{prop}
Let $\Sy_n$ be the symmetric group of degree $n$. For $\Bbbk$-vector spaces $V_1,\dots, V_n$ and $\sigma\in \Sy_n$, define the $\Bbbk$-linear map $P_{\sigma}$ as 
\[
P_{\sigma}\colon V_1\otimes \cdots\otimes V_n\to V_{\sigma^{-1}(1)}\otimes \cdots\otimes V_{\sigma^{-1}(n)},\ v_1\otimes \cdots\otimes v_n\mapsto v_{\sigma^{-1}(1)}\otimes \cdots\otimes v_{\sigma^{-1}(n)}.
\]

Suppose that  $V_1,\dots, V_n$ are finite dimensional $\Uq$-modules. For $1\leq i<j\leq n$, define the $\Bbbk$-linear map $\cR_{i, j}$ as 
\[
\cR_{i, j}:=P_{\sigma^{-1}}\circ (\cR_{V_{i}, V_{j}}\otimes \id)\circ P_{\sigma}\colon V_1\otimes \cdots\otimes V_{i}\otimes \cdots \otimes V_{j}\otimes\cdots \otimes V_n\to V_1\otimes \cdots\otimes V_{j}\otimes \cdots \otimes V_{i}\otimes\cdots \otimes V_n,
\]
where $\sigma\in \Sy_n$ is an element satisfying $\sigma(i)=1$ and $\sigma(j)=2$. Namely, the subscript of $\cR_{i, j}$ indicates the position where the operator $\cR$ acts. Note that $\cR_{i, i+1}$ is an isomorphism of $\Uq$-modules for $i=1,\dots, n-1$.  
\begin{prop}[{see \cite[Theorem 7.5]{Jantzen:Quantum}, \cite[Proposition 32.2.4]{Lusztig:Intro}}]\label{p:YB}
For finite dimensional $\Uq$-modules $V_1, V_2, V_3$, we have 
\begin{align*}
&\cR_{12}\circ \cR_{23}\circ \cR_{12}=\cR_{23}\circ \cR_{12}\circ \cR_{23}\colon V_1\otimes V_2\otimes V_3\to  V_3\otimes V_2\otimes V_1, \\
&\cR_{V_1\otimes V_2, V_3}= \cR_{12}\circ \cR_{23}\colon V_1\otimes V_2\otimes V_3\to  V_3\otimes V_1\otimes V_2, \\
&\cR_{V_1, V_2\otimes V_3}= \cR_{23}\circ \cR_{12}\colon V_1\otimes V_2\otimes V_3\to  V_2\otimes V_3\otimes V_1.
\end{align*}
\end{prop}
Set
\[
\Theta^{\sfS}:=\sum_{\alpha\in Q_+}(\sfS^{-1}\otimes \sfS^{-1})(\Theta_{\alpha})
=\sum_{\alpha\in Q_+}\sum_{i=1,\dots, k_{\alpha}}\sfS^{-1}(B_{\alpha, i})\otimes \sfS^{-1}(B_{\alpha, i}^+)
\in \Uq\otimes_{\mathrm{comp}} \Uq.
\]
For finite dimensional $\Uq$-modules $U$ and $V$, define 
\[
\tcR_{U, V}:=P_{U, V}\circ \Pi\circ \Theta^{\sfS}\colon U\otimes V\to V\otimes U. 
\]
Then we have the following proposition. 
\begin{prop}\label{p:tildeR}
    For finite dimensional $\Uq$-modules $U$ and $V$, we have the following commutative diagram:
 		\[
		\begin{tikzcd}
            U^{\ast}\ut V^{\ast}
            \arrow[r,"\tcR_{U^{\ast}, V^{\ast}}", "\sim"']
            \arrow[d,"\sim" sloped, "\Upsilon_{U, V}"']
            \arrow[rd, phantom, "\circlearrowleft" marking] 
            &V^{\ast}\ut U^{\ast}
            \arrow[d,"\Upsilon_{V, U}", "\sim"' sloped]\\
            (U\otimes V)^{\ast}
            \arrow[r,"\cR_{V, U}^{\ast}"', "\sim"]
            &(V\otimes U)^{\ast},
		\end{tikzcd}
		\]
  where $\cR_{V, U}^{\ast}\colon (U\otimes V)^{\ast}\to (V\otimes U)^{\ast}$ denotes the isomorphism of $\Uq$-modules induced from $\cR_{V, U}\colon V\otimes U\to U\otimes V$.
  In particular, $\tcR_{U^{\ast}, V^{\ast}}$ is an isomorphism of $\Uq$-modules. 
\end{prop}
\begin{proof}
For $\xi\in (U^{\ast})_{\mu}$, $\zeta\in (V^{\ast})_{\nu}$, $u\in U$ and $v\in V$, we have 
\begin{align*}
    &\langle \cR_{V, U}^{\ast}(\Upsilon_{U, V}(\xi \otimes \zeta)), v\otimes u\rangle\\
    &=\langle \Upsilon_{U, V}(\xi \otimes \zeta), \cR_{V, U}(v\otimes u)\rangle\\
    &=\langle \Upsilon_{U, V}(\xi \otimes \zeta), \Theta. (\Pi(P_{V, U}(v\otimes u)))\rangle\\
    &=\sum\nolimits_{\alpha\in Q_+}\qq^{-(-\mu+\alpha, -\nu-\alpha)}\sum\nolimits_{i=1,\dots, k_{\alpha}}\langle \Upsilon_{U, V}(\xi \otimes \zeta),
    B_{\alpha, i}.u\otimes B_{\alpha, i}^+.v\rangle\\
    &=\sum\nolimits_{\alpha\in Q_+}\qq^{-(\mu-\alpha, \nu+\alpha)}\sum\nolimits_{i=1,\dots, k_{\alpha}}\langle \Upsilon_{U, V}(\sfS^{-1}(B_{\alpha, i}).\xi\otimes \sfS^{-1}(B_{\alpha, i}^+).\zeta), 
    u\otimes v\rangle\\
    &=\sum\nolimits_{\alpha\in Q_+}\langle \Upsilon_{U, V}(\Pi((\sfS^{-1}\otimes \sfS^{-1})(\Theta_{\alpha}).(\xi \otimes \zeta))), 
    u\otimes v\rangle\\
    &= \langle \Upsilon_{U, V}(\Pi(\Theta^{\sfS}.(\xi \otimes \zeta))), 
    u\otimes v\rangle\\
    &=\langle \Upsilon_{V, U}(\tcR_{U^{\ast}, V^{\ast}}(\xi \otimes \zeta)), 
    v\otimes u\rangle.
\end{align*}
Therefore, we have $\cR_{V, U}^{\ast}\circ \Upsilon_{U, V}=\Upsilon_{V, U}\circ \tcR_{U^{\ast}, V^{\ast}}$.
\end{proof}
The following is a well known property of universal $R$-matrices, which can be checked straightforwardly.  
\begin{prop}\label{prop:Rop}
For finite dimensional $\Uq$-modules $U$ and $V$, 
\[
\tcR_{U, V}=\cR_{U, V}^{\mathrm{op}},
\]
where $\cR_{U, V}^{\mathrm{op}}= P_{U,V} \circ \cR_{V, U}\circ P_{U,V}$. Indeed, $\Pi\circ \Theta^{\sfS}=\Theta\circ \Pi$ on $U\otimes V$. 
\end{prop}
\begin{cor}\label{c:YB}
    For finite dimensional $\Uq$-modules $V_1, V_2, V_3$, we have 
\begin{align*}
&\tcR_{12}\circ \tcR_{23}\circ \tcR_{12}=\tcR_{23}\circ \tcR_{12}\circ \tcR_{23}\colon V_1\ut V_2\ut V_3\to  V_3\ut V_2\ut V_1, \\
&\tcR_{V_1\ut V_2, V_3}= \tcR_{12}\circ \tcR_{23}\colon V_1\ut V_2\ut V_3\to  V_3\ut V_1\ut V_2, \\
&\tcR_{V_1, V_2\ut V_3}= \tcR_{23}\circ \tcR_{12}\colon V_1\ut V_2\ut V_3\to  V_2\ut V_3\ut V_1.
\end{align*}
Here we define $\tcR_{i, j}$ in the same way as $\cR_{i, j}$.
\end{cor}
\subsection{Quantum coordinate rings}\label{subsec:QCR}
Let $G$ be a connected, simply-connected complex simple algebraic group $G$ whose Lie algebra is $\lieg$. 
In this subsection, we recall the definition of quantum coordinate ring $\cO_\qq(G)$ which can be regarded as a quantum analogue of the coordinate ring $\cO(G)$ of $G$.  
\begin{dfn}
Let $V$ be a $\Uq$-module. For $\xi\in V^{\ast}$ and $v\in V$, define a $\Bbbk$-linear map $c^V(\xi, v)\in \Uq^{\ast}$ by 
\[
X\mapsto \langle \xi, X.v\rangle.
\]
\end{dfn}
The dual space $\Uq^{\ast}$ is an associative algebra whose multiplication map is induced from the coproduct $\Delta$ of $\Uq$. The following lemma immediately follows from the definition. 
\begin{lem}[{\cite[Lemma 7.10]{Jantzen:Quantum}}]\label{p:coordprod}
Let $V_1$ and $V_2$ be $\Uq$-modules. Then, for any $\xi_i\in V_i^{\ast}$ and $v_i\in V_i$ $(i=1,2)$, we have 
\[
c^{V_1}(\xi_1, v_1)c^{V_2}(\xi_2, v_2)=c^{V_1\otimes V_2}(\xi_1\otimes \xi_2, v_1\otimes v_2).
\]
Here $\xi_1\otimes \xi_2\in V_1^{\ast}\otimes V_2^{\ast}=V_1^{\ast}\ut V_2^{\ast}$ is considered as an element of $(V_1\otimes V_2)^{\ast}$ via $\Upsilon_{V_1, V_2}$. 
\end{lem}

Let $\cO_\qq(G)$ be a subspace of $\Uq^{\ast}$ spanned by  
\[
\left\{c^{V(\lambda)}(\xi, v) \mid \xi \in V(\lambda)^*, v \in V(\lambda), \lambda\in P_+\right\}.
\]
It is called \emph{the quantum coordinate ring of $G$} associated with $\Uq$. The quantized coordinate ring $\cO_\qq(G)$ has a Hopf algebra structure over $\Bbbk$ whose product and coproduct are induced from the coproduct and the product of $\Uq$, respectively. Moreover, there is an isomorphism of $\Uq\otimes \Uq$-modules
\begin{align}
\bigoplus_{\lambda\in P^+} V(\lambda)^{\ast}\boxtimes V(\lambda)\xrightarrow{\sim} \cO_\qq(G),\ \xi\otimes v\mapsto c^{V(\lambda)}(\xi, v)\ \text{for }\xi\otimes v\in V(\lambda)^{\ast}\boxtimes V(\lambda), \label{eq:PW} 
\end{align}
where $V(\lambda)^{\ast}\boxtimes V(\lambda)$ is the tensor product $V(\lambda)^{\ast}\otimes V(\lambda)$ with a $\Uq\otimes \Uq$-module structure
\[
(X_1\otimes X_2).(\xi\otimes v)=(X_1.\xi)\otimes (X_2.v),
\]
for $\xi\in V(\lambda)^{\ast}$, $v\in V(\lambda)$, $X_1, X_2\in \Uq$, and the $\Uq\otimes \Uq$-module structure on $\cO_\qq(G)$ is given by 
\[
\langle (X_1\otimes X_2).\phi, Y\rangle:= \langle \phi, \sfS(X_1)YX_2\rangle,
\]
for $\phi\in \cO_\qq(G)$ and $X_1, X_2, Y\in \Uq$. See \cite[Section 7.2]{Kashiwara:Global} for more details. 

Let $V_i$ be a $\Uq$-module, $\xi_i\in V_i^{\ast}$ and $v_i\in V_i$ for $i=1,2$. For $X\in \Uq$, write $\Delta(X)=\sum_{(X)} X^{(1)}\otimes X^{(2)}$ and $\uDelta(X)=\sum_{(X)} X_{(1)}\otimes X_{(2)}(=\sum_{(X)} X^{(2)}\otimes X^{(1)})$. Then,  
\begin{align*}
    (1\otimes X).\left(c^{V_1}(\xi_1, v_1)c^{V_2}(\xi_2, v_2)\right)
    &=(1\otimes X).c^{V_1\otimes V_2}(\xi_1\otimes \xi_2, v_1\otimes v_2)\notag\\
    &=\sum_{(X)} c^{V_1\otimes V_2}(\xi_1\otimes \xi_2, X^{(1)}.v_1\otimes X^{(2)}.v_2)\notag\\
    &=\sum_{(X)} c^{V_1}(\xi_1, X^{(1)}.v_1)c^{V_2}(\xi_2, X^{(2)}.v_2)\notag\\
    &=\sum_{(X)} \left((1\otimes X^{(1)}).c^{V_1}(\xi_1, v_1)\right)\left((1\otimes X^{(2)}).c^{V_2}(\xi_2, v_2)\right),
\end{align*}
and
\begin{align}
    (X\otimes 1).\left(c^{V_1}(\xi_1, v_1)c^{V_2}(\xi_2, v_2)\right)
    &=(X\otimes 1).c^{V_1\otimes V_2}(\xi_1\otimes \xi_2, v_1\otimes v_2)\notag\\
    &=\sum_{(X)} c^{V_1\otimes V_2}(X^{(2)}.\xi_1\otimes X^{(1)}.\xi_2, v_1\otimes v_2)\notag\\
    &=\sum_{(X)} c^{V_1}(X^{(2)}.\xi_1, v_1)c^{V_2}(X^{(1)}.\xi_2, v_2)\notag\\
    &=\sum_{(X)} \left((X^{(2)}\otimes 1).c^{V_1}(\xi_1, v_1)\right)\left((X^{(1)}\otimes 1).c^{V_2}(\xi_2, v_2)\right)\notag\\
    &=\sum_{(X)} \left((X_{(1)}\otimes 1).c^{V_1}(\xi_1, v_1)\right)\left((X_{(2)}\otimes 1).c^{V_2}(\xi_2, v_2)\right).\label{eq:modalg}
\end{align} 
In particular, the weight space decomposition of $\cO_\qq(G)$ with respect to the action of $\Uq^0\otimes \Uq^0$ gives a $P\oplus P$-graded algebra structure. 

\begin{ex}
Assume that $G=SL_n$. For $i, j\in [1, n]$, set
\[
c_{ij}:=c^{V(\varpi_1)}(\sfS^{-1}(E_1E_2\cdots E_{i-1}).\xi_{\varpi_1}, F_{j-1}\cdots F_2F_1.v_{\varpi_1}).
\]
Then $c_{ij}$'s satisfy the following relations in $\cO_\qq(SL_n)$, and these are actually the defining relations of $\cO_\qq(SL_n)$ (i.e. $\cO_\qq(SL_n)$ is isomorphic to the $\Bbbk$-algebra generated by $\{c_{ij}\mid i, j=1,\dots, n\}$ with the following defining relations). 
\begin{itemize}
    \item $c_{ij}c_{ik}=\qq c_{ik}c_{ij}$ if $j<k$, 
    \item $c_{ik}c_{jk}=\qq c_{jk}c_{ik}$ if $i<j$, 
    \item $c_{ij}c_{k\ell}=c_{k\ell}c_{ij}$ if $i<k$ and $j>\ell$, 
    \item $c_{ij}c_{k\ell}-c_{k\ell}c_{ij}=(\qq -\qq ^{-1})c_{i\ell}c_{kj}$ if $i<k$ and $j<\ell$, 
    \item $\sum_{\sigma\in \mathfrak{S}_n}(-\qq )^{\ell (\sigma)}c_{1\sigma(1)}c_{2\sigma(2)}\cdots c_{n\sigma(n)}=1$.  
\end{itemize}
\end{ex}
Let $U^+$ be the maximal unipotent subgroup of $G$ associated with $\Phi_+$. We set $\sA_G:=G/U^+$ and $\sA_G^{-}:=G/U^-$. These homogeneous $G$-spaces are called \emph{base affine spaces}, and their elements are called \emph{decorated flags}. The quantum analogue of the ring of regular functions on $\sA_G$ and $\sA_G^-$ are defined as   
\begin{align}
&\cO_\qq(\sA_G)
:=\{\phi\in \cO_\qq(G)\mid (1\otimes X).\phi=\varepsilon(X)\phi\text{ for all }X\in \Uq^+\},\label{eq:def_AG}\\
&\cO_\qq(\sA_G^-)
:=\{\phi\in \cO_\qq(G)\mid (1\otimes X).\phi=\varepsilon(X)\phi\text{ for all }X\in \Uq^-\}.\notag
\end{align}
By \eqref{eq:PW}, we can show that $\cO_\qq(\sA_G)$ (resp.~$\cO_\qq(\sA_G^-)$) is a subspace of $\cO_\qq(G)$ spanned by 
\begin{align*}
&\left\{c^{\lambda}(\xi, v_{\lambda})\in \Uq^{\ast}\mid \xi\in V(\lambda)^{\ast}, \lambda\in P_+\right\}&
&\left(
\text{resp.~}\ \left\{c^{\lambda}(\xi, v_{w_0\lambda})\in \Uq^{\ast} \mid \xi\in V(\lambda)^{\ast}, \lambda\in P_+\right\}\right),
\end{align*}
and it is a subalgebra of $\cO_\qq(G)$. Moreover, 
\begin{align}
\cO_\qq(\sA_G)\simeq \bigoplus_{\lambda\in P_+} V(\lambda)^{\ast}\simeq \cO_\qq(\sA_G^-)\label{eq:AGmod}
\end{align}
as $\Uq$-modules. 

Next we define quantum coordinate rings of Borel subgroups $B_+$ and $B_-$ of $G$, associated with $\Phi_+$ and $-\Phi_+$, respectively. Let $\widetilde{\pi}_{+}\colon \Uq^{\ast}\to (\Uq^{\geq 0})^{\ast}$ (resp.~$\widetilde{\pi}_{-}\colon \Uq^{\ast}\to (\Uq^{\leq 0})^{\ast}$) be the restriction map. These are $\Bbbk$-algebra homomorphisms with respect to the $\Bbbk$-algebra structures induced from the coproduct $\Delta$. Set $\cO_\qq(B^+):=\widetilde{\pi}_{+}(\cO_\qq(G))$ and $\cO_\qq(B^-):=\widetilde{\pi}_{-}(\cO_\qq(G))$. Note that $\cO_\qq(B^+)$ and $\cO_\qq(B^-)$ are also $P\oplus P$-graded algebras by the action of $\Uq^0\otimes \Uq^0$. 

We can show that 
\[
\pi_+:=\widetilde{\pi}_{+}|_{\cO_\qq(\sA_G^-)}\colon \cO_\qq(\sA_G^-)\to \cO_\qq(B^+)\quad \text{and}\quad 
\pi_-:=\widetilde{\pi}_{-}|_{\cO_\qq(\sA_G)}\colon \cO_\qq(\sA_G)\to \cO_\qq(B^-)
\]
are injective $\Bbbk$-algebra homomorphism. Moreover, $\mathbf{\Delta}_{w_0, w_0}:=\{\qq^ac^{\lambda}(\xi_{w_0\lambda}, v_{w_0\lambda})\mid  a\in \frac{1}{2d}\Z, \lambda\in P_+\}$ (resp.~$\mathbf{\Delta}_{e, e}:=\{\qq^ac^{\lambda}(\xi_{\lambda}, v_{\lambda})\mid  a\in \frac{1}{2d}\Z, \lambda\in P_+\}$) forms an Ore set of $\cO_\qq(\sA_G^-)$ (resp.~$\cO_\qq(\sA_G)$), and  $\pi_{+}|_{\cO_\qq(\sA_G^-)}$ (resp.~$\pi_{-}|_{\cO_\qq(\sA_G)}$) can be extended to an isomorphism of $\Bbbk$-algebras
\begin{align}
\pi_{+}\colon \cO_\qq(\sA_G^-)[\mathbf{\Delta}_{w_0, w_0}^{-1}]\xrightarrow{\sim} \cO_\qq(B^+)\quad (\text{resp.~}\pi_{-}\colon \cO_\qq(\sA_G)[\mathbf{\Delta}_{e, e}^{-1}]\xrightarrow{\sim} \cO_\qq(B^-)).\label{eq:AG_B}
\end{align}

\section{Quantum configuration space of decorated flags}
\subsection{Configuration spaces of decorated flags}\label{subsec:config_classical}
For $K \in \Z_{\geq 2}$, the configuration space of $K$ decorated flags in $\sA_G$ is defined to be the set 
\[
\Conf_K \sA_G := (\overbrace{\sA_G \times \dots \times \sA_G}^{K\text{ times}})/G,
\]
where the action of $G$ is the diagonal left action. An element of $\Conf_K \sA_G$ is written as $[\sfA_1,\dots, \sfA_K]$ for $\sfA_1,\dots, \sfA_K\in \sA_G$. A pair $(\sfA, \sfA')\in \sA_G\times \sA_G$ is said to be \emph{generic} if $[\sfA, \sfA']=[[U^+], h\overline{w_0}.[U^+]]$ for some $h\in H$ in $\Conf_2 \sA_G$. For $I\subset [1,K]$, set 
\[
\Conf_K^I \sA_G := \{[\sfA_1=\sfA_{K+1},\sfA_2, \dots, \sfA_K]\in \Conf_K \sA_G\mid (\sfA_i,\sfA_{i+1})\ \text{is generic for }i\in I\},
\]
When $I= [1,K]$, we write $\Conf_K^{[1,K]} \sA_G=:\Conf_K^{\times} \sA_G$.
\begin{rem}
    The configuration spaces of decorated flags can be also defined as stacks. 
    However, we do not discuss the geometric structure of $\Conf_K \sA_G$ directly in this paper (we use these symbols only to capture the impressions of the geometric objects underlying the algebras considered below). Hence we introduce them just as a set to keep the prerequisites minimal. 
    See \cite{IO:Wilson, IOS} for more details. 
\end{rem}

\subsection{Braided tensor products}\label{subsec:braided_product}
To formulate the quantum configuration spaces of decorated flags, let us recall the construction of braided tensor products; cf.~\cite{Majid}. 
Let $\A_i$ ($i=1, 2$) be locally finite $\Uq$-module algebras, that is, $\A_i$ is a $\Bbbk$-algebra endowed with a $\Uq$-module structure such that 
\begin{itemize}
    \item as a $\Uq$-module, $\A_i$ is decomposed into a direct sum of (possibly infinitely many) finite dimensional $\Uq$-modules, 
    \item $X.1=\varepsilon(X)1$ for $X\in \Uq$, and 
    \item \[
    X.(aa')=\sum_{(X)} (X_{(1)}.a)(X_{(2)}.a')
    \]
    for all $a, a'\in \A_i$ and $X\in \Uq$, where $\uDelta(X)=\sum_{(X)} X_{(1)}\otimes X_{(2)}$. 
\end{itemize}
The second and third conditions can be rephrased as saying that the unit map $\Bbbk\to \A_i$ and the multiplication map $m_{\A_i}\colon \A_i\ut \A_i\to\A_i$ are $\Uq$-module homomorphisms.

Define \emph{the braided tensor product} $\A_1\bt \A_2$ of $\A_1$ and $\A_2$ as follows.
\begin{itemize}
    \item As a $\Uq$-module, $\A_1\bt \A_2:=\A_1\ut \A_2$.
    \item The multiplication $m_{\A_1\bt \A_2}\colon (\A_1\bt \A_2)\ut (\A_1\bt \A_2)\to \A_1\bt \A_2$ is defined by 
    \[
m_{\A_1\bt \A_2}:= (m_{\A_1}\otimes m_{\A_2})\circ (\id_{\A_1}\otimes \tcR_{\A_2, \A_1}\otimes \id_{\A_2}). 
    \]
\end{itemize}
Note that $m_{\A_1\bt \A_2}$ is a $\Uq$-module homomorphism. It follows from \cref{c:YB} that  $\A_1\bt \A_2$ is again a locally finite $\Uq$-module algebra. Moreover, for locally finite $\Uq$-module algebras $\A_1, \A_2, \A_3$, we have an isomorphism of locally finite $\Uq$-module algebras
\[
(\A_1\bt\A_2)\bt \A_3\xrightarrow{\sim}\A_1\bt (\A_2\bt \A_3),\qquad (a_1\otimes a_2)\otimes a_3\mapsto a_1\otimes (a_2\otimes a_3)
\]
by \cref{c:YB}. Hence we simply write these algebras as $\A_1\bt\A_2\bt \A_3$.

Let $\A_1, \dots, \A_k$ be locally finite $\Uq$-module algebras. For $\ell=1,\dots, k$ and $a\in \A_{\ell}$, we write 
\[
a^{[\ell]}:=1\otimes \cdots \otimes 1\otimes \overset{\text{$\ell$-th}}{\overset{\vee}{{a}}} \otimes 1\otimes \cdots \otimes 1\in \A_1\bt\cdots\bt\A_{k}.
\]
Note that  
\[
a_1^{[1]}\cdots a_k^{[k]}=a_1\otimes\cdots\otimes a_k,
\]
for $a_{\ell}\in \A_{\ell}$ ($\ell=1,\dots, k$), but it may not be equal to $a_k^{[k]}\cdots a_1^{[1]}$. 
\subsection{Quantum configuration spaces of decorated flags}
The algebra $\cO_\qq(\sA_G)$ is a locally finite $\Uq$($=\Uq\otimes 1$)-module algebra (see \eqref{eq:modalg}). Henceforth, we always regard $\cO_\qq(\sA_G)$ as a $\Uq$-module by the action induced from that of $\Uq\otimes 1$ on $\cO_\qq(G)$. Define
\begin{align*}
    &\cO_\qq(\sA_G)^{\bt K}:=\underset{K\text{ times}}{\underbrace{\cO_\qq(\sA_G)\bt\cdots \bt \cO_\qq(\sA_G)}}  
\end{align*}
for $K\in \Z_{>0}$
with respect to this $\Uq$-module structure. In the following, we write
\[
\xi(\lambda):=c^{V(\lambda)}(\xi, v_{\lambda})\in \cO_\qq(\sA_G)(\subset \cO_\qq(G))
\]
for $\xi\in V(\lambda)^{\ast}$, $\lambda\in P_+$. For $\lambda_1, \lambda_2\in P_+$, there exists a $\Uq$-module homomorphism $V(\lambda_1+\lambda_2)\to V(\lambda_1)\otimes V(\lambda_2)$ given by $v_{\lambda_1+\lambda_2}\mapsto v_{\lambda_1}\otimes v_{\lambda_2}$. Therefore, for $\xi_1\in V(\lambda_1)^{\ast}$ and $\xi_2\in V(\lambda_2)^{\ast}$, there uniquely exists $\xi_1\diamond \xi_2\in V(\lambda_1+\lambda_2)^{\ast}$ satisfying
\[
\xi_1(\lambda_1)\xi_2(\lambda_2)=(\xi_1\diamond \xi_2)(\lambda_1+\lambda_2)
\]
in $\cO_\qq(\sA_G)$. The correspondence $(\xi_1, \xi_2)\mapsto \xi_1\diamond \xi_2$ is called the Cartan
product.  
By definition, the Cartan product makes $\bigoplus_{\lambda\in P_+} V(\lambda)^{\ast}$ into a $\Bbbk$-algebra which is isomorphic to $\cO_\qq(\sA_G)$ (see \eqref{eq:AGmod}). 
Note that $\xi_{w\lambda_1}\diamond \xi_{w\lambda_2}=\xi_{w(\lambda_1+\lambda_2)}$ for $\lambda_1, \lambda_2\in P_+$ and $w\in W$. 

By \eqref{eq:AGmod}, we have an isomorphism of $\Uq$-modules
\[
\Psi_K\colon\bigoplus_{\lambda_1,\dots, \lambda_K\in P_+}V(\lambda_1)^{\ast}\ut \cdots\ut V(\lambda_K)^{\ast}\xrightarrow{\sim}\cO_\qq(\sA_G)^{\bt K},
\]
given by 
\begin{align}
\xi_1\otimes \cdots \otimes \xi_K\mapsto \xi_1(\lambda_1)\otimes \cdots \otimes \xi_K(\lambda_K)   \label{eq:Psi_k} 
\end{align}
for $\xi_i\in V(\lambda_i)^{\ast}$, $i=1,\dots, K$. In the following, we write 
\begin{align}
    \cO(\lambda_1,\dots, \lambda_K)=\Psi_K(V(\lambda_1)^{\ast}\ut \cdots\ut V(\lambda_K)^{\ast})\label{eq:O_notation}
\end{align}
for $\lambda_1,\dots, \lambda_K\in P_+$. 
For $\phi\in \cO(\lambda_1,\dots, \lambda_K)$, write 
\begin{align}
\deg(\phi)=(\lambda_1,\dots, \lambda_K).\label{eq:deg}
\end{align}
\begin{lem}\label{lem:gradedalg}
The algebra $\cO_\qq(\sA_G)^{\bt K}$ is a graded algebra with respect to the grading \eqref{eq:deg}.
\end{lem}
\begin{proof}
For $\xi_i\in V(\lambda_i)^{\ast}$ and $\zeta_i\in V(\mu_i)^{\ast}$, $i=1,\dots, K$, the element 
\[
\left(c^{V(\lambda_1)}(\xi_1, v_{\lambda_1})\otimes\cdots \otimes c^{V(\lambda_K)}(\xi_K, v_{\lambda_K})\right)\cdot \left(c^{V(\mu_1)}(\zeta_1, v_{\mu_1})\otimes\cdots \otimes c^{V(\mu_K)}(\zeta_K, v_{\mu_K})\right)
\]
is a finite sum of elements of the form
\[
c^{V(\lambda_1)\otimes V(\mu_1)}(\tilde{\eta}_1, v_{\lambda_1}\otimes v_{\mu_1})\otimes\cdots \otimes c^{V(\lambda_K)\otimes V(\mu_K)}(\tilde{\eta}_K, v_{\lambda_K}\otimes v_{\mu_K})
\]
with $\tilde{\eta}_i\in V(\lambda_i)^{\ast}\ut V(\mu_i)^{\ast}$ by \cref{p:coordprod} and the definition of the multiplication in $\cO_\qq(\sA_G)^{\bt K}$. Since there exists a $\Uq$-module homomorphism $V(\lambda_i+\mu_i)\to V(\lambda_i)\otimes V(\mu_i)$ given by $v_{\lambda_i+\mu_i}\mapsto v_{\lambda_i}\otimes v_{\mu_i}$, there exists $\eta_i\in V(\lambda_i+\mu_i)^{\ast}$ such that 
\[
c^{V(\lambda_i)\otimes V(\mu_i)}(\tilde{\eta}_i, v_{\lambda_i}\otimes v_{\mu_i})=c^{V(\lambda_i+\mu_i)}(\eta_i, v_{\lambda_i+\mu_i})
\]
in $\cO_\qq(\sA_G)$. Hence we obtain the assertion. 
\end{proof}

The following statements are immediate from the definition of the multiplication in $\cO_\qq(\sA_G)^{\bt K}$. 

\begin{prop}\label{p:emb}
For an order-preserving injective map $p\colon [1,K]\to [1,L]$, there exists an injective homomorphism of $\Uq$-module algebras
    \[
    \iota_p\colon \cO_\qq(\sA_G)^{\bt K}\to \cO_\qq(\sA_G)^{\bt L}
    \]
    given by 
    \[
    \phi_1^{[1]}\cdots \phi_K^{[K]}\mapsto \phi_1^{[p(1)]}\cdots \phi_K^{[p(K)]}
    \]
    for $\phi_1,\dots, \phi_K\in \cO_\qq(\sA_G)$.
\end{prop}
\begin{prop}\label{p:comm}
Let $\lambda_i\in P_+$ and $\xi_i\in (V(\lambda_i)^{\ast})_{\nu_i}$, $i=1,2$. Then, for $k<\ell$, 
\[
\xi_2(\lambda_2)^{[\ell]}\xi_1(\lambda_1)^{[k]}=\qq^{-(\nu_1, \nu_2)} \xi_1(\lambda_1)^{[k]}\xi_2(\lambda_2)^{[\ell]}+\sum_{j}\zeta_j(\lambda_1)^{[k]}\zeta'_j(\lambda_2)^{[\ell]}
\]
for some $\zeta_j\in (V(\lambda_1)^{\ast})_{>\nu_1}$ and $\zeta'_j\in (V(\lambda_2)^{\ast})_{<\nu_2}$.
\end{prop}
The following domain property is an important corollary of \cref{p:comm}. 
\begin{prop}\label{prop:domain}
For $K\in \Z_{>0}$, $\cO_\qq(\sA_G)^{\bt K}$ is an Ore domain. 
\end{prop}
\begin{proof}
First, we show that the usual tensor product algebra $\cO_\qq(\sA_G)^{\otimes K}$ is a domain. Indeed, recall that $\pi_-\colon \cO_\qq(\sA_G)\to \cO_\qq(B^-)$ is an injective $\Bbbk$-algebra homomorphism. Hence, the assertion follows from the fact that  $\cO_\qq(B^-)^{\otimes K}$ is a domain (see, for example, the proof of \cite[Lemma 2.3]{OQY}). In the following, the multiplication of $\phi, \psi\in \cO_\qq(\sA_G)^{\otimes K}$ in $\cO_\qq(\sA_G)^{\otimes K}$ is denoted by $\phi\cdot_{\mathrm{usual}}  \psi$. 

Next, we prove that the braided tensor product $\cO_\qq(\sA_G)^{\bt K}$ is a domain. 
Write $S=\{1,\dots, n\}$ in this proof, that is, we arbitrarily fix a total order on the set of simple roots $\Pi$. Since $\Pi$ is a $\Q$-basis of $P\otimes_\Z\Q$, each
$\lambda\in P$ has a unique expression
\[
 \lambda=\sum_{s=1}^n c_s(\lambda)\alpha_s,\qquad
 c(\lambda):=(c_1(\lambda),\ldots,c_n(\lambda))\in\Q^n.
\]
We fix a total order on $P^{\oplus K}$ as follows. 
For $\boldsymbol{\lambda}=(\lambda_1,\dots,\lambda_K)\in P^{\oplus K}$, we set 
\begin{align*}
 c(\boldsymbol{\lambda})
 &:=(c(\lambda_1),\dots,c(\lambda_K))\\
 &=(\underbrace{c_1(\lambda_1),\dots,c_n(\lambda_1)}, \underbrace{c_1(\lambda_2),\dots,c_n(\lambda_2)},\dots, \underbrace{c_1(\lambda_K),\dots,c_n(\lambda_K)})\in\Q^{nK}.
 \end{align*}
Then, for $\boldsymbol{\lambda}, \boldsymbol{\lambda}'\in P^{\oplus K}$, define
\begin{equation*}
 \boldsymbol{\lambda}\preceq\boldsymbol{\lambda}'
 \quad\Longleftrightarrow
 \quad
 c(\boldsymbol{\lambda})
       \leq c(\boldsymbol{\lambda}'),
\end{equation*}
here $\leq$ on $\Q^{nK}$ denotes the lexicographic order reading from the left. Then it is a total order on $P^{\oplus K}$ satisfying the following. 
\begin{itemize}
    \item $ \boldsymbol{\lambda}\preceq \boldsymbol{\lambda}',\ \boldsymbol{\mu}\preceq \boldsymbol{\mu}'
 \quad\Longrightarrow\quad
 \boldsymbol{\lambda}+\boldsymbol{\mu}
       \preceq \boldsymbol{\lambda}'+\boldsymbol{\mu}'$. 
      \item It is a refinement of the partial order on $P^{\oplus K}$ induced from the partial order $\leq$ on $P$. 
\end{itemize}
For $\boldsymbol{\mu}=(\mu_1,\dots,\mu_K)\in P^{\oplus K}$, let
\[
 E_{\boldsymbol\mu}=\bigoplus_{\lambda_1,\dots, \lambda_K\in P_+}\cO(\lambda_1)_{\mu_1}\ut \cdots\ut \cO(\lambda_K)_{\mu_K}\subset \cO_\qq(\sA_G)^{\bt K}.
\]
Then, for $\phi\in E_{\boldsymbol\mu}$ and $\psi\in E_{\boldsymbol\nu}$ with $\boldsymbol{\mu}=(\mu_1,\dots,\mu_K), \boldsymbol{\nu}=(\nu_1,\dots,\nu_K)\in P^{\oplus K}$, repeated
use of \cref{p:comm} gives
\begin{equation}\label{eq:leading}
 \phi\psi-
 \qq^{-\sum_{i<j}(\mu_j,\nu_i)}\phi\cdot_{\mathrm{usual}}\psi
 \in
 \bigoplus_{\boldsymbol\gamma\succeq
              \boldsymbol\mu+\boldsymbol\nu, \boldsymbol\gamma\neq 
              \boldsymbol\mu+\boldsymbol\nu} E_{\boldsymbol\gamma}.
\end{equation}
Note that $\qq^{-\sum_{i<j}(\mu_j,\nu_i)}\phi\cdot_{\mathrm{usual}}\psi
 \in
E_{\boldsymbol\mu+\boldsymbol\nu}$. 

Let $\phi, \psi\in \cO_\qq(\sA_G)^{\bt K}\setminus \{0\}$. Since $\cO_\qq(\sA_G)^{\bt K}=\bigoplus_{\boldsymbol{\mu}\in P^{\oplus K}}E_{\boldsymbol\mu}$, we can write $\phi, \psi$ as 
\[
 \phi=\sum_{\boldsymbol\mu\in P^{\oplus K}} \phi_{\boldsymbol\mu},
 \qquad
 \psi=\sum_{\boldsymbol\nu\in P^{\oplus K}} \psi_{\boldsymbol\nu},
\]
where $\phi_{\boldsymbol\mu}\in E_{\boldsymbol\mu}$ and $\psi_{\boldsymbol\nu}\in E_{\boldsymbol\nu}$. 
Set 
\begin{align*}
&\boldsymbol{\mu}_0=(\mu_1^0,\dots, \mu_K^0):=\min\nolimits_{\preceq}\{\boldsymbol{\mu}\in P^{\oplus K}\mid \phi_{\boldsymbol\mu}\neq 0\},\\ 
&\boldsymbol{\nu}_0=(\nu_1^0,\dots, \nu_K^0):=\min\nolimits_{\preceq}\{\boldsymbol{\nu}\in P^{\oplus K}\mid \psi_{\boldsymbol\nu}\neq 0\}.
\end{align*}
Then, by \eqref{eq:leading}, 
\[
\phi\psi-\qq^{-\sum_{i<j}(\mu_j^0,\nu_i^0)}\phi_{\boldsymbol{\mu}_0}\cdot_{\mathrm{usual}}\psi_{\boldsymbol{\nu}_0}\in  \bigoplus_{\boldsymbol\gamma\succeq
              \boldsymbol{\mu}_0+\boldsymbol{\nu}_0, \boldsymbol\gamma\neq 
              \boldsymbol{\mu}_0+\boldsymbol{\nu}_0} E_{\boldsymbol\gamma}.
\]
Since the usual tensor product algebra $\cO_\qq(\sA_G)^{\otimes K}$ is a domain, $\qq^{-\sum_{i<j}(\mu_j^0,\nu_i^0)}\phi_{\boldsymbol{\mu}_0}\cdot_{\mathrm{usual}}\psi_{\boldsymbol{\nu}_0}\neq 0$. Hence $\phi\psi\neq 0$, which proves that $\cO_\qq(\sA_G)^{\bt K}$ is a domain.

Finally, we show that $\cO_\qq(\sA_G)^{\bt K}$ is an Ore domain. We set 
\begin{align*}
&\mathcal{F}_{m}:=\bigoplus_{\substack{\lambda_1,\dots, \lambda_K\in P_+\\ 
\sum_{i=1}^{K}\sum_{s\in S}\langle \alpha_s^\vee,\lambda_i\rangle =m}}\cO(\lambda_1,\dots, \lambda_K),\quad 
\mathcal{F}_{\leq m}:=\bigoplus_{0\leq m'\leq m}\mathcal{F}_{m'}
\end{align*}
for $m\in \Z_{\geq 0}$. Then, $\{\mathcal{F}_{\leq m}\}_{m\geq 0}$ is an increasing filtration of $\cO_\qq(\sA_G)^{\bt K}$ such that $\cO_\qq(\sA_G)^{\bt K}=\bigcup_{m\geq 0}\mathcal{F}_{\leq m}$. Hence, by \cite[Proposition A.1]{BZ}, it suffices to show that $\{\dim_{\Bbbk}\mathcal{F}_{\leq m}\}_{m\geq 0}$ has polynomial growth. It can be shown by considering its $\qq=1$ counterpart. Namely, we consider the direct sum $\cO_{\C}(\sA_G):=\bigoplus_{\lambda\in P_+} V_{\C}(\lambda)^{\ast}$ endowed with the Cartan product, where $V_{\C}(\lambda)^{\ast}$ denote the duals of finite dimensional irreducible $\lieg$-modules. We can define $\mathcal{F}_{m}^{\qq=1}$ and $\mathcal{F}_{\leq m}^{\qq=1}$ for $\cO_{\C}(\sA_G)^{\otimes K}$ in the same way as $\mathcal{F}_{m}$ and $\mathcal{F}_{\leq m}$, respectively. Then, $\mathcal{F}_{m_1}^{\qq=1}\cdot \mathcal{F}_{m_2}^{\qq=1}=\mathcal{F}_{m_1+m_2}^{\qq=1}$ for $m_1, m_2\geq 0$, and $\cO_{\C}(\sA_G)^{\otimes K}$ is a commutative $\C$-algebra generated by $\mathcal{F}_{1}^{\qq=1}$. Therefore, $\{\dim_{\Bbbk}\mathcal{F}_{\leq m}^{\qq=1}\}_{m\geq 0}$ has polynomial growth. Since $\dim_{\Bbbk}\mathcal{F}_{\leq m}=\dim_{\C}\mathcal{F}_{\leq m}^{\qq=1}$ for $m\neq 0$, we obtain the desired result. 
\end{proof}
Let us introduce the main object of this paper. 
\begin{dfn}\label{def:quantum_config_space}
    For $K\in \Z_{\geq 2}$, we define
    \begin{align*}
    \cO_\qq(\Conf_K\sA_G)&:=\left(\cO_\qq(\sA_G)^{\bt K}\right)^{\inv}\\
    &:=\left\{\phi\in \cO_\qq(\sA_G)^{\bt K}\ \middle|\ X.\phi=\varepsilon(X)\phi\text{ for all }X\in \Uq\right\}.
    \end{align*}
    and call it \emph{a quantum configuration space of $K$ decorated flags} (or, \emph{quantum configuration space}, for short). 
\end{dfn}
Recall that the multiplication map is a $\Uq$-module homomorphism in the braided tensor product. Hence we have the following. 
\begin{prop}\label{prop:qcsgrading}
For $K\in \Z_{\geq 2}$, the quantum configuration space $\cO_\qq(\Conf_K\sA_G)$ is a graded $\Bbbk$-subalgebra of $\cO_\qq(\sA_G)^{\bt K}$. In particular, $\cO_\qq(\Conf_K\sA_G)$ is an Ore domain. 
\end{prop}
Note that $\Psi_K$ restricts to
\begin{align}\label{eq:quantum_Conf_k}
\bigoplus_{\lambda_1,\dots, \lambda_K\in P_+}(V(\lambda_1)^{\ast}\ut \cdots\ut V(\lambda_K)^{\ast})^{\inv}\xrightarrow{\sim}\cO_\qq(\Conf_K\sA_G).
\end{align}
We write 
\begin{align}
    \cO^{\inv}(\lambda_1,\dots, \lambda_K)=\Psi_K((V(\lambda_1)^{\ast}\ut \cdots\ut V(\lambda_K)^{\ast})^{\inv})\label{eq:Oinv_notation}
\end{align}

\section{Basic structure of invariant algebras}
In this section, we study the $\Bbbk$-algebra structure of  $(\cO_\qq(\sA_G)\bt \A\bt \cO_\qq(\sA_G))^{\inv}$ for an arbitrary locally finite $\Uq$-module algebra $\A$. 
The quantum configuration space $\cO_\qq(\Conf_K\sA_G)$ corresponds to the case when $\A=\cO_\qq(\sA_G)^{\bt (K-2)}$. 
\subsection{Basic structure of invariant algebras}
Our starting point is the following injective $\Bbbk$-linear map, which will be upgraded into an injective algebra homomorphism in \cref{t:invstr} after some modifications.
\begin{prop}[{cf.~\cite[Proposition 31.2.6]{Lusztig:Intro}}]\label{p:invsp}
    Let $\lambda_1, \lambda_2\in P_+$ and $M$ be an integrable $\Uq$-module. Then the $\Bbbk$-linear map 
    \[
    \Phi_{\lambda_1, \lambda_2}^M\colon (V(\lambda_1)^{\ast}\ut V(\lambda_2)^{\ast}\ut M)^{\inv}\to M
    \]
    given by 
    \[
    \xi_{w_0\lambda_1}\otimes \xi_{\lambda_2}\otimes m+(\text{other terms})\mapsto m
    \]
    is injective. Here $\mathrm{(}$other terms$\mathrm{)}$ belong to 
    \[
    (V(\lambda_1)^{\ast})_{< \lambda_1^{\ast}}\ut V(\lambda_2)^{\ast}\ut M+V(\lambda_1)^{\ast}\ut (V(\lambda_2)^{\ast})_{> -\lambda_2}\ut M.
    \]
    Moreover, $\sum_{\mu_1, \mu_2\in P_+} \Ima \Phi_{\mu_1, \mu_2}^M=M$. 
\end{prop}
\begin{proof}
    We have an isomorphism of $\Bbbk$-vector spaces 
    \[
    V(\lambda_1)^{\ast}\ut V(\lambda_2)^{\ast}\ut M\xrightarrow{\sim} \Hom_{\Bbbk}(V(\lambda_1)\otimes V(\lambda_2),M)
    \]
    given by 
    \[
    \xi_1\otimes \xi_2\otimes m\mapsto (v_1\otimes v_2\mapsto \xi_1(v_1)\xi_2(v_2)m)
    \]
    which induces an isomorphism 
    \begin{align}
    (V(\lambda_1)^{\ast}\ut V(\lambda_2)^{\ast}\ut M)^{\inv}\xrightarrow{\sim}\Hom_{\Uq}(V(\lambda_1)\otimes V(\lambda_2),M).\label{eq:linmap1}
    \end{align}
    Moreover, the $\Bbbk$-linear map 
    \begin{align}
    \Hom_{\Uq}(V(\lambda_1)\otimes V(\lambda_2),M)\to M,\ \phi\mapsto \phi(v_{w_0\lambda_1}\otimes v_{\lambda_2})\label{eq:linmap2}
    \end{align}
    is injective since $V(\lambda_1)\otimes V(\lambda_2)$ is a cyclic $\Uq$-module generated by $v_{w_0\lambda_1}\otimes v_{\lambda_2}$. Therefore $\Phi_{\lambda_1, \lambda_2}^M$ injective since it is the composite of \eqref{eq:linmap1} and \eqref{eq:linmap2}. 
    Note that, by \cite[Proposition 31.2.6]{Lusztig:Intro}, 
\begin{align}
\Ima \Phi_{\lambda_1, \lambda_2}^M=\left\{m\in M_{-\lambda_1^{\ast}+\lambda_2}\ \middle|\begin{array}{l}E_s^{k}.m=0 \text{\ for\ all\ }k>\langle \lambda_1^{\ast}, \alpha_s^{\vee}\rangle,  s\in S, \text{\ and}\\
F_s^{k}.m=0 \text{\ for\ all\ }k>\langle \lambda_2, \alpha_s^{\vee}\rangle, s\in S\end{array}\right\}. \label{eq:image_phi}
\end{align}
Let $m\in M_{\mu}$ be a weight vector. There exists $N\in \Z_{\geq 0}$ such that 
\[
E_s^{N+1}.m=F_s^{N+1}.m=0\text{ for all }s\in S,\ \text{and}\ \mu+N\rho\in P_+. 
\]
Then $\mu_1:=2N\rho=\mu_1^{\ast}$ and $\mu_2=\mu+2N\rho$ satisfy
\[
\mu_1, \mu_2\in P_+,\ -\mu_1^{\ast}+\mu_2=\mu,\ \text{and}\ \langle \mu_1^{\ast}, \alpha_s^{\vee}\rangle \geq N, \langle \mu_2, \alpha_s^{\vee}\rangle\geq N\ \text{for all}\ s\in S.
\]
Therefore, by \eqref{eq:image_phi}, $m\in \Ima \Phi_{\mu_1, \mu_2}^M$. This proves the last equality in the assertion.
\end{proof}
\begin{prop}[{cf.~\cite[Section 28.2]{Lusztig:Intro}}]\label{p:cyclic}
    Let $M_1, M_2$ be integrable $\Uq$-modules. Then the $\Bbbk$-linear isomorphism 
    \[
    \widetilde{\sigma}_{M_1, M_2}\colon M_1\ut M_2\xrightarrow{\sim} M_2\ut M_1,\ 
    m_1\otimes m_2 \mapsto  (-1)^{\langle 2\rho^{\vee}, \wt m_2\rangle}\qq^{(2\rho, \wt m_2)}m_2\otimes m_1
    \]    
restricts to the $\Bbbk$-linear isomorphism 
    \[
    \sigma_{M_1, M_2}\colon (M_1\ut M_2)^{\inv}\xrightarrow{\sim} (M_2\ut M_1)^{\inv}.
    \]
\end{prop}
\begin{rem}
The factor $(-1)^{\langle 2\rho^{\vee}, \wt m_2\rangle}$ in $\widetilde{\sigma}_{M_1, M_2}$ is not necessary for this assertion. However, this definition should be natural when we consider the  canonical basis (cf.~\cite[Proposition 28.2.4]{Lusztig:Intro}). We will discuss the canonical basis of $\cO_\qq(\Conf_K\sA_G)$ in a future work. See also \cref{cor:sigma_2}.
\end{rem}

Let $\A$ be a locally finite $\Uq$-module algebra. For $\lambda_1, \lambda_2\in P_+$, define a $\Bbbk$-linear map 
\begin{align}
    \Phi_{\lambda_1;\A; \lambda_2}\colon (\cO(\lambda_1)\ut \A\ut \cO(\lambda_2))^{\inv}\to \A
\end{align}
by 
\[
    \xi_{\lambda_1}(\lambda_1)\otimes a\otimes \xi_{w_0\lambda_2}(\lambda_2)+(\text{other terms})\mapsto \qq^{-(\lambda_1, \lambda_1)/2}a
\]
where $\mathrm{(}$other terms$\mathrm{)}$ belong to 
    \[
    \cO(\lambda_1)_{> -\lambda_1}\ut \A\ut \cO(\lambda_2)+\cO(\lambda_1)\ut \A\ut \cO(\lambda_2)_{<\lambda_2^{\ast}}.
    \]
Recall the notation \eqref{eq:O_notation}. By \cref{p:invsp,p:cyclic}, $\Phi_{\lambda_1;\A ;\lambda_2}$ is injective. 

\begin{prop}\label{p:key}
    Let $\lambda_1, \lambda_2, \lambda'_1, \lambda'_2\in P_+$. Then, for $\phi\in (\cO(\lambda_1)\ut \A\ut \cO(\lambda_2))^{\inv}$ and $\phi'\in (\cO(\lambda'_1)\ut \A\ut \cO(\lambda'_2))^{\inv}$, we have 
    \begin{align*}
        \Phi_{\lambda_1;\A;\lambda_2}(\phi)\Phi_{\lambda'_1;\A;\lambda'_2}(\phi')=\qq^{(\lambda'_1-\lambda^{\prime \ast}_2, \lambda_2^{\ast})}\Phi_{\lambda_1+\lambda'_1;\A;\lambda_2+\lambda'_2}(\phi\phi').
    \end{align*}
    where the multiplication $\phi\phi'$ is considered in $(\cO_\qq(\sA_G)\bt \A\bt \cO_\qq(\sA_G))^{\inv}$.
\end{prop}
\begin{proof}
    Let $\xi_i\in V(\lambda_i)^{\ast}, \xi'_i\in V(\lambda'_i)^{\ast}$ ($i=1, 2$), $a, a'\in \A$ be weight vectors. By \cref{p:comm}, the multiplication
    \[
    (\xi_1(\lambda_1)\otimes a \otimes \xi_2(\lambda_2))\cdot (\xi'_1(\lambda'_1)\otimes a' \otimes \xi'_2(\lambda'_2))
    \]
    in $\cO_\qq(\sA_G)\bt \A\bt \cO_\qq(\sA_G)$ is written as 
    \[
    \qq^{-(\wt \xi'_1, \wt a +\wt \xi_2)-(\wt a', \wt \xi_2)}(\xi_1\diamond \xi'_1)(\lambda_1+\lambda'_1)\otimes aa'\otimes (\xi_2\diamond \xi'_2)(\lambda_2+\lambda'_2)+(\text{other terms}),
    \]
    where $\mathrm{(}$other terms$\mathrm{)}$ belong to 
    \begin{align*}
    &\cO(\lambda_1+\lambda'_1)_{>\wt \xi_1\diamond \xi'_1}\ut \A \ut \cO(\lambda_2+\lambda'_2)+\cO(\lambda_1+\lambda'_1)\ut \A \ut \cO(\lambda_2+\lambda'_2)_{<\wt \xi_2\diamond \xi'_2}.
    \end{align*}
    Therefore, if we set $a_0:=\qq^{(\lambda_1, \lambda_1)/2}\Phi_{\lambda_1;\A;\lambda_2}(\phi)$ and  $a'_0:=\qq^{(\lambda'_1, \lambda'_1)/2}\Phi_{\lambda'_1;\A; \lambda'_2}(\phi')$, we have 
    \begin{align}
    \phi\phi'=&\qq^{(\lambda'_1, \wt a_0 +\lambda_2^{\ast})-(\wt a'_0, \lambda_2^{\ast})}\xi_{\lambda_1+\lambda_1'}(\lambda_1+\lambda_1')\otimes a_0a'_0\otimes \xi_{w_0(\lambda_2+\lambda'_2)}(\lambda_2+\lambda'_2)\notag\\
    &+(\text{other terms}),\label{eq:product}
    \end{align}
    and the first term in the right-hand side is equal to 
\begin{align*}
\qq^{(\lambda_1+\lambda'_1, \lambda_1+\lambda'_1)/2}
\xi_{\lambda_1+\lambda_1'}(\lambda_1+\lambda_1')\otimes \Phi_{\lambda_1+\lambda'_1, ;\A;\lambda_2+\lambda'_2}(\phi\phi')\otimes \xi_{w_0(\lambda_2+\lambda'_2)}(\lambda_2+\lambda'_2)
\end{align*}
by definition of $\Phi_{\lambda_1+\lambda'_1, \lambda_2+\lambda'_2}^{\A}$. Hence,  
\begin{align*}
    \qq^{(\lambda'_1, \wt a_0 +\lambda_2^{\ast})-(\wt a'_0, \lambda_2^{\ast})}a_0a'_0=\qq^{(\lambda_1+\lambda'_1, \lambda_1+\lambda'_1)/2}\Phi_{\lambda_1+\lambda'_1;\A; \lambda_2+\lambda'_2}(\phi\phi').
\end{align*}
Since $\phi\in (\cO(\lambda_1)\ut \A\ut \cO(\lambda_2))^{\inv}$ and $\phi'\in (\cO(\lambda'_1)\ut \A\ut \cO(\lambda'_2))^{\inv}$, we have $\wt a_0=\lambda_1-\lambda_2^{\ast}$ and $\wt a'_0=\lambda'_1-\lambda^{\prime \ast}_2$. Therefore, 
\begin{align*}
&\Phi_{\lambda_1;\A ;\lambda_2}(\phi)\Phi_{\lambda'_1;\A ;\lambda'_2}(\phi')\\
&=\qq^{-(\lambda_1, \lambda_1)/2-(\lambda'_1, \lambda'_1)/2}a_0a'_0\\
&=\qq^{-(\lambda_1, \lambda_1)/2-(\lambda'_1, \lambda'_1)/2-(\lambda'_1, \lambda_1)+(\lambda'_1-\lambda^{\prime \ast}_2, \lambda_2^{\ast})+(\lambda_1+\lambda'_1, \lambda_1+\lambda'_1)/2}\Phi_{\lambda_1+\lambda'_1; \A ; \lambda_2+\lambda'_2}(\phi\phi')\\
&=\qq^{(\lambda'_1-\lambda^{\prime \ast}_2, \lambda_2^{\ast})}\Phi_{\lambda_1+\lambda'_1;\A ; \lambda_2+\lambda'_2}(\phi\phi').
\end{align*}
\end{proof}
Let $\Bbbk[P]=\bigoplus_{\lambda\in P}\Bbbk e(\lambda)$ be the group algebra of the additive group $P$, and  $\A$ a $\Bbbk$-algebra endowed with a weight space decomposition. We define the $\Bbbk$-algebra $\A^{\ext}$ as follows. 
\begin{itemize}
    \item As a vector space, $\A^{\ext}=\A\otimes \Bbbk[P]$.
    \item The multiplication on $\A^{\ext}$ is determined by the following conditions.   
    \begin{itemize}
        \item $\A\to \A^{\ext}, a\mapsto a\otimes 1$ and $\Bbbk[P]\to \A^{\ext}, x\mapsto 1\otimes x$ are injective $\Bbbk$-algebra homomorphisms. 
        \item For $\lambda\in P$ and a weight vector $a\in \A$, we have 
        \[
        (1\otimes e(\lambda))(a\otimes 1)=\qq^{-(\lambda, \wt a)}a\otimes e(\lambda)=\qq^{-(\lambda, \wt a)}(a\otimes 1)(1\otimes e(\lambda)).
        \]
    \end{itemize}
\end{itemize}
In the following, we simply write $a\otimes x$ as $ax$ in $\A^{\ext}$ for $a\in \A$ and $x\in \Bbbk[P]$. 
\begin{thm}\label{t:invstr}
Let $\A$ be a locally finite $\Uq$-module algebra. 
\begin{enumerate}
\item 
Define a $\Bbbk$-linear map
\begin{align*}
    \Phi_{\A}^{\ext}\colon (\cO_\qq(\sA_G)\bt \A\bt \cO_\qq(\sA_G))^{\inv}\to \A^{\ext}
\end{align*}
by 
\[
\Phi_{\A}^{\ext}(\phi)=\Phi_{\lambda_1;\A ;\lambda_2}(\phi)e(\lambda_2^{\ast})
\]
for $\phi\in (\cO(\lambda_1)\ut \A\ut \cO(\lambda_2))^{\inv}$. 
Then 
$\Phi^{\A, \ext}$ is an injective $\Bbbk$-algebra homomorphism. 
\item 
For $\lambda\in P_+$, there uniquely exists an element 
\[
D_{\lambda}\in (\cO(\lambda)\ut 1\ut \cO(\lambda^{\ast}))^{\inv}
\]
of the form 
\begin{align*}
D_{\lambda}=\qq^{(\lambda, \lambda)/2}\xi_{\lambda}(\lambda)\otimes 1\otimes \xi_{-\lambda}(\lambda^{\ast})+(\text{other terms}),
\end{align*}
where $($other terms$)$ belong to $\cO(\lambda)_{>-\lambda}\ut 1\ut \cO(\lambda^{\ast})_{<\lambda}$. Then, for $\lambda\in P_+$,    
    \begin{align}
    \Phi_{\A}^{\ext}(D_{\lambda})=e(\lambda),\label{eq:D_e}
    \end{align}
    and $\cD:= \{\qq^aD_{\lambda}\mid a\in \frac{1}{2d}\Z, \lambda\in P_+\}$ forms an Ore set of $(\cO_\qq(\sA_G)\bt \A\bt \cO_\qq(\sA_G))^{\inv}$. 
Moreover, $\Phi_{\A}^{\ext}$ can be extended to an isomorphism of $\Bbbk$-algebras
\[
\Phi_{\A}^{\ext}\colon (\cO_\qq(\sA_G)\bt \A\bt \cO_\qq(\sA_G))^{\inv}[\cD^{-1}]\xrightarrow{\sim} \A^{\ext}.
\]
\end{enumerate}
\end{thm}
\begin{proof}
\item{\underline{(1)}}   
Let $\phi\in (\cO(\lambda_1)\ut \A\ut \cO(\lambda_2))^{\inv}$ and $\phi'\in (\cO(\lambda'_1)\ut \A\ut \cO(\lambda'_2))^{\inv}$. Then, by \cref{p:key},  
    \begin{align*}
    \Phi_{\A}^{\ext}(\phi)\Phi_{\A}^{\ext}(\phi')
    &=\Phi_{\lambda_1;\A ; \lambda_2}(\phi)e(\lambda_2^{\ast})\Phi_{\lambda'_1;\A ; \lambda'_2}(\phi')e(\lambda_2^{\prime\ast})\\
    &=\qq^{-(\lambda_2^{\ast}, \lambda'_1-\lambda^{\prime \ast}_2)}
    \Phi_{\lambda_1;\A;  \lambda_2}(\phi)\Phi_{\lambda'_1;\A; \lambda'_2}(\phi')e((\lambda_2+\lambda'_2)^{\ast})\\
    &=\Phi_{\lambda_1+\lambda'_1;\A; \lambda_2+\lambda'_2}(\phi\phi')e((\lambda_2+\lambda'_2)^{\ast})\\
    &=\Phi_{\A}^{\ext}(\phi\phi').
    \end{align*}
    Therefore $\Phi_{\A}^{\ext}$ is an algebra homomorphism. 
    The injectivity of $\Phi_{\A}^{\ext}$ follows from the injectivity of $\Phi_{\lambda_1;\A; \lambda_2}$ and the fact 
    \begin{align}
    \Phi_{\A}^{\ext}((\cO(\lambda_1)\ut \A\ut \cO(\lambda_2))^{\inv})\subset \A_{\lambda_1-\lambda_2^{\ast}}e(\lambda_2^{\ast}).\label{eq:weight}
    \end{align}
\item{\underline{(2)}}   
    The existence of $D_{\lambda}$ for $\lambda\in P_+$ is a standard fact of the representation theory of $\Uq$, and \eqref{eq:D_e} follows from the definition of $\Phi_{\A}^{\ext}$. The equality \eqref{eq:D_e} implies 
    \[
    D_{\lambda}\phi=\qq^{-(\lambda, \lambda_1-\lambda_2^{\ast})}\phi D_{\lambda}
    \]
    for $\lambda\in P_+$ and $\phi\in (\cO(\lambda_1)\ut \A\ut \cO(\lambda_2))^{\inv}$ (see \eqref{eq:weight}). Therefore, $\cD$ forms an Ore set of $(\cO_\qq(\sA_G)\bt \A\bt \cO_\qq(\sA_G))^{\inv}$, and $(\cO_\qq(\sA_G)\bt \A\bt \cO_\qq(\sA_G))^{\inv}[\cD^{-1}]\to \A^{\ext}$ is well-defined. 

    The injectivity of this extended homomorphism follows from that of $\Phi_{\A}^{\ext}$ and the definition of localization. The surjectivity follows from the last equality in \cref{p:invsp}.
\end{proof}
Note that \eqref{eq:D_e} implies 
\begin{align*}
    D_{\lambda}D_{\lambda'}=D_{\lambda+\lambda'}
\end{align*}
for $\lambda, \lambda'\in P_+$. 

We use the following lemma to show \cref{t:cyclic} below. The morphism $\sigma_{\A}$ in \cref{t:cyclic} will specialize to the \emph{quantum cyclic shift} in \eqref{eq:q-cyclic_shift}.
The proof of \cref{l:antipode} is given by a straightforward calculation. 

\begin{lem}\label{l:antipode}
       Let $M_1, M_2$ be integrable $\Uq$-modules. Then an element $\phi$ of $(M_1\ut M_2)^{\inv}$ satisfies 
       \[
       (X\otimes 1). \phi=(1\otimes \sfS(X)). \phi
       \]
       for $X\in \Uq$. 
\end{lem}

\begin{thm}\label{t:cyclic}
Let $\A$ be a locally finite $\Uq$-module algebra. Define a $\Bbbk$-linear map
\begin{align*}
    \sigma_{\A}\colon (\cO_\qq(\sA_G)\bt \A\bt \cO_\qq(\sA_G))^{\inv}\to (\cO_\qq(\sA_G)\bt \cO_\qq(\sA_G)\bt \A)^{\inv}
\end{align*}
by 
\[
\sigma_{\A}(\phi)=\sigma_{\cO(\lambda_1)\ut \A, \cO(\lambda_2)}(\phi)
\]
for $\phi\in (\cO(\lambda_1)\ut \A\ut \cO(\lambda_2))^{\inv}$ (see \cref{p:cyclic} for the notation). Then 
$\sigma_{\A}$ is an isomorphism of $\Bbbk$-algebras.
\end{thm}
\begin{proof}
    By definition, $\sigma_{\A}$ is an isomorphism of vector spaces. Hence it suffices to show that $\sigma_{\A}$ is an homomorphism of $\Bbbk$-algebras. 
    
    Let $\phi\in (\cO(\lambda_1)\ut \A\ut \cO(\lambda_2))^{\inv}$ and $\phi'\in (\cO(\lambda'_1)\ut \A\ut \cO(\lambda'_2))^{\inv}$. 
    Write 
    \begin{align*}
        &\phi=\sum_{i\in I}c_i^{[1]}a_i^{[2]}d_i^{[3]}&
        &\phi'=\sum_{j\in J}{c'}_j^{[1]} {a'}_j^{[2]}{d'}_j^{[3]}
    \end{align*}
where $I$ and $J$ are certain index sets, and $c_i, c'_j, a_i, a'_j, d_i, d'_j$ are weight vectors. By \cref{l:antipode}, 
    \begin{align*}
        &\sum_{i\in I}(X.(c_i^{[1]}a_i^{[2]}))d_i^{[3]}=\sum_{i\in I}c_i^{[1]}a_i^{[2]}(\sfS(X).d_i^{[3]})
    \end{align*}
for $X\in \Uq$. Hence, for $X\in (\Uq)_{\alpha}$ and $\mu\in P$, we have 
\begin{align*}
        &\sum_{i\in I, \wt d_i=\mu}(X.(c_i^{[1]}a_i^{[2]}))d_i^{[3]}=\sum_{i\in I, \wt d_i=\mu-\alpha}c_i^{[1]}a_i^{[2]}(\sfS(X).d_i^{[3]}),
\end{align*}
hence, in $ \cO_\qq(\sA_G)\bt \cO_\qq(\sA_G)\bt \A$, 
\begin{align}
        &\sum_{i\in I, \wt d_i=\mu}d_i^{[1]}(X.(c_i^{[2]}a_i^{[3]}))=\sum_{i\in I, \wt d_i=\mu-\alpha}(\sfS(X).d_i^{[1]})c_i^{[2]}a_i^{[3]}.\label{eq:XSX1}
\end{align}
Similarly, for $X\in (\Uq)_{\alpha}$ and $\mu\in P$, we have  
\begin{align}
        &\sum_{j\in J, \wt d'_j=\mu+\alpha}{d'}_j^{[1]}(\sfS^{-1}(X).({c'}_j^{[2]}{a'}_j^{[3]}))=\sum_{j\in J, \wt d'_j=\mu}(X.{d'}_j^{[1]}){c'}_j^{[2]}{a'}_j^{[3]}.\label{eq:XSX2}
\end{align}
in $ \cO_\qq(\sA_G)\bt \cO_\qq(\sA_G)\bt \A$. We may assume that there uniquely exists $i_0\in I$ (resp.~$j_0\in J$) such that
\begin{align*}
c_{i_0}\otimes d_{i_0}\in \cO(\lambda_1)_{-\lambda_1}\ut \cO(\lambda_2)_{\lambda_2^{\ast}}\quad 
(\text{resp.~} c'_{j_0}\otimes d'_{j_0}\in \cO(\lambda'_1)_{-\lambda'_1}\ut \cO(\lambda'_2)_{\lambda^{\prime \ast}_2}).
\end{align*}
Moreover, in this case, we may set 
\[
c_{i_0}=\xi_{\lambda_1}(\lambda_1),\ d_{i_0}=\xi_{w_0\lambda_2}(\lambda_2),\ c'_{j_0}=\xi_{\lambda'_1}(\lambda'_1),\ d'_{j_0}=\xi_{w_0\lambda'_2}(\lambda'_2).
\]
Note that $\wt a_{i_0}=\lambda_1-\lambda_2^{\ast}$ and $\wt a'_{j_0}=\lambda'_1-\lambda^{\prime \ast}_2$.

For elements $\phi_1, \phi_2\in \cO(\nu_1)\ut \cO(\nu_2)\ut \A$ ($\nu_1, \nu_2\in P_+$), we will write $\phi_1\equiv \phi_2$ if and only if $\phi_1-\phi_2$ is in 
\[
    \cO(\nu_1)_{<\nu_1^{\ast}}\ut\cO(\nu_2)\ut \A +\cO(\nu_1)\ut \cO(\nu_2)_{>-\nu_2}\ut \A.
\]
Note that \cref{p:invsp} implies that 
\begin{center}
$\phi_1\equiv \phi_2$ if and only if $\phi_1= \phi_2$     
\end{center}
for $\phi_1, \phi_2\in (\cO(\nu_1)\ut \cO(\nu_2)\ut \A)^{\inv}$. 
Then, by \cref{p:cyclic} (see also \eqref{eq:product}), 
\begin{align*}
        &\sigma_{\A}(\phi\phi')\equiv \\
        &(-1)^{\langle 2\rho^{\vee}, (\lambda_2+\lambda'_2)^{\ast}\rangle}\qq^{( 2\rho, (\lambda_2+\lambda'_2)^{\ast})+(\lambda'_1, \lambda_1)-(\lambda'_1-\lambda^{\prime \ast}_2, \lambda_2^{\ast})}
        \xi_{w_0(\lambda_2+\lambda'_2)}(\lambda_2+\lambda'_2)^{[1]} \xi_{\lambda_1+\lambda'_1}(\lambda_1+\lambda'_1)^{[2]} (a_{i_0}a'_{j_0})^{[3]}.
    \end{align*}
On the other hand, by the definition of the multiplication of $\cO_\qq(\sA_G)\bt \cO_\qq(\sA_G)\bt \A$ and \cref{prop:Rop}, \eqref{eq:XSX1}, \eqref{eq:XSX2}, we have
    \begin{align*}
        &\sigma_{\A}(\phi)\sigma_{\A}(\phi')\\
        &=(\sum_{i\in I}(-1)^{\langle 2\rho^{\vee}, \wt d_i\rangle}\qq^{(2\rho, \wt d_i)}d_i^{[1]}c_i^{[2]}a_i^{[3]})(\sum_{j\in J}(-1)^{\langle 2\rho^{\vee}, \wt d'_j\rangle}\qq^{(2\rho, \wt d'_j)}{d'}_j^{[1]}{c'}_j^{[2]}{a'}_j^{[3]})\\
        &=\sum_{\mu, \mu'\in P}\sum_{\substack{\alpha\in Q_+\\ \ell=1,\dots, k_{\alpha}}}\sum_{\substack{i\in I, j\in J\\ \wt d_i=\mu\\ \wt d'_j=\mu'}}(-1)^{\langle 2\rho^{\vee}, \mu+\mu'\rangle}\qq^{(2\rho, \mu+\mu')+(\mu, \mu')}d_i^{[1]}(B_{\alpha, \ell}^+.{d'}_j^{[1]})(B_{\alpha, \ell}. (c_i^{[2]}a_i^{[3]})){c'}_j^{[2]}{a'}_j^{[3]}\\
        &=\sum_{\mu, \mu'\in P}\sum_{\substack{\alpha\in Q_+\\ \ell=1,\dots, k_{\alpha}}}\sum_{\substack{i\in I, j\in J\\ \wt d_i=\mu+\alpha\\ \wt d'_j=\mu'+\alpha}}(-1)^{\langle 2\rho^{\vee}, \mu+\mu'\rangle}\qq^{(2\rho, \mu+\mu')+(\mu, \mu')}(\sfS(B_{\alpha, \ell}).d_i^{[1]}){d'}_j^{[1]}c_i^{[2]}a_i^{[3]}(\sfS^{-1}(B_{\alpha, \ell}^+).({c'}_j^{[2]}{a'}_j^{[3]}))\\
        &\equiv \sum_{\substack{i\in I, j\in J\\ \wt d_i=\lambda_2^{\ast}\\\wt d'_j=\lambda^{\prime \ast}_2}}
        (-1)^{\langle 2\rho^{\vee}, (\lambda_2+\lambda'_2)^{\ast}\rangle}\qq^{(2\rho, (\lambda_2+\lambda'_2)^{\ast})+(\lambda^{\ast}_2, \lambda^{\prime \ast}_2)}
        d_i^{[1]}{d'}_j^{[1]} c_i^{[2]}a_i^{[3]}{c'}_j^{[2]}{a'}_j^{[3]}\\
        &=\sum_{\substack{i\in I, j\in J\\ \wt d_i=\lambda_2^{\ast}\\\wt d'_j=\lambda^{\prime \ast}_2}}\sum_{\substack{\alpha\in Q_+\\ \ell=1,\dots, k_{\alpha}}}
        (-1)^{\langle 2\rho^{\vee}, (\lambda_2+\lambda'_2)^{\ast}\rangle}\qq^{(2\rho, (\lambda_2+\lambda'_2)^{\ast})+(\lambda^{\ast}_2, \lambda^{\prime \ast}_2)-(\wt a_i, \wt c'_j)}
        d_i^{[1]}{d'}_j^{[1]} c_i^{[2]}(B_{\alpha, \ell}^+.{c'}_j^{[2]})(B_{\alpha, \ell}.a_i^{[3]}){a'}_j^{[3]}\\
        &\equiv (-1)^{\langle 2\rho^{\vee}, (\lambda_2+\lambda'_2)^{\ast}\rangle}\qq^{(2\rho, (\lambda_2+\lambda'_2)^{\ast})+(\lambda^{\ast}_2, \lambda^{\prime \ast}_2)+(\lambda_1-\lambda_2^{\ast}, \lambda'_1)} 
        \xi_{w_0(\lambda_2+\lambda'_2)}(\lambda_2+\lambda'_2)^{[1]} \xi_{\lambda_1+\lambda'_1}(\lambda_1+\lambda'_1)^{[2]} (a_{i_0}a'_{j_0})^{[3]}.
    \end{align*}
    Therefore, $\sigma_{\A}(\phi\phi')\equiv \sigma_{\A}(\phi)\sigma_{\A}(\phi')$, which implies $\sigma_{\A}(\phi\phi')=\sigma_{\A}(\phi)\sigma_{\A}(\phi')$.
\end{proof}
\subsection{Localizations of $\cO_\qq(\Conf_K\sA_G)$}
Let $K\in \Z_{\geq 2}$. In the case when  $\A=\cO_\qq(\sA_G)^{\bt (K-2)}$, we write $\Phi_{\A}^{\ext}$ in \cref{t:invstr} (resp.~$\sigma_{\A}$ in \cref{t:cyclic}) as $\Phi_{K}^{\ext}$ (resp.~$\sigma_K$). Then, 
\begin{align*}
\Phi_{K}^{\ext}\colon\cO_\qq(\Conf_K\sA_G)\to (\cO_\qq(\sA_G)^{\bt (K-2)})^{\ext},
\end{align*}
and 
\begin{align*}
\sigma_K\colon \cO_\qq(\Conf_K\sA_G)\xrightarrow{\sim}\cO_\qq(\Conf_K\sA_G). 
\end{align*}
The automorphism $\sigma_K$ is called the \emph{quantum cyclic shift}. For $\lambda\in P_+$, we will write the element $D_{\lambda}\in \cO_\qq(\Conf_K\sA_G)$ in \cref{t:invstr} as $D_{K;\lambda}$, and set 
\begin{align}
    D_{i;\lambda}:=\sigma_K^{i}(D_{K;\lambda})\label{eq:Di}
\end{align}
for $i=1,\dots, K-1$. Note that $\sigma_K^{K}(D_{K;\lambda})=D_{K;\lambda}$. 
It follows from \cref{t:invstr,t:cyclic} that, for $I=\{i_1,\dots, i_l\}\subset [1,K]$,  
\[
\cD_I:= \left\{\qq^a D_{i_1;\lambda_1}\cdots D_{i_l;\lambda_l}\ \middle|\  a\in \frac{1}{2d}\Z, \lambda_1,\dots, \lambda_l\in P_+\right\}
\]
forms an Ore set of $\cO_\qq(\Conf_K\sA_G)$. We set 
\begin{align}
\cO_\qq(\Conf_K^I\sA_G):=\cO_\qq(\Conf_K\sA_G)[\cD_I^{-1}].\label{eq:Conf_I}
\end{align}
When $I=[1,K]$, we write $\cO_\qq(\Conf_K^{[1,K]}\sA_G)=:\cO_\qq(\Conf_K^{\times}\sA_G)$. By definition, $\sigma_K$ induces 
\[
\sigma_K\colon \cO_\qq(\Conf_K^I\sA_G)\xrightarrow{\sim} \cO_\qq(\Conf_K^{\tau_K(I)}\sA_G),
\]
where $\tau_K\in \mathfrak{S}_K$ is the cyclic permutation $(1\ 2\ \cdots\ K)$.

\begin{rem}
The element $D_{i;\lambda} \in \cO_\qq(\Conf_K \sA_G)$ is a quantum analogue of the function on $\Conf_K \sA_G$ given by $[\flA_1,\dots,\flA_K] \mapsto h(\flA_{i+1},\flA_{i})^{\lambda}$, where $i=1,\dots,K$ is considered modulo $K$. See \cref{fig:D_picture} for an illustration and \cref{lem:hw-distance} for the notation. 
\end{rem}

\begin{figure}
    \centering
\begin{tikzpicture}
\def\R{2}
\foreach \i in {1,2,3,4,5,6}{
    \draw(60*\i:\R) -- (60*\i+60:\R);
    \node at (-60*\i+180:\R+0.4) {$\flA_{\i}$};
    \node[scale=0.9] at (-60*\i+150:\R+0.2) {$D_{\i;\lambda}$};
}
\end{tikzpicture}
    \caption{The location of the elements $D_{i;\lambda} \in \cO_\qq(\Conf_K \sA_G)$ with $K=6$.}
    \label{fig:D_picture}
\end{figure}

\subsection{Summary}
Let $K\in \Z_{\geq 2}$. 
For the convenience of the reader, we summarize here the morphisms and formulas for $\cO_\qq(\Conf_K\sA_G)$ established in this section. 
We have an injective $\Bbbk$-algebra homomorphism 
\[
\Phi_{K}^{\ext}\colon\cO_\qq(\Conf_K\sA_G)\to (\cO_\qq(\sA_G)^{\bt (K-2)})^{\ext}
\]
determined by 
\begin{align}
    \xi_{\lambda_1}(\lambda_1)\otimes \phi\otimes \xi_{w_0\lambda_K}(\lambda_K)+(\text{other terms})\mapsto \qq^{-(\lambda_1, \lambda_1)/2}\phi e(\lambda_K^{\ast})
\label{eq:Phi_ke}
\end{align}
where $\mathrm{(}$other terms$\mathrm{)}$ belong to 
    \[
    \cO(\lambda_1)_{> -\lambda_1}\ut \cO_\qq(\sA_G)^{\bt (K-2)}\ut \cO(\lambda_K)+\cO(\lambda_1)\ut \cO_\qq(\sA_G)^{\bt (K-2)}\ut \cO(\lambda_K)_{<\lambda_K^{\ast}}.
    \]
    Moreover, in this case, $\phi\in (\cO_\qq(\sA_G)^{\bt (K-2)})_{\lambda_1-\lambda_K^{\ast}}$. The homomorphism $\Phi_{K}^{\ext}$ can be extended to an isomorphism $\cO_\qq(\Conf_K^{\{K\}}\sA_G)\xrightarrow{\sim} (\cO_\qq(\sA_G)^{\bt (K-2)})^{\ext}$.
    
The quantum cyclic shift is a $\Bbbk$-algebra isomorphism 
\begin{align*}
\sigma_K\colon \cO_\qq(\Conf_K\sA_G)\xrightarrow{\sim}\cO_\qq(\Conf_K\sA_G). 
\end{align*}
determined by 
\begin{align}
    \sum_{i}\xi_{i}(\lambda_1)\otimes \phi_i\otimes \xi'_{i}(\lambda_K)\mapsto 
     \sum_{i}(-1)^{\langle 2\rho^{\vee}, \wt \xi'_{i}\rangle}\qq^{(2\rho, \wt \xi'_{i})}\xi'_{i}(\lambda_K)\otimes \xi_{i}(\lambda_1)\otimes \phi_i
     \label{eq:q-cyclic_shift}
\end{align}
    for $\lambda_1, \lambda_K\in P_+$ and weight vectors $\xi_{i}\in V(\lambda_1)^{\ast}$, $\xi'_{i}\in V(\lambda_K)^{\ast}$. 

For $\lambda\in P_+$ and $\phi\in \cO^{\inv}(\lambda_1,\dots, \lambda_K)$, 
\begin{align}
D_{K;\lambda}\phi=\qq^{-(\lambda, \lambda_1-\lambda_K^{\ast})}\phi D_{K;\lambda}.
     \label{eq:D_k_qcomm}
\end{align}

\section{Disjoint embeddings of quantum configuration spaces}
For an order-preserving injective map $p\colon [1,K]\to [1,L]$, there exists an injective $\Bbbk$-algebra homomorphism
\begin{align}\label{eq:Conf_embedding}
\iota_p^{\inv}\colon \cO_\qq(\Conf_K\sA_G)\to \cO_\qq(\Conf_L\sA_G)
\end{align}
induced from $\iota_p$ in \cref{p:emb}. In this section, we study the commutation relations among images of this kind of embeddings. The following lemmas play a key role. 
\begin{lem}\label{lem:OA_comm}
Write the multiplication map $\cO_\qq(\sA_G)\ut \cO_\qq(\sA_G)\to \cO_\qq(\sA_G)$ as $m_{\cO_\qq(\sA_G)}$. Then, for $\lambda_1, \lambda_2\in P_+$,  
\[
m_{\cO_\qq(\sA_G)} =\qq^{(\lambda_1, \lambda_2)}m_{\cO_\qq(\sA_G)}\circ \tcR_{\cO(\lambda_1), \cO(\lambda_2)}
\] 
on $\cO(\lambda_1)\ut \cO(\lambda_2)$. 
\end{lem}
\begin{proof}
For $\xi_1\in V(\lambda_1)^{\ast}, \xi_2\in V(\lambda_2)^{\ast}$ and $X\in \Uq$, 
\begin{align*}
    &\langle (m_{\cO_\qq(\sA_G)}\circ \tcR_{\cO(\lambda_1), \cO(\lambda_2)})(\xi_1(\lambda_1)\otimes \xi_2(\lambda_2)), X\rangle\\
    &=\langle \Upsilon_{V(\lambda_2), V(\lambda_1)}(\tcR_{V(\lambda_1)^{\ast}, V(\lambda_2)^{\ast}}(\xi_1\otimes \xi_2)), \Delta(X).v_{\lambda_2}\otimes v_{\lambda_1}\rangle\\
    &=\langle \Upsilon_{V(\lambda_1), V(\lambda_2)}(\xi_1\otimes \xi_2),\cR_{V(\lambda_2), V(\lambda_1)}(\Delta(X).v_{\lambda_2}\otimes v_{\lambda_1})\rangle\ \text{by \cref{p:tildeR}}\\
    &=\langle \Upsilon_{V(\lambda_1), V(\lambda_2)}(\xi_1\otimes \xi_2),\Delta(X).\cR_{V(\lambda_2), V(\lambda_1)}(v_{\lambda_2}\otimes v_{\lambda_1})\rangle\\
    &=\qq^{-(\lambda_1, \lambda_2)}\langle \Upsilon_{V(\lambda_1), V(\lambda_2)}(\xi_1\otimes \xi_2), \Delta(X).v_{\lambda_1}\otimes v_{\lambda_2}\rangle\\
    &=\qq^{-(\lambda_1, \lambda_2)}\langle m_{\cO_\qq(\sA_G)}(\xi_1(\lambda_1)\otimes \xi_2(\lambda_2)), X\rangle,
\end{align*}
which proves the assertion. 
 \end{proof} 
 \begin{lem}\label{lem:Conf_comm}
 Let $K\in \Z_{\geq 2}$. 
 \begin{itemize}
     \item[(1)] For $\phi\in (\cO_\qq(\sA_G)^{\bt K})^{\inv}\bt 1\subset \cO_\qq(\sA_G)^{\bt (K+1)}$ and $\psi\in 1^{\bt K}\bt \cO_\qq(\sA_G)\subset \cO_\qq(\sA_G)^{\bt (K+1)}$, $\phi \psi=\psi\phi$ in $\cO_\qq(\sA_G)^{\bt (K+1)}$. 
          \item[(2)] For $\phi\in 1\bt (\cO_\qq(\sA_G)^{\bt K})^{\inv}\subset \cO_\qq(\sA_G)^{\bt (K+1)}$ and $\psi\in \cO_\qq(\sA_G)\bt 1^{\bt K}\subset \cO_\qq(\sA_G)^{\bt (K+1)}$, $\phi \psi=\psi\phi$ in $\cO_\qq(\sA_G)^{\bt (K+1)}$. 
     \item[(3)] For $\phi\in \cO^{\inv}(\lambda_1, \dots, \lambda_K)\subset \cO_\qq(\sA_G)^{\bt K}$ and $\xi\in V(\lambda)^{\ast}$, 
     \[
     \phi \cdot \xi(\lambda)^{[1]}=\qq^{(\lambda_1, \lambda)}\xi(\lambda)^{[1]}\cdot \phi
     \]
     in $\cO_\qq(\sA_G)^{\bt K}$.
     \item[(4)] For $\phi\in \cO^{\inv}(\lambda_1, \dots, \lambda_K)\subset \cO_\qq(\sA_G)^{\bt K}$ and $\xi\in V(\lambda)^{\ast}$,  
     \[
     \phi \cdot \xi(\lambda)^{[K]}=\qq^{-(\lambda, \lambda_K)} \xi(\lambda)^{[K]}\cdot \phi
     \]
     in $\cO_\qq(\sA_G)^{\bt K}$.
 \end{itemize}
\end{lem}
\begin{proof}
\item[\underline{(1)}] Write $\phi=\phi'\otimes 1$ and $\psi=\underset{K}{\underbrace{1\otimes \cdots \otimes 1}}\otimes \psi'$. Then,
\begin{align*}
    \psi\phi&=(1\otimes \cdots \otimes 1\otimes \psi')\cdot (\phi'\otimes 1)
    =\tcR_{\cO_\qq(\sA_G), \cO_\qq(\sA_G)^{\bt K}}(\psi'\otimes \phi')
    =\phi'\otimes \psi'=\phi\psi. 
\end{align*}
Here the third equality follows from $\phi'\in (\cO_\qq(\sA_G)^{\bt K})^{\inv}$ and the definition of $\tcR_{\cO_\qq(\sA_G), \cO_\qq(\sA_G)^{\bt K}}$. 
\item[\underline{(2)}] This follows from the same argument as in the proof of (1).
\item[\underline{(3)}] Write $\phi=\sum_{i\in I} \xi_1^{(i)}(\lambda_1)\otimes \cdots \otimes \xi_K^{(i)}(\lambda_K)$ for some index set $I$. Then, 
\begin{align*}
    \phi \cdot \xi(\lambda)^{[1]}&=\left(\sum_{i\in I} \xi_1^{(i)}(\lambda_1)\otimes \cdots \otimes \xi_K^{(i)}(\lambda_K) \right)\cdot \left(\xi(\lambda)\otimes 1\otimes\cdots \otimes1 \right)\\
    &=\left((m_{\cO_\qq(\sA_G)}\otimes \id_{\cO_\qq(\sA_G)^{\bt (K-1)}})\circ (\id_{\cO_\qq(\sA_G)}\otimes \tcR_{\cO_\qq(\sA_G)^{\bt (K-1)}, \cO_\qq(\sA_G)}) \right)\\
    &\phantom{===================}\left(\sum_{i\in I} \xi_1^{(i)}(\lambda_1)\otimes \cdots \otimes \xi_K^{(i)}(\lambda_K) \otimes \xi(\lambda)\right)\\
    &=\qq^{(\lambda_1, \lambda)}\left((m_{\cO_\qq(\sA_G)}\otimes \id_{\cO_\qq(\sA_G)^{\bt (K-1)}})\right.\circ \\
    &\phantom{=}\left. (\tcR_{\cO_\qq(\sA_G), \cO_\qq(\sA_G)}\otimes \id_{\cO_\qq(\sA_G)^{\bt (K-1)}})\circ (\id_{\cO_\qq(\sA_G)}\otimes \tcR_{\cO_\qq(\sA_G)^{\bt (K-1)}, \cO_\qq(\sA_G)}) \right)\\
    &\phantom{===================}\left(\sum_{i\in I} \xi_1^{(i)}(\lambda_1)\otimes \cdots \otimes \xi_K^{(i)}(\lambda_K) \otimes \xi(\lambda)\right)\\
    &=\qq^{(\lambda_1, \lambda)}\left((m_{\cO_\qq(\sA_G)}\otimes \id_{\cO_\qq(\sA_G)^{\bt (K-1)}})\circ \tcR_{\cO_\qq(\sA_G)^{\bt K}, \cO_\qq(\sA_G)}\right)\\
    &\phantom{===================}\left(\sum_{i\in I} \xi_1^{(i)}(\lambda_1)\otimes \cdots \otimes \xi_K^{(i)}(\lambda_K) \otimes \xi(\lambda)\right)\\
    &=\qq^{(\lambda_1, \lambda)}\sum_{i\in I} \xi(\lambda)\xi_1^{(i)}(\lambda_1)\otimes \cdots \otimes \xi_K^{(i)}(\lambda_K)=\qq^{(\lambda_1, \lambda)}\xi(\lambda)^{[1]}\cdot \phi.    
\end{align*}
Here the third equality follows from \cref{lem:OA_comm}, the fourth equality follows from \cref{c:YB}, and the fifth equality follows from $\phi\in (\cO_\qq(\sA_G)^{\bt K})^{\inv}$ and the definition of $\tcR_{\cO_\qq(\sA_G)^{\bt K}, \cO_\qq(\sA_G)}$. 
\item[\underline{(4)}] This follows from the same argument as in the proof of (3).

 \end{proof}

The following theorem and corollaries provide the commutation relations among images of the ``disjoint'' embeddings of quantum configuration spaces.
\cref{fig:polygon_embedding} illustrates the geometric situations for \cref{thm:edge,cor:disjoint,cor:meet} below.

\begin{figure}[ht]
    \centering
\begin{tikzpicture}
\def\R{2}
\draw[thick] (0,0) circle(\R);
\foreach \i in {1,2,3,4,5,6,7}{
    \fill(141.42-51.42*\i:\R) circle(1.5pt) coordinate(A\i);
    \node at (141.42-51.42*\i:\R+0.3) {\scriptsize \i};
}
\filldraw[fill=myblue!15,draw=blue] (A1) -- (A2) -- (A3) -- (A4) --cycle;
\node at ($(A1)!0.6!(A3)$) {$p$};

\filldraw[fill=myblue!15,draw=blue] (A1) -- (A4) -- (A5) -- (A6) -- (A7) --cycle;
\node at ($(A1)!0.6!(A5)$) {$p'$};

\node at (0,-\R-1) {$K=4$, $K'=5$};

\begin{scope}[xshift=6cm]
\draw[thick] (0,0) circle(\R);
\foreach \i in {1,2,3,4,5,6,7}{
    \fill(141.42-51.42*\i:\R) circle(1.5pt) coordinate(A\i);
    \node at (141.42-51.42*\i:\R+0.3) {\scriptsize \i};
}
\filldraw[fill=myblue!15,draw=blue] (A1) -- (A2) -- (A3) -- (A4) --cycle;
\node at ($(A1)!0.6!(A3)$) {$p$};

\filldraw[fill=myblue!15,draw=blue] (A5) -- (A6) -- (A7) --cycle;
\node[left] at ($(A7)!0.5!(A5)$) {$p'$};

\node at (0,-\R-1) {$K=4$, $K'=3$};
\end{scope}

\begin{scope}[xshift=12cm]
\draw[thick] (0,0) circle(\R);
\foreach \i in {1,2,3,4,5,6,7}{
    \fill(141.42-51.42*\i:\R) circle(1.5pt) coordinate(A\i);
    \node at (141.42-51.42*\i:\R+0.3) {\scriptsize \i};
}
\filldraw[fill=myblue!15,draw=blue] (A1) -- (A2) -- (A3) -- (A4) --cycle;
\node at ($(A1)!0.6!(A3)$) {$p$};

\filldraw[fill=myblue!15,draw=blue] (A4) -- (A5) -- (A6) -- (A7) --cycle;
\node at ($(A7)!0.5!(A5)$) {$p'$};

\node at (0,-\R-1) {$K=4$, $K'=4$};
\end{scope}
\end{tikzpicture}
    \caption{The situations for \cref{thm:edge,cor:disjoint,cor:meet} from the left.}
    \label{fig:polygon_embedding}
\end{figure}

\begin{thm}\label{thm:edge}
Let $K, K'\in \Z_{\geq 2}$. Define $p\colon [1,K]\to[1,K+K'-2]$ and $p'\colon [1,K']\to [1,K+K'-2]$ by 
  \begin{align*}
      &p(i)=i\ \text{for}\ i\in [1, K],\\ 
      &p'(1)=1,\ p'(j)=K+j-2\ \text{for}\ j\in [2,K'].
  \end{align*}
 Then, for $\phi\in \cO^{\inv}(\lambda_1,\dots, \lambda_K)$ and $\phi'\in \cO^{\inv}(\lambda'_1,\dots, \lambda'_{K'})$, we have 
  \[
  \iota_{p}^{\inv}(\phi)\iota_{p'}^{\inv}(\phi')=\qq^{(\lambda_1, \lambda'_1)-(\lambda_K, \lambda'_2)}\iota_{p'}^{\inv}(\phi')\iota_p^{\inv}(\phi)
  \]
  in $\cO_\qq(\Conf_{K+K'-2}\sA_G)$.
\end{thm}
\begin{proof}
The element $\iota_{p'}^{\inv}(\phi')$ is of the following form.
\begin{align*}
    \iota_{p'}^{\inv}(\phi')&=\sum_{i\in I}\eta_1^{(i)}(\lambda'_1)\otimes \underset{K-2}{\underbrace{1\otimes \cdots \otimes 1}}\otimes  \eta_2^{(i)}(\lambda'_2)\otimes \cdots \otimes \eta_{K'}^{(i)}(\lambda_{K'})\\
    &=\sum_{i\in I}\eta_1^{(i)}(\lambda'_1)^{[1]}\eta_2^{(i)}(\lambda'_2)^{[K]}\eta_3^{(i)}(\lambda'_3)^{[K+1]} \cdots \eta_{K'}^{(i)}(\lambda_{K'})^{[K+K'-2]}.
\end{align*}
Here $I$ is some index set. Therefore, by \cref{lem:Conf_comm}, 
\begin{align*}
    \iota_p^{\inv}(\phi)\iota_{p'}^{\inv}(\phi')&=\iota_p^{\inv}(\phi)\left( \sum_{i\in I}\eta_1^{(i)}(\lambda'_1)^{[1]}\eta_2^{(i)}(\lambda'_2)^{[K]}\eta_3^{(i)}(\lambda'_3)^{[K+1]} \cdots \eta_{K'}^{(i)}(\lambda_{K'})^{[K+K'-2]}\right)\\
    &=\qq^{(\lambda_1, \lambda'_1)}\sum_{i\in I} \eta_1^{(i)}(\lambda'_1)^{[1]}\iota_p^{\inv}(\phi)\eta_2^{(i)}(\lambda'_2)^{[K]}\eta_3^{(i)}(\lambda'_3)^{[K+1]} \cdots \eta_{K'}^{(i)}(\lambda_{K'})^{[K+K'-2]}\\
    &=\qq^{(\lambda_1, \lambda'_1)-(\lambda_K, \lambda'_2)}\sum_{i\in I} \eta_1^{(i)}(\lambda'_1)^{[1]}\eta_2^{(i)}(\lambda'_2)^{[K]}\iota_p^{\inv}(\phi)\eta_3^{(i)}(\lambda'_3)^{[K+1]} \cdots \eta_{K'}^{(i)}(\lambda_{K'})^{[K+K'-2]}\\
    &=\qq^{(\lambda_1, \lambda'_1)-(\lambda_K, \lambda'_2)}\sum_{i\in I} \eta_1^{(i)}(\lambda'_1)^{[1]}\eta_2^{(i)}(\lambda'_2)^{[K]}\eta_3^{(i)}(\lambda'_3)^{[K+1]} \cdots \eta_{K'}^{(i)}(\lambda_{K'})^{[K+K'-2]}\iota_p^{\inv}(\phi)\\
    &=\qq^{(\lambda_1, \lambda'_1)-(\lambda_K, \lambda'_2)}\iota_{p'}^{\inv}(\phi')\iota_p^{\inv}(\phi).
\end{align*}
\end{proof}
The following corollaries can be regarded as special cases of \cref{thm:edge}. See \cref{fig:polygon_embedding}. 
\begin{cor}\label{cor:disjoint}
  Let $K, K'\in \Z_{\geq 2}$. Define $p\colon [1,K]\to[1,K+K']$ and $p'\colon [1,K']\to [1,K+K']$ by  
 \begin{align*}
      &p(i)=i\ \text{for}\ i\in [1,K],\\ 
      &p'(j)=K+j\ \text{for}\ j\in [1,K'].
  \end{align*}
  Then $\Ima \iota_p^{\inv}$ and $\Ima \iota_{p'}^{\inv}$ commute with each other in $\cO_\qq(\Conf_{K+K'}\sA_G)$. 
\end{cor}
\begin{cor}\label{cor:meet}
  Let $K, K'\in \Z_{\geq 2}$. Define $p\colon [1,K]\to [1, K+K'-1]$ and $p'\colon [1,K']\to [1,K+K'-1]$ by  
 \begin{align*}
      &p(i)=i\ \text{for}\ i\in [1, K],\\ 
      &p'(j)=K+j-1\ \text{for}\ j\in [1,K'].
  \end{align*}
   Then, for $\phi\in \cO^{\inv}(\lambda_1,\dots, \lambda_K)$ and $\phi'\in \cO^{\inv}(\lambda'_1,\dots, \lambda'_{K'})$, we have 
  \[
  \iota_{p}^{\inv}(\phi)\iota_{p'}^{\inv}(\phi')=\qq^{-(\lambda_K, \lambda'_1)}\iota_{p'}^{\inv}(\phi')\iota_p^{\inv}(\phi)
  \]
  in $\cO_\qq(\Conf_{K+K'-1}\sA_G)$.
\end{cor}

\section{Structure of the quantum configuration space \texorpdfstring{$\cO_\qq(\Conf_2\sA_G)$}{Oq(Conf2)}}
In this section, we first describe the algebraic structure of $\cO_\qq(\Conf_2\sA_G)$. We have 
\[
\Psi_2\colon\bigoplus_{\lambda_1,\lambda_2\in P_+}(V(\lambda_1)^{\ast}\ut V(\lambda_2)^{\ast})^{\inv}\xrightarrow{\sim}\cO_\qq(\Conf_2\sA_G),
\]
The standard fact in the representation theory of $\Uq$ gives 
\[
\dim_{\Bbbk}(V(\lambda_1)^{\ast}\ut V(\lambda_2)^{\ast})^{\inv}=
\begin{cases}
1&\text{if }\lambda_2=\lambda_1^{\ast},\\
0&\text{otherwise.}
\end{cases}
\]
Hence, by \cref{t:invstr} for $\A=\Bbbk$ (the trivial $\Uq$-module algebra), 
\begin{align*}
\cO_\qq(\Conf_2\sA_G)&=\Psi_2\left(\bigoplus_{\lambda_1,\lambda_2\in P_+}(V(\lambda_1)^{\ast}\ut V(\lambda_2)^{\ast})^{\inv}\right)\\
&=\Psi_2\left(\bigoplus_{\lambda\in P_+}(V(\lambda)^{\ast}\ut V(\lambda^{\ast})^{\ast})^{\inv}\right)=\bigoplus_{\lambda\in P_+}\Bbbk D_{\lambda}.
\end{align*}
Therefore, by \eqref{eq:D_e}, we obtain the following. 
\begin{thm}\label{t:conf2str}
    The quantum configuration space $\cO_\qq(\Conf_2\sA_G)$ is isomorphic to the polynomial ring $\Bbbk[D_{\varpi_s}\mid s\in S]$.
\end{thm}

\begin{rem}
If we further invert $D_{\varpi_s}$ for $s \in S$, then
\cref{t:conf2str} gives a quantum analogue of the isomorphism of algebraic varieties $H \xrightarrow{\sim} \Conf_2^\times \sA_G$, $h \mapsto [h.[U^+],\vw.[U^+]]$. The element $D_{\varpi_s} \in \cO_\qq(\Conf_2 \sA_G)$ is a quantum analogue of the function on $\Conf_2 \sA_G$ given by $[\flA_1,\flA_2] \mapsto h(\flA_1,\flA_2)^{\varpi_s}$. See \cref{lem:hw-distance} for the notation. 
\end{rem}
A more precise description of $D_{\lambda}$ will be useful also in the study of $\cO_\qq(\Conf_K\sA_G)$ for $K\in \Z_{\geq 2}$. Hence, the rest of this section is devoted to the proof of the following theorem.  
\begin{thm}\label{thm:D_lambda_extremal}
Let $\lambda\in P_+$. Then 
\begin{align*}
D_{\lambda}=\sum_{\mu\in W\lambda}(-1)^{\langle\rho^{\vee}, \lambda-\mu\rangle}q^{(\lambda, \lambda)/2+(\rho, \lambda-\mu)}\xi_{\mu}(\lambda)\otimes \xi_{-\mu}(\lambda^{\ast})+(\text{other terms}),
\end{align*}
where $($other terms$)$ belong to $\bigoplus_{\nu\in P; \nu\not\in W\lambda}\cO(\lambda)_{-\nu}\ut \cO(\lambda^{\ast})_{\nu}$. 
\end{thm}

For $s\in S$, denote by $T_s$ the $\Bbbk$-algebra automorphism on $\Uq$ given by 
\begin{align*}
T_{s}(K_{\alpha}) & =K_{r_s(\alpha)}\quad \text{for }\alpha\in Q,\\
T_{s}(E_{t}) & =\begin{cases}
-K_{s}^{-1}F_{s} & \text{for}\;t=s,\\
{\displaystyle \sum_{k+l=-c_{st}}(-1)^{k}\qq_{s}^{-k}E_{s}^{(k)}E_{t}E_{s}^{(l)}} & \text{for}\;t\neq s,
\end{cases}\\
T_{s}(F_{t}) & =\begin{cases}
-E_{s}K_{s} & \text{for}\;s=t,\\
{\displaystyle \sum_{k+l=-c_{st}}(-1)^{k}\qq_{s}^{k}F_{s}^{(l)}F_{t}F_{s}^{(k)}} & \text{for}\;t\neq s.
\end{cases}
\end{align*}
Here $X_s^{(k)} := X_s^k/[k]_{\qq_s}!$ for $s\in S, k \in \Z_{\geq 0}, X=E, F$. This automorphism corresponds to $T'_{s,-1}$ in \cite{Lusztig:Intro}. Let $V=\bigoplus_{\mu\in P}V_{\mu}$ be an integrable $\Uq$-module. For $s\in S$, there exists a $\Bbbk$-linear automorphism $T_{s}$ on $V$ given by 
\begin{align*}
 & T_{s}(v)=\sum_{a, b, c\in \Z_{\geq 0}; a-b+c=\langle \alpha_s^{\vee},\mu\rangle}(-1)^{b}\qq_{s}^{ac-b}F_{s}^{(a)}E_{s}^{(b)}F_{s}^{(c)}.v
\end{align*}
for $v\in V_{\mu}$. It is denoted by $T'_{i,-1}$ in \cite{Lusztig:Intro}. Note that 
\begin{align}
T_{s}(V_{\mu})=V_{r_s\mu}\label{eq:braid_action_weight}
\end{align}
for $\mu\in P$.
\begin{prop}[{\cite[Proposition 5.3.4, Proposition 37.1.2, Theorem 39.4.3]{Lusztig:Intro}}]\label{p:braid} Let $V, V'$ be  integrable $\Uq$-modules. 
\begin{itemize}
    \item[(1)] For $s\in S$, $X\in\Uq$ and $v\in V$, $T_{s}(X.v)=T_{s}(X).T_{s}(v)$.
    \item[(2)] For $s\in S$ and $v\in V, v'\in V'$, 
    \[
    T_s(v\otimes v')=\sum_{k\in \Z_{\geq 0}}(-1)^k\qq_s^{-k(k-1)/2}\left(\prod_{l=1}^k(\qq_s^{l}-\qq_s^{-l})\right)F_s^{(k)}.T_s(v)\otimes E_s^{(k)}.T_s(v')
    \]
    on $V\otimes V'$, and 
    \[
    T_s(v\otimes v')=\sum_{k\in \Z_{\geq 0}}(-1)^k\qq_s^{-k(k-1)/2}\left(\prod_{l=1}^k(\qq_s^{l}-\qq_s^{-l})\right)E_s^{(k)}.T_s(v)\otimes F_s^{(k)}.T_s(v')
    \]
    on $V\ut V'$. 
    \item[(3)] For $w\in W$ and $(s_1,\dots, s_{\ell})\in S(w)$, the composite map $T_{w}:=T_{s_{1}}\cdots T_{s_{\ell}}\colon \Uq\to \Uq$ and $T_{w}:=T_{s_{1}}\cdots T_{s_{\ell}}\colon V\to V$ do not depend on the choice of $(s_1,\dots, s_{\ell})\in S(w)$.
\end{itemize}
\end{prop}
Let $V, V'$ be integrable $\Uq$-modules. For $w\in W$ and $v\in V, v'\in V'$, we have
\begin{align}
&T_{w}(v\otimes v')=\sum_{\bm{k}=(k_1,\dots, k_\ell)\in \Z_{\geq 0}^\ell}
(-1)^{k_1+\cdots +k_\ell}a_{\bm{s}}(\bm{k})
F(\bm{k}, \bm{s}).T_{w}(v)\otimes E(\bm{k}, \bm{s}).T_{w}(v'),\label{eq:w_t_action}
\end{align}
on $V\otimes V'$, and 
\begin{align}
&T_{w}(v\otimes v')=\sum_{\bm{k}=(k_1,\dots, k_\ell)\in \Z_{\geq 0}^\ell}
(-1)^{k_1+\cdots +k_\ell}a_{\bm{s}}(\bm{k})
E(\bm{k}, \bm{s}).T_{w}(v)\otimes F(\bm{k}, \bm{s}).T_{w}(v'),\label{eq:w_ut_action}
\end{align}
on $V\ut V'$, where $\bm{s}=(s_1,\dots, s_\ell)\in S(w)$, and 
\begin{align*}
    &a_{\bm{s}}(\bm{k}):=\prod_{i=1}^\ell\left(\qq_{s_i}^{-k_i(k_i-1)/2}\prod_{l_i=1}^{k_i}(\qq_{s_i}^{l_i}-\qq_{s_i}^{-l_i})\right),\\
    &E(\bm{k}, \bm{s}):=E_{s_1}^{(k_1)}T_{s_1}(E_{s_2}^{(k_2)})\cdots T_{s_1}\cdots T_{s_{\ell-1}}(E_{s_\ell}^{(k_\ell)}),\\
    &F(\bm{k}, \bm{s}):=F_{s_1}^{(k_1)}T_{s_1}(F_{s_2}^{(k_2)})\cdots T_{s_1}\cdots T_{s_{\ell-1}}(F_{s_\ell}^{(k_\ell)}).
\end{align*}
by \cref{p:braid}. By \cite[Section 37.2.4, Proposition 40.2.1]{Lusztig:Intro}, $E(\bm{k}, \bm{s})\in \Uq^+$ and $F(\bm{k}, \bm{s})\in \Uq^-$ for $\bm{k}\in \Z_{\geq 0}^\ell$. 
\begin{lem}[{\cite[Proposition 5.2.2]{Lusztig:Intro}}]\label{lem:extremal_braid}
For $\lambda\in P_+$ and $w\in W$,   
    \[
    T_s(v_{w\lambda})=
    \begin{cases}
        v_{r_sw\lambda}&\text{if }\ell(r_sw)=\ell(w)+1,\\
        (-1)^{\langle \rho^{\vee}, w\lambda-r_sw\lambda\rangle}\qq^{(\rho, w\lambda-r_sw\lambda)}v_{r_sw\lambda}&\text{if }\ell(r_sw)=\ell(w)-1.
    \end{cases}
    \]
\end{lem}
\begin{lem}\label{lem:xi_braid}
    Let $\lambda\in P_+$ and $w\in W$. Then, 
    \[
    T_s(\xi_{w\lambda})=
    \begin{cases}
        \qq^{-2( \rho, w\lambda-r_sw\lambda)}\xi_{r_sw\lambda}&\text{if }\ell(r_sw)=\ell(w)+1,\\
        (-1)^{-\langle \rho^{\vee}, w\lambda-r_sw\lambda\rangle}\qq^{-( \rho, w\lambda-r_sw\lambda)}\xi_{r_sw\lambda}&\text{if }\ell(r_sw)=\ell(w)-1.
    \end{cases}.
    \]
     Moreover, 
     \[
     T_{w^{-1}}(\xi_{w\lambda})=(-1)^{-\langle\rho^{\vee}, w\lambda-\lambda\rangle}\qq^{-(\rho, w\lambda-\lambda)}\xi_{\lambda},\quad T_{w^{-1}}(\xi_{-w\lambda})=\qq^{2(\rho, w\lambda-\lambda)}\xi_{-\lambda}. 
    \]
\end{lem}
\begin{proof}
In this proof, we write 
\[
m:=\langle\alpha_s^{\vee}, w\lambda\rangle=\frac{1}{2}\langle \alpha_s^{\vee}, w\lambda-r_sw\lambda\rangle=\langle \rho^{\vee}, w\lambda-r_sw\lambda\rangle
=\frac{2}{(\alpha_s, \alpha_s)}( \rho, w\lambda-r_sw\lambda).
\]
When $\ell(r_sw)=\ell(w)+1$ (resp.~$\ell(r_sw)=\ell(w)-1$), $m\geq 0$ (resp.~$\leq 0$), hence $F_s.\xi_{w\lambda}=0$ (resp.~$E_s.\xi_{w\lambda}=0$). 
Therefore, by \cite[Proposition 5.2.2]{Lusztig:Intro}, 
\[
(V(\lambda)^{\ast})_{-r_sw\lambda}\ni T_s(\xi_{w\lambda})=
\begin{cases}
(-1)^{m}E_s^{(m)}\qq_{s}^{-m}.\xi_{w\lambda}&\text{if }\ell(r_sw)=\ell(w)+1,\\
F_s^{(-m)}.\xi_{w\lambda}&\text{if }\ell(r_sw)=\ell(w)-1.
\end{cases}
\]
Since $(V(\lambda)^{\ast})_{-r_sw\lambda}=\Bbbk\xi_{r_sw\lambda}$, $T_s(\xi_{w\lambda})=a\xi_{r_sw\lambda}$ for some $a\in \Bbbk$. 
When $\ell(r_sw)=\ell(w)+1$, 
\begin{align*}
    \langle E_s^{(m)}.\xi_{w\lambda}, v_{r_sw\lambda}\rangle&
    =\langle \xi_{w\lambda}, \sfS(E_s^{(m)}).v_{r_sw\lambda}\rangle\\
    &=(-1)^{m}\qq_s^{-m(m+1)}\langle \xi_{w\lambda}, E_s^{(m)}K_s^{-m}.v_{r_sw\lambda}\rangle\\
    &=(-1)^{m}\qq_s^{-m(m+1)+m^2}\langle \xi_{w\lambda}, E_s^{(m)}.v_{r_sw\lambda}\rangle\\
    &=(-1)^{m}\qq_s^{-m}\langle \xi_{w\lambda}, v_{w\lambda}\rangle=(-1)^{m}\qq_s^{-m}. 
\end{align*}
Therefore, we have $T_s(\xi_{w\lambda})=\qq_s^{-2m}\xi_{r_sw\lambda}$. Similarly, we can show $T_s(\xi_{w\lambda})=(-1)^{-m}\qq_s^{-m}\xi_{r_sw\lambda}$ when $\ell(r_sw)=\ell(w)-1$. 

Let $(s_{\ell},\dots, s_1)\in S(w)$. Then, $T_{w^{-1}}=T_{s_1}\cdots T_{s_{\ell}}$. Hence, by the above results, $T_{w^{-1}}(\xi_{w\lambda})=(-1)^A\qq^B\xi_{w_0\lambda}$, $T_{w^{-1}}(\xi_{-w\lambda})=\qq^C\xi_{\lambda}$, where
 \begin{align*}
   A&=-\sum_{i=1}^{\ell}\langle \rho^{\vee}, r_{s_{i}}\cdots r_{s_1}\lambda-r_{s_{i-1}}\cdots r_{s_1}\lambda \rangle 
     =-\langle \rho^{\vee}, w\lambda-\lambda\rangle,\\
     B&=-\sum_{i=1}^{\ell}(\rho, r_{s_{i}}\cdots r_{s_1}\lambda-r_{s_{i-1}}\cdots r_{s_1}\lambda)
     =-(\rho, w\lambda-\lambda),\\
     C&=-2\sum_{i=1}^{\ell}(\rho, -r_{s_i}\cdots r_{s_1}\lambda +r_{s_{i-1}}\cdots r_{s_1}\lambda)
     =2(\rho, w\lambda-\lambda).
 \end{align*} 
\end{proof}
\begin{proof}[{Proof of \cref{thm:D_lambda_extremal}}]
Write 
\[
D_{\lambda}=\sum_{\mu\in W\lambda}a_{\mu}\xi_{\mu}(\lambda)\otimes \xi_{-\mu}(\lambda^{\ast})+(\text{other terms}),
\]
where $a_{\lambda}\in \Bbbk$ and $($other terms$)$ belong to $\bigoplus_{\nu\in P; \nu\not\in W\lambda}\cO(\lambda)_{\nu}\ut \cO(\lambda^{\ast})_{-\nu}$. Since $D_{\lambda}\in \cO^{\inv}(\lambda, \lambda^{\ast})$, we have $T_{w^{-1}}(D_{\lambda})=D_{\lambda}$ for $w\in W$. On the other hand, by \eqref{eq:w_ut_action}, 
\begin{align*}
T_{w^{-1}}(D_{\lambda})=\sum_{\bm{k}=(k_1,\dots, k_{\ell})\in \Z_{\geq 0}^{\ell}}
(-1)^{k_1+\cdots +k_\ell}a_{\bm{s}}(\bm{k})
(E(\bm{k}, \bm{s})\otimes  F(\bm{k}, \bm{s}))(T_{w^{-1}}\otimes T_{w^{-1}})(D_{\lambda})
\end{align*}
where $\bm{s}=(s_1,\dots, s_\ell)\in S(w^{-1})$. Therefore, by weight reason, 
\[
a_{w\lambda}(T_{w^{-1}}(\xi_{w\lambda}))(\lambda)\otimes (T_{w^{-1}}(\xi_{-w\lambda}))(\lambda^{\ast})=a_{\lambda}\xi_{\lambda}(\lambda)\otimes \xi_{-\lambda}(\lambda^{\ast}). 
\]
\cref{lem:xi_braid} implies that the left-hand side is equal to 
\[
(-1)^{-\langle\rho^{\vee}, w\lambda-\lambda\rangle}\qq^{(\rho, w\lambda-\lambda)}
a_{w\lambda}\xi_{\lambda}(\lambda)\otimes \xi_{-\lambda}(\lambda^{\ast}).
\]
Therefore, $a_{w\lambda}=(-1)^{\langle\rho^{\vee}, \lambda-w\lambda\rangle}\qq^{(\rho, \lambda-w\lambda)}a_{\lambda}=(-1)^{\langle\rho^{\vee}, \lambda-w\lambda\rangle}\qq^{(\lambda, \lambda)/2+(\rho, \lambda-w\lambda)}$. 
\end{proof}
\begin{cor}\label{cor:sigma_2}
For $\lambda\in P_+$, $\sigma_2(D_{\lambda})=D_{\lambda^{\ast}}$. 
\end{cor}
\begin{proof}
By \cref{thm:D_lambda_extremal}, 
\[
D_{\lambda}=\qq^{(\lambda, \lambda)/2}\xi_{\lambda}(\lambda)\otimes \xi_{-\lambda}(\lambda^{\ast})+(-1)^{2\langle\rho^{\vee}, \lambda\rangle}\qq^{(\lambda, \lambda)/2+2(\rho, \lambda)}\xi_{-\lambda^{\ast}}(\lambda)\otimes \xi_{\lambda^{\ast}}(\lambda^{\ast})+\sum_{i\in I}\xi_i(\lambda)\otimes \xi_i'(\lambda^{\ast}),
\]
where $I$ is some index set, $\xi_i$ and $\xi_i'$ are weight vectors such that $\wt \xi_i\neq -\lambda, \lambda^{\ast}$ and $\wt \xi'_i\neq \lambda, -\lambda^{\ast}$. Hence, 
\begin{align*}
    &\sigma_2(D_{\lambda})\\
    &=(-1)^{2\langle\rho^{\vee}, \lambda\rangle}\qq^{2(\rho, \lambda)+(\lambda, \lambda)/2}\xi_{-\lambda}(\lambda^{\ast})\otimes \xi_{\lambda}(\lambda)\\
    &\phantom{=}
    +(-1)^{2\langle\rho^{\vee}, -\lambda^{\ast}\rangle+2\langle\rho^{\vee}, \lambda\rangle}\qq^{2(\rho, -\lambda^{\ast})+(\lambda, \lambda)/2+2(\rho, \lambda)}\xi_{\lambda^{\ast}}(\lambda^{\ast})\otimes \xi_{-\lambda^{\ast}}(\lambda)\\
    &\phantom{=}+\sum_{i\in I}(-1)^{2\langle\rho^{\vee}, \wt \xi_i'\rangle}\qq^{2(\rho, \wt \xi_i')}\xi_i'(\lambda^{\ast})\otimes \xi_i(\lambda) \\
    &=\qq^{(\lambda^{\ast}, \lambda^{\ast})/2}\xi_{\lambda^{\ast}}(\lambda^{\ast})\otimes \xi_{-\lambda^{\ast}}(\lambda)+(-1)^{2\langle\rho^{\vee}, \lambda\rangle}\qq^{2(\rho, \lambda)+(\lambda, \lambda)/2}\xi_{-\lambda}(\lambda^{\ast})\otimes \xi_{\lambda}(\lambda)\\
    &\phantom{=}   +\sum_{i\in I}(-1)^{2\langle\rho^{\vee}, \wt \xi_i'\rangle}\qq^{2(\rho, \wt \xi_i')}\xi_i'(\lambda^{\ast})\otimes \xi_i(\lambda).
\end{align*}
This is an element of $\cO^{\inv}(\lambda^{\ast}, \lambda)$, hence the uniqueness of $D_{\lambda^{\ast}}$ implies $\sigma_2(D_{\lambda})=D_{\lambda^{\ast}}$. 
\end{proof}
\begin{cor}\label{cor:D_lambda}
Let $\A$ be a locally finite $\Uq$-module algebra. Then, for $\lambda\in P_+$, the element $D_{\lambda}$ in \cref{t:invstr} is of the form
\begin{align*}
D_{\lambda}=\sum_{\mu\in W\lambda}(-1)^{\langle\rho^{\vee}, \lambda-\mu\rangle}q^{(\lambda, \lambda)/2+(\rho, \lambda-\mu)}\xi_{\mu}(\lambda)\otimes 1\otimes \xi_{-\mu}(\lambda^{\ast})+(\text{other terms}),
\end{align*}
where $($other terms$)$ belong to $\bigoplus_{\nu\in P; \nu\not\in W\lambda}\cO(\lambda)_{-\nu}\ut 1\ut \cO(\lambda^{\ast})_{\nu}$. 
\end{cor}
\section{Quantum Wilson lines}
We have an injective morphism 
\[
G\times H\times H \to \Conf_4\sA_G,\ (g, h, h')\mapsto [[U^+], gh.[U^+], g\overline{w_0}.[U^+], \overline{w_0}{h'}^{\ast}.[U^+]],
\]
which induces an isomorphism of algebraic varieties 
\[
\imath_{4}\colon G\times H\times H\xrightarrow{\sim} \Conf_4^{\{2, 4\}}\sA_G.
\]
See \cref{fig:band_config}. 
In this subsection, we construct a quantum analogue of this morphism. 
The morphism 
\begin{align}\label{eq:Wilson_line}
\Conf_4^{\{2, 4\}}\sA_G \xrightarrow{\imath_{4}^{-1}} G\times H\times H\xrightarrow{\text{projection}} G
\end{align}
is called a \emph{Wilson line} in \cite{IOS}.

\begin{figure}[ht]
    \centering
\begin{tikzpicture}[rotate=90]
\path (3,3) coordinate(A);
\path (0,3) coordinate(B);
\path (0,0) coordinate(C);
\path (3,0) coordinate(D);
\bline{-0.5,0}{3.5,0}{0.15};
\tline{-0.5,3}{3.5,3}{0.15};
\draw[red,->-,thick] (1.5,3) --node[midway,above]{$g$} (1.5,0);
\filldraw(A) circle(1.5pt) node[left=0.3em,anchor=east]{$[U^+]=\flA_1$}; 
\filldraw(B) circle(1.5pt) node[left=0.3em,anchor=east]{$\vw {h'}^{\ast}.[U^+]=\flA_4$};
\filldraw(D) circle(1.5pt) node[right=0.3em,anchor=west]{$\flA_2=gh.[U^+]$};
\filldraw(C) circle(1.5pt) node[right=0.3em,anchor=west]{$\flA_3=g\vw.[U^+]$};
\begin{pgfonlayer}{bg}  
\filldraw[fill=myblue!15,draw=blue,dashed,thick] (B) -- (C) -- (D) -- (A) --cycle;
\end{pgfonlayer}
\end{tikzpicture}
    \caption{A parameterization of $\Conf_4^{\{2,4\}} \sA_G$.}
    \label{fig:band_config}
\end{figure}
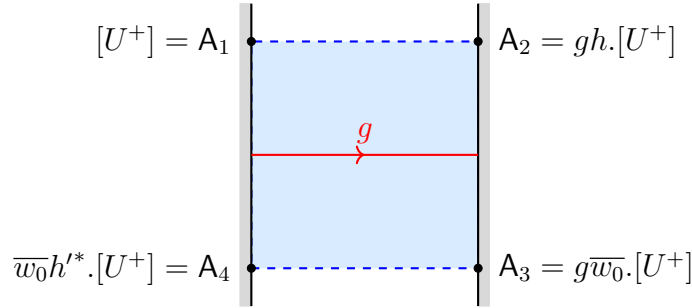

By \cref{t:invstr} (see \eqref{eq:Phi_ke}), there exists an injective algebra homomorphism
\[
\Phi_4^{\ext}\colon\cO_\qq(\Conf_4\sA_G)\to (\cO_\qq(\sA_G)\bt \cO_\qq(\sA_G))^{\ext},
\]
which induces the isomorphism 
\begin{align*}
\Phi_4^{\ext}\colon\cO_\qq(\Conf_4^{\{4\}}\sA_G)\xrightarrow{\sim} (\cO_\qq(\sA_G)\bt \cO_\qq(\sA_G))^{\ext}.
\end{align*}
We are going to relate $\cO_\qq(\sA_G)\bt \cO_\qq(\sA_G)$ with $\cO_\qq(G)$. 
For $\lambda\in P_+$ and $\xi\in V(\lambda^{\ast})^{\ast}$, we write
\[
\xi(-\lambda):=c^{V(\lambda^{\ast})}(\xi, v_{-\lambda})\in \cO_\qq(\sA_G^-)(\subset \cO_\qq(G)).
\]
Denote by $\Psi_1^-$ the isomorphism of $\Uq$-modules 
\[
\bigoplus_{\lambda\in P_+}V(\lambda^{\ast})^{\ast}\xrightarrow{\sim} \cO_\qq(\sA_G^-),\quad  V(\lambda^{\ast})^{\ast}\ni \xi\mapsto \xi(-\lambda)\in \cO_\qq(\sA_G^-). 
\]
We will write 
\begin{align}
    \cO^-(-\lambda):=\Psi_1^-(V(\lambda^{\ast})^{\ast}).\label{eq:Ominus_notation}
\end{align}
The $\Uq$-module isomorphism 
\begin{align}
\Psi^{+,-}:=\Psi_1^-\circ \Psi_1^{-1}\colon \cO_\qq(\sA_G)\to \cO_\qq(\sA_G^-),\quad  \xi(\lambda)\mapsto \xi(w_0\lambda)\ \text{for}\ \xi\in V(\lambda)^{\ast}\label{eq:Psi_+-}
\end{align}
is an isomorphism of $\Uq$-module algebras. Hence we have an isomorphism of $\Uq$-module algebras
\[
\mathrm{id}\otimes \Psi^{+,-}\colon \cO_\qq(\sA_G)\bt \cO_\qq(\sA_G)\xrightarrow{\sim} \cO_\qq(\sA_G)\bt \cO_\qq(\sA_G^-).
\]
Let $\sfm\colon \cO_\qq(\sA_G)\bt \cO_\qq(\sA_G^-)\to \cO_\qq(G)$ be the $\Bbbk$-linear map defined by 
\[
\xi_1(\lambda_1)\otimes \xi_2(-\lambda_2)\mapsto \qq^{-(\lambda_1, \lambda_1)/2}\xi_1(\lambda_1)\xi_2(-\lambda_2)
\]
for $\xi_1\in V(\lambda_1)^{\ast}, \xi_2\in V(\lambda_2^{\ast})^{\ast}$. It is known that $\sfm$ is surjective (cf.~\cite[Proposition 9.2.2]{Joseph:book}). 
\begin{prop}\label{p:reltoG}
For $\psi\in \cO(\lambda_1)\ut \cO^-(-\lambda_2)$, $\psi'\in \cO(\lambda'_1)\ut \cO^-(-\lambda_2')$, we have 
\[
\sfm(\psi\psi')=\qq^{(\lambda'_1, -\lambda_1+\lambda_2)}\sfm(\psi)\sfm(\psi').
\]
\end{prop}
\begin{proof}
For $\xi_1\in V(\lambda_1)^{\ast}, \xi_2\in V(\lambda_2^{\ast})^{\ast}, \xi'_1\in V(\lambda'_1)^{\ast}, \xi'_2\in V(\lambda_2^{\prime\ast})^{\ast}$,  
\begin{align*}
    &\sfm(\xi_1(\lambda_1)^{[1]}\xi_2(-\lambda_2)^{[2]}\xi'_1(\lambda'_1)^{[1]} \xi'_2(-\lambda'_2)^{[2]})\\
    &=\qq^{-(\lambda_1+\lambda'_1, \lambda_1+\lambda'_1)/2}\xi_1(\lambda_1)c^{V(\lambda'_1)\otimes V(\lambda_2^{\ast})}(\tcR_{V(\lambda_2^{\ast})^{\ast}, V(\lambda'_1)^{\ast}}(\xi_2\otimes \xi'_1), v_{\lambda'_1}\otimes v_{-\lambda_2})\xi'_2(-\lambda'_2).
\end{align*}
For $X\in \Uq$, 
\begin{align*}
 &\langle c^{V(\lambda'_1)\otimes V(\lambda_2^{\ast})}(\tcR_{V(\lambda_2^{\ast})^{\ast}, V(\lambda'_1)^{\ast}}(\xi_2\otimes \xi'_1), v_{\lambda'_1}\otimes v_{-\lambda_2}), X\rangle\\
    &=\langle \Upsilon_{V(\lambda'_1), V(\lambda_2^{\ast})}(\tcR_{V(\lambda_2^{\ast})^{\ast}, V(\lambda'_1)^{\ast}}(\xi_2\otimes \xi'_1)), \Delta(X).v_{\lambda'_1}\otimes v_{-\lambda_2}\rangle\\
    &=\langle \Upsilon_{V(\lambda_2^{\ast}), V(\lambda'_1)}(\xi_2\otimes \xi'_1),\cR_{V(\lambda'_1), V(\lambda_2^{\ast})}(\Delta(X).v_{\lambda'_1}\otimes v_{-\lambda_2})\rangle\ \text{by \cref{p:tildeR}}\\
    &=\langle \Upsilon_{V(\lambda_2^{\ast}), V(\lambda'_1)}(\xi_2\otimes \xi'_1),\Delta(X).\cR_{V(\lambda'_1), V(\lambda_2^{\ast})}(v_{\lambda'_1}\otimes v_{-\lambda_2})\rangle\\
    &=\qq^{(\lambda_2, \lambda'_1)}\langle \Upsilon_{V(\lambda_2^{\ast}), V(\lambda'_1)}(\xi_2\otimes \xi'_1),\Delta(X).v_{-\lambda_2}\otimes v_{\lambda'_1}\rangle\\
    &=\qq^{(\lambda_2, \lambda'_1)}\langle \xi_2(-\lambda_2)\xi'_1(\lambda'_1), X\rangle.
\end{align*}
Therefore, 
\begin{align*}
    &\sfm(\xi_1(\lambda_1)^{[1]}\xi_2(-\lambda_2)^{[2]}\xi'_1(\lambda'_1)^{[1]} \xi'_2(-\lambda'_2)^{[2]})\\
    &=\qq^{-(\lambda_1+\lambda'_1, \lambda_1+\lambda'_1)/2+(\lambda_2, \lambda'_1)}\xi_1(\lambda_1)\xi_2(-\lambda_2)\xi'_1(\lambda'_1)\xi'_2(-\lambda'_2)\\
    &=\qq^{-(\lambda_1+\lambda'_1, \lambda_1+\lambda'_1)/2+(\lambda_2, \lambda'_1)+(\lambda_1, \lambda_1)/2+(\lambda'_1, \lambda'_1)/2}\sfm(\xi_1(\lambda_1)^{[1]}\xi_2(-\lambda_2)^{[2]})\sfm(\xi'_1(\lambda'_1)^{[1]}\xi'_2(-\lambda'_2)^{[2]})\\
    &=\qq^{(\lambda'_1, -\lambda_1+\lambda_2)}\sfm(\xi_1(\lambda_1)^{[1]}\xi_2(-\lambda_2)^{[2]})\sfm(\xi'_1(\lambda'_1)^{[1]}\xi'_2(-\lambda'_2)^{[2]}).
\end{align*}
    \end{proof}
    We define the $\Bbbk$-algebra  ${}^{\ext'}\cO_\qq(G)$ as follows. 
\begin{itemize}
    \item ${}^{\ext'}\cO_\qq(G)=\Bbbk[P]\otimes \cO_\qq(G)$ as a vector space. Here $\Bbbk[P]=\oplus_{\lambda\in P}\Bbbk e'(\lambda)$ is the group algebra of the additive group $P$. 
    \item The multiplication on ${}^{\ext'}\cO_\qq(G)$ is determined by the following conditions. 
    \begin{itemize}
        \item $\cO_\qq(G)\to {}^{\ext'}\cO_\qq(G), \phi\mapsto 1\otimes \phi$ and $\Bbbk [P]\to {}^{\ext'}\cO_\qq(G), x\mapsto x\otimes 1$ are injective $\Bbbk$-algebra homomorphisms.
        \item For $\lambda\in P$ and $\phi\in \cO_\qq(G)$ satisfying $(1\otimes K_s)\cdot \phi=\qq_s^{\langle \alpha_s^{\vee}, \nu\rangle}\phi$ for $s\in S$, 
        \[
        (e'(\lambda)\otimes 1)(1\otimes \phi)=e'(\lambda)\otimes \phi=\qq^{(\lambda, \nu)}(1\otimes \phi)(e'(\lambda)\otimes 1).
        \]
    \end{itemize}
\end{itemize}
In the following, we simply write $x\otimes \phi$ as $x\phi$ in ${}^{\ext'}\cO_\qq(G)$ for $x\in \Bbbk [P]$ and $\phi\in \cO_\qq(G)$. When $\phi\in \cO_\qq(G)$ satisfies $(1\otimes K_s)\cdot \phi=\qq_s^{\langle \alpha_s^{\vee}, \nu\rangle}\phi$ for $s\in S$, we write 
\[
\wt_r\phi=\nu.
\]
\begin{thm}\label{t:Gext}
Define a $\Bbbk$-linear map
\begin{align*}
    \wtmu\colon \cO_\qq(\sA_G)\bt \cO_\qq(\sA_G^-)\to {}^{\ext'}\cO_\qq(G)
\end{align*}
by 
\[
\wtmu(\psi)=e'(\lambda_1)\sfm(\psi)
\]
for $\psi\in \cO(\lambda_1)\ut  \cO^-(-\lambda_2)$. Then 
$\wtmu$ is an injective $\Bbbk$-algebra homomorphism. Moreover, for $\lambda\in P_+$, there uniquely exists an element 
\[
D'_{\lambda}\in (\cO(\lambda)\ut \cO^-(-\lambda))^{\inv}\subset \cO_\qq(\sA_G)\bt \cO_\qq(\sA_G^-)
\]
of the form 
\begin{align*}
D'_{\lambda}=\qq^{(\lambda, \lambda)/2}\xi_{\lambda}(\lambda)\otimes \xi_{-\lambda}(-\lambda)+(\text{other terms}),
\end{align*}
here $($other terms$)$ belong to $\cO(\lambda)_{>-\lambda}\ut \cO^-(-\lambda)_{<\lambda}$. Then $\wtmu(D'_{\lambda})=e'(\lambda)$ and  $\cD':= \{\qq^aD'_{\lambda}\mid a\in \frac{1}{2d}\Z, \lambda\in P_+\}$ forms an Ore set of $\cO_\qq(\sA_G)\bt \cO_\qq(\sA_G^-)$ and $\wtmu$ can be extended to an isomorphism of $\Bbbk$-algebras
\[
\wtmu\colon \cO_\qq(\sA_G)\bt \cO_\qq(\sA_G^-)[\cD^{\prime -1}]\xrightarrow{\sim} {}^{\ext'}\cO_\qq(G).
\]
\end{thm}
\begin{proof}
    For $\psi\in \cO(\lambda_1)\ut \cO^-(-\lambda_2)$ and $\psi'\in \cO(\lambda'_1)\ut \cO^-(-\lambda'_2)$,  
    \begin{align*}
    \wtmu(\psi)\wtmu(\psi')
    &=e'(\lambda_1)\sfm(\psi)e'(\lambda'_1)\sfm(\psi')\\
    &=\qq^{-(\lambda'_1, \lambda_1-\lambda_2)}e'(\lambda_1+\lambda'_1)\sfm(\psi)\sfm(\psi')\\
    &=e'(\lambda_1+\lambda'_1)\sfm(\psi\psi')\ \text{by \cref{p:reltoG}}\\
    &=\wtmu(\psi\psi').
    \end{align*}
    Note that $\wt_r \mu(\psi)=\lambda_1-\lambda_2$. Therefore $\wtmu$ is a $\Bbbk$-algebra homomorphism. Next, we prove the injectivity of $\wtmu$. By definition, 
    \[
    \wtmu(\cO(\lambda_1)\ut \cO^-(-\lambda_2))\subset \{e'(\lambda_1)\phi\mid \phi\in \cO_\qq(G), \wt_r \mu(\psi)=\lambda_1-\lambda_2\}.
    \]
 Therefore, 
    \begin{align}
\wtmu(\cO_\qq(\sA_G)\bt \cO_\qq(\sA_G^-))&=\sum_{\lambda_1, \lambda_2\in P_+}\wtmu(\cO(\lambda_1)\ut \cO^-(-\lambda_2))\notag\\
&=\bigoplus_{\lambda_1, \lambda_2\in P_+}\wtmu(\cO(\lambda_1)\ut \cO^-(-\lambda_2)).\label{eq:direct_sum}
    \end{align}
    Moreover, for any non-zero element $\widetilde{\xi}\in V(\lambda_1)^{\ast}\ut V(\lambda_2^{\ast})^{\ast}$, $c^{V(\lambda_1)\otimes V(\lambda_2^{\ast})}(\Upsilon_{V(\lambda_1), V(\lambda_2^{\ast})}(\widetilde{\xi}), v_{\lambda_1}\otimes v_{-\lambda_2})$ is a non-zero element in $\cO_\qq(G)$ since $v_{\lambda_1}\otimes v_{-\lambda_2}$ generates $V(\lambda_1)\otimes V(\lambda_2^{\ast})$ as a $\Uq$-module. Hence, $\sfm$ is injective on $\cO(\lambda_1)\ut \cO^-(-\lambda_2)$. Therefore, $\wtmu$ is also injective by \eqref{eq:direct_sum}. 

   The existence of $D'_{\lambda}$ for $\lambda\in P_+$ is a standard fact of the representation theory of $\Uq$. For $\lambda\in P_+$, $\sfm(D'_{\lambda})=1$. Hence $\wtmu(D'_{\lambda})=e'(\lambda)$. Therefore, for $\psi\in \cO(\lambda_1)\ut \cO^-(-\lambda_2)$, 
\[
D'_{\lambda}\psi=\qq^{(\lambda, \lambda_1-\lambda_2)}\psi D'_{\lambda}.
\]
Therefore, $\cD'$ forms an Ore set of $\cO_\qq(\sA_G)\bt \cO_\qq(\sA_G^-)$ and $\wtmu$ can be extended to an injective $\Bbbk$-algebra homomorphism
\[
\wtmu\colon \cO_\qq(\sA_G)\bt \cO_\qq(\sA_G^-)[\cD^{\prime -1}]\to {}^{\ext'}\cO_\qq(G).
\]
The surjectivity of this homomorphism follows from that of $\sfm$. 
\end{proof}

\cref{cor:D_lambda} and \eqref{eq:q-cyclic_shift} imply that  
\[
D_{2; \lambda}=\qq^{(\lambda,\lambda)/2}\cdot 1\otimes \xi_{\lambda^{\ast}}(\lambda^{\ast})\otimes \xi_{w_0\lambda}(\lambda)\otimes 1+(\text{other terms}),
\]
where $($other terms$)$ belong to $1\ut \cO(\lambda^{\ast})_{>-\lambda^{\ast}}\ut \cO(\lambda)_{<\lambda^{\ast}}\ut 1$. Therefore, by \eqref{eq:Phi_ke},
\begin{align}
((\mathrm{id}\otimes \Psi^{+,-})\otimes \mathrm{id})(\Phi_4^{\ext}(D_{2; \lambda}))=D'_{\lambda^{\ast}}\in (\cO_\qq(\sA_G)\bt \cO_\qq(\sA_G^-))^{\ext}.\label{eq:D_2_D_prime}
\end{align}

Therefore, we have the following theorem. 
\begin{thm}\label{thm:qW}
    The $\Bbbk$-algebra homomorphism
    \begin{align*}
    \cO_\qq(\Conf_4\sA_G)&\xrightarrow{\Phi_4^{\ext}} (\cO_\qq(\sA_G)\bt \cO_\qq(\sA_G))^{\ext}\\
    &\xrightarrow{(\mathrm{id}\otimes \Psi^{+,-})\otimes \mathrm{id}} (\cO_\qq(\sA_G)\bt \cO_\qq(\sA_G^-))^{\ext}\xrightarrow{\wtmu\otimes \mathrm{id}} ({}^{\ext'}\cO_\qq(G))^{\ext}
    \end{align*}
    induces a $\Bbbk$-algebra isomorphism
    \[
    \imath_{4,\qq}\colon\cO_\qq(\Conf_4^{\{2,4\}}\sA_G)\xrightarrow{\sim}({}^{\ext'}\cO_\qq(G))^{\ext}.
    \]
    Here $({}^{\ext'}\cO_\qq(G))^{\ext}$ is defined with respect to the weight space decomposition of ${}^{\ext'}\cO_\qq(G)$ induced from the $\Uq^0\otimes 1$-module structure on $\cO_\qq(G)$ and $\wt (\Bbbk[P])=0$. 
\end{thm}
In the following, we simply write $({}^{\ext'}\cO_\qq(G))^{\ext}$ as ${}^{\ext'}\cO_\qq(G)^{\ext}$. The isomorphism $\imath_{4,\qq}$ gives an injective $\Bbbk$-algebra homomorphism 
    \[
    \cO_\qq(G)\xrightarrow{\mathrm{inclusion}} {}^{\ext'}\cO_\qq(G)^{\ext}\xrightarrow{\imath_{4,\qq}^{-1}}\cO_\qq(\Conf_4^{\{2, 4\}}\sA_G). 
    \]
This injective $\Bbbk$-algebra homomorphisms can be regarded as a quantum analogue of the Wilson line \eqref{eq:Wilson_line}, which we call \emph{the quantum Wilson line} for short. 

We conclude this section by providing a concrete formula for $\imath_{4,\qq}$ and its cyclic shift $\isigma$. For $\phi\in \cO^{\inv}(\lambda_1, \lambda_2,\lambda_3, \lambda_4)$ of the form 
\begin{align*}
    \phi=\xi_{\lambda_1}(\lambda_1)\otimes \left(\sum_{i}\xi_i(\lambda_2)\otimes\xi'_i(\lambda_3)\right)\otimes \xi_{w_0\lambda_4}(\lambda_4)+(\text{other terms})
\end{align*}
where $\mathrm{(}$other terms$\mathrm{)}$ belong to 
    \[
    \cO(\lambda_1)_{> -\lambda_1}\ut \cO(\lambda_2)\ut\cO(\lambda_3)\ut \cO(\lambda_4)+\cO(\lambda_1)\ut \cO(\lambda_2)\ut\cO(\lambda_3)\ut \cO(\lambda_4)_{<\lambda_4^{\ast}},
    \]
we have 
\begin{align}
    \imath_{4,\qq}(\phi)=\qq^{-(\lambda_1, \lambda_1)/2-(\lambda_2, \lambda_2)/2}e'(\lambda_2)\left(\sum_{i}\xi_i(\lambda_2)\xi'_i(w_0\lambda_3)\right)e(\lambda_4^{\ast}).\label{eq:imath_formula}
\end{align}
We set $\isigma:=\imath_{4,\qq}\circ \sigma_4^{-1}$. Then, for $\phi\in \cO^{\inv}(\lambda_1, \lambda_2,\lambda_3, \lambda_4)$ of the form 
\begin{align*}
    \phi=\xi_{w_0\lambda_1}(\lambda_1)\otimes \xi_{\lambda_2}(\lambda_2)\otimes \left(\sum_{i}\xi_i(\lambda_3)\otimes\xi'_i(\lambda_4)\right)+(\text{other terms})
\end{align*}
where $\mathrm{(}$other terms$\mathrm{)}$ belong to 
    \[
    \cO(\lambda_1)_{<\lambda_1^{\ast}}\ut \cO(\lambda_2)\ut \cO(\lambda_3)\ut\cO(\lambda_4)+
    \cO(\lambda_1)\ut\cO(\lambda_2)_{> -\lambda_2}\ut \cO(\lambda_3)\ut\cO(\lambda_4),
    \]
we have 
\begin{align}
    \isigma(\phi)=(-1)^{\langle 2\rho^{\vee}, \lambda_1^{\ast}\rangle}\qq^{-(2\rho, \lambda_1^{\ast})-(\lambda_2, \lambda_2)/2-(\lambda_3, \lambda_3)/2}e'(\lambda_3)\left(\sum_{i}\xi_i(\lambda_3)\xi'_i(w_0\lambda_4)\right)e(\lambda_1^{\ast})\label{eq:isigma_formula}
\end{align}
by \eqref{eq:imath_formula}. At the classical level,  $\imath_{4,\qq}$ corresponds to \cref{fig:band_config}, while $\isigma$ corresponds to \cref{fig:band_config_appendix} in \cref{subsub:cluster_config}. In the following, we will mainly use $\isigma$ in the study of quantum cluster structures on the quantum configuration spaces, since it is more compatible with the amalgamation of cluster structures. See \cref{sec:polygon}.  

\section{Quantum cluster structure on \texorpdfstring{$\cO_\qq(\Conf_3\sA_G)$ and $\cO_\qq(\Conf_4\sA_G)$}{Oq(Conf3) and Oq(Conf4)}}
In this section, we establish quantum cluster algebra structures on $\cO_\qq(\Conf_3^{\times}\sA_G)$ and $\cO_\qq(\Conf_4^{\times}\sA_G)$. They can be regarded as a quantum analogue of Goncharov--Shen cluster structures, established in \cite{GS:Quantum} (cf.~\cref{subsub:cluster_config}). We first construct quantum seeds for $\cO_\qq(\Conf_4^{\times}\sA_G)$, and then define quantum seeds for $\cO_\qq(\Conf_3^{\times}\sA_G)$ by restriction. See \cref{app:QCA} for our notation on the quantum cluster algebras. 

\subsection{Quantum seeds for $\cO_\qq(\Conf_4^{\times}\sA_G)$}\label{subsec:Qseed_square}
We define a quantum cluster structure of $\cO_\qq(\Conf_4^{\times}\sA_G)$ by transporting that of $\cO_\qq(G)$ via $(\isigma)^{-1}\colon {}^{\ext'}\cO_\qq(G)^{\ext}\xrightarrow{\sim}\cO_\qq(\Conf_4^{\{1,3\}}\sA_G)$. Let us recall the definition of the Berenstein--Zelevinsky quantum seeds \cite{BZ} for $\cO_{\qq}(G)$. 

For $\lambda\in P_+$ and $w, w'\in W$, we write 
\[
\Delta_{w\lambda, w'\lambda}:=c^{V(\lambda)}(\xi_{w\lambda}, v_{w'\lambda}). 
\]
Recall that $\mathbf{\Delta}_{w_0,e}=\{\qq^{a}\Delta_{w_0\lambda, \lambda} \mid a\in \frac{1}{2d}\Z, \lambda\in P_+\}$, and set $\mathbf{\Delta}_{e, w_0}:=\{\qq^{a}\Delta_{\lambda, w_0\lambda} \mid a\in \frac{1}{2d}\Z, \lambda\in P_+\}$. It is known that $\mathbf{\Delta}_{w_0,e}\cup \mathbf{\Delta}_{e, w_0}$ forms an Ore set of $\cO_\qq(G)$, and set 
\[
\cO_\qq(G^{w_0, w_0}):=\cO_\qq(G)[\mathbf{\Delta}_{w_0,e}^{-1}, \mathbf{\Delta}_{e, w_0}^{-1}]
\]
(see \cite[Definition 9.5]{BZ}). Note that $\cO_\qq(G^{w_0, w_0})$ has a  $P\oplus P$-graded algebra structure induced from that of $\cO_\qq(G)$. 

Set $N:=\ell (w_0)$. Write $\olS:=\{\ols\mid s\in S\}$. For $s\in S$, we set $|s|=|\ols|:=s$, $\sgn(s):=1$, $\sgn(\ols):=-1$, and $\overline{\ols}:=s$. A word $\bbs=(\bbs_1,\dots, \bbs_{2N})$ in the alphabet $S\sqcup \olS$ is called a \emph{double reduced word} of $(w_0,w_0) \in W \times W$ if the subwords of $\bbs$ in $S$ and in $\olS$ (with bars removed) are reduced words of $w_0$. The set of the double reduced words of $(w_0,w_0)$ is denoted by $S(w_0, w_0)$. 

Fix $\bbs=(\bbs_1,\dots, \bbs_{2N})\in S(w_0, w_0)$. For $1\leq k\leq 2N$, set 
\[
w_{\leq k}=r_{\bbs_1}\cdots r_{\bbs_k},\quad w^{\leq k}=r_{\overline{\bbs}_1}\cdots r_{\overline{\bbs}_k},
\]
where we set $r_{\ols}=e$ for $s\in S$. Define the set of vertices as $J(\bbs)=S_0\sqcup [1,2N]$. For $j\in J(\bbs)$, set
\[
j^+=\begin{cases}
    \min\{l\in [1, 2N]\mid |\bbs_l|=s \}&\text{if }j=s_0\in S_0,\\
    \min(\{l\in [1, 2N]\mid l>j, |\bbs_l|=|\bbs_j| \}\sqcup \{\infty\})&\text{if }j\in [1, 2N].
\end{cases}
\]
Define the set of unfrozen vertices as $J(\bbs)^{\uf}:=\{j\in [1,2N]\mid j^+\neq \infty\}$ and the set of frozen vertices as $J(\bbs)^{\fr}:=J(\bbs)\setminus J(\bbs)^{\uf}$. Then $J(\bbs)^{\fr}=S_0\sqcup \{j\in J(\bbs)\mid j^+= \infty\}$. 

For $j\in J(\bbs)$, set 
\[
\gamma_{\bbs; j}=
\begin{cases}
        \varpi_{s}&\text{ if }j=s_0\in S_0,\\
        w^{\leq j}\varpi_{|\bbs_j|}&\text{ if }j\in [1,2N],
    \end{cases}\quad 
\delta_{\bbs; j}=
\begin{cases}
        w_0\varpi_{s}&\text{ if }j=s_0\in S_0,\\
        w_0w_{\leq j}\varpi_{|\bbs_j|}&\text{ if }j\in [1,2N],
    \end{cases}
\]
and define 
\[
\Delta_{\bbs; j}:=\Delta_{\gamma_{\bbs; j}^{\ast}, \delta_{\bbs; j}^{\ast}}. 
\]
Define the $J(\bbs) \times J(\bbs)$ skew-symmetric matrix $\Lambda(\bbs)=(\Lambda_{\bbs; ij})_{i,j\in J(\bbs)}$ by 
\[
\Lambda_{\bbs; ij}=(\gamma_{\bbs;i}^{\ast},\gamma_{\bbs;j}^{\ast})-(\delta_{\bbs;i}^{\ast},\delta_{\bbs;j}^{\ast})=(\gamma_{\bbs;i},\gamma_{\bbs;j})-(\delta_{\bbs;i},\delta_{\bbs;j})
\]
for $i>j$, where we fix an arbitrary total order on $S_0$ and we set $s_0<k$ for all $s\in S$ and $k\in [1, 2N]$. Note that $\Lambda_{\bbs; ij}=0$ if $i, j\in S_0$, hence $\Lambda(\bbs)$ does not depend on the choice of the total order on $S_0$. 

Define the matrix $\cE(\bbs)=(\varepsilon_{\bbs;ij})_{i\in J(\bbs)^\uf,j\in J(\bbs)}$ by
\begin{align}
    \varepsilon_{\bbs;ij}=
    \begin{cases}
        \epsilon_i & i=j^+ \\
        -\epsilon_j & j=i^+\\
        \epsilon_i c_{|\bbs_j|,|\bbs_i|} & \epsilon_i=\epsilon_{j^+},\quad j<i<j^+<i^+\\
        -\epsilon_j c_{|\bbs_j|,|\bbs_i|} & \epsilon_j=\epsilon_{i^+},\quad i<j<i^+<j^+\\
        \epsilon_i c_{|\bbs_j|,|\bbs_i|} & \epsilon_i=-\epsilon_{i^+},\quad j<i<i^+<j^+\\
        -\epsilon_j c_{|\bbs_j|,|\bbs_i|} & \epsilon_j=-\epsilon_{j^+},\quad i<j<j^+<i^+\\
        0 & \text{otherwise,}       
    \end{cases}.\label{eq:cE_bbs}
\end{align}
where $\bbs_{s_0}:=s$ for $s_0\in S_0$ and $\epsilon_j:=\sgn(\bbs_j)$ for $j\in J(\bbs)$. 
For a more diagrammatic description of $\cE(\bbs)$, see \cref{sec:quiver}. By \cite[Theorem 8.3]{BZ}, the collection $\bi(\bbs):=(\Lambda(\bbs), \cE(\bbs),(A_{\bi(\bbs); j})_{j\in J(\bbs)})$ forms a quantum seed.

\begin{thm}[{\cite[Main Theorem]{GY:BZ}, \cite[Theorem A]{QY}}]\label{thm:clusterDouble}
Let $\bbs\in S(w_0, w_0)$. There exists an algebra isomorphism 
\[
\kappa_{\bbs}:U_{\qq}(\bi(\bbs); J(\bbs)^{\fr})\simeq \cO_\qq(G^{w_0, w_0})
\]
satisfying $\kappa_{\bbs}(A_{\bi(\bbs); j})=\Delta_{\bbs; j}$ for $j\in J(\bbs)$. Moreover, $A_{\qq}(\bi(\bbs); J(\bbs)^{\fr})=U_{\qq}(\bi(\bbs); J(\bbs)^{\fr})$. 
\end{thm}
\begin{rem}
Our choice of quantum cluster variables is twisted from the Berenstein--Zelevinsky convention by the $\ast$-involution in order to match the convention of Goncharov--Shen. See \cref{prop:GS_variable_minor}.   
\end{rem}

In the following, we identify $U_{\qq}(\bi(\bbs); J(\bbs)^{\fr})$ (and $A_{\qq}(\bi(\bbs); J(\bbs)^{\fr})$) with $\cO_\qq(G^{w_0, w_0})$ through $\kappa_{\bbs}$. In particular, we will write $A_{\bi(\bbs); j}=\Delta_{\bbs; j}$ and regard $\bi(\bbs)$ as a quantum seed in the skew field of fractions $\cF(\cO_\qq(G^{w_0, w_0}))$ of $\cO_\qq(G^{w_0, w_0})$. Under this identification, we have the following. 
\begin{thm}[{\cite[Theorem 4.11]{QY}}]\label{thm:indep_Double}
For $\bbs, \bbs'\in S(w_0, w_0)$, $\bi(\bbs)$ and $\bi(\bbs')$ are mutation equivalent. 
\end{thm}
Suppose that $ \lambda, \mu, \nu \in P_+$ satisfies 
\begin{align}
w\lambda=\mu^{\ast}-\nu\label{eq:weightcond}
\end{align}
for some $w\in W$. For example, for $w\in W$ and $\lambda\in P_+$, 
\[
\mu=[w\lambda]_+^{\ast},\ \nu=[w\lambda]_-
\]
satisfies \eqref{eq:weightcond}. Here, for $\gamma=\sum_{s\in S}m_s\varpi_s\in P$, we set 
\begin{align}
[\gamma]_+:=\sum_{s\in S}\max\{m_s, 0\}\varpi_s,\quad [\gamma]_-:=\sum_{s\in S}\max\{-m_s, 0\}\varpi_s.\label{eq:+-notation}
\end{align}
Note that $[\gamma]_+,[\gamma]_-\in P_+$ and $\gamma=[\gamma]_+-[\gamma]_-$. When $ \lambda,\mu, \nu\in P_+$ satisfy \eqref{eq:weightcond}, $\dim_{\Bbbk} (V(\mu)^{\ast}\ut V(\nu)^{\ast}\ut V(\lambda)^{\ast})^{\inv}=1$ by \cref{p:invsp}, \eqref{eq:image_phi} and $\dim_{\Bbbk} (V(\lambda)^{\ast})_{-w\lambda}=1$. Therefore, by \cref{p:cyclic}, $\dim_{\Bbbk} (V(\nu)^{\ast}\ut V(\lambda)^{\ast}\ut V(\mu)^{\ast})^{\inv}=1$. 

The following theorem plays an essential role in relating elements of $\cO_\qq(\Conf_4^{\{1,3\}}\sA_G)$ to those of ${}^{\ext'}\cO_\qq(G)^{\ext}$ through $\isigma$. 
\begin{thm}\label{thm:isigma}
Assume that $\lambda, \mu, \nu \in P_+$ satisfies $w\lambda=\mu^{\ast}-\nu$ for some $w\in W$. 
\begin{itemize}
    \item[(1)] There uniquely exists $\phi_{\mu, \nu, \lambda}\in \cO^\inv(\mu, \nu, \lambda, 0)$ such that  
\[
    \phi_{\mu, \nu, \lambda}=\xi_{w_0\mu}(\mu)\otimes \xi_{\nu}(\nu)\otimes \xi_{w\lambda}(\lambda)\otimes 1+(\text{other terms})
\]
where $\mathrm{(}$other terms$\mathrm{)}$ belong to 
    \[
    \cO(\mu)_{<\mu^{\ast}}\ut \cO(\nu)\ut \cO(\lambda)\ut 1+\cO(\mu)\otimes \cO(\nu)_{> -\nu}\ut \cO(\lambda)\ut 1.
    \]
Then 
\[
\isigma(\phi_{\mu, \nu, \lambda})=(-1)^{\langle 2\rho^{\vee}, \mu^{\ast}\rangle}\qq^{-(2\rho, \mu^{\ast})-(\nu, \nu)/2-(\lambda, \lambda)/2}e'(\lambda)\Delta_{w\lambda, \lambda}e(\mu^{\ast}).
\]
    \item[(1)${}^\prime$] There uniquely exists $\phi'_{\mu, \nu, \lambda}\in \cO^\inv(\mu, \nu, 0, \lambda)$ such that  
\[
    \phi'_{\mu, \nu, \lambda}=\xi_{w_0\mu}(\mu)\otimes \xi_{\nu}(\nu)\otimes 1\otimes \xi_{w\lambda}(\lambda)+(\text{other terms})
\]
where $\mathrm{(}$other terms$\mathrm{)}$ belong to 
    \[
    \cO(\mu)_{<\mu^{\ast}}\ut\cO(\nu)\ut 1\ut\cO(\lambda)+\cO(\mu)\ut \cO(\nu)_{> -\nu}\ut 1\ut\cO(\lambda).
    \]
Then 
\[
\isigma(\phi'_{\mu, \nu, \lambda})=(-1)^{\langle 2\rho^{\vee}, \mu^{\ast}\rangle}\qq^{-(2\rho, \mu^{\ast})-(\nu, \nu)/2}\Delta_{w\lambda, w_0\lambda}e(\mu^{\ast}).
\]
\item[(2)] There uniquely exists $\psi_{\lambda, \mu, \nu}\in \cO^\inv(\lambda, 0, \mu, \nu)$ such that  
\[
    \psi_{\lambda, \mu, \nu}=\xi_{w\lambda}(\lambda)\otimes 1\otimes \xi_{w_0\mu}(\mu)\otimes \xi_{\nu}(\nu)+(\text{other terms})
\]
where $\mathrm{(}$other terms$\mathrm{)}$ belong to 
    \[
    \cO(\lambda)\ut 1\ut \cO(\mu)_{<\mu^{\ast}}\ut \cO(\nu)+\cO(\lambda)\ut 1\ut \cO(\mu)\ut \cO(\nu)_{> -\nu}.
    \]
Then 
\[
\isigma(\psi_{\lambda, \mu, \nu})=(-1)^{\langle 2\rho^{\vee}, \nu\rangle}q^{-(2\rho, \nu)-(\mu, \mu)/2}e'(\mu)\Delta_{\lambda^{\ast}, w^{\ast}\lambda^{\ast}}e(\lambda^{\ast}), 
\]
where $w^{\ast}:=w_0ww_0$.
\item[(2)${}^\prime$] There uniquely exists $\psi'_{\lambda, \mu, \nu}\in \cO^\inv(0, \lambda, \mu, \nu)$ such that  
\[
    \psi'_{\lambda, \mu, \nu}=1\otimes \xi_{w\lambda}(\lambda)\otimes \xi_{w_0\mu}(\mu)\otimes \xi_{\nu}(\nu)+(\text{other terms})
\]
where $\mathrm{(}$other terms$\mathrm{)}$ belong to 
    \[
    1\ut \cO(\lambda)\ut \cO(\mu)_{<\mu^{\ast}}\ut \cO(\nu)+1\ut \cO(\lambda)\ut \cO(\mu)\ut \cO(\nu)_{> -\nu}.
    \]
Then 
\[
\isigma(\psi'_{\lambda, \mu, \nu})=(-1)^{\langle 2\rho^{\vee}, \nu\rangle}q^{-(2\rho, \nu)-(\lambda, \lambda)/2-(\mu, \mu)/2}e'(\mu)\Delta_{w_0\lambda^{\ast}, w^{\ast}\lambda^{\ast}}. 
\]
\end{itemize}
\end{thm}
The correspondence of elements up to scalars and the extended part is summarized in \cref{tab:Conf4_OG}.

\begin{table}[ht]
    \centering
    \begin{tabular}{l|l}
        \qquad $\cO_\qq(\Conf_4\sA_G)$ &  $\cO_\qq(G)$ \\ \hline
        $\phi_{\mu, \nu,\lambda} \in \cO^{\inv}(\mu, \nu,\lambda,0)$ & $\Delta_{w\lambda,\lambda}$ \\
        $\phi'_{\mu, \nu,\lambda} \in \cO^{\inv}(\mu, \nu,0,\lambda)$ & $\Delta_{w\lambda,w_0\lambda}$ \\
        $\psi_{\lambda, \mu,\nu} \in \cO^{\inv}(\lambda, 0,\mu,\nu)$ & $\Delta_{\lambda^\ast,w^\ast\lambda^\ast}$ \\
        $\psi'_{\lambda,\mu,\nu} \in \cO^{\inv}(0, \lambda,\mu,\nu)$ & $\Delta_{w_0\lambda^\ast,w^\ast\lambda^\ast}$ \\ \vspace{0.5mm}
    \end{tabular}
    \caption{Correspondence of elements up to scalars and the extended part.}
    \label{tab:Conf4_OG}
\end{table}

\begin{proof}
The statements (1) and (1)${}^\prime$ immediately follow from \eqref{eq:isigma_formula}. 

Let us show (2). There exists a $\Uq$-module homomorphism $\pi\colon V(\mu)\otimes V(\nu)\to V(\lambda^{\ast})$ satisfying $v_{w_0\mu}\otimes v_{\nu}\mapsto v_{-w\lambda}$ (see \eqref{eq:image_phi}). Hence, it induces a $\Uq$-module homomorphism 
\[
\pi^{\ast}\colon V(\lambda^{\ast})^{\ast}\to (V(\mu)\otimes V(\nu))^{\ast}\xrightarrow[\sim]{\Upsilon_{V(\mu), V(\nu)}}  V(\mu)^{\ast}\ut V(\nu)^{\ast}. 
\]
Hence, through $\Psi_1$ (recall \eqref{eq:Psi_k}), we obtain a $\Uq$-module homomorphism
\[
\pi^{\ast}\colon \cO(\lambda^{\ast})\to \cO(\mu)\ut \cO(\nu).\ \text{(abuse of notation)}
\]
Note that, by definition of $\pi$, 
\begin{align}
    \pi^{\ast}(\xi_{-w\lambda}(\lambda^{\ast}))=\xi_{w_0\mu}(\mu)\otimes \xi_{\nu}(\nu)+(\text{other terms}),\label{eq:pi_ast_xi}
\end{align}
where $\mathrm{(}$other terms$\mathrm{)}$ belong to $\cO(\mu)_{<\mu^{\ast}}\ut \cO(\nu)+\cO(\mu)\ut \cO(\nu)_{>-\nu}$. 
Since $\dim_{\Bbbk}\cO^\inv(\lambda, 0, \mu, \nu)=1=\dim_{\Bbbk} (\cO(\lambda)\ut 1\ut \pi^{\ast}(\cO(\lambda^{\ast})))^\inv$, we have 
\[
\psi_{\lambda, \mu, \nu}\in \cO^\inv(\lambda, 0, \mu, \nu)=(\cO(\lambda)\ut 1\ut \pi^{\ast}(\cO(\lambda^{\ast})))^\inv. 
\]
Therefore, by \cref{thm:D_lambda_extremal} and \eqref{eq:pi_ast_xi}, 
\begin{align*}
\psi_{\lambda, \mu, \nu}=\sum_{\lambda'\in W\lambda}(-1)^{\langle\rho^{\vee}, w\lambda-\lambda'\rangle}q^{(\rho, w\lambda-\lambda')}\xi_{\lambda'}(\lambda)\otimes 1\otimes \pi^{\ast}(\xi_{-\lambda'}(\lambda^{\ast}))+(\text{other terms}),
\end{align*}
where $($other terms$)$ belong to $\bigoplus_{\gamma\in P; \gamma\not\in W\lambda}\cO(\lambda)_{-\gamma}\ut 1\otimes \pi^{\ast}(\cO(\lambda^{\ast})_{\gamma})$. Therefore, by \eqref{eq:isigma_formula}, 
\begin{align}
&\isigma(\psi_{\lambda, \mu, \nu})\notag\\
&=(-1)^{\langle\rho^{\vee}, w\lambda+\lambda^{\ast}\rangle+\langle 2\rho^{\vee}, \lambda^{\ast}\rangle}q^{(\rho, w\lambda+\lambda^{\ast})-(2\rho, \lambda^{\ast})-(\mu, \mu)/2}e'(\mu)(\qq^{(\mu, \mu)/2}(\sfm\circ (\mathrm{id}\otimes \Psi^{+,-}))(\pi^{\ast}(\xi_{\lambda^{\ast}}(\lambda^{\ast}))))e(\lambda^{\ast})\notag\\
&=(-1)^{\langle\rho^{\vee}, \mu^{\ast}-\nu+\lambda^{\ast}\rangle+\langle 2\rho^{\vee}, \lambda^{\ast}\rangle}q^{(\rho, \mu^{\ast}-\nu-\lambda^{\ast})}e'(\mu)((\sfm\circ (\mathrm{id}\otimes \Psi^{+,-}))(\pi^{\ast}(\xi_{\lambda^{\ast}}(\lambda^{\ast}))))e(\lambda^{\ast}).\label{eq:imath_psi_1}
\end{align}
Moreover, 
\begin{align}
    &(\sfm\circ (\mathrm{id}\otimes \Psi^{+,-}))(\pi^{\ast}(\xi_{\lambda^{\ast}}(\lambda^{\ast})))\notag\\
    &=\qq^{-(\mu, \mu)/2}c^{V(\mu)\otimes V(\nu)}(\pi^{\ast}(\xi_{\lambda^{\ast}}), v_{\mu}\otimes v_{w_0\nu})\notag\\
    &=\qq^{-(\mu, \mu)/2}c^{V(\lambda^\ast)}(\xi_{\lambda^{\ast}}, \pi(v_{\mu}\otimes v_{w_0\nu})).\label{eq:imath_psi_2}
\end{align}
By \eqref{eq:w_t_action} and \cref{lem:extremal_braid}, 
\begin{align*}
    T_{w_0}(v_{\mu}\otimes v_{w_0\nu})&=T_{w_0}(v_{\mu})\otimes T_{w_0}(v_{w_0\nu})\notag\\
    &=(-1)^{\langle\rho^{\vee}, w_0\nu-\nu\rangle}q^{(\rho, w_0\nu-\nu)}v_{w_0\mu}\otimes v_{\nu},\\
    T_{w_0}(v_{w^{\ast}\lambda^{\ast}})&=T_{w}(v_{-\lambda})
    =(-1)^{\langle\rho^{\vee}, -\lambda+w\lambda\rangle}q^{(\rho, -\lambda+w\lambda)}v_{-w\lambda}.
\end{align*}
Therefore, 
\begin{align}
\pi(v_{\mu}\otimes v_{w_0\nu})&=(-1)^{\langle\rho^{\vee}, w_0\nu-\nu\rangle}q^{(\rho, w_0\nu-\nu)}\pi(T_{w_0}^{-1}(v_{w_0\mu}\otimes v_{\nu}))\notag\\
    &=(-1)^{\langle\rho^{\vee}, w_0\nu-\nu\rangle}q^{(\rho, w_0\nu-\nu)}T_{w_0}^{-1}\pi(v_{w_0\mu}\otimes v_{\nu})\notag\\
    &=(-1)^{\langle\rho^{\vee}, w_0\nu-\nu\rangle}q^{(\rho, w_0\nu-\nu)}T_{w_0}^{-1}v_{-w\lambda}\notag\\
    &=(-1)^{\langle\rho^{\vee}, w_0\nu-\nu+\lambda-w\lambda\rangle}q^{(\rho, w_0\nu-\nu+\lambda-w\lambda)}v_{w^{\ast}\lambda^{\ast}}\notag\\
    &=(-1)^{\langle\rho^{\vee}, -\mu^{\ast}-\nu^{\ast}+\lambda\rangle}q^{(\rho, -\mu^{\ast}-\nu^{\ast}+\lambda)}v_{w^{\ast}\lambda^{\ast}}\notag\\
    &=(-1)^{\langle\rho^{\vee}, -\mu-\nu+\lambda\rangle}q^{(\rho, -\mu-\nu+\lambda)}v_{w^{\ast}\lambda^{\ast}}.\label{eq:pi_extremal_image}
\end{align}
Hence, by \eqref{eq:imath_psi_1}, \eqref{eq:imath_psi_2} and \eqref{eq:pi_extremal_image}, 
\begin{align*}
    &\isigma(\psi_{\lambda, \mu, \nu})\\
    &=(-1)^{\langle\rho^{\vee}, \mu^{\ast}-\nu+\lambda^{\ast}\rangle+\langle 2\rho^{\vee}, \lambda^{\ast}\rangle+\langle\rho^{\vee}, -\mu-\nu+\lambda\rangle}q^{(\rho, \mu^{\ast}-\nu-\lambda^{\ast})-(\mu, \mu)/2+(\rho, -\mu-\nu+\lambda)}e'(\mu)\Delta_{\lambda^{\ast}, w^{\ast}\lambda^{\ast}}e(\lambda^{\ast})\\
    &=(-1)^{\langle 2\rho^{\vee}, \nu\rangle}q^{-(2\rho, \nu)-(\mu, \mu)/2}e'(\mu)\Delta_{\lambda^{\ast}, w^{\ast}\lambda^{\ast}}e(\lambda^{\ast}).
\end{align*}
We can show (2)${}^\prime$ in a parallel manner. 
\end{proof}
Let us relate $\cO_\qq(\Conf_4^{\times}\sA_G)$ with $\cO_\qq(G^{w_0, w_0})$ via $\isigma$. 
\cref{t:invstr}, \eqref{eq:q-cyclic_shift}, and \cref{cor:D_lambda} imply that  
\begin{align*}
D_{1; \lambda}&=(-1)^{\langle 2\rho^{\vee}, \lambda\rangle}\qq^{(\lambda, \lambda)/2+(2\rho, \lambda)} \xi_{-\lambda}(\lambda^{\ast})\otimes \xi_{\lambda}(\lambda)\otimes 1\otimes 1+(\text{other terms})\\
&=(-1)^{\langle 2\rho^{\vee}, \lambda\rangle}\qq^{(\lambda, \lambda)/2+(2\rho, \lambda)}\phi_{\lambda^{\ast}, \lambda, 0},\\
D_{2; \lambda}&=\qq^{(\lambda, \lambda)/2}\cdot 1\otimes \xi_{\lambda^{\ast}}(\lambda^{\ast})\otimes \xi_{w_0\lambda}(\lambda)\otimes 1+(\text{other terms}) \\
&=\qq^{(\lambda, \lambda)/2}\phi_{0, \lambda^{\ast}, \lambda},\\
D_{3; \lambda}&=(-1)^{\langle 2\rho^{\vee}, \lambda\rangle}\qq^{(\lambda, \lambda)/2+(2\rho, \lambda)}\cdot 1\otimes 1\otimes \xi_{-\lambda}(\lambda^{\ast})\otimes \xi_{\lambda}(\lambda)+(\text{other terms}),\\
&=(-1)^{\langle 2\rho^{\vee}, \lambda\rangle}\qq^{(\lambda, \lambda)/2+(2\rho, \lambda)}\psi_{0, \lambda^{\ast}, \lambda},\\
D_{4; \lambda}&=(-1)^{\langle 2\rho^{\vee}, \lambda\rangle}\qq^{(\lambda, \lambda)/2+(2\rho, \lambda)} \xi_{w_0\lambda}(\lambda)\otimes 1\otimes 1\otimes \xi_{\lambda^{\ast}}(\lambda^{\ast})+(\text{other terms}),\\
&=(-1)^{\langle 2\rho^{\vee}, \lambda\rangle}\qq^{(\lambda, \lambda)/2+(2\rho, \lambda)}\psi_{\lambda, 0, \lambda^{\ast}}.
\end{align*}
Therefore, \cref{thm:isigma} implies that 
\begin{align}
&\isigma(D_{1; \lambda})=e(\lambda),\label{eq:isigma_D1}\\    
&\isigma(D_{2; \lambda})=\qq^{-(\lambda, \lambda)/2}e'(\lambda)\Delta_{w_0\lambda, \lambda},  \label{eq:isigma_D2} \\
&\isigma(D_{3; \lambda})=e'(\lambda^{\ast}),\label{eq:isigma_D3}\\    
&\isigma(D_{4; \lambda})=\qq^{(\lambda, \lambda)/2}\Delta_{\lambda^{\ast}, w_0\lambda^{\ast}}e(\lambda^{\ast}).  \label{eq:isigma_D4} 
\end{align}
Therefore, $\isigma$ can be extended to 
\[
\isigma\colon\cO_\qq(\Conf_4^{\times}\sA_G)\xrightarrow{\sim} {}^{\ext'}\cO_\qq(G)^{\ext}[\mathbf{\Delta}_{w_0,e}^{-1}, \mathbf{\Delta}_{e, w_0}^{-1}]\xrightarrow{\sim}{}^{\ext'}\cO_\qq(G^{w_0, w_0})^{\ext}.
\]
Let $\bbs\in S(w_0, w_0)$. We set $S_{-\infty}:=\{s_{-\infty}\mid s\in S\}, S_{\infty}:=\{s_{\infty}\mid s\in S\}$, and 
\begin{align*}
&J_{\square}(\bbs):=S_{-\infty}\sqcup J(\bbs)\sqcup S_{\infty}=S_{-\infty}\sqcup S_0\sqcup [1,2N]\sqcup S_{\infty},\\
&J_{\square}(\bbs)^{\uf}:=J(\bbs)^{\uf},\ J_{\square}(\bbs)^{\fr}:=J_{\square}(\bbs)\setminus J_{\square}(\bbs)^{\uf}=S_{-\infty}\sqcup J(\bbs)^{\fr}\sqcup S_{\infty},\\
&\Lambda_{\dsquare}(\bbs)=(\Lambda_{\dsquare, \bbs; ij})_{i,j\in J_{\square}(\bbs)},\  \Lambda_{\dsquare,\bbs;ij}:=
\begin{cases}
    \Lambda_{\bbs; ij}&\text{if }i, j\in  J(\bbs),\\
    0&\text{if }i, j\in  S_{-\infty}\sqcup S_{\infty},\\
    (\varpi_s^{\ast}, \gamma_{\bbs; j}^{\ast})&\text{if }i=s_{-\infty}\in  S_{-\infty},\ j\in  J(\bbs),\\
    -(\gamma_{\bbs; i}^{\ast}, \varpi_s^{\ast})&\text{if }i\in  J(\bbs), j=s_{-\infty}\in  S_{-\infty},\\
     (\varpi_s^{\ast}, \delta_{\bbs; j}^{\ast})&\text{if }i=s_{\infty}\in  S_{\infty},\ j\in  J(\bbs),\\
    -(\delta_{\bbs; i}^{\ast}, \varpi_s^{\ast})&\text{if }i\in  J(\bbs), j=s_{\infty}\in  S_{\infty},\\
\end{cases}\\
&\cE_{\dsquare}(\bbs)=(\varepsilon_{\dsquare, \bbs;ij})_{i\in J_{\square}(\bbs)^{\uf},j\in J_{\square}(\bbs)},\ \varepsilon_{\dsquare, \bbs;ij}:=
\begin{cases}
    \varepsilon_{\bbs;ij}&\text{if }i\in J_{\square}(\bbs)^{\uf}, j\in  J(\bbs),\\
    0&\text{if }i\in J_{\square}(\bbs)^{\uf}, j\in  S_{-\infty}\sqcup S_{\infty}.
\end{cases}\\
&A_{\dsquare, \bbs; j}:=
\begin{cases}
(\isigma)^{-1}(A_{\bi(\bbs); j})&\text{if }j\in  J(\bbs),\\
    D_{1; \varpi_s^{\ast}}&\text{if }j=s_{-\infty}\in  S_{-\infty},\\
    D_{3; \varpi_s}&\text{if }j=s_{\infty}\in  S_{\infty}.
\end{cases}
\end{align*}
Note that, by \cref{thm:clusterDouble} and \eqref{eq:isigma_D1}, \eqref{eq:isigma_D3}, we have 
\begin{align}
    A_{\dsquare, \bbs; i}A_{\dsquare, \bbs; j}=\qq^{\Lambda_{\dsquare,\bbs;ij}}A_{\dsquare, \bbs; j}A_{\dsquare, \bbs; i}\label{eq:qcomm_dsquare}
\end{align}
for $i, j\in J_{\square}(\bbs)$. As explained in \cref{sec:quiver}, we can describe  $\cE_{\dsquare}(\bbs)$ by a quiver. An example of the quiver $\bJ_{\dsquare}(\bbs)$ associated with $\cE_{\dsquare}(\bbs)$ is shown in \cref{fig:quiver_dsquare}. Observe that the frozen vertices in $S_{-\infty}$ and $S_\infty$ are isolated.

\begin{figure}[ht]
    \centering
\begin{tikzpicture}
\def\R{3}
\coordinate(A) at (0,0);
\coordinate(B) at (\R,0);
\coordinate(C) at (0,\R);
\coordinate(D) at (\R,\R);
\fill[gray!40] ($(A)!0.5!(B)$) ellipse[
      x radius=.4*\R,
      y radius=.06*\R
   ];
\node[below=3pt] at ($(A)!0.5!(B)$) {$S_{\infty}$};
\fill[gray!40] ($(C)!0.5!(D)$) ellipse[
      x radius=.4*\R,
      y radius=.06*\R
   ];
\node[above=3pt] at ($(C)!0.5!(D)$) {$S_{-\infty}$};
\fill[gray!40] ($(A)!0.5!(C)$) ellipse[
      x radius=.06*\R,
      y radius=.4*\R
   ];
\node[left=3pt] at ($(A)!0.5!(C)$) {$S_0$};
\node[below left] at ($(A)$) {$4$};
\node[below right] at ($(B)$) {$3$};
\node[above left] at ($(C)$) {$1$};
\node[above right] at ($(D)$) {$2$};

\draw[thick] (A) -- (B) -- (D) -- (C) --cycle;
\draw[thick] ($(A)!0.1!(C)$) -- ($(B)!0.1!(D)$);
\draw[thick] ($(A)!0.9!(C)$) -- ($(B)!0.9!(D)$);

\begin{scope}[xshift=7cm]
\coordinate(A) at (0,0);
\coordinate(B) at (\R,0);
\coordinate(C) at (0,\R);
\coordinate(D) at (\R,\R);

\draw[thick,dashed] (B) -- (C);
\draw[thick] (A) -- (B) -- (D) -- (C) --cycle;
\draw[thick] ($(A)!0.1!(C)$) -- ($(B)!0.1!(D)$);
\draw[thick] ($(A)!0.9!(C)$) -- ($(B)!0.9!(D)$);

\vertexA{C}{B}{A}{3};
\foreach \i in {1,2,3}{
    \filldraw[fill=white](a\i) circle(2pt);
    \fnode{b\i}{black};
    \fnode{V\i 0}{myblue};
    }
\foreach \i in {11,21,12}
\draw(V\i) circle(2pt);
\qarrow{a3}{b3};
\qarrow{a2}{V12}; \qarrow{V12}{b2};
\qarrow{a1}{V11}; \qarrow{V11}{V21}; \qarrow{V11}{b1};
\qarrow{b2}{V21};
\qarrow{b3}{V12}; 
\qarrow{V12}{V11};
\qarrow{V21}{V12};
\qarrow{V12}{a3};
\qarrow{V11}{a2};

\vertexA{B}{C}{D}{3};
\foreach \i in {1,2,3}{
    \fnode{b\i}{black};
    \fnode{V\i 0}{myblue};
    }
\foreach \i in {11,21,12}
\draw(V\i) circle(2pt);
\qarrow{a3}{b3};
\qarrow{a2}{V12}; \qarrow{V12}{b2};
\qarrow{a1}{V11}; \qarrow{V11}{V21}; \qarrow{V11}{b1};
\qarrow{b2}{V21};
\qarrow{b3}{V12}; 
\qarrow{V12}{V11};
\qarrow{V21}{V12};
\qarrow{V12}{a3};
\qarrow{V11}{a2};
\end{scope}
\end{tikzpicture}
    \caption{The location of the frozen subsets $S_0,S_{-\infty},S_\infty$ (Left) and the quiver $\bJ_{\dsquare}(\bbs)$ with $\bbs=(1,2,3,1,2,1,\overline{3},\overline{2},\overline{1},\overline{3},\overline{2},\overline{3})$.}
    \label{fig:quiver_dsquare}
\end{figure}

\begin{thm}\label{thm:cluster_dsquare}
Let $\bbs\in S(w_0, w_0)$. The collection $\bi_{\dsquare}(\bbs):=(\Lambda_{\dsquare}(\bbs), \cE_{\dsquare}(\bbs),(A_{\dsquare, \bbs; j})_{j\in J_{\square}(\bbs)})$ forms a quantum seed in $\cF(\cO_\qq(\Conf_4^{\times}\sA_G))$, and  
\[
A_{\qq}(\bi_{\dsquare}(\bbs); J_{\square}(\bbs)^{\fr})=U_{\qq}(\bi_{\dsquare}(\bbs); J_{\square}(\bbs)^{\fr})=\cO_\qq(\Conf_4^{\times}\sA_G).
\]
Moreover, for $\bbs'\in S(w_0, w_0)$, $\bi_{\dsquare}(\bbs)$ and $\bi_{\dsquare}(\bbs')$ are mutation equivalent. 
\end{thm}
\begin{proof}
    By \cref{thm:clusterDouble}, the definition of ${}^{\ext'}\cO_\qq(G^{w_0, w_0})^{\ext}$, and \eqref{eq:qcomm_dsquare}, we can immediately show that $\cF(\cO_\qq(\Conf_4^{\times}\sA_G))$ is generated by $A_{\dsquare, \bbs; j}$, $j\in J_{\square}(\bbs)$ as a skew field, and the $\Bbbk$-subalgebra of $\cF(\cO_\qq(\Conf_4^{\times}\sA_G))$ generated by $A_{\dsquare, \bbs; j}^{\pm 1}$, $j\in J_{\square}(\bbs)$ is isomorphic to the quantum torus $\cT(\Lambda_{\dsquare}(\bbs))$ associated with $\Lambda_{\dsquare}(\bbs)$. 
    Let us show that $(\Lambda_{\dsquare}(\bbs), \cE_{\dsquare}(\bbs))$ is a compatible pair. As in \eqref{eq:Xvariable}, we set 
    \begin{align*}
X_{\bi_{\dsquare}(\bbs); i}&:=A_{\bi_{\dsquare}(\bbs)}^{\sum_{j\in J_{\square}(\bbs)}\varepsilon_{\dsquare, \bbs;ij}f_j}\\
&:=\qq^{-\frac{1}{2}\sum_{j_1 < j_2}\varepsilon_{\dsquare, \bbs;ij_1} \varepsilon_{\dsquare, \bbs;ij_2} \Lambda_{\dsquare,\bbs;j_1j_2}} \dprod_{j \in J_{\square}(\bbs)} A_{\dsquare, \bbs; j}^{\varepsilon_{\dsquare, \bbs;ij}}\in \cF(\cO_\qq(\Conf_4^{\times}\sA_G))
\end{align*}
for $i\in J_{\square}(\bbs)^{\uf}$, where we fix a total ordering on $J_{\square}(\bbs)$. As mentioned in \eqref{eq:AXrel}, it suffices to show that, for $i\in J_{\square}(\bbs)^{\uf}$, there exists $d_i\in \Z_{>0}$ such that 
\begin{align}
X_{\bi_{\dsquare}(\bbs); i}A_{\dsquare, \bbs; j}=\qq^{-d_i\delta_{i,j}/d}A_{\dsquare, \bbs; j}X_{\bi_{\dsquare}(\bbs); i}\label{eq:AXrel_dsquare}
\end{align}
for $j\in J_{\square}(\bbs)$. By the definition of $\cE_{\dsquare}(\bbs)$ and $(A_{\dsquare, \bbs; j})_{j\in J_{\square}(\bbs)}$, we have 
\[
X_{\bi_{\dsquare}(\bbs); i}=(\isigma)^{-1}(X_{\bi(\bbs); i})
\]
for $i\in J_{\square}(\bbs)^{\uf}=J(\bbs)^{\uf}$. 
Since $\bi(\bbs)$ is a quantum seed, for $i\in J_{\square}(\bbs)^{\uf}$, there exists $d_i\in \Z_{>0}$ satisfying \eqref{eq:AXrel_dsquare} for $j\in J(\bbs)$. Hence, it remains to show that 
\begin{align}
X_{\bi_{\dsquare}(\bbs); i}A_{\dsquare, \bbs; j}=A_{\dsquare, \bbs; j}X_{\bi_{\dsquare}(\bbs); i}\label{eq:AXrel_dsquare_remain}
\end{align}
for $i\in J_{\square}(\bbs)^{\uf}$ and $j\in S_{-\infty}\sqcup S_{\infty}$. By \cref{thm:clusterDouble} and \cref{prop:weight0}, we have
\begin{align*}
&0=-\sum_{j\in J(\bbs)}\varepsilon_{\bbs;ij}\gamma_{\bbs; j}^{\ast}=-\sum_{j\in J(\bbs)}\varepsilon_{\dsquare, \bbs;ij}\gamma_{\bbs; j}^{\ast},\\
&0=\sum_{j\in J(\bbs)}\varepsilon_{\bbs;ij}\delta_{\bbs; j}^{\ast}=\sum_{j\in J(\bbs)}\varepsilon_{\dsquare, \bbs;ij}\delta_{\bbs; j}^{\ast},
\end{align*}
for $i\in J_{\square}(\bbs)^{\uf}$. Here we consider the $P$-gradings on $\cO_\qq(G^{w_0, w_0})$ induced from those on $\cO_\qq(G)$ given by $\wt$ and $\wt_r$. Hence, by the definition of $\Lambda_{\dsquare}(\bbs)$ and \eqref{eq:qcomm_dsquare}, it implies \eqref{eq:AXrel_dsquare_remain} for $i\in J_{\square}(\bbs)^{\uf}$ and $j\in S_{-\infty}\sqcup S_{\infty}$. Therefore, we conclude that $\bi_{\dsquare}(\bbs)$ is a quantum seed. 

By the definition of $\cE_{\dsquare}(\bbs)$ and \cref{thm:clusterDouble}, we have 
\begin{align*}
    A_{\qq}(\bi_{\dsquare}(\bbs); J_{\square}(\bbs)^{\fr})&=(\isigma)^{-1}(A_{\qq}(\bi(\bbs); J(\bbs)^{\fr}))[D_{1; \varpi_s^{\ast}}^{\pm1}, D_{3; \varpi_s}^{\pm 1}\mid s\in S]\\
    &=(\isigma)^{-1}(\cO_\qq(G^{w_0, w_0}))[D_{1; \varpi_s^{\ast}}^{\pm 1}, D_{3; \varpi_s}^{\pm 1}\mid s\in S]\\
    &=(\isigma)^{-1}({}^{\ext'}\cO_\qq(G^{w_0, w_0})^{\ext})=\cO_\qq(\Conf_4^{\times}\sA_G).
\end{align*}
Similarly, $U_{\qq}(\bi_{\dsquare}(\bbs); J_{\square}(\bbs)^{\fr})=(\isigma)^{-1}(U_{\qq}(\bi(\bbs); J(\bbs)^{\fr}))[D_{1; \varpi_s^{\ast}}^{\pm 1}, D_{3; \varpi_s}^{\pm 1}\mid s\in S]=\cO_\qq(\Conf_4^{\times}\sA_G)$. The mutation equivalence of $\bi_{\dsquare}(\bbs), \bbs\in S(w_0,w_0)$ immediately follows from \cref{thm:indep_Double}.
\end{proof}
Next we construct a quantum analogue of Goncharov--Shen seed \cite{GS:Quantum} for $\cO_\qq(\Conf_4^{\times}\sA_G)$ by a strict similarity transform of $\bi_{\dsquare}(\bbs)$ in the sense of \cref{def:coeffmod}. In the following, we fix $\bbs=(\bbs_1,\dots, \bbs_{2N})\in S(w_0, w_0)$. For $k\in J(\bbs)$ and $s\in S$, define $m_{k; s}, n_{k; s}\in \Z_{\geq 0}$ by 
\begin{align*}
&\sum_{s\in S} m_{k; s}\varpi_s=[\gamma_{\bbs; k}]_+,
&&\sum_{s\in S} n_{k; s}\varpi_s=[\delta_{\bbs; k}]_+.
\end{align*}
For $j\in J_{\square}(\bbs)$, define $f_{\bbs; j}\in \Z^{J_{\square}(\bbs)}$ and $A_{\square, \bbs; j}\in \cO_\qq(\Conf_4^{\times}\sA_G)$ by 
\[
f_{\bbs; j}:=\begin{cases}
    f_j+\sum_{s\in S} m_{j; s}f_{s_{-\infty}}+\sum_{s\in S} n_{j; s}f_{s_{\infty}}&\text{if }j\in J(\bbs),\\
    f_j&\text{if }j\in S_{-\infty}\sqcup S_{\infty},
\end{cases} 
\]
\begin{align*}
A_{\square, \bbs; j}&:=A_{\bi_{\dsquare}(\bbs)}^{f_{\bbs; j}}\\
&=\begin{cases}
\qq^{\frac{1}{2}(\gamma_{\bbs; j}^{\ast}, [\gamma_{\bbs; j}^{\ast}]_+)-\frac{1}{2}(\delta_{\bbs; j}^{\ast}, [\delta_{\bbs; j}^{\ast}]_+)}\left(\prod_{s\in S}A_{\dsquare, \bbs; s_{\infty}}^{n_{j; s}}\right)A_{\dsquare, \bbs; j}\left(\prod_{s\in S}A_{\dsquare, \bbs; s_{-\infty}}^{m_{j; s}}\right)&\text{if }j\in J(\bbs),\\
A_{\dsquare, \bbs; j}&\text{if }j\in S_{-\infty}\sqcup S_{\infty}.
\end{cases}
\end{align*}
Note that, by \eqref{eq:isigma_D1} and \eqref{eq:isigma_D3},
\begin{align}
\isigma(A_{\square, \bbs; j})
&=\begin{cases}
\qq^{\frac{1}{2}(\gamma_{\bbs; j}^{\ast}, [\gamma_{\bbs; j}^{\ast}]_+)-\frac{1}{2}(\delta_{\bbs; j}^{\ast}, [\delta_{\bbs; j}^{\ast}]_+)}e'([\delta_{\bbs; j}^{\ast}]_+)\Delta_{\gamma_{\bbs; j}^{\ast}, \delta_{\bbs; j}^{\ast}}e([\gamma_{\bbs; j}^{\ast}]_+)&\text{if }j\in J(\bbs),\\
e(\varpi_s^{\ast})&\text{if }j=s_{-\infty}\in S_{-\infty},\\
e'(\varpi_s^{\ast})&\text{if }j=s_{\infty}\in S_{\infty}.
\end{cases}\label{eq:isigma_A_square}
\end{align}

Define $\Lambda_{\square}(\bbs)=(\Lambda_{\square, \bbs; ij})_{i,j\in J_{\square}(\bbs)}$ and $\cE_{\square}(\bbs)=(\varepsilon_{\square, \bbs;ij})_{i\in J_{\square}(\bbs)^{\uf},j\in J_{\square}(\bbs)}$ by 
\begin{align}
    &A_{\square, \bbs; i}A_{\square, \bbs; j}=\qq^{\Lambda_{\square,\bbs;ij}}A_{\square, \bbs; j}A_{\square, \bbs; i}\quad \text{for }i, j\in J_{\square}(\bbs),\label{eq:Lambda_square_def}\\
    &\sum_{j\in J_{\square}(\bbs)}\varepsilon_{\square, \bbs;ij}f_{\bbs; j}=\sum_{j\in J_{\square}(\bbs)}\varepsilon_{\dsquare, \bbs;ij}f_j
    \quad \text{for }i\in J_{\square}(\bbs)^{\uf}.\label{eq:varepsilon_square_def}
\end{align}
Note that 
\[
    \varepsilon_{\square, \bbs;ij}=
    \varepsilon_{\dsquare, \bbs;ij}=
    \varepsilon_{\bbs;ij}
  \]  
for $i\in J_{\square}(\bbs)^{\uf}$ and $j\in J(\bbs)$. The following theorem follows from \cref{thm:cluster_dsquare} and \cref{prop:modified_seed}.
\begin{thm}\label{thm:cluster_square}
Let $\bbs\in S(w_0, w_0)$. The collection $\bi_{\square}(\bbs):=(\Lambda_{\square}(\bbs), \cE_{\square}(\bbs),(A_{\square, \bbs; j})_{j\in J_{\square}(\bbs)})$ forms a quantum seed in $\cF(\cO_\qq(\Conf_4^{\times}\sA_G))$. Moreover, 
\[
A_{\qq}(\bi_{\square}(\bbs); J_{\square}(\bbs)^{\fr})=U_{\qq}(\bi_{\square}(\bbs); J_{\square}(\bbs)^{\fr})=\cO_\qq(\Conf_4^{\times}\sA_G).
\]
\end{thm}

\begin{ex}\label{ex:SST}
In the case of type $A_2$ and $\bbs=(1,2,1,\overline{2},\overline{1},\overline{2})$, the relation between $\bi_{\dsquare}(\bbs)$ and $\bi_{\square}(\bbs)$ is shown in  \cref{fig:SST}. Here we express the exchange matrices $\cE_{\dsquare}(\bbs)$ and $\cE_{\square}(\bbs)$ by quivers as in \cref{sec:quiver} and label each vertex with the image under $\isigma$ of the corresponding quantum cluster variable.
\begin{figure}[ht]
    \centering
    \scalebox{0.85}{
\begin{tikzpicture}
\def\R{4.6}
\def\H{1.5}
\node(V10) at (0,0) {$\Delta_{\varpi_2, w_0\varpi_2}$};
\node(V11) at (\R,0) {$\Delta_{\varpi_2, r_2\varpi_2}$};
\node(V12) at (2*\R,0) {$\Delta_{\varpi_2, \varpi_2}$};
\node(V13) at (3*\R,0) {$\Delta_{w_0\varpi_2, \varpi_2}$};
\node(V20) at (0.5*\R,\H) {$\Delta_{\varpi_1, w_0\varpi_1}$};
\node(V21) at (1.5*\R,\H) {$\Delta_{\varpi_1, \varpi_1}$};
\node(V22) at (2.5*\R,\H) {$\Delta_{r_1\varpi_1, \varpi_1}$};
\node(V23) at (3.5*\R,\H) {$\Delta_{w_0\varpi_1, \varpi_1}$};
\qarrow{V11}{V10};
\qarrow{V12}{V11};
\qarrow{V21}{V20};
\qarrow{V21}{V22};
\qarrow{V22}{V23};
\qarrow{V12}{V13};
\qarrow{V20}{V11};
\qarrow{V11}{V21};
\qarrow{V13}{V22};
\qarrow{V22}{V12};
\qdarrow{V10}{V20};
\qdarrow{V23}{V13};
{\color{myblue}
\node(Y1) at (0.5*\R,-\H) {$e'(\varpi_1)$};
\node(Y2) at (1.5*\R,-\H) {$e'(\varpi_2)$};
\node(Y3) at (2*\R,2*\H) {$e(\varpi_1)$};
\node(Y4) at (3*\R,2*\H) {$e(\varpi_2)$};
}

\begin{scope}[yshift=-4.5*\H cm]
\node(V10) at (0,0) {$\qq^{\frac{1}{3}}\Delta_{\varpi_2, w_0\varpi_2}e(\varpi_2)$};
\node(V11) at (\R,0) {$\qq^{\frac{1}{6}}e'(\varpi_1)\Delta_{\varpi_2, r_2\varpi_2}e(\varpi_2)$};
\node(V12) at (2*\R,0) {$e'(\varpi_2)\Delta_{\varpi_2, \varpi_2}e(\varpi_2)$};
\node(V13) at (3*\R,0) {$\qq^{-\frac{1}{3}}e'(\varpi_2)\Delta_{w_0\varpi_2, \varpi_2}$};
\node(V20) at (0.5*\R,\H) {$\qq^{\frac{1}{3}}\Delta_{\varpi_1, w_0\varpi_1}e(\varpi_1)$};
\node(V21) at (1.5*\R,\H) {$e'(\varpi_1)\Delta_{\varpi_1, \varpi_1}e(\varpi_1)$};
\node(V22) at (2.5*\R,\H) {$\qq^{-\frac{1}{6}}e'(\varpi_1)\Delta_{r_1\varpi_1, \varpi_1}e(\varpi_2)$};
\node(V23) at (3.5*\R,\H) {$\qq^{-\frac{1}{3}}e'(\varpi_1)\Delta_{w_0\varpi_1, \varpi_1}$};
\qarrow{V11}{V10};
\qarrow{V12}{V11};
\qarrow{V21}{V20};
\qarrow{V21}{V22};
\qarrow{V22}{V23};
\qarrow{V12}{V13};
\qarrow{V20}{V11};
\qarrow{V11}{V21};
\qarrow{V13}{V22};
\qarrow{V22}{V12};
\qdarrow{V10}{V20};
\qdarrow{V23}{V13};
{\color{myblue}
\node(Y1) at (0.5*\R,-\H) {$e'(\varpi_1)$};
\node(Y2) at (1.5*\R,-\H) {$e'(\varpi_2)$};
\node(Y3) at (2*\R,2*\H) {$e(\varpi_1)$};
\node(Y4) at (3*\R,2*\H) {$e(\varpi_2)$};
\qarrow{V10}{Y1}; 
\qarrow{Y1}{V11};
\qarrow{V11}{Y2};
\qarrow{Y2}{V12};
\qdarrow{Y2}{Y1};
\qarrow{V22}{Y3};
\qarrow{Y3}{V21};
\qarrow{V23}{Y4};
\qarrow{Y4}{V22};
\qdarrow{Y3}{Y4};
}
\end{scope}
\end{tikzpicture}}
    \caption{$\bi_{\dsquare}(\bbs)$ (above) and $\bi_{\square}(\bbs)$ (below).}
    \label{fig:SST}
\end{figure}
\end{ex}

The following theorem verifies that the quantum seed $\bi_{\square}(\bbs)$ is a quantum analogue of the Goncharov--Shen seed $\bi^\GS(\bbs)=(\cE^{\GS}(\bbs), (A_{\bbs; j}^{\GS})_{j\in J_{\square}(\bbs)})$ for $\Conf_4^{\times}\sA_G$. See \cref{subsub:cluster_config} for its precise definition. 

\begin{thm}\label{thm:GS_quantum}
For $\bbs\in S(w_0, w_0)$, $\cE_{\square}(\bbs)=\cE^{\GS}(\bbs)$. Moreover, for $\bbs, \bbs'\in S(w_0, w_0)$, $\bi_{\square}(\bbs)$ and $\bi_{\square}(\bbs')$ are mutation equivalent. 
\end{thm}
\begin{proof}
    We have a decomposition 
    \[
    {}^{\ext'}\cO_\qq(G^{w_0, w_0})^{\ext}=\bigoplus_{(\lambda, \mu)\in P\oplus P }{}^{\ext'}\cO_\qq(G^{w_0, w_0})^{\ext}_{\lambda, \mu},
    \]
    where ${}^{\ext'}\cO_\qq(G^{w_0, w_0})^{\ext}_{\lambda, \mu}:=e'(\mu)\cO_\qq(G^{w_0, w_0})e(\lambda^{\ast})$. 
    This decomposition endows ${}^{\ext'}\cO_\qq(G^{w_0, w_0})^{\ext}$ with a $P\oplus P$-graded algebra structure. Through the isomorphism $\isigma\colon\cO_\qq(\Conf_4^{\times}\sA_G)\xrightarrow{\sim}{}^{\ext'}\cO_\qq(G^{w_0, w_0})^{\ext}$, $\cO_\qq(\Conf_4^{\times}\sA_G)$ is also endowed with the $P\oplus P$-graded algebra structure. 
    By \eqref{eq:isigma_formula}, this grading is the same as the one induced by the first and third components of the $P^{\oplus 4}$-grading on $\cO_\qq(\Conf_4\sA_G)$ (recall \cref{prop:qcsgrading}). 
    For $\phi\in (\isigma)^{-1}({}^{\ext'}\cO_\qq(G^{w_0, w_0})^{\ext}_{\lambda, \mu})$, we write $\Deg(\phi)=(\Deg_1(\phi), \Deg_3(\phi))=(\lambda, \mu)$ in this proof. By \eqref{eq:isigma_A_square}, 
    \begin{align*}
        \Deg(A_{\square, \bbs; j})=
        \begin{cases}
            ([\gamma_{\bbs; j}]_+, [\delta_{\bbs; j}^{\ast}]_+)&\text{if }j\in J(\bbs),\\
            (\varpi_s, 0)&\text{if }j=s_{-\infty}\in S_{-\infty},\\
            (0, \varpi_s^{\ast})&\text{if }j=s_{\infty}\in S_{\infty}.
        \end{cases}
    \end{align*}
    These data coincide with $(\deg_{\flA^{\sfL}}(A_{\bbs; j}^\GS), \deg_{\flA_{\sfR}}(A_{\bbs; j}^\GS))=:(\deg_{1}(A_{\bbs; j}^\GS), \deg_{3}(A_{\bbs; j}^\GS))$ for $j\in J_{\square}(\bbs)$ in \cref{subsec:weights} by \cref{prop:weight}. By \eqref{eq:weight_Poisson}, 
    \begin{align*}
    \sum_{j\in J_{\square}(\bbs)} \varepsilon_{\bbs;ij}^\GS \Deg_k(A_{\square, \bbs; j})=
        \sum_{j\in J_{\square}(\bbs)} \varepsilon_{\bbs;ij}^\GS \deg_k(A_{\bbs; j}^\GS)=0
    \end{align*}
    for $k=1, 3$ and $i\in J_{\square}(\bbs)^{\uf}$. On the other hand, by \cref{prop:weight0}, 
    \begin{align*}
    \sum_{j\in J_{\square}(\bbs)} \varepsilon_{\square, \bbs;ij} \Deg_k(A_{\square, \bbs; j})=0
    \end{align*}
    for $k=1, 3$ and $i\in J_{\square}(\bbs)^{\uf}$. Moreover, 
    \[
    \varepsilon_{\bbs;ij}^\GS=\varepsilon_{\bbs;ij}=\varepsilon_{\square, \bbs;ij}
    \]
     for $i\in J_{\square}(\bbs)^{\uf}$ and $j\in J(\bbs)$ (see \cref{conv:labeling_corresp}). 
    Therefore, we obtain 
    \begin{align*}
        \sum_{j\in S_{-\infty}\sqcup S_{\infty}} \varepsilon_{\bbs;ij}^\GS \Deg_k(A_{\square, \bbs; j})=
        \sum_{j\in S_{-\infty}\sqcup S_{\infty}} \varepsilon_{\square, \bbs;ij} \Deg_k(A_{\square, \bbs; j})
    \end{align*}
    for $k=1, 3$ and $i\in J_{\square}(\bbs)^{\uf}$, which implies 
    \[
    \varepsilon_{\bbs;ij}^\GS=\varepsilon_{\square, \bbs;ij}
    \]
     for $i\in J_{\square}(\bbs)^{\uf}$ and $j\in S_{-\infty}\sqcup S_{\infty}$. Therefore, $\cE_{\square}(\bbs)=\cE^{\GS}(\bbs)$. 

     Let $\bbs, \bbs'\in S(w_0, w_0)$. By \cref{thm:cluster_dsquare}, there exist a sequence $(t_1,\dots, t_k)$ of elements of $J(\bbs)^{\uf}=J_{\square}(\bbs)^{\uf}$ and an admissible permutation $\sigma$ of $J_{\square}(\bbs)$ (see \cref{app:qca}) such that 
     \[
     \sigma\mu_{t_k}\cdots \mu_{t_1}\bi_{\dsquare}(\bbs)=\bi_{\dsquare}(\bbs').
     \]
     Then, by \cref{prop:modified_seed} (2), $\sigma\mu_{t_k}\cdots \mu_{t_1}\bi_{\square}(\bbs)$ is a strict similarity transform of $\bi_{\dsquare}(\bbs')$. For $j\in J_{\square}(\bbs)^{\uf}$, we have 
     \begin{align*}
     \bg_{\bi^\GS(\bbs)}(A_{\sigma\mu_{t_k}\cdots \mu_{t_1}\bi^\GS(\bbs); j})&=
     \bg_{\bi_{\dsquare}(\bbs)}(A_{\sigma\mu_{t_k}\cdots \mu_{t_1}\bi_{\dsquare}(\bbs); j})\\
     &=     \bg_{\bi_{\dsquare}(\bbs)}(A_{\dsquare, \bbs'; j})\\
     &=\bg_{\bi(\bbs)}(\Delta_{\gamma_{\bbs'; j}^{\ast}, \delta_{\bbs'; j}^{\ast}})\\
     &=\bg_{\bi^\GS(\bbs)}(A_{\bi^\GS(\bbs'); j}).
    \end{align*}
    Here the first equality follows from $\cE_{\square}(\bbs)=\cE^{\GS}(\bbs)$ and \cref{prop:modified_seed} (3), and the last equality follows from \cref{thm:indep_GS} and \cref{prop:GS_variable_minor}. Note that quantum generalized minors specialize to the corresponding generalized minors under the specialization at $\qq=1$ (see, for example, \cite[Section 2.3]{OQY} for the details of the specialization). Therefore, by \cref{thm:separation}, we have $A_{\sigma\mu_{t_k}\cdots \mu_{t_1}\bi^\GS(\bbs); j}=A_{\bi^\GS(\bbs'); j}$. By definition of the mutation, we have 
    \begin{align*}
        \Deg(A_{\sigma\mu_{t_k}\cdots \mu_{t_1}\bi_{\square}(\bbs); j})&=(\deg_1(A_{\sigma\mu_{t_k}\cdots \mu_{t_1}\bi^\GS(\bbs); j}), \deg_3(A_{\sigma\mu_{t_k}\cdots \mu_{t_1}\bi^\GS(\bbs); j}))\\
        &=(\deg_1(A_{\bi^\GS(\bbs'); j}), \deg_3(A_{\bi^\GS(\bbs'); j}))\\
        &=([\gamma_{\bbs'; j}]_+, [\delta_{\bbs'; j}^{\ast}]_+).
    \end{align*}
    for $j\in J_{\square}(\bbs)$. Therefore, the similarity datum for making $\sigma\mu_{t_k}\cdots \mu_{t_1}\bi_{\square}(\bbs)$ from $\bi_{\dsquare}(\bbs')$ is the same as that for making $\bi_{\square}(\bbs')$ from $\bi_{\dsquare}(\bbs')$. Therefore, $\sigma\mu_{t_k}\cdots \mu_{t_1}\bi_{\square}(\bbs)=\bi_{\square}(\bbs')$. 
\end{proof}
We conclude this subsection by stating the following refinement of \cref{thm:cluster_square}, which will be used in \cref{sec:polygon}. 
\begin{thm}\label{thm:cluster_square_cpt}
    For $\bbs\in S(w_0, w_0)$, we have  
\[
A_{\qq}(\bi_{\square}(\bbs); S_{-\infty}\sqcup S_{\infty})=U_{\qq}(\bi_{\square}(\bbs); S_{-\infty}\sqcup S_{\infty})=\cO_\qq(\Conf_4^{\{1, 3\}}\sA_G).
\]
\end{thm}
\begin{proof}
    By \cref{prop:modified_seed}, 
    \begin{align*}
        &A_{\qq}(\bi_{\square}(\bbs); S_{-\infty}\sqcup S_{\infty})=A_{\qq}(\bi_{\dsquare}(\bbs); S_{-\infty}\sqcup S_{\infty})=\bigoplus_{(\lambda, \mu)\in P\oplus P}e'(\mu)(\isigma)^{-1}(A_{\qq}(\bi(\bbs); \varnothing))e(\lambda^{\ast}),\\
        &U_{\qq}(\bi_{\square}(\bbs); S_{-\infty}\sqcup S_{\infty})=U_{\qq}(\bi_{\dsquare}(\bbs); S_{-\infty}\sqcup S_{\infty})=\bigoplus_{(\lambda, \mu)\in P\oplus P}e'(\mu)(\isigma)^{-1}(U_{\qq}(\bi(\bbs); \varnothing))e(\lambda^{\ast}).
    \end{align*}
    By \cite[Theorem B]{QY}, \cite[Theorem A]{OQY}, and \cite[Corollary B]{Dey}, we have 
    \[
    A_{\qq}(\bi(\bbs); \varnothing)=U_{\qq}(\bi(\bbs); \varnothing)=\cO_\qq(G).
    \]
    Therefore, by \cref{thm:qW}, 
    \[
    A_{\qq}(\bi_{\square}(\bbs); S_{-\infty}\sqcup S_{\infty})=U_{\qq}(\bi_{\square}(\bbs); S_{-\infty}\sqcup S_{\infty})=\cO_\qq(\Conf_4^{\{1, 3\}}\sA_G). 
    \]
\end{proof}
\subsection{Quantum seeds for $\cO_\qq(\Conf_3^{\times}\sA_G)$} Next we construct quantum seed for $\cO_\qq(\Conf_3^{\times}\sA_G)$ by restricting $\bi_{\square}(\bbs)$. 
Fix $\bs=(s_1,\dots, s_N), \bs'=(s'_1,\dots, s'_N)\in S(w_0)$. Set $\bbs:=\bs\ast \overline{\bs'}=(s_1,\dots, s_N, \overline{s'_1},\dots, \overline{s'_N})$. Then $\bbs\in S(w_0, w_0)$. 
Write 
\begin{align}
 \iota_{43}^{1\bullet }:=\iota_{p_{1,3,4}}^{\inv},\ \text{where}\ p_{1,3,4}\colon [1,3]\to [1, 4], 1\mapsto 1, 2\mapsto 3, 3\mapsto 4.  \label{eq:triangle_emb}
\end{align}
Set $J(\bs):=S_0\sqcup [1,N]$, and $J_{\triangle}(\bs):=S_{\infty}\sqcup J(\bs)=S_{\infty}\sqcup S_0\sqcup [1,N]$. We regard these sets as subsets of $J_{\square}(\bbs)$ in an obvious manner. 

Let us describe $A_{\square, \bbs; j}$ for $j\in J(\bs)$ explicitly. In this case, $\delta_{\bbs; j}$ does not depend on the choice of $\bs'$. Hence, we write 
\[
\delta_{\bs; j}:=\delta_{\bbs; j}=\begin{cases}
        w_0\varpi_{s}&\text{ if }j=s_0\in S_0,\\
        w_0r_{s_1}\cdots r_{s_{j}}\varpi_{s_j}=r_{s_N}\cdots r_{s_{j+1}}\varpi_{s_j}&\text{ if }j\in [1,N].
    \end{cases}
\]

By \eqref{eq:isigma_A_square}, 
\begin{align*}
\isigma(A_{\square, \bbs; j})
&=\qq^{\frac{1}{2}(\varpi_{s_j}^{\ast}, \varpi_{s_j}^{\ast})-\frac{1}{2}(\delta_{\bs; j}^{\ast}, [\delta_{\bs; j}^{\ast}]_+)}e'([\delta_{\bs; j}^{\ast}]_+)\Delta_{\varpi_{s_j}^{\ast}, \delta_{\bs; j}^{\ast}}e(\varpi_{s_j}^{\ast}).
\end{align*}
On the other hand, by \cref{thm:isigma}, 
\begin{align*}
\isigma(\psi_{\varpi_{s_j}, [\delta_{\bs; j}^{\ast}]_+, [\delta_{\bs; j}]_-})
&=(-1)^{\langle 2\rho^{\vee}, [\delta_{\bs; j}]_-\rangle}\qq^{-(2\rho, [\delta_{\bs; j}]_-)-\frac{1}{2}([\delta_{\bs; j}^{\ast}]_+,[\delta_{\bs; j}^{\ast}]_+)}e'([\delta_{\bs; j}^{\ast}]_+)\Delta_{\varpi_{s_j}^{\ast}, \delta_{\bs; j}^{\ast}}e(\varpi_{s_j}^{\ast}).
\end{align*}
Therefore, 
\begin{align*}
    A_{\square, \bbs; j}&=(-1)^{\langle 2\rho^{\vee}, [\delta_{\bs; j}]_-\rangle}\qq^{(2\rho, [\delta_{\bs; j}]_-)+\frac{1}{2}(\varpi_{s_j}^{\ast}, \varpi_{s_j}^{\ast})+\frac{1}{2}([\delta_{\bs; j}^{\ast}]_-,[\delta_{\bs; j}^{\ast}]_+)}\psi_{\varpi_{s_j}, [\delta_{\bs; j}^{\ast}]_+, [\delta_{\bs; j}]_-}\\
    &=(-1)^{\langle 2\rho^{\vee}, [\delta_{\bs; j}]_-\rangle}\qq^{(2\rho, [\delta_{\bs; j}]_-)+\frac{1}{2}(\varpi_{s_j}^{\ast}, \varpi_{s_j}^{\ast})+\frac{1}{2}([\delta_{\bs; j}^{\ast}]_-,[\delta_{\bs; j}^{\ast}]_+)}\\
    &\phantom{=}\times \xi_{\delta_{\bs; j}}(\varpi_{s_j})\otimes 1\otimes \xi_{w_0[\delta_{\bs; j}^{\ast}]_+}([\delta_{\bs; j}^{\ast}]_+)\otimes \xi_{[\delta_{\bs; j}]_-}([\delta_{\bs; j}]_-)\\
    &\phantom{=}+(\text{other terms}),
\end{align*}
where $\mathrm{(}$other terms$\mathrm{)}$ belong to 
    \[
    \cO(\varpi_{s_j})\ut 1\ut \cO([\delta_{\bs; j}^{\ast}]_+)_{<[\delta_{\bs; j}]_+}\ut \cO([\delta_{\bs; j}]_-)+\cO(\varpi_{s_j})\ut 1\ut \cO([\delta_{\bs; j}^{\ast}]_+)\ut \cO([\delta_{\bs; j}]_-)_{> -[\delta_{\bs; j}]_-}.
    \]
Therefore, for $j\in J_{\triangle}(\bs)$, there uniquely exists $A_{\triangle, \bs; j}\in \cO_\qq(\Conf_3\sA_G)$ such that $\iota_{43}^{1\bullet }(A_{\triangle, \bs; j})=A_{\square, \bbs; j}$. More explicitly, 
\begin{align}
A_{\triangle, \bs; j}&=\begin{cases}
\begin{array}{ll}
(-1)^{\langle 2\rho^{\vee}, [\delta_{\bs; j}]_-\rangle}\qq^{(2\rho, [\delta_{\bs; j}]_-)+\frac{1}{2}(\varpi_{s_j}^{\ast}, \varpi_{s_j}^{\ast})+\frac{1}{2}([\delta_{\bs; j}^{\ast}]_-,[\delta_{\bs; j}^{\ast}]_+)}\\
    \phantom{=}\times \xi_{\delta_{\bs; j}}(\varpi_{s_j})\otimes \xi_{w_0[\delta_{\bs; j}^{\ast}]_+}([\delta_{\bs; j}^{\ast}]_+)\otimes \xi_{[\delta_{\bs; j}]_-}([\delta_{\bs; j}]_-)\\
    \phantom{=}+(\text{other terms})    
\end{array}&\text{if }j\in J(\bs),\\
D_{2; \varpi_s}&\text{if }j=s_{\infty}\in S_{\infty},
\end{cases}\label{eq:A_triangle}
\end{align}
    where $\mathrm{(}$other terms$\mathrm{)}$ in the first case belong to 
    \[
    \cO(\varpi_{s_j})\ut \cO([\delta_{\bs; j}^{\ast}]_+)_{<[\delta_{\bs; j}]_+}\ut \cO([\delta_{\bs; j}]_-)+\cO(\varpi_{s_j})\ut \cO([\delta_{\bs; j}^{\ast}]_+)\ut \cO([\delta_{\bs; j}]_-)_{> -[\delta_{\bs; j}]_-}.
    \]
We set
\begin{align*}
&J_{\triangle}(\bs)^{\uf}:=\{j\in [1,N]\mid j^+\leq N\},\ J_{\triangle}(\bs)^{\fr}:=J_{\triangle}(\bs)\setminus J_{\triangle}(\bs)^{\uf},\\
&\Lambda_{\triangle}(\bs)=(\Lambda_{\triangle, \bs; ij})_{i,j\in J_{\triangle}(\bs)},\  \Lambda_{\triangle,\bs;ij}:=\Lambda_{\square,\bbs;ij}\ \text{for}\ i, j\in J_{\triangle}(\bs),\\
&\cE_{\triangle}(\bs)=(\varepsilon_{\triangle, \bs;ij})_{i\in J_{\triangle}(\bs)^{\uf},j\in J_{\triangle}(\bs)},\ \varepsilon_{\triangle, \bs;ij}:=\varepsilon_{\square, \bbs;ij}\ \text{for}\ i\in J_{\triangle}(\bs)^{\uf},j\in J_{\triangle}(\bs).
\end{align*}
Then these data give a quantum seed for $\cO_\qq(\Conf_3^{\times}\sA_G)$ as follows. 

\begin{thm}\label{thm:cluster_triangle}
Let $\bs\in S(w_0)$. The collection $\bi_{\triangle}(\bs):=(\Lambda_{\triangle}(\bs), \cE_{\triangle}(\bs),(A_{\triangle, \bs; j})_{j\in J_{\triangle}(\bs)})$ forms a quantum seed in $\cF(\cO_\qq(\Conf_3^{\times}\sA_G))$, and  
\[
A_{\qq}(\bi_{\triangle}(\bs); J_{\triangle}(\bs)^{\fr})=U_{\qq}(\bi_{\triangle}(\bs); J_{\triangle}(\bs)^{\fr})=\cO_\qq(\Conf_3^{\times}\sA_G).
\]
Moreover, for $\bs_1, \bs_2\in S(w_0)$, $\bi_{\triangle}(\bs_1)$ and $\bi_{\triangle}(\bs_2)$ are mutation equivalent.
\end{thm}

The rest of this subsection is devoted to the proof of \cref{thm:cluster_triangle}. First, we recall the Berenstein--Zelevinsky quantum seeds for $\cO_\qq(B^-)$. For $\lambda\in P_+$ and $w, w'\in W$, we write 
\[
\Delta_{w\lambda, w'\lambda}^-:=\pi_-(\Delta_{w\lambda, w'\lambda})\in \cO_\qq(B^-). 
\]
Write $\mathbf{\Delta}_{w_0,e}^-:=\{\qq^{a}\Delta_{w_0\lambda, \lambda}^- \mid a\in \frac{1}{2d}\Z, \lambda\in P_+\}$. It is known that $\mathbf{\Delta}_{w_0,e}^-$ forms an Ore set of $\cO_\qq(B^-)$, and set 
\[
\cO_\qq(B_{\ast}^-):=\cO_\qq(B^-)[(\mathbf{\Delta}_{w_0,e}^-)^{-1}].
\]
See \cite[Definition 9.5]{BZ}. Note that $\cO_\qq(B_{\ast}^-)$ has a  $P\oplus P$-graded algebra structure induced from that of $\cO_\qq(B^-)$.

Write $\bs^\op:=(s_N,\dots, s_1)\in S(w_0)$. For $j\in J(\bs^\op)$\footnote{As sets, $J(\bs^\op)=S_0\sqcup [1, N]=J(\bs)$. However, we distinguish between them to avoid possible confusion. }, set 
\begin{align*}
\Delta_{\bs^\op; j}^-:=\Delta_{\gamma_{\bs^{\op}; j}, \varpi_{s_{N+1-j}}}^-, 
\end{align*}
where
\begin{align}
\gamma_{\bs^{\op}; j}:=\begin{cases}
\varpi_s&\text{if}\ j=s_0\in S_0,\\
r_{s_N}\cdots r_{s_{N+1-j}}\varpi_{s_{N+1-j}}&\text{if}\ j\in [1, N],
\end{cases}\quad\text{and}\quad  
s_{N+1-s_0}:=s\ \text{for}\ s_0\in S_0\label{eq:gamma_op}
\end{align}
We define an injective map $\sfI_{\bs}\colon J(\bs) \to S\times \Z_{\geq 0}$ by 
\begin{align}
\sfI_{\bs}(j):=
\begin{cases}
    (s, 0)&\text{if }j=s_0\in S_0,\\
    (s_j, n[j])&\text{if }j\in [1, N]\text{ where } n[j] := |\{1 \leq i \leq j\mid s_i = s_j\}|.
\end{cases}  \label{eq:sfI_def}
\end{align}
Define $\sfI_{\bs^\op}\colon J(\bs^\op) \to S\times \Z_{\geq 0}$ in the same way. For $s\in S$, we write $n^s(\bs) := |\{j\in  [1,N] \mid s_j = s\}|$. Then, $\Ima \sfI_{\bs}=\Ima \sfI_{\bs^{\op}}=\{(s, m)\mid s\in S, 0\leq m\leq n^s(\bs)\}$. We can define an involution $\sfRev$ on $\Ima \sfI_{\bs}$ by 
\begin{align*}
\sfRev((s, m)):=(s, n^s(\bs)-m)\ \text{for}\ s\in S,\ 0\leq m\leq n^s(\bs).
\end{align*}
We can define mutually inverse bijections $\sfJ_{\bs\to \bs^{\op}}:J(\bs)\to J(\bs^{\op})$ and $\sfJ_{\bs^{\op}\to \bs}:J(\bs^{\op})\to J(\bs)$ by $\sfJ_{\bs\to \bs^{\op}}:=\sfI_{\bs^\op}^{-1}\circ \sfRev \circ \sfI_{\bs}$ and $\sfJ_{\bs^{\op}\to \bs}:=\sfI_{\bs}^{-1}\circ \sfRev\circ \sfI_{\bs^{\op}}$.

Note that 
\begin{align}
    \gamma_{\bs^{\op}; j}=\delta_{\bs; \sfJ_{\bs^{\op}\to \bs}(j)}\ \text{for}\ j\in J(\bs^\op).\label{eq:gamma_delta}
\end{align}
\begin{ex}\label{ex:gamma_delta}
  In the case of type $A_3$, $\bs:=(2,1,2,3,2,1)\in S(w_0)$. Then $\bs^\op:=(1,2,3,2,1,2)$, and $n^1(\bs)=2, n^2(\bs)=3, n^3(\bs)=1$. We have the following correspondence. 
  \[
  \begin{array}{ccccccc}
       J(\bs)& \overset{\sfI_{\bs}}{\longrightarrow}&  \Ima \sfI_{\bs}&\overset{\sfRev}{\longleftrightarrow}& \Ima \sfI_{\bs}&\overset{\sfI_{\bs^\op}}{\longleftarrow}&J(\bs^{\op}) \\
       \rotatebox{90}{$\in$}& &\rotatebox{90}{$\in$}& &\rotatebox{90}{$\in$}& &\rotatebox{90}{$\in$}\\
       1_0&\mapsto &(1,0)&\leftrightarrow&(1,2)&\mapsfrom &5\\
       2_0&\mapsto &(2,0)&\leftrightarrow&(2,3)&\mapsfrom &6\\
       3_0&\mapsto &(3,0)&\leftrightarrow&(3,1)&\mapsfrom &3\\
       1&\mapsto &(2,1)&\leftrightarrow&(2,2)&\mapsfrom &4\\
       2&\mapsto &(1,1)&\leftrightarrow&(1,1)&\mapsfrom &1\\
       3&\mapsto &(2,2)&\leftrightarrow&(2,1)&\mapsfrom &2\\
       4&\mapsto &(3,1)&\leftrightarrow&(3,0)&\mapsfrom &3_0\\
       5&\mapsto &(2,3)&\leftrightarrow&(2,0)&\mapsfrom &2_0\\
       6&\mapsto &(1,2)&\leftrightarrow&(1,0)&\mapsfrom &1_0
  \end{array}
  \]
Moreover, 
\begin{align*}
    &\gamma_{\bs^{\op}; 1_0}=\varpi_1=w_0r_2r_1r_2r_3r_2r_1\varpi_1=\delta_{\bs; 6},\\
    &\gamma_{\bs^{\op}; 2_0}=\varpi_2=w_0r_2r_1r_2r_3r_2r_1\varpi_2=w_0r_2r_1r_2r_3r_2\varpi_2=\delta_{\bs; 5},\\
    &\gamma_{\bs^{\op}; 3_0}=\varpi_3=w_0r_2r_1r_2r_3r_2r_1\varpi_3=w_0r_2r_1r_2r_3\varpi_3=\delta_{\bs; 4},\\
    &\gamma_{\bs^{\op}; 1}=r_1\varpi_1=w_0r_2r_1r_2r_3r_2\varpi_1=w_0r_2r_1\varpi_1=\delta_{\bs; 2},\\
    &\gamma_{\bs^{\op}; 2}=r_1r_2\varpi_2=w_0r_2r_1r_2r_3\varpi_2=w_0r_2r_1r_2\varpi_2=\delta_{\bs; 3},\\
    &\gamma_{\bs^{\op}; 3}=r_1r_2r_3\varpi_3=w_0r_2r_1r_2\varpi_3=w_0\varpi_3=\delta_{\bs; 3_0},\\
    &\gamma_{\bs^{\op}; 4}=r_1r_2r_3r_2\varpi_2=w_0r_2r_1\varpi_2=w_0r_2\varpi_2=\delta_{\bs; 1},\\
    &\gamma_{\bs^{\op}; 5}=r_1r_2r_3r_2r_1\varpi_1=w_0r_2\varpi_1=w_0\varpi_1=\delta_{\bs; 1_0},\\
    &\gamma_{\bs^{\op}; 6}=r_1r_2r_3r_2r_1r_2\varpi_2=w_0\varpi_2=\delta_{\bs; 2_0}.
\end{align*}

\end{ex}
For $j\in J(\bs^\op)$, set 
\begin{align*}
&j^+:=\begin{cases}
    \min\{l\in [1, N]\mid s_{N+1-l}=s \}&\text{if }j=s_0\in S_0,\\
    \min(\{l\in [1, N]\mid l>j, s_{N+1-l}=s_{N+1-j} \}\sqcup \{\infty\})&\text{if }j\in [1, N],
\end{cases}\\
&J(\bs^{\op})^\uf:=\{j\in [1, N]\mid j^+\neq \infty\},\\
&J(\bs^{\op})^\fr:=J(\bs^{\op})\setminus J(\bs^{\op})^\uf=S_0\sqcup \{j\in J(\bs^{\op})\mid j^+=\infty\}.
\end{align*}

Define the $J(\bs^\op) \times J(\bs^\op)$ skew-symmetric matrix $\Lambda^-(\bs^\op)=(\Lambda_{\bs^\op; ij}^-)_{i,j\in J(\bs^\op)}$ by 
\begin{align*}
\Lambda_{\bs^\op; ij}^-:=(\gamma_{\bs^{\op}; i},\gamma_{\bs^{\op}; j})-(\varpi_{s_{N+1-i}},\varpi_{s_{N+1-j}})
\end{align*}
for $i>j$, where we fix an arbitrary total order on $S_0$ and we set $s_0<j$ for all $s\in S$ and $j\in [1, N]$. Note that $\Lambda_{\bs^\op; ij}^-=0$ if $i, j\in S_0$, hence $\Lambda^-(\bs^\op)$ does not depend on the choice of the total order on $S_0$. Define the matrix $\cE^-(\bs^{\op})=(\varepsilon_{\bs^{\op};ij}^-)_{i\in J(\bs^{\op})^\uf,j\in J(\bs^{\op})}$ by
\begin{align*}
    \varepsilon_{\bs^{\op};ij}^-=
    \begin{cases}
        -1 & \text{if}\ i=j^+, \\
        1 & \text{if}\ j=i^+,\\
        -c_{s_{N+1-j},s_{N+1-i}} & \text{if}\ j<i<j^+<i^+,\\
        c_{s_{N+1-j},s_{N+1-i}} &  \text{if}\ i<j<i^+<j^+,\\
        0 & \text{otherwise.}       
    \end{cases}
\end{align*}
By \cite[Theorem 8.3]{BZ}, the collection $\bi^-(\bs^\op):=(\Lambda^-(\bs^\op), \cE^-(\bs^\op),(A_{\bi^-(\bs^\op); j})_{j\in J(\bs^\op)})$ forms a quantum seed. 
We can directly check that $\cE^-(\bs^\op)$ coincides with the exchange matrix of $\bJ(\overline{\bs}^\op)$ defined in \cref{sec:quiver}, by the labeling in \cref{conv:labeling_corresp}. See also \cref{ex:cE^-} below.

\begin{thm}[{\cite[Main Theorem]{GY:BZ}}]\label{thm:clusterBorel}
There exists an algebra isomorphism 
\[
\kappa_{\bs^\op}^-:U_{\qq}(\bi^-(\bs^\op); J(\bs^\op)^{\fr})\simeq \cO_\qq(B_{\ast}^-)
\]
satisfying $\kappa_{\bs^\op}^-(A_{\bi^-(\bs^\op); j})=\Delta_{\bs^\op; j}^-$ for $j\in J(\bs^\op)$. Moreover, $A_{\qq}(\bi^-(\bs^\op); J(\bs^\op)^{\fr})=U_{\qq}(\bi^-(\bs^\op); J(\bs^\op)^{\fr})$. 
\end{thm}
In the following, we identify $U_{\qq}(\bi^-(\bs^\op); J(\bs^\op)^{\fr})$ (and $A_{\qq}(\bi^-(\bs^\op); J(\bs^\op)^{\fr})$) with $\cO_\qq(B_{\ast}^-)$ through $\kappa_{\bs^\op}^-$. In particular, we will write $A_{\bi^-(\bs^\op); j}=\Delta_{\bs^\op; j}^-$ and regard $\bi^-(\bs^\op)$ as a quantum seed in $\cF(\cO_\qq(B_{\ast}^-))$. We extend the quantum structure of $\cO_\qq(B_\ast^-)$ to that of $\cO_\qq(B_\ast^-)^{\ext}$, as follows. We set
\begin{align*}
&\widetilde{J}(\bs^\op):=S_{\infty}\sqcup J(\bs^\op)=S_{\infty}\sqcup S_0\sqcup [1,N],\\
&\widetilde{J}(\bs^\op)^{\uf}:=J(\bs^\op)^{\uf},\ \widetilde{J}(\bs^\op)^{\fr}:=\widetilde{J}(\bs^\op)\setminus \widetilde{J}(\bs^\op)^{\uf}=S_{\infty}\sqcup J(\bs^\op)^{\fr},\\
&\widetilde{\Lambda}^-(\bs^\op)=(\widetilde{\Lambda}_{\bs^\op; ij}^-)_{i,j\in \widetilde{J}(\bs^\op)},\  \widetilde{\Lambda}_{\bs^\op; ij}^-:=\begin{cases}
    \Lambda_{\bs^\op; ij}^-&\text{if }i, j\in  J(\bs^\op),\\
    0&\text{if }i, j\in  S_{\infty},\\
    (\varpi_s, \gamma_{\bs^\op; j})&\text{if }i=s_{\infty}\in  S_{\infty},\ j\in  J(\bs^\op),\\
    -(\gamma_{\bs^\op; i}, \varpi_s)&\text{if }i\in  J(\bs^\op), j=s_{\infty}\in  S_{\infty}.
\end{cases}\\
&\widetilde{\cE}^-(\bs^\op)=(\widetilde{\varepsilon}_{\bs^\op;ij}^-)_{i\in \widetilde{J}(\bs^\op)^{\uf},j\in \widetilde{J}(\bs^\op)},\ \widetilde{\varepsilon}_{\bs^\op;ij}^-:=
\begin{cases}
    \varepsilon_{\bs^\op;ij}^-&\text{if }i\in \widetilde{J}(\bs^\op)^{\uf}, j\in J(\bs^\op),\\
    0&\text{if }i\in \widetilde{J}(\bs^\op)^{\uf}, j\in  S_{\infty}.
\end{cases}\\
&\widetilde{\Delta}_{\bs^\op; j}^-:=
\begin{cases}
\Delta_{\bs^\op; j}^-&\text{if }j\in  J(\bs^\op),\\
e(\varpi_s)&\text{if }j=s_{\infty}\in  S_{\infty}.
\end{cases}
\end{align*}
Note that, by \cref{thm:clusterBorel} and the definition of $\cO_\qq(B_{\ast}^-)^{\ext}$, we have 
\begin{align*}
\widetilde{\Delta}_{\bs^\op; i}^-\widetilde{\Delta}_{\bs^\op; j}^-=\qq^{\widetilde{\Lambda}_{\bs^\op; ij}^-}\widetilde{\Delta}_{\bs^\op; j}^-\widetilde{\Delta}_{\bs^\op; i}^-
\end{align*}
for $i, j\in \widetilde{J}(\bs^\op)$. The following theorem can be proved in the same way as \cref{thm:cluster_dsquare}. 
\begin{thm}\label{thm:cluster_Be}
The collection $\widetilde{\bi}^-(\bs^\op):=(\widetilde{\Lambda}^-(\bs^\op), \widetilde{\cE}^-(\bs^\op),(\Delta_{\bs^\op; j}^-)_{j\in \widetilde{J}(\bs^\op)})$ forms a quantum seed in $\cF(\cO_\qq(B_{\ast}^-)^{\ext})$, and  
\[
A_{\qq}(\widetilde{\bi}^-(\bs^\op); \widetilde{J}(\bs^\op)^{\fr})=U_{\qq}(\widetilde{\bi}^-(\bs^\op); \widetilde{J}(\bs^\op)^{\fr})=\cO_\qq(B_{\ast}^-)^{\ext}.
\]
\end{thm}
\begin{ex}\label{ex:cE^-}
In the case of type $A_3$ and $\bs=(2,1,2,3,2,1)$, the quantum seed $\widetilde{\bi}^-(\bs^\op)$ is described by \cref{ex:SST_triangle} (above). Here we express the exchange matrices $\widetilde{\cE}^-(\bs^\op)$ by a quiver as in \cref{sec:quiver} and label each vertex with the corresponding quantum cluster variable.

\end{ex}
Let us relate $\cO_\qq(\Conf_3^{\times}\sA_G)$ with $\cO_\qq(B_{\ast}^-)^{\ext}$. Recall that, by \eqref{eq:Phi_ke} and \eqref{eq:q-cyclic_shift}, we have an $\Bbbk$-algebra isomorphism 
\begin{align*}
\Phisigma:=\Phi_{3}^{\ext}\circ \sigma_3\colon\cO_\qq(\Conf_3^{\{2\}}\sA_G)\xrightarrow{\sim} \cO_\qq(\sA_G)^{\ext},
\end{align*}
determined by 
\begin{align}
    \xi(\lambda_1)\otimes \xi_{w_0\lambda_2}(\lambda_2)\otimes \xi_{\lambda_3}(\lambda_3)+(\text{other terms})\mapsto (-1)^{-\langle 2\rho^{\vee}, \lambda_3\rangle}\qq^{-( 2\rho, \lambda_3)-(\lambda_3, \lambda_3)/2}\xi(\lambda_1) e(\lambda_2^{\ast}),\label{eq:Phisigma_formula} 
\end{align}
where $\mathrm{(}$other terms$\mathrm{)}$ belong to 
    \[
    \cO(\lambda_1)\ut \cO(\lambda_2)_{<\lambda_2^{\ast}}\ut \cO(\lambda_3)+\cO(\lambda_1)\ut \cO(\lambda_2)\ut \cO(\lambda_3)_{>-\lambda_3}. 
    \]
\cref{t:invstr}, \eqref{eq:q-cyclic_shift}, and \cref{cor:D_lambda} imply that  
\begin{align*}
D_{1; \lambda}&=\qq^{(\lambda, \lambda)/2}\xi_{\lambda^{\ast}}(\lambda^{\ast})\otimes \xi_{w_0\lambda}(\lambda)\otimes 1+(\text{other terms}),\\
D_{2; \lambda}&=(-1)^{\langle 2\rho^{\vee}, \lambda\rangle}\qq^{(\lambda, \lambda)/2+(2\rho, \lambda)}\cdot 1\otimes \xi_{-\lambda}(\lambda^{\ast})\otimes \xi_{\lambda}(\lambda)+(\text{other terms}),\\
D_{3; \lambda}&=(-1)^{\langle 2\rho^{\vee}, \lambda\rangle}\qq^{(\lambda, \lambda)/2+(2\rho, \lambda)} \xi_{w_0\lambda}(\lambda)\otimes 1\otimes \xi_{\lambda^{\ast}}(\lambda^{\ast})+(\text{other terms}),
\end{align*}
where $\mathrm{(}$other terms$\mathrm{)}$ belong to the similar spaces as above. Therefore, \eqref{eq:Phisigma_formula} implies that 
\begin{align}
&\Phisigma(D_{1; \lambda})=\qq^{(\lambda, \lambda)/2}\Delta_{\lambda^{\ast}, \lambda^{\ast}}e(\lambda^{\ast}),\label{eq:Phisigma_D1}\\    
&\Phisigma(D_{2; \lambda})=e(\lambda),\label{eq:Phisigma_D2}\\    
&\Phisigma(D_{3; \lambda})=\Delta_{w_0\lambda, \lambda}.  \label{eq:Phisigma_D3} 
\end{align}
Therefore, \eqref{eq:AG_B} implies that $\Phisigma$ can be extended to 
\[
\Phisigma\colon\cO_\qq(\Conf_3^{\times}\sA_G)\xrightarrow{\sim} \cO_\qq(\sA_G)^{\ext}[\mathbf{\Delta}_{w_0,e}^{-1}, \mathbf{\Delta}_{e,e}^{-1}]\xrightarrow{\sim}\cO_\qq(B_{\ast}^-)^{\ext},
\]
where the second isomorphism is induced from $\pi_-$. 

\begin{proof}[{Proof of \cref{thm:cluster_triangle}}]
    Let us show the assertion in the first sentence. By \cref{thm:cluster_Be} and \cref{prop:modified_seed} (1), it suffices to show that 
    \[
    \Phisigma(\bi_{\triangle}(\bs)):=(\Lambda_{\triangle}(\bs), \cE_{\triangle}(\bs),(\Phisigma(A_{\triangle, \bs; j}))_{j\in J_{\triangle}(\bs)})
    \]
    is a strict similarity transform of $\widetilde{\bi}^-(\bs^\op)$, up to a suitable identification between $J_{\triangle}(\bs)$ and $\widetilde{J}(\bs^\op)$. By \eqref{eq:A_triangle} and \eqref{eq:Phisigma_formula},
    \begin{align}
\Phisigma(A_{\triangle, \bs; j})&=\begin{cases}
\qq^{\frac{1}{2}([\delta_{\bs; j}]_+,\delta_{\bs; j})}\Delta_{\delta_{\bs; j}, \varpi_{s_j}}^-e([\delta_{\bs; j}]_+)&\text{if }j\in J(\bs),\\
e(\varpi_s)&\text{if }j=s_{\infty}\in S_{\infty}.
\end{cases}\label{eq:Phisigma_cluster}
\end{align}
Here, note that 
\[
\frac{1}{2}(\varpi_{s_j}^{\ast}, \varpi_{s_j}^{\ast})+\frac{1}{2}([\delta_{\bs; j}^{\ast}]_-,\delta_{\bs; j}^{\ast})=
\frac{1}{2}(\varpi_{s_j}^{\ast}, \varpi_{s_j}^{\ast})-\frac{1}{2}(\delta_{\bs; j}^{\ast},\delta_{\bs; j}^{\ast})+\frac{1}{2}([\delta_{\bs; j}^{\ast}]_+,\delta_{\bs; j}^{\ast})=\frac{1}{2}([\delta_{\bs; j}]_+,\delta_{\bs; j}).
\]
Therefore, if we identify $J_{\triangle}(\bs)$ with $\widetilde{J}(\bs^\op)$ through the bijection
\[
\widetilde{\sfJ}_{\bs\to \bs^{\op}}:J_{\triangle}(\bs)\to \widetilde{J}(\bs^\op),\ 
\widetilde{\sfJ}_{\bs\to \bs^{\op}}(j):=\begin{cases}
    \sfJ_{\bs\to \bs^{\op}}(j)&\text{if}\ j\in J(\bs),\\
    j&\text{if}\ j\in S_{\infty},
\end{cases}
\]
which maps $J_{\triangle}(\bs)^\fr$ to $\widetilde{J}(\bs^\op)^\fr$, our claim follows from \eqref{eq:varepsilon_square_def}, \eqref{eq:gamma_delta}, and the definitions of $\Lambda_{\triangle}(\bs)$ and $\cE_{\triangle}(\bs)$. See also the construction of $\bJ_{\std}(\bs), \bJ_{\std}(\overline{\bs}^{\op})$ in \cref{sec:quiver} and \cref{ex:SST_triangle} below. 

 The mutation equivalence among $\bi_{\triangle}(\bs), \bs\in S(w_0)$ follows from \cref{thm:GS_quantum} and the fact that the mutation sequence connecting $\bi_{\square}(\bs_1\ast\overline{\bs'})$ and $\bi_{\square}(\bs_2\ast\overline{\bs'})$ for $\bs_1, \bs_2\in S(w_0)$ consists only of mutations at the vertices in $J_{\triangle}(\bs)^\uf$ (see \cite[Proposition 3.7]{SW}, \cite[Theorem 4.10]{QY}).
\end{proof}    
\begin{ex}\label{ex:SST_triangle}
In the case of type $A_3$ and $\bs=(2,1,2,3,2,1)$, the relation between $\widetilde{\bi}^-(\bs^\op)$ and $\Phisigma(\bi_{\triangle}(\bs))$ is shown in \cref{fig:SST_triangle}. 
    \begin{figure}[ht]
    \centering
\begin{tikzpicture}
\def\R{4}
\def\H{1.5}
\node(V10) at (-0.5*\R,0) {$\Delta_{w_0\varpi_1, \varpi_1}^-$};
\node(V11) at (-2.5*\R,0) {$\Delta_{r_1\varpi_1, \varpi_1}^-$};
\node(V12) at (-3.5*\R,0) {$\Delta_{\varpi_1, \varpi_1}^-$};
\node(V20) at (0,\H) {$\Delta_{w_0\varpi_2, \varpi_2}^-$};
\node(V21) at (-\R,\H) {$\Delta_{r_1r_2r_3r_2\varpi_2, \varpi_2}^-$};
\node(V22) at (-2*\R,\H) {$\Delta_{r_1r_2\varpi_2, \varpi_2}^-$};
\node(V23) at (-3*\R,\H) {$\Delta_{\varpi_2, \varpi_2}^-$};
\node(V30) at (-1.5*\R,2*\H) {$\Delta_{w_0\varpi_3, \varpi_3}^-$};
\node(V31) at (-2.5*\R,2*\H) {$\Delta_{\varpi_3, \varpi_3}^-$};
\qarrow{V12}{V11};
\qarrow{V11}{V10};
\qarrow{V23}{V22};
\qarrow{V22}{V21};
\qarrow{V21}{V20};
\qarrow{V31}{V30};
\qarrow{V10}{V21};
\qarrow{V21}{V11};
\qarrow{V11}{V23};
\qarrow{V30}{V22};
\qarrow{V22}{V31};
\qdarrow{V20}{V10};
\qdarrow{V23}{V12};
\qdarrow{V20}{V30};
\qdarrow{V31}{V23};
{\color{myblue}
\node(Y1) at (-3*\R,-\H) {$e(\varpi_1)$};
\node(Y2) at (-1.5*\R,-\H) {$e(\varpi_3)$};
\node(Y3) at (-0.5*\R,-\H) {$e(\varpi_2)$};
}
\begin{scope}[xshift=-3.5*\R cm, yshift=-4.5*\H cm]
\def\R{4}
\def\H{1.5}
\node(V10) at (0.5*\R,0) {$\Delta_{w_0\varpi_1, \varpi_1}^-$};
\node(V11) at (1.5*\R,0) {$\Delta_{r_1\varpi_1, \varpi_1}^-e(\varpi_2)$};
\node(V12) at (3.5*\R,0) {$\Delta_{\varpi_1, \varpi_1}^-e(\varpi_1)$};
\node(V20) at (0,\H) {$\Delta_{w_0\varpi_2, \varpi_2}^-$};
\node(V21) at (\R,\H) {$\Delta_{r_1r_2r_3r_2\varpi_2, \varpi_2}^-e(\varpi_2)$};
\node(V22) at (2*\R,\H) {$\Delta_{r_1r_2\varpi_2, \varpi_2}^-e(\varpi_3)$};
\node(V23) at (3*\R,\H) {$\Delta_{\varpi_2, \varpi_2}^-e(\varpi_2)$};
\node(V30) at (0.5*\R,2*\H) {$\Delta_{w_0\varpi_3, \varpi_3}^-$};
\node(V31) at (2.5*\R,2*\H) {$\Delta_{\varpi_3, \varpi_3}^-e(\varpi_3)$};
\qarrow{V12}{V11};
\qarrow{V11}{V10};
\qarrow{V23}{V22};
\qarrow{V22}{V21};
\qarrow{V21}{V20};
\qarrow{V31}{V30};
\qarrow{V10}{V21};
\qarrow{V21}{V11};
\qarrow{V11}{V23};
\qarrow{V30}{V22};
\qarrow{V22}{V31};
\qdarrow{V20}{V10};
\qdarrow{V23}{V12};
\qdarrow{V20}{V30};
\qdarrow{V31}{V23};
{\color{myblue}
\node(Y1) at (3*\R,-\H) {$e(\varpi_1)$};
\node(Y2) at (1.5*\R,-\H) {$e(\varpi_3)$};
\node(Y3) at (0.5*\R,-\H) {$e(\varpi_2)$};
\qarrow{V20}{Y3};
\qarrow{Y3}{V21};
\draw[qarrow] (V21) to[bend right=15] (Y2);
\draw[qarrow] (Y2) to[bend right=15] (V22);
\qarrow{V11}{Y1};
\qarrow{Y1}{V12};
\qdarrow{Y2}{Y3};
\draw[qarrow,dashed] (Y1) to[bend left=5] (Y3);
}
\end{scope}
\end{tikzpicture}
    \caption{$\widetilde{\bi}^-(\bs^\op)$ (above) and $\Phisigma(\bi_{\triangle}(\bs))$ (below).}
    \label{fig:SST_triangle}
\end{figure}
\end{ex}
\begin{ex}\label{ex:triangle_smallrak}
    Let us describe the quantum seeds for $\cO_{\qq}(\Conf_3^{\times}\sA_{SL_2})$ (type $A_1$) and  $\cO_{\qq}(\Conf_3^{\times}\sA_{SL_3})$ (type $A_2$). 
\begin{enumerate}
    \item
    In the case of type $A_1$, $S(w_0)=\{(1)\}$. Hence the quantum seed $\bi_{\triangle}((1))$ is described as follows. 
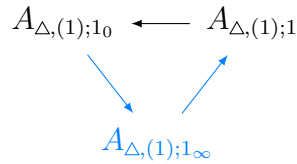
\[
\begin{tikzpicture}
\def\R{2.5}
\def\H{2}
\node(V10) at (0,0) {$A_{\triangle, (1); 1_0}$};
\node(V11) at (\R,0) {$A_{\triangle, (1); 1}$};
\qarrow{V11}{V10};
{\color{myblue}
\node(Y1) at (0.5*\R,-0.8*\H) {$A_{\triangle, (1); 1_{\infty}}$};
\qarrow{V10}{Y1};
\qarrow{Y1}{V11};
}
\end{tikzpicture}
\]    
All of these variables are frozen, namely $J_{\triangle}((1))=J_{\triangle}((1))^\fr$. By \eqref{eq:Phisigma_D1}, \eqref{eq:Phisigma_D2}, \eqref{eq:Phisigma_D3} and \eqref{eq:Phisigma_cluster}, we have 
\[
A_{\triangle, (1); 1_0}=D_{3; \varpi_1},\quad A_{\triangle, (1); 1}=D_{1; \varpi_1},\quad 
A_{\triangle, (1); 1_{\infty}}=D_{2; \varpi_1}. 
\]
By \eqref{eq:Phisigma_D1}, \eqref{eq:Phisigma_D2}, \eqref{eq:Phisigma_D3}, the $\qq$-commutation relations among these variables are given as follows. 
\begin{align*}
    &A_{\triangle, (1); 1_0}A_{\triangle, (1); 1}=\qq^{-\frac{1}{2}} A_{\triangle, (1); 1}A_{\triangle, (1); 1_0},\\
    &A_{\triangle, (1); 1}A_{\triangle, (1); 1_{\infty}}=\qq^{-\frac{1}{2}} A_{\triangle, (1); 1_{\infty}}A_{\triangle, (1); 1},\\
    &A_{\triangle, (1); 1_{\infty}}A_{\triangle, (1); 1_0}=\qq^{-\frac{1}{2}} A_{\triangle, (1); 1_0}A_{\triangle, (1); 1_{\infty}}.
\end{align*}
There relations agree with the Muller's quantum cluster algebra \cite{Muller} (resp. Jordan--Le--Schrader--Shapiro's one \cite[Corollary 3.56]{JLSS}) by identifying $\qq^{-1/2}$ (resp. $\qq^{-1}$) with the quantum parameters \emph{loc.~cit}.
    \item
    In the case of type $A_2$, $S(w_0)=\{\bs_1:=(1, 2, 1),\ \bs_2:=(2, 1, 2)\}$. The quantum seed $\bi_{\triangle}(\bs_1)$ is described as follows. 
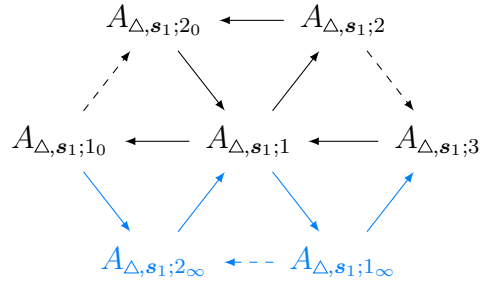
\[
\begin{tikzpicture}
\def\R{2.5}
\def\H{2}
\node(V10) at (0,0) {$A_{\triangle, \bs_1; 1_0}$};
\node(V11) at (\R,0) {$A_{\triangle, \bs_1; 1}$};
\node(V12) at (2*\R,0) {$A_{\triangle, \bs_1; 3}$};
\node(V20) at (0.5*\R,0.8*\H) {$A_{\triangle, \bs_1; 2_0}$};
\node(V21) at (1.5*\R,0.8*\H) {$A_{\triangle, \bs_1; 2}$};
\qarrow{V11}{V10};
\qarrow{V12}{V11};
\qarrow{V21}{V20};
\qdarrow{V10}{V20};
\qarrow{V20}{V11};
\qarrow{V11}{V21};
\qdarrow{V21}{V12};
{\color{myblue}
\node(Y1) at (0.5*\R,-0.8*\H) {$A_{\triangle, \bs_1; 2_{\infty}}$};
\node(Y2) at (1.5*\R,-0.8*\H) {$A_{\triangle, \bs_1; 1_{\infty}}$};
\qarrow{V10}{Y1};
\qarrow{Y1}{V11};
\qarrow{V11}{Y2};
\qarrow{Y2}{V12};
\qdarrow{Y2}{Y1};
}
\end{tikzpicture}
\]    
The skew-symmetric matrix $\Lambda_{\triangle}(\bs_1)$ is equal to 
\[
\begin{array}{c}
    1_0\\
    2_0\\
   1\\
   2\\
   3\\
   1_{\infty}\\
   2_{\infty}
\end{array}
\left[
\begin{array}{ccccccc}
    0&0&-1/3&-1/3&-2/3&1/3&2/3\\
    0&0&1/3&-2/3&-1/3&2/3&1/3\\
    1/3&-1/3&0&1/3&-1/3&1/3&-1/3\\
    1/3&2/3&-1/3&0&0&-1/3&-2/3\\
    2/3&1/3&1/3&0&0&-2/3&-1/3\\
    -1/3&-2/3&-1/3&1/3&2/3&0&0\\
    -2/3&-1/3&1/3&2/3&1/3&0&0
\end{array}
\right].
\]
We have $J_{\triangle}(\bs_1)^{\uf}=\{1\}$, and 
\begin{align}\label{eq:A2_exchange}
A_{\triangle, \bs_1; 1}A_{\mu_1(\bi_{\triangle}(\bs_1)); 1}&=\qq^{\frac{2}{3}}A_{\triangle, \bs_1; 1_0}A_{\triangle, \bs_1; 2}A_{\triangle, \bs_1; 1_{\infty}}+\qq^{-\frac{1}{3}}A_{\triangle, \bs_1; 2_0}A_{\triangle, \bs_1; 3}A_{\triangle, \bs_1; 2_{\infty}}.
\end{align}
Then we have $A_{\mu_1(\bi_{\triangle}(\bs_1)); 1}=A_{\triangle, \bs_2; 1}$. Indeed, we can show this equality by $\Phisigma$ and the equality
\[
\Delta_{r_1\varpi_1}^-\Delta_{r_2\varpi_2, \varpi_2}^-=\qq\Delta_{w_0\varpi_1, \varpi_1}^-\Delta_{\varpi_2, \varpi_2}^-+\Delta_{w_0\varpi_2, \varpi_2}^-\Delta_{\varpi_1, \varpi_1}^-
\]
in $\cO_\qq(B^-)$. We leave the details of the calculation to the reader. Hence $\mu_1(\bi_{\triangle}(\bs_1))$ is described as follows. 
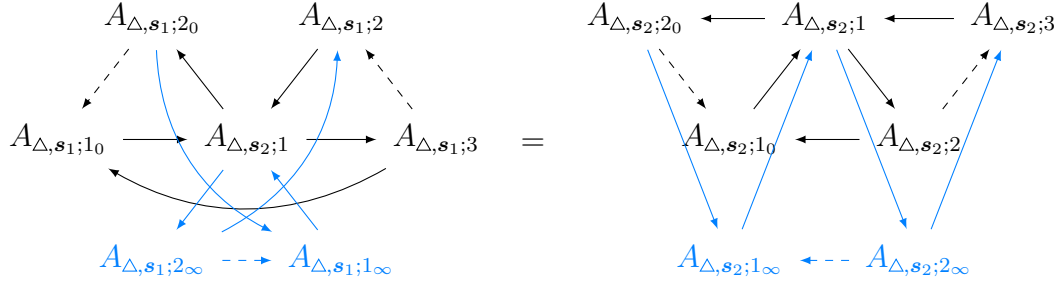
\[
\begin{tikzpicture}
\def\R{2.5}
\def\H{2}
\node(V10) at (0,0) {$A_{\triangle, \bs_1; 1_0}$};
\node(V11) at (\R,0) {$A_{\triangle, \bs_2; 1}$};
\node(V12) at (2*\R,0) {$A_{\triangle, \bs_1; 3}$};
\node(V20) at (0.5*\R,0.8*\H) {$A_{\triangle, \bs_1; 2_0}$};
\node(V21) at (1.5*\R,0.8*\H) {$A_{\triangle, \bs_1; 2}$};
\qarrow{V10}{V11};
\qarrow{V11}{V12};
\qdarrow{V20}{V10};
\qarrow{V11}{V20};
\qarrow{V21}{V11};
\qdarrow{V12}{V21};
\draw[qarrow,bend left] (V12) to (V10);
{\color{myblue}
\node(Y1) at (0.5*\R,-0.8*\H) {$A_{\triangle, \bs_1; 2_{\infty}}$};
\node(Y2) at (1.5*\R,-0.8*\H) {$A_{\triangle, \bs_1; 1_{\infty}}$};
\qarrow{V11}{Y1};
\qarrow{Y2}{V11};
\qdarrow{Y1}{Y2};
\draw[qarrow,bend right] (Y1) to (V21);
\draw[qarrow,bend right] (V20) to (Y2);
}
\end{tikzpicture}\quad \raisebox{1.8cm}{$=$}\quad
\begin{tikzpicture}
\def\R{2.5}
\def\H{2}
\node(V10) at (0.5*\R,0) {$A_{\triangle, \bs_2; 1_0}$};
\node(V11) at (\R,0.8*\H) {$A_{\triangle, \bs_2; 1}$};
\node(V12) at (1.5*\R,0) {$A_{\triangle, \bs_2; 2}$};
\node(V20) at (0,0.8*\H) {$A_{\triangle, \bs_2; 2_0}$};
\node(V21) at (2*\R,0.8*\H) {$A_{\triangle, \bs_2; 3}$};
\qarrow{V10}{V11};
\qarrow{V11}{V12};
\qdarrow{V20}{V10};
\qarrow{V11}{V20};
\qarrow{V21}{V11};
\qdarrow{V12}{V21};
\qarrow{V12}{V10};
{\color{myblue}
\node(Y1) at (1.5*\R,-0.8*\H) {$A_{\triangle, \bs_2; 2_{\infty}}$};
\node(Y2) at (0.5*\R,-0.8*\H) {$A_{\triangle, \bs_2; 1_{\infty}}$};
\qarrow{V11}{Y1};
\qarrow{Y2}{V11};
\qdarrow{Y1}{Y2};
\qarrow{Y1}{V21};
\qarrow{V20}{Y2};
}
\end{tikzpicture}
\]
Therefore, $\mu_1(\bi_{\triangle}(\bs_1))=\bi_{\triangle}(\bs_2)$. 
The relation \eqref{eq:A2_exchange} agrees with the quantum cluster algebra of Ishibashi--Yuasa \cite{IYsl3} by identifying $\qq^{-1/3}=q_{\mathrm{IY}}$ with the quantum parameter \emph{loc.~cit}. Compare with the third displayed equation in \cite[p.72]{IYsl3},
\begin{align*}
    e_i e'_i = q_{\mathrm{IY}}^{-3/2}[e_{k_2}e_{k_4}e_{k_6}] + q_{\mathrm{IY}}^{3/2}[e_{k_1}e_{k_3}e_{k_5}] = q_{\mathrm{IY}}^{-2} e_{k_6}e_{k_4}e_{k_2} + q_{\mathrm{IY}}e_{k_5}e_{k_3}e_{k_1}.
\end{align*}
\end{enumerate}
\end{ex}
\subsection{Further study of the relation between $\bi_{\triangle}(\bs)$ and $\bi_{\square}(\bbs)$}
In this subsection, we show that $\bi_{\square}(\bbs)$ can be obtained as an ``amalgamation'' of two seeds $\bi_{\triangle}(\bs)$. 

For $1\leq i<j<k\leq 4$, we define 
\begin{align}
\iota_{i,j,k}^{\inv}\colon \cO_\qq(\Conf_3\sA_G)\to \cO_\qq(\Conf_4\sA_G)    \label{eq:iota_ijk}
\end{align}
by $\iota_{i,j,k}^{\inv}:=\iota_{p_{i, j, k}}^{\inv}$, where $p_{i, j, k}\colon [1,3]\to [1, 4],\ 1\mapsto i,\ 2\mapsto j,\ 3\mapsto k$. 

Let $\bs=(s_1,\dots, s_N)\in S(w_0)$. We define an injective map $\sfI_{\triangle, \bs}\colon J_{\triangle}(\bs) \to S\times (\Z_{\geq 0}\sqcup \{\infty\})$ by 
\[
\sfI_{\triangle, \bs}(j):=
\begin{cases}
    \sfI_{\bs}(j)&\text{if }j\in J(\bs)\ \text{(see \eqref{eq:sfI_def})},\\
    (s, \infty)&\text{if }j=s_{\infty}\in S_{\infty}.
\end{cases}  
\]
Recall that we write $n^s(\bs) := |\{1 \leq j \leq N \mid s_j = s\}|$ for $s\in S$. Then,  
\[
\sfI_{\triangle, \bs}(J_{\triangle}(\bbs)^\fr)=\{(s, \infty), (s, 0), (s, n^s(\bs))\mid s\in S\}. 
\]
Let $\bbs=(\bbs_1,\dots, \bbs_{2N})\in S(w_0, w_0)$. We define an injective map $\sfI_{\square, \bbs}\colon J_{\square}(\bbs) \to S\times (\Z_{\geq 0}\sqcup \{\infty, -\infty\})$ by 
\[
\sfI_{\square, \bbs}(j):=
\begin{cases}
    (s, 0)&\text{if }j=s_0\in S_0,\\
    (|\bbs_j|, n[j])&\text{if }j\in [1, 2N]\text{ where } n[j] := |\{1 \leq i \leq j\mid |\bbs_i| = |\bbs_j|\}|,\\
    (s, -\infty)&\text{if }j=s_{-\infty}\in S_{-\infty},\\
    (s, \infty)&\text{if }j=s_{\infty}\in S_\infty.
\end{cases}  
\]
For $s\in S$, we write $n^s(\bbs) := |\{1 \leq j \leq 2N \mid |\bbs_j| = s\}|$. Then,  
\[
\sfI_{\square, \bbs}(J_{\square}(\bbs)^\fr)=\{(s, -\infty), (s, 0), (s, n^s(\bbs)), (s, \infty)\mid s\in S\}. 
\]

The following theorem gives a precise statement of the amalgamation. It is illustrated in \cref{fig:triangle_square_1}. 

\begin{thm}\label{thm:triangle_square_1}
Let $\bs_1=(s_{1;1},\dots, s_{1; N}), \bs_2=(s_{2; 1},\dots, s_{2;N})\in S(w_0)$. Set 
\begin{align*}
&\bbs:=\bs_{1}\ast(\overline{\bs}_2^\ast)^\op=(s_{1;1},\dots, s_{1;N}, \overline{s}_{2;N}^{\ast},\dots, \overline{s}_{2;1}^{\ast})\in S(w_0, w_0).    
\end{align*}
\begin{itemize}
    \item[(1)] The injective $\Bbbk$-algebra homomorphism (see \eqref{eq:triangle_emb})
    \[
    \iota_{43}^{1\bullet }\colon \cO_\qq(\Conf_3\sA_G)\to \cO_\qq(\Conf_4\sA_G)\ 
    \]
     is an embedding of seeds (in the sense of \cref{def:clusteremb}) from $\bi_{\triangle}(\bs_1)$ to $\bi_{\square}(\bbs)$. The index correspondence associated to $\iota_{43}^{1\bullet }$ is given as $\sfJ_{43}^{1\bullet }\colon J_{\triangle}(\bs_1) \to J_{\square}(\bbs), j\mapsto j'$, where 
    \[
    \sfI_{\square, \bbs}(\sfJ_{43}^{1\bullet }(j'))=\sfI_{\triangle, \bs_1}(j).
    \]
    \item[(2)] The injective $\Bbbk$-algebra homomorphism 
    \[
    \iota_{\bullet 3}^{12}:=\sigma_4^2\circ \iota_{43}^{1\bullet }=\iota_{1,2,3}^{\inv}\circ \sigma_3^2\colon \cO_\qq(\Conf_3\sA_G)\to \cO_\qq(\Conf_4\sA_G)
    \]
    is an embedding of seeds from $\bi_{\triangle}(\bs_2)$ to $\bi_{\square}(\bbs)$.  The index correspondence associated to $ \iota_{\bullet 3}^{12}$ is given as $\sfJ_{\bullet 3}^{12}\colon J_{\triangle}(\bs_2) \to J_{\square}(\bbs), j\mapsto j'$, where 
    \[
    \sfI_{\square, \bbs}(\sfJ_{\bullet 3}^{12}(j'))=    \begin{cases}
        (s^{\ast}, -\infty)&\text{if}\ \sfI_{\triangle, \bs_2}(j)=(s, \infty),\\
        (s^{\ast}, n^{s^{\ast}}(\bbs)-m)&\text{if}\  \sfI_{\triangle, \bs_2}(j)=(s, m),\ 0\leq m\leq n^s(\bs_2).
    \end{cases}
    \]
\end{itemize}
\end{thm}
\begin{figure}[ht]
    \centering
\begin{tikzpicture}
\def\R{3}
\coordinate(A) at (0,0);
\coordinate(B) at (1.2*\R,0);
\coordinate(C) at (0.6*\R,\R);
\node[scale=0.7] at (0.6*\R,0.9*\R) {$\bullet$};
\draw[thick,right hook->] (1.2*\R,0.5*\R) --node[midway,above]{$\iota_{43}^{1\bullet}$ or $\iota_{\bullet 3}^{12}$} ++(2,0);
\draw[thick] (B) -- (C);
\draw[thick] (A) -- (B) -- (C) --cycle;

\vertexA{C}{B}{A}{3};
\foreach \i in {1,2,3}{
    \fnode{a\i}{black};
    \fnode{b\i}{black};
    \fnode{V\i 0}{myblue};
    }
\foreach \i in {11,21,12}
\draw(V\i) circle(2pt);
\qarrow{a3}{b3};
\qarrow{a2}{V12}; \qarrow{V12}{b2};
\qarrow{a1}{V11}; \qarrow{V11}{V21}; \qarrow{V21}{b1};
\qarrow{b2}{V21};
\qarrow{b3}{V12}; 
\qarrow{V12}{V11};
\qarrow{V21}{V12};
\qarrow{V12}{a3};
\qarrow{V11}{a2};
{\color{myblue}
\qarrow{V10}{a1};
\qarrow{V11}{V10};
\qarrow{V20}{V11};
\qarrow{V21}{V20};
\qarrow{V30}{V21};
\qarrow{b1}{V30};
}
\begin{scope}[xshift=2.2*\R cm]
\coordinate(A) at (0,0);
\coordinate(B) at (\R,0);
\coordinate(C) at (0,\R);
\coordinate(D) at (\R,\R);
\node[below left] at ($(A)$) {$4$};
\node[below right] at ($(B)$) {$3$};
\node[above left] at ($(C)$) {$1$};
\node[above right] at ($(D)$) {$2$};
\node[scale=0.7] at (0.05*\R,0.86*\R) {$\bullet$};
\node[scale=0.7] at (0.95*\R,0.14*\R) {$\bullet$};

\draw[thick,dashed] (B) -- (C);
\draw[thick] (A) -- (B) -- (D) -- (C) --cycle;

\vertexA{C}{B}{A}{3};
\foreach \i in {1,2,3}{
    \filldraw[fill=white](a\i) circle(2pt);
    \fnode{b\i}{black};
    \fnode{V\i 0}{myblue};
    }
\foreach \i in {11,21,12}
\draw(V\i) circle(2pt);
\qarrow{a3}{b3};
\qarrow{a2}{V12}; \qarrow{V12}{b2};
\qarrow{a1}{V11}; \qarrow{V11}{V21}; \qarrow{V21}{b1};
\qarrow{b2}{V21};
\qarrow{b3}{V12}; 
\qarrow{V12}{V11};
\qarrow{V21}{V12};
\qarrow{V12}{a3};
\qarrow{V11}{a2};
{\color{myblue}
\qarrow{V10}{a1};
\qarrow{V11}{V10};
\qarrow{V20}{V11};
\qarrow{V21}{V20};
\qarrow{V30}{V21};
\qarrow{b1}{V30};
}

\vertexA{B}{C}{D}{3};
\foreach \i in {1,2,3}{
    \fnode{b\i}{black};
    \fnode{V\i 0}{myblue};
    }
\foreach \i in {11,21,12}
\draw(V\i) circle(2pt);
\qarrow{a3}{b3};
\qarrow{a2}{V12}; \qarrow{V12}{b2};
\qarrow{a1}{V11}; \qarrow{V11}{V21}; \qarrow{V21}{b1};
\qarrow{b2}{V21};
\qarrow{b3}{V12}; 
\qarrow{V12}{V11};
\qarrow{V21}{V12};
\qarrow{V12}{a3};
\qarrow{V11}{a2};
{\color{myblue}
\qarrow{V10}{a1};
\qarrow{V11}{V10};
\qarrow{V20}{V11};
\qarrow{V21}{V20};
\qarrow{V30}{V21};
\qarrow{b1}{V30};
}
\end{scope}
\end{tikzpicture}
    \caption{Type $A_3$, $\bs=(1,2,3,1,2,1)$, $\bbs=(1,2,3,1,2,1,\overline{3},\overline{2},\overline{3},\overline{1},\overline{2},\overline{3})$. The quantum seed $\bi_{\square}(\bbs)$ can be obtained as an amalgamation of two copies of $\bi_{\triangle}(\bs)$.}
    \label{fig:triangle_square_1}
\end{figure}
\begin{proof}
The statement (1) follows from the definition of $\bi_{\triangle}(\bs_1)$. Let us show (2). By \eqref{eq:A_triangle},
\begin{align*}
\iota_{\bullet 3}^{12}(A_{\triangle, \bs_2; j})
&=\begin{cases}
\begin{array}{ll}
(-1)^{\langle 2\rho^{\vee}, [\delta_{\bs_2; j}]_+\rangle}\qq^{(2\rho, [\delta_{\bs_2; j}]_+)+\frac{1}{2}(\varpi_{s_{2;j}}^{\ast}, \varpi_{s_{2; j}}^{\ast})+\frac{1}{2}([\delta_{\bs_2; j}^{\ast}]_-,[\delta_{\bs_2; j}^{\ast}]_+)}\\
    \phantom{=}\times \xi_{w_0[\delta_{\bs_2; j}^{\ast}]_+}([\delta_{\bs_2; j}^{\ast}]_+)\otimes \xi_{[\delta_{\bs_2; j}]_-}([\delta_{\bs_2; j}]_-)\otimes \xi_{\delta_{\bs_2; j}}(\varpi_{s_{2;j}})\otimes 1\\
    \phantom{=}+(\text{other terms})    
\end{array}&\text{if }j\in J(\bs_2),\\
D_{1; \varpi_s}=A_{\square, \bbs; s^{\ast}_{-\infty}}&\text{if }j=s_{\infty}\in S_{\infty},
\end{cases}
\end{align*}
    where $\mathrm{(}$other terms$\mathrm{)}$ in the first case belong to 
    \[
    \cO([\delta_{\bs_2; j}^{\ast}]_+)_{<[\delta_{\bs_2; j}]_+}\ut\cO([\delta_{\bs_2; j}]_-)\ut\cO(\varpi_{s_{2; j}})\ut 1 +
    \cO([\delta_{\bs_2; j}^{\ast}]_+)\ut\cO([\delta_{\bs_2; j}]_-)_{> -[\delta_{\bs_2; j}]_-}\ut\cO(\varpi_{s_{2; j}})\ut 1.
    \]
For $j\in J(\bs_2)$, 
\[
\iota_{\bullet 3}^{12}(A_{\triangle, \bs_2; j})=(-1)^{\langle 2\rho^{\vee}, [\delta_{\bs_2; j}]_+\rangle}\qq^{(2\rho, [\delta_{\bs_2; j}]_+)+\frac{1}{2}(\varpi_{s_{2;j}}^{\ast}, \varpi_{s_{2; j}}^{\ast})+\frac{1}{2}([\delta_{\bs_2; j}^{\ast}]_-,[\delta_{\bs_2; j}^{\ast}]_+)}\phi_{[\delta_{\bs_2; j}^{\ast}]_+, [\delta_{\bs_2; j}]_-, \varpi_{s_{2; j}}}
\]
in the notation of \cref{thm:isigma}. Therefore, by \cref{thm:isigma},  
\begin{align*}
    &\isigma(\iota_{\bullet 3}^{12}(A_{\triangle, \bs_2; j}))\\
    &=\qq^{\frac{1}{2}(\varpi_{s_{2;j}}^{\ast}, \varpi_{s_{2; j}}^{\ast})+\frac{1}{2}([\delta_{\bs_2; j}^{\ast}]_-,[\delta_{\bs_2; j}^{\ast}]_+)-\frac{1}{2}([\delta_{\bs_2; j}]_-, [\delta_{\bs_2; j}]_-)-\frac{1}{2}(\varpi_{s_{2; j}}, \varpi_{s_{2; j}})}e'(\varpi_{s_{2; j}})\Delta_{\delta_{\bs_2; j}, \varpi_{s_{2; j}}}e([\delta_{\bs_2; j}]_+)\\
    &=\qq^{\frac{1}{2}(\delta_{\bs_2; j}, [\delta_{\bs_2; j}]_+)-\frac{1}{2}(\varpi_{s_{2; j}}, \varpi_{s_{2; j}})}e'(\varpi_{s_{2; j}})\Delta_{\delta_{\bs_2; j}, \varpi_{s_{2; j}}}e([\delta_{\bs_2; j}]_+).
\end{align*}
Moreover, we can directly check that 
\[
\delta_{\bs_2; j}=\gamma_{\bbs; \sfJ_{\bullet 3}^{12}(j)}^{\ast},\quad \varpi_{s_{2; j}} = \delta_{\bbs; \sfJ_{\bullet 3}^{12}(j)}^{\ast},
\]
for $j\in J(\bs_2)$ (cf.~\eqref{eq:gamma_delta}). Therefore, we have confirmed the condition (2) in \cref{def:clusteremb}. The conditions (1) and (3) in \cref{def:clusteremb} can be straightforwardly checked by the construction of $\bJ_{\std}(\bbs)$ in \cref{sec:quiver} (cf.~\cref{fig:triangle_square_1}). 
\end{proof}
The following theorem gives another way of amalgamation, which is illustrated in \cref{fig:triangle_square_2}. The proof is identical to that of \cref{thm:triangle_square_2}, hence we leave the details to the reader. 
\begin{thm}\label{thm:triangle_square_2}
Let $\bs_1=(s_{1;1},\dots, s_{1; N}), \bs_2=(s_{2; 1},\dots, s_{2;N})\in S(w_0)$. Set 
\begin{align*}
&\bbs':=(\overline{\bs}_2^\ast)^\op\ast \bs_{1}=(\overline{s}_{2;N}^{\ast},\dots, \overline{s}_{2;1}^{\ast}, s_{1;1},\dots, s_{1;N})\in S(w_0, w_0).    
\end{align*}
\begin{itemize}
    \item[(1)] The injective $\Bbbk$-algebra homomorphism 
    \[
    \iota_{4\bullet}^{12}:=\sigma_4^2\circ \iota_{2,3,4}^{\inv} = \iota_{1,2,4}^{\inv}\circ \sigma_3^2\colon \cO_\qq(\Conf_3\sA_G)\to \cO_\qq(\Conf_4\sA_G)
    \]
    is an embedding of seeds from $\bi_{\triangle}(\bs_2)$ to $\bi_{\square}(\bbs')$. The index correspondence associated to $ \iota_{4\bullet}^{12}$ is given by $\sfJ_{4\bullet}^{12}\colon J_{\triangle}(\bs_2) \to J_{\square}(\bbs'), j\mapsto j'$, where 
    \[
    \sfI_{\square, \bbs'}(\sfJ_{4\bullet}^{12}(j'))=    \begin{cases}
        (s^{\ast}, -\infty)&\text{if}\ \sfI_{\triangle, \bs_2}(j)=(s, \infty),\\
        (s^{\ast}, n^s(\bs_2)-m)&\text{if}\  \sfI_{\triangle, \bs_2}(j)=(s, m),\ 0\leq m\leq n^s(\bs_2).
    \end{cases}
    \]
    \item[(2)] The injective $\Bbbk$-algebra homomorphism 
    \[
    \iota_{43}^{\bullet 2}:=\iota_{2,3,4}^{\inv}=\sigma_4^2\circ \iota_{4\bullet}^{12}\colon \cO_\qq(\Conf_3\sA_G)\to \cO_\qq(\Conf_4\sA_G)
    \]
    is an embedding of seeds from $\bi_{\triangle}(\bs_1)$ to $\bi_{\square}(\bbs')$. The index correspondence associated to $\iota_{43}^{\bullet 2}$ is given by $\sfJ_{43}^{\bullet 2}\colon J_{\triangle}(\bs_1) \to J_{\square}(\bbs'), j\mapsto j'$, where 
    \[
    \sfI_{\square, \bbs'}(\sfJ_{43}^{\bullet 2}(j'))=    \begin{cases}
        (s, \infty)&\text{if}\ \sfI_{\triangle, \bs_1}(j)=(s, \infty),\\
        (s, n^{s^{\ast}}(\bs_2)+m)&\text{if}\  \sfI_{\triangle, \bs_1}(j)=(s, m),\ 0\leq m\leq n^{s}(\bs_1).
    \end{cases}
    \]
\end{itemize}
\end{thm}
\begin{figure}[ht]
    \centering
\begin{tikzpicture}
\def\R{3}
\coordinate(A) at (0,0);
\coordinate(B) at (1.2*\R,0);
\coordinate(C) at (0.6*\R,\R);
\node[scale=0.7] at (0.6*\R,0.9*\R) {$\bullet$};
\draw[thick,right hook->] (1.2*\R,0.5*\R) --node[midway,above]{$\iota_{4\bullet}^{12}$ or $\iota_{43}^{\bullet 2}$} ++(2,0);
\draw[thick] (B) -- (C);
\draw[thick] (A) -- (B) -- (C) --cycle;

\vertexA{C}{B}{A}{3};
\foreach \i in {1,2,3}{
    \fnode{a\i}{black};
    \fnode{b\i}{black};
    \fnode{V\i 0}{myblue};
    }
\foreach \i in {11,21,12}
\draw(V\i) circle(2pt);
\qarrow{a3}{b3};
\qarrow{a2}{V12}; \qarrow{V12}{b2};
\qarrow{a1}{V11}; \qarrow{V11}{V21}; \qarrow{V21}{b1};
\qarrow{b2}{V21};
\qarrow{b3}{V12}; 
\qarrow{V12}{V11};
\qarrow{V21}{V12};
\qarrow{V12}{a3};
\qarrow{V11}{a2};
{\color{myblue}
\qarrow{V10}{a1};
\qarrow{V11}{V10};
\qarrow{V20}{V11};
\qarrow{V21}{V20};
\qarrow{V30}{V21};
\qarrow{b1}{V30};
}
\begin{scope}[xshift=2.2*\R cm]
\coordinate(A) at (0,0);
\coordinate(B) at (\R,0);
\coordinate(C) at (0,\R);
\coordinate(D) at (\R,\R);
\node[below left] at ($(A)$) {$4$};
\node[below right] at ($(B)$) {$3$};
\node[above left] at ($(C)$) {$1$};
\node[above right] at ($(D)$) {$2$};
\node[scale=0.7] at (0.05*\R,0.14*\R) {$\bullet$};
\node[scale=0.7] at (0.95*\R,0.86*\R) {$\bullet$};

\draw[thick,dashed] (D) -- (A);
\draw[thick] (A) -- (B) -- (D) -- (C) --cycle;

\vertexA{D}{B}{A}{3};
\foreach \i in {1,2,3}{
    \filldraw[fill=white](b\i) circle(2pt);
    \fnode{a\i}{black};
    \fnode{V\i 0}{myblue};
    }
\foreach \i in {11,21,12}
\draw(V\i) circle(2pt);
\qarrow{a3}{b3};
\qarrow{a2}{V12}; \qarrow{V12}{b2};
\qarrow{a1}{V11}; \qarrow{V11}{V21}; \qarrow{V21}{b1};
\qarrow{b2}{V21};
\qarrow{b3}{V12}; 
\qarrow{V12}{V11};
\qarrow{V21}{V12};
\qarrow{V12}{a3};
\qarrow{V11}{a2};
{\color{myblue}
\qarrow{V10}{a1};
\qarrow{V11}{V10};
\qarrow{V20}{V11};
\qarrow{V21}{V20};
\qarrow{V30}{V21};
\qarrow{b1}{V30};
}

\vertexA{A}{C}{D}{3};
\foreach \i in {1,2,3}{
    \fnode{a\i}{black};
    \filldraw[fill=white](b\i) circle(2pt);
    \fnode{V\i 0}{myblue};
    }
\foreach \i in {11,21,12}
\draw(V\i) circle(2pt);
\qarrow{a3}{b3};
\qarrow{a2}{V12}; \qarrow{V12}{b2};
\qarrow{a1}{V11}; \qarrow{V11}{V21}; \qarrow{V21}{b1};
\qarrow{b2}{V21};
\qarrow{b3}{V12}; 
\qarrow{V12}{V11};
\qarrow{V21}{V12};
\qarrow{V12}{a3};
\qarrow{V11}{a2};
{\color{myblue}
\qarrow{V10}{a1};
\qarrow{V11}{V10};
\qarrow{V20}{V11};
\qarrow{V21}{V20};
\qarrow{V30}{V21};
\qarrow{b1}{V30};
}
\end{scope}
\end{tikzpicture}
    \caption{Type $A_3$, $\bs=(1,2,3,1,2,1)$, $\bbs=(\overline{3},\overline{2},\overline{3},\overline{1},\overline{2},\overline{3}, 1,2,3,1,2,1)$. The quantum seed $\bi_{\square}(\bbs)$ can be obtained as an amalgamation of two copies of $\bi_{\triangle}(\bs)$.}
    \label{fig:triangle_square_2}
\end{figure}

\section{Quantum cluster structure on \texorpdfstring{$\cO_\qq(\Conf_{K+2}\sA_G)$}{Oq(Conf(K+2))}} \label{sec:polygon}
Let $K\in \Z_{>0}$. By \cref{prop:qcsgrading}, we can define the skew-field of fractions $\cF(\cO_\qq(\Conf_{K+2}^{\times}\sA_G))$.
In this section, we construct a quantum seed in $\cF(\cO_\qq(\Conf_{K+2}^{\times}\sA_G))$, which  generalizes $\bi_{\triangle}(\bs)$ and $\bi_{\square}(\bbs)$. 

Write the $(K+2)$-gon as $\bbD_{K+2}$. We label the vertices of $\bbD_{K+2}$ as follows.  

\begin{itemize}
    \item The case where $K+2$ is even ($K+2=:2L$): 
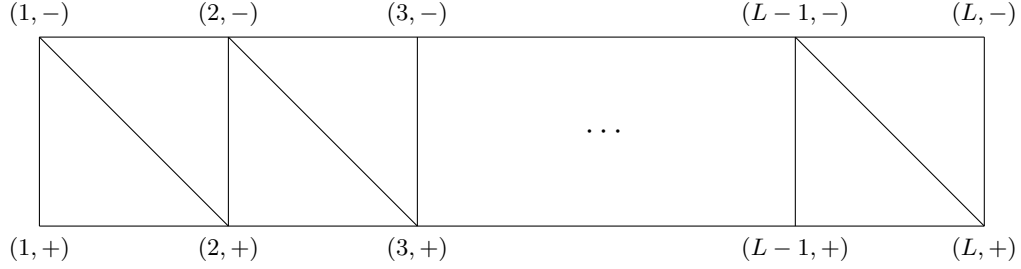
    \[
\begin{tikzpicture}[scale=2.5]
  \draw (0,1) -- (5,1);
  \draw (0,0) -- (5,0);
  \draw (0,0) -- (0,1);
  \draw (5,0) -- (5,1);

  \draw (1,0) -- (1,1);
  \draw (2,0) -- (2,1);
  \draw (4,0) -- (4,1);
  
  \draw (0,1) -- (1,0);
  \draw (1,1) -- (2,0);
  \draw (4,1) -- (5,0);
  
  \node at (3,0.5) {$\cdots$};

  \node[above] at (0,1) {\scriptsize $(1,-)$};
  \node[above] at (1,1) {\scriptsize $(2,-)$};
  \node[above] at (2,1) {\scriptsize $(3,-)$};
  \node[above] at (4,1) {\scriptsize $(L-1,-)$};
  \node[above] at (5,1) {\scriptsize $(L,-)$};  

  \node[below] at (0,0) {\scriptsize $(1,+)$};
  \node[below] at (1,0) {\scriptsize $(2,+)$};
  \node[below] at (2,0) {\scriptsize $(3,+)$};
  \node[below] at (4,0) {\scriptsize $(L-1,+)$};
  \node[below] at (5,0) {\scriptsize $(L,+)$};
\end{tikzpicture}
\]
    \item The case where $K+2$ is odd ($K+2=:2L+1$): 
    
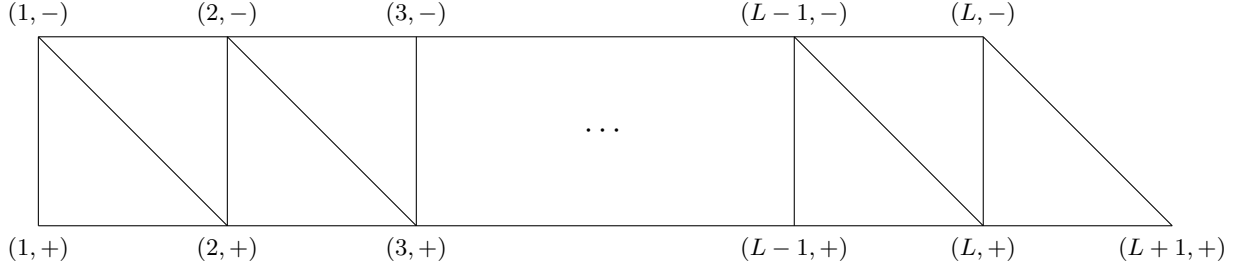
\[
\begin{tikzpicture}[scale=2.5]
  \draw (0,1) -- (5,1);
  \draw (0,0) -- (6,0);
  \draw (0,0) -- (0,1);
  \draw (5,0) -- (5,1);

  \draw (1,0) -- (1,1);
  \draw (2,0) -- (2,1);
  \draw (4,0) -- (4,1);
  
  \draw (0,1) -- (1,0);
  \draw (1,1) -- (2,0);
  \draw (4,1) -- (5,0);
  \draw (5,1) -- (6,0);
  
  \node at (3,0.5) {$\cdots$};

  \node[above] at (0,1) {\scriptsize $(1,-)$};
  \node[above] at (1,1) {\scriptsize $(2,-)$};
  \node[above] at (2,1) {\scriptsize $(3,-)$};
  \node[above] at (4,1) {\scriptsize $(L-1,-)$};
  \node[above] at (5,1) {\scriptsize $(L,-)$};  

  \node[below] at (0,0) {\scriptsize $(1,+)$};
  \node[below] at (1,0) {\scriptsize $(2,+)$};
  \node[below] at (2,0) {\scriptsize $(3,+)$};
  \node[below] at (4,0) {\scriptsize $(L-1,+)$};
  \node[below] at (5,0) {\scriptsize $(L,+)$};
  \node[below] at (6,0) {\scriptsize $(L+1,+)$};
\end{tikzpicture}
\]
\end{itemize}
Moreover, we consider another labeling of the vertices of $\bbD_{K+2}$ obtained by labeling $(1, -)$ by $1$, and the remaining vertices consecutively by $2,\dots, K+2$ in the clockwise order (e,g.~the vertex $(1,+)$ is labeled by $K+2$). These two labeling gives a bijection between $[1, K+2]$ and $\{(1, -),\dots, (L, -), (1, +),\dots, (L+\varepsilon_K, +)\}$ ($\varepsilon_K:=(1-(-1)^K)/2$), and we will identify these index sets by this bijection. 
For $1\leq i_1<i_2<i_3\leq K+2$\ ($K\geq 1$), set 
\begin{align}
    &\iota_{i_3 i_2}^{i_1}:=\iota_{i_1, i_2, i_3}^{\inv}\colon \cO_\qq(\Conf_3\sA_G)\to \cO_\qq(\Conf_{K+2}\sA_G),\\
    &\iota_{i_3}^{i_1 i_2}:=\iota_{i_1, i_2, i_3}^{\inv}\circ\sigma_3^2\colon \cO_\qq(\Conf_3\sA_G)\to \cO_\qq(\Conf_{K+2}\sA_G).
\end{align}
Recall \eqref{eq:iota_ijk}. By using this notation, the $\Bbbk$-algebra homomorphisms in \cref{thm:triangle_square_1,thm:triangle_square_1} can be written as follows. 
\begin{align}
    \iota_{43}^{1\bullet }=\iota_{4 3}^{1},\quad \iota_{\bullet 3}^{12}=\iota_{3}^{1 2},\quad 
    \iota_{4\bullet}^{12}=\iota_{4}^{1 2},\quad \iota_{43}^{\bullet 2}=\iota_{4 3}^{2}.\label{eq:iota_notation}
\end{align}

For $1\leq i_1<i_2<i_3<i_4\leq K+2$ ($K\geq 2$), set 
\begin{align}
    &\iota_{i_4 i_3}^{i_1 i_2}:=\iota_{p_{i_1, i_2, i_3, i_4}}^{\inv}\colon \cO_\qq(\Conf_4\sA_G)\to \cO_\qq(\Conf_{K+2}\sA_G),\label{eq:square_emb}
\end{align}
where $p_{i_1, i_2, i_3, i_4}\colon [1,4]\to [1, K+2], j\mapsto i_j\ (j=1,2,3,4)$. 
Although the targets of these morphisms depend on $K$, we omit $K$ from the notation for simplicity, as this should cause no confusion. 
 
Fix $\bs_1,\dots, \bs_K\in S(w_0)$, and write $\bs_k=(s_{k; 1},\dots, s_{k; N})$ for $k=1,\dots,K$. We set 
\begin{align*}
\tbbs&:=\bs_1\ast (\overline{\bs}_2^{\ast})^\op\ast \bs_3\ast (\overline{\bs}_4^\ast)^\op\ast \cdots\\
&=
\begin{cases}
\bs_1\ast (\overline{\bs}_2^\ast)^\op\ast \bs_3\ast (\overline{\bs}_4^\ast)^\op\ast \cdots \ast (\overline{\bs}_K^\ast)^\op&\text{when $K$ is even},\\
\bs_1\ast (\overline{\bs}_2^\ast)^\op\ast \bs_3\ast (\overline{\bs}_4^\ast)^\op\ast \cdots \ast \bs_K&\text{when $K$ is odd},
\end{cases}\\
&=(\underbrace{s_{1; 1},\dots, s_{1; N}}, \underbrace{\overline{s_{2; N}^\ast},\dots, \overline{s_{2; 1}^\ast}}, \underbrace{s_{3; 1},\dots, s_{3; N}}, \underbrace{\overline{s_{4; N}^\ast},\dots, \overline{s_{4; 1}^\ast}},\dots).
\end{align*}
Consider the weighted quiver $\bJ_\std(\tbbs)$ associated with $\tbbs$ as in \cref{sec:quiver} (see \eqref{eq:std_tbbs}). Set 
\begin{align*}
 n^s(\tbbs):=\sum_{k=1,\dots, K}n^{s^{\ast (k-1)}}(\bs_k), \quad 
n^s(\tbbs_{\leq K'}):=\sum_{k=1,\dots, K'}n^{s^{\ast (k-1)}}(\bs_k)   
\end{align*}
for $s\in S$ and $K'=0, 1,\dots, K$. Here $s^{\ast (k-1)}:=s^{\overset{\text{$(k-1)$ times}}{\overbrace{\scalebox{0.8}{$\ast\cdots \cdots \ast$}}}}$.  
The vertices of $\bJ_\std(\tbbs)$ are labeled by the following subset of $S\times (\Z_{\geq 0}\sqcup \{\infty_{k}\mid k=1,\dots, K\})$; 
\begin{align*}
    J_{\bbD_{K+2}}(\tbbs)&:=
    \{(s, m)\mid 0\leq m\leq n^s(\tbbs), s\in S\}\sqcup \{(s, \infty_k)\mid s\in S, k=1,\dots, K\},
\end{align*}
where 
\begin{align*}
    J_{\bbD_{K+2}}(\tbbs)^\uf&:=
    \{(s, m)\in J_{\bbD_{K+2}}(\tbbs)\mid 1\leq  m\leq n^s(\tbbs)-1, s\in S\},\\
    J_{\bbD_{K+2}}(\tbbs)^\fr&:=J_{\bbD_{K+2}}(\tbbs)\setminus J_{\bbD_{K+2}}(\tbbs)^\uf\\
    &=\{(s, 0), (s, n^s(\tbbs)), (s, \infty_k) \mid s\in S, k=1,\dots, K\}.    
\end{align*}
Write the corresponding exchange matrix as $\cE_{\bbD_{K+2}}(\tbbs)=(\varepsilon_{\bbD_{K+2}, \tbbs;ij})_{i\in J_{\bbD_{K+2}}(\tbbs)^{\uf},j\in J_{\bbD_{K+2}}(\tbbs)}$. 
Define $A_{\bbD_{K+2}, \tbbs; (s, m)}\in \cO_\qq(\Conf_{K+2}^{\times}\sA_G)$ for $(s, m)\in J_{\bbD_{K+2}}(\tbbs)$ as follows. \vspace{5pt}

    \noindent (1) For $\ell=1,\dots, L-1+\varepsilon_K$, 
    \begin{align*}
    &A_{\bbD_{K+2}, \tbbs; (s, m)}:=\iota_{(\ell, +), (\ell+1, +)}^{(\ell, -)}(A_{\triangle, \bs_{2\ell-1}; j})\ \text{if $\sfI_{\bs_{2\ell-1}}(j)=(s, r)$ and $m=n^s(\tbbs_{\leq 2\ell-2})+r$},\\
     &A_{\bbD_{K+2}, \tbbs; (s, \infty_{2\ell-1})}:=\iota_{(\ell, +), (\ell+1, +)}^{(\ell, -)}(A_{\triangle, \bs_{2\ell-1}; s_{\infty}}). 
    \end{align*}

   \noindent (2) For $\ell=1,\dots, L-1$, 
    \begin{align*}
    &A_{\bbD_{K+2}, \tbbs; (s, m)}:=\iota_{(\ell+1, +)}^{(\ell, -), (\ell+1, -)}(A_{\triangle, \bs_{2\ell}; j})\ \text{if $\sfI_{\bs_{2\ell}}(j)=(s^\ast, r)$ and $m=n^s(\tbbs_{\leq 2\ell})-r$},\\
     &A_{\bbD_{K+2}, \tbbs; (s, \infty_{2\ell})}:=\iota_{(\ell+1, +)}^{(\ell, -), (\ell+1, -)}(A_{\triangle, \bs_{2\ell}; s^{\ast}_{\infty}}). 
    \end{align*}

\noindent By \cref{thm:edge,thm:triangle_square_1,thm:triangle_square_2}, the elements $A_{\bbD_{K+2}, \tbbs; (s, m)}$, for $(s, m)\in J_{\bbD_{K+2}}(\tbbs)$, are well-defined and mutually $\qq$-commute. 

Define $\Lambda_{\bbD_{K+2}}(\tbbs)=(\Lambda_{\bbD_{K+2}, \tbbs; (s, m), (s', m')})_{(s, m), (s', m')\in J_{\bbD_{K+2}}(\tbbs)}$ by 
\begin{align*}
    A_{\bbD_{K+2}, \tbbs; (s, m)}A_{\bbD_{K+2}, \tbbs; (s', m')}=\qq^{\Lambda_{\bbD_{K+2},\tbbs;(s, m), (s', m')}}A_{\bbD_{K+2}, \tbbs; (s', m')}A_{\bbD_{K+2}, \tbbs; (s, m)}
\end{align*}
for $(s, m), (s', m')\in J_{\bbD_{K+2}}(\tbbs)$. Then we have a $\Bbbk$-algebra homomorphism (see \cref{ssec:Qtorus})
\[
M_{\tbbs}:\cT(\Lambda_{\bbD_{K+2}}(\tbbs))\to \cF(\cO_\qq(\Conf_{K+2}^{\times}\sA_G)),\ A_{(s, m)}\mapsto A_{\bbD_{K+2}, \tbbs; (s, m),}\ (s, m)\in J_{\bbD_{K+2}}(\tbbs).
\]

The following is the main theorem of this section.
\begin{thm}\label{thm:cluster_polygon}
The collection 
\[
\bi_{\bbD_{K+2}}(\tbbs):=(\Lambda_{\bbD_{K+2}}(\tbbs), \cE_{\bbD_{K+2}}(\tbbs),(A_{\bbD_{K+2}, \tbbs; (s, m)})_{(s, m)\in J_{\bbD_{K+2}}(\tbbs)})
\]
forms a quantum seed in $\cF(\cO_\qq(\Conf_{K+2}^{\times}\sA_G))$, and  
\[
\cO_\qq(\Conf_{K+2}^{\times}\sA_G)\subset A_{\qq}(\bi_{\bbD_{K+2}}(\tbbs); J_{\bbD_{K+2}}(\tbbs)^{\fr}).
\]
\end{thm}
\begin{proof}
When $K=1, 2$, the statement is a (weaker version of) \cref{thm:cluster_triangle,thm:cluster_square}. We show the case $K\in \Z_{\geq 3}$. 
Set
\[
\cO_\qq(\Conf_{K+2}^{\overline{\times}}\sA_G):=\cO_\qq(\Conf_{K+2}^{[1, L-1]\cup [L+1, K+1]}\sA_G)=\cO_\qq(\Conf_{K+2}^{[1, K+2]\setminus \{L, K+2\}}\sA_G).
\]    
The following claim plays a key role in the proof. 
\begin{claim}\label{claim}
    For $K\in \Z_{\geq 3}$, $\cO_\qq(\Conf_{K+2}^{\overline{\times}}\sA_G)$ is generated by $\iota^{(\ell, -),(\ell+1, -)}_{(\ell, +), (\ell+1, +)}(\cO_\qq(\Conf_{4}^{\overline{\times}}\sA_G))$ ($\ell =1,\dots, L-1$), and  $\iota^{(\ell', -),(\ell'+1, -)}_{(\ell'+1, +),(\ell'+2, +)}(\cO_\qq(\Conf_{4}^{\overline{\times}}\sA_G))$ ($\ell'=1,\dots, L-2+\varepsilon_K$) as a $\Bbbk$-algebra.
\end{claim}
\begin{proof}[Proof of \cref{claim}]
Define $\iota_{<}\colon \cO_\qq(\Conf_{K+1}^{\overline{\times}}\sA_G)\to \cO_\qq(\Conf_{K+2}^{\overline{\times}}\sA_G)$ by $\iota_{<}:=\iota_{p_{<}}^{\inv}$, where 
\[
p_{<}\colon [1,K+1]\to [1,K+2],\ p_{<}(j)=
\begin{cases}
    j&\text{if}\ 1\leq j\leq L-1+\varepsilon_K,\\
    j+1&\text{if}\ L+\varepsilon_K\leq j\leq K+2.
\end{cases}
\]
It suffices to show the following statements.
        \item[(1)] If $K$ is odd, $\cO_\qq(\Conf_{K+2}^{\overline{\times}}\sA_G)$ is generated by $\iota_{<}(\cO_\qq(\Conf_{K+1}^{\overline{\times}}\sA_G))$ and $\iota^{(L-1, -),(L, -)}_{(L, +), (L+1, +)}(\cO_\qq(\Conf_{4}^{\overline{\times}}\sA_G))$ as a $\Bbbk$-algebra.
        \item[(2)] If $K$ is even, $\cO_\qq(\Conf_{K+2}^{\overline{\times}}\sA_G)$ is generated by $\iota_{<}(\cO_\qq(\Conf_{K+1}^{\overline{\times}}\sA_G))$ and $\iota^{(L-1, -),(L, -)}_{(L-1, +), (L, +)}(\cO_\qq(\Conf_{4}^{\overline{\times}}\sA_G))$ as a $\Bbbk$-algebra.        

Indeed, the assertion follows from the iterated application of (1) and (2). Let us show (1). By \cref{t:invstr}, $\cO_\qq(\Conf_{K+2}^{\{L-1\}}\sA_G)$ is generated by $\iota_{<}(\cO_\qq(\Conf_{K+1}^{\{L-1\}}\sA_G))$ and $\iota^{L-1}_{L+1, L}(\cO_\qq(\Conf_{3}^{\{1\}}\sA_G))$. (We regard the $L$-th tensor component as the first tensor component in \cref{t:invstr} by the quantum cyclic shifts.) 
    Since $\iota^{L-1}_{L+1, L}(\cO_\qq(\Conf_{3}^{\{1\}}\sA_G))\subset \iota^{L-1, L}_{L+2, L+1}(\cO_\qq(\Conf_{4}^{\overline{\times}}\sA_G))$ in $\cO_\qq(\Conf_{K+2}^{\overline{\times}}\sA_G)$ and 
    \begin{align*}
    &D_{j;\lambda}^{-1}\in \iota_{<}(\cO_\qq(\Conf_{K+1}^{\overline{\times}}\sA_G))\ \text{for}\ j\in [1,K+1]\setminus \{L, L+1\}, \lambda\in P_+,\\
    &D_{L+1;\lambda}^{-1}\in \iota^{L-1, L}_{L+2, L+1}(\cO_\qq(\Conf_{4}^{\overline{\times}}\sA_G))\ \text{for}\ \lambda\in P_+,
    \end{align*}
    we obtain the assertion. The statement (2) can be proved in the same way by regarding $L+2$-th tensor component as the first tensor component in \cref{t:invstr} by the quantum cyclic shifts. 
\end{proof}
Let us return to the proof of \cref{thm:cluster_polygon}. 
Define the set of monomials in $\cT(\Lambda_{\bbD_{K+2}}(\tbbs))$ as 
\begin{align*}
\cM:=\{\qq^aA^\ba\mid a\in \frac{1}{2d}\Z, \ba\in \Z_{\geq 0}^{J_{\bbD_{K+2}}(\tbbs)}\},
\end{align*}
and $\cT_+(\Lambda_{\bbD_{K+2}}(\tbbs)):=\sum_{m\in \cM}\Bbbk m$. By \cref{thm:triangle_square_1} and the construction of $A_{\bbD_{K+2}, \tbbs; (s, m)}$, 
\[
\iota^{(\ell, -),(\ell+1, -)}_{(\ell, +), (\ell+1, +)}(A_{\square, \bs_{2\ell-1}\ast (\overline{\bs}_{2\ell}^\ast)^\op; j})\in\mathbf{A}_{\tbbs}:=\{A_{\bbD_{K+2}, \tbbs; (s, m)}\mid (s, m)\in J_{\bbD_{K+2}}(\tbbs)\}
\]
for $j\in J_{\square}(\bs_{2\ell-1}\ast (\overline{\bs}_{2\ell}^\ast)^\op)$ and $\ell=1,\dots, L-1$. Such an element will be called an element of $\mathbf{A}_{\tbbs}$ on the square $((\ell, -),(\ell +1, -), (\ell +1, +),(\ell, +))$. Therefore, it follows from \cref{thm:cluster_square} that, for $\phi\in \iota^{(\ell, -),(\ell+1, -)}_{(\ell, +), (\ell+1, +)}(\cO_\qq(\Conf_{4}^{\overline{\times}}\sA_G))$, there exists $m\in \cM$ such that $M_{\tbbs}(m)\phi\in M_{\tbbs}(\cT_+(\Lambda_{\bbD_{K+2}}(\tbbs)))$. 
Similarly, by \cref{thm:triangle_square_2} and the construction of $A_{\bbD_{K+2}, \tbbs; (s, m)}$, 
\[
\iota^{(\ell', -),(\ell'+1, -)}_{(\ell'+1, +), (\ell'+2, +)}(A_{\square,  (\overline{\bs}_{2\ell'}^\ast)^\op\ast \bs_{2\ell'+1}; j})\in \mathbf{A}_{\tbbs}
\]
for $j\in J_{\square}((\overline{\bs}_{2\ell'}^\ast)^\op\ast \bs_{2\ell'+1})$ and $\ell'=1,\dots, L-2+\varepsilon_K$. Such an element will be called an element of $\mathbf{A}_{\tbbs}$ on the square $((\ell', -),(\ell'+1, -), (\ell'+2, +),(\ell'+1, +))$. Therefore, it follows from \cref{thm:cluster_square} that, for $\phi\in \iota^{(\ell', -),(\ell'+1, -)}_{(\ell'+1, +), (\ell'+2, +)}(\cO_\qq(\Conf_{4}^{\overline{\times}}\sA_G))$, there exists $m\in \cM$ such that $M_{\tbbs}(m)\phi\in M_{\tbbs}(\cT_+(\Lambda_{\bbD_{K+2}}(\tbbs)))$. 

Therefore, by Claim,  
we conclude that $M_{\tbbs}(\cM)$ forms an Ore set of $\cO_\qq(\Conf_{K+2}^{\times}\sA_G)$. 
Moreover, $M_{\tbbs}$ induces a surjective $\Bbbk$-algebra homomorphism
\[
M_{\tbbs}\colon \cT(\Lambda_{\bbD_{K+2}}(\tbbs))\twoheadrightarrow \cO_\qq(\Conf_{K+2}^{\times}\sA_G)[M_{\tbbs}(\cM)^{-1}]. 
\]
We shall show that this $M_{\tbbs}$ is an isomorphism. By \cite[Lemma 2.2 (b)]{LeYu}, it suffices to show that the Gelfand--Kirillov dimension $\GKdim (\cO_\qq(\Conf_{K+2}^{\times}\sA_G)[M_{\tbbs}(\cM)^{-1}])$ of $\cO_\qq(\Conf_{K+2}^{\times}\sA_G)[M_{\tbbs}(\cM)^{-1}]$ satisfies 
\begin{align*}
\GKdim (\cO_\qq(\Conf_{K+2}^{\times}\sA_G)[M_{\tbbs}(\cM)^{-1}])&\geq 
\GKdim (\cT(\Lambda_{\bbD_{K+2}}(\tbbs)))\\
&=|J_{\bbD_{K+2}}(\tbbs)|=NK+(K+1)|S|. 
\end{align*}
Recall \eqref{eq:gamma_op}. Given $\bs=(s_1,\dots, s_N)\in S(w_0)$, set 
\[
\Delta_{j}^{(k)}:=1^{\otimes (k-1)}\otimes \Delta_{\gamma_{\bs^{\op}; j}, \varpi_{s_{N+1-j}}}\otimes 1^{\otimes (K-k)} \in \cO_\qq(\sA_G)^{\bt K}\subset (\cO_\qq(\sA_G)^{\bt K})^{\ext} 
\]
for $j\in J(\bs^\op)$ and $k=1,\dots, K$. Then, by \cref{thm:clusterBorel}, the elements 
\[
\bigg(\dprod_{k=1,\dots, K}\dprod_{j\in J(\bs^\op)}(\Delta_{j}^{(k)})^{a_{j; k}}\bigg)e\bigg(\sum_{s\in S}a_s\varpi_s\bigg)\quad \text{for}\ a_{j; k}, a_s\in \Z_{\geq 0},
\]
are mutually distinct and linearly independent in $(\cO_\qq(\sA_G)^{\bt K})^{\ext}$, where we fix an arbitrary total order on $J(\bs^\op)$. Recall that, by \cref{t:invstr}, we have 
\begin{align*}
 (\cO_\qq(\sA_G)^{\bt K})^{\ext}\xrightarrow[\sim]{(\Phi_{K+2}^{\ext})^{-1}} \cO_\qq(\Conf_{K+2}^{\{K+2\}}\sA_G)\hookrightarrow \cO_\qq(\Conf_{K+2}^{\times}\sA_G)[M_{\tbbs}(\cM)^{-1}]. 
\end{align*}
Therefore, if we take a finite set $\mathcal{G}$ of $\Bbbk$-algebra generators of  $\cO_\qq(\Conf_{K+2}^{\times}\sA_G)[M_{\tbbs}(\cM)^{-1}]$ such that 
\[
\mathcal{G}\supset \{(\Phi_{K+2}^{\ext})^{-1}(\Delta_{j}^{(k)})\mid j\in J(\bs^\op), k\in [1,K]\}\cup \{ (\Phi_{K+2}^{\ext})^{-1}(e(\varpi_s))\mid s\in S\},
\]
then the above argument implies that   
\[
\dim_{\Bbbk}\bigg(\sum_{m\in \mathcal{G}_{\leq r(K(N+|S|)+|S|)}} \Bbbk m\bigg)\geq (r+1)^{K(N+|S|)+|S|}
\]
for $r\in \Z_{>0}$, where $\mathcal{G}_{\leq r'}:=\{x_1\cdots x_i\mid x_1,\dots, x_i\in \mathcal{G}, 0\leq i\leq r'\}$. Therefore, 
\begin{align*}
&\GKdim (\cO_\qq(\Conf_{K+2}^{\times}\sA_G)[M_{\tbbs}(\cM)^{-1}])\\
&\geq \lim_{r\to \infty }\left(\log\dim_{\Bbbk}\bigg(\sum_{m\in \mathcal{G}_{\leq r(K(N+|S|)+|S|)}} \Bbbk m\bigg)\middle/ \log (r(K(N+|S|)+|S|))  \right) \\
&\geq \lim_{r\to \infty }\frac{\log((r+1)^{K(N+|S|)+|S|})}{\log (r(K(N+|S|)+|S|))}  \\
&= \lim_{r\to \infty }\frac{(K(N+|S|)+|S|)\log(r+1)}{\log (r)+\log(K(N+|S|)+|S|)}=K(N+|S|)+|S|=NK+(K+1)|S|,  
\end{align*}
which proves that $M_{\tbbs}\colon \cT(\Lambda_{\bbD_{K+2}}(\tbbs))\twoheadrightarrow \cO_\qq(\Conf_{K+2}^{\times}\sA_G)[M_{\tbbs}(\cM)^{-1}]$ is an isomorphism. 
In particular, $\cF(\cO_\qq(\Conf_{K+2}^{\times}\sA_G))$ is isomorphic to $\cF(\cT(\Lambda_{\bbD_{K+2}}(\tbbs)))$. 

Next, we show that $(\Lambda_{\bbD_{K+2}}(\tbbs), \cE_{\bbD_{K+2}}(\tbbs))$ is a compatible pair.  As in \eqref{eq:Xvariable}, we set 
    \begin{align*}
X_{\bbD_{K+2}, \tbbs; (s, m)}&:=M_{\tbbs}(A^{\sum_{(s',m')\in J_{\bbD_{K+2}}(\tbbs)}\varepsilon_{\bbD_{K+2}, \tbbs;(s, m), (s', m')}f_{(s', m')}})\in \cF(\cO_\qq(\Conf_{K+2}^{\times}\sA_G))
\end{align*}
for $(s, m)\in J_{\bbD_{K+2}}(\tbbs)^\uf$. As mentioned in \eqref{eq:AXrel}, it suffices to show that there exists $d_{(s, m)}\in \Z_{>0}$ such that 
\begin{align}
X_{\bbD_{K+2}, \tbbs; (s, m)}A_{\bbD_{K+2}, \tbbs; (s', m')}=\qq^{-d_{(s, m)}\delta_{(s,m),(s',m')}/d}A_{\bbD_{K+2}, \tbbs; (s', m')}X_{\bbD_{K+2}, \tbbs; (s, m)}\label{eq:AXrel_polygon}
\end{align}
for $(s', m')\in J_{\bbD_{K+2}}(\tbbs)$. By the definition of $\cE_{\bbD_{K+2}}(\tbbs)$ and $\mathbf{A}_{\tbbs}$ (see \cref{thm:triangle_square_1,thm:triangle_square_2}), there exists a pair $((v_1, v_2, v_3, v_4), j)$ such that $(v_1, v_2, v_3, v_4)=((\ell, -),(\ell+1, -), (\ell+1, +),(\ell, +))$ for some $\ell=1,\dots, L-1$ or $((\ell', -),(\ell'+1, -), (\ell'+2, +),(\ell'+1, +))$ for some $\ell'=1,\dots, L-2+\varepsilon_K$, 
\[
\bbs:=
\begin{cases}
    \bs_{2\ell-1}\ast (\overline{\bs}_{2\ell}^\ast)^\op&\text{ if }(v_1, v_2, v_3, v_4)=((\ell, -),(\ell+1, -), (\ell+1, +),(\ell, +))\\
    (\overline{\bs}_{2\ell'}^\ast)^\op\ast \bs_{2\ell'+1}&\text{ if }(v_1, v_2, v_3, v_4)=((\ell', -),(\ell'+1, -), (\ell'+2, +),(\ell'+1, +))
\end{cases}
\]
and $j\in J_{\square}(\bbs)^\uf$ satisfying  
\[
X_{\bbD_{K+2}, \tbbs; (s, m)}=\iota^{v_1, v_2}_{v_4, v_3} (X_{\bi_{\square}(\bbs); j}),
\]
where we extend $\iota^{v_1, v_2}_{v_4, v_3} $ to the homomorphism between the skew-field of fractions. If $A_{\bbD_{K+2}, \tbbs; (s', m')}\in \mathbf{A}_{\tbbs}$ is on the square $(v_1, v_2, v_3, v_4)$, then 
\[
X_{\bbD_{K+2}, \tbbs; (s, m)}A_{\bbD_{K+2}, \tbbs; (s', m')}=\qq^{-d_{(s, m)}\delta_{(s,m),(s',m')}}A_{\bbD_{K+2}, \tbbs; (s', m')}X_{\bbD_{K+2}, \tbbs; (s, m)}
\]
for some $d_{(s, m)}\in \Z_{>0}$ by \cref{thm:cluster_square}. If $A_{\bbD_{K+2}, \tbbs; (s', m')}\in \mathbf{A}_{\tbbs}$ is not on the square $(v_1, v_2, v_3, v_4)$, then 
\[
X_{\bbD_{K+2}, \tbbs; (s, m)}A_{\bbD_{K+2}, \tbbs; (s', m')}=A_{\bbD_{K+2}, \tbbs; (s', m')}X_{\bbD_{K+2}, \tbbs; (s, m)}
\]
by \cref{thm:edge} and the degree $0$ property shown in \cref{prop:weight0}. Therefore, \eqref{eq:AXrel_polygon} is proved, hence $\bi_{\bbD_{K+2}}(\tbbs)$ forms a quantum seed in $\cF(\cO_\qq(\Conf_{K+2}^{\times}\sA_G))$. 

Finally, we show the inclusion $\cO_\qq(\Conf_{K+2}^{\times}\sA_G)\subset A_{\qq}(\bi_{\bbD_{K+2}}(\tbbs); J_{\bbD_{K+2}}(\tbbs)^{\fr})$. By the construction of $\bi_{\bbD_{K+2}}(\tbbs)$  (see \cref{thm:triangle_square_1,thm:triangle_square_2}), $\iota^{(\ell, -),(\ell+1, -)}_{(\ell, +), (\ell+1, +)}$ (resp.~$\iota^{(\ell', -),(\ell'+1, -)}_{(\ell'+1, +),(\ell'+2, +)}$) is an embedding of seeds in the sense of \cref{def:clusteremb} from $\bi_{\square}(\bs_{2\ell-1}\ast (\overline{\bs}_{2\ell}^\ast)^\op)$ (resp.~$\bi_{\square}((\overline{\bs}_{2\ell'}^\ast)^\op\ast \bs_{2\ell'+1})$) to $\bi_{\bbD_{K+2}}(\tbbs)$ for $\ell =1,\dots, L-1$ (resp.~$\ell'=1,\dots, L-2+\varepsilon_K$). Therefore (see a remark in \cref{def:clusteremb}), $A_{\qq}(\bi_{\bbD_{K+2}}(\tbbs); J_{\bbD_{K+2}}(\tbbs)^{\fr})$ contains 
\begin{align*}
&\iota^{(\ell, -),(\ell+1, -)}_{(\ell, +), (\ell+1, +)}(A_{\qq}(\bi_{\square}(\bs_{2\ell-1}\ast (\overline{\bs}_{2\ell}^\ast)^\op); S_{-\infty}\sqcup S_{\infty}))=\iota^{(\ell, -),(\ell+1, -)}_{(\ell, +), (\ell+1, +)}(\cO_\qq(\Conf_{4}^{\overline{\times}}\sA_G)),\\
&\iota^{(\ell', -),(\ell'+1, -)}_{(\ell'+1, +),(\ell'+2, +)}(A_{\qq}(\bi_{\square}((\overline{\bs}_{2\ell'}^\ast)^\op\ast \bs_{2\ell'+1})); S_{-\infty}\sqcup S_{\infty}))=\iota^{(\ell', -),(\ell'+1, -)}_{(\ell'+1, +),(\ell'+2, +)}(\cO_\qq(\Conf_{4}^{\overline{\times}}\sA_G))
\end{align*}
for $\ell=1,\dots, L-1$ and $\ell'=1,\dots, L-2+\varepsilon_K$. Here, the equalities follow from \cref{thm:cluster_square_cpt}. Therefore, the desired result follows from Claim.
\end{proof}

\appendix
\section{Quantum cluster algebras}\label{app:QCA}
In this appendix, we fix our notation concerning the quantum cluster algebras. 

\subsection{Quantum torus}\label{ssec:Qtorus}
Let $J$ be a finite set. 
For a $\frac{1}{d}\Z$-valued skew-symmetric $J\times J$-matrix 
$\Lambda = (\Lambda_{ij})_{i,j \in J}$, we define a \emph{quantum torus} 
$\cT(\Lambda)$ as the $\Bbbk$-algebra defined by the generators 
\[
A_j,\ A_j^{-1}\ (j \in J)
\]
and the relations:
\begin{itemize}
\item $A_j A_j^{-1} = A_j^{-1} A_j  =1$ for $j \in J$, 
\item $A_{i} A_{j} = \qq^{\Lambda_{ij}} A_j A_i$ for $i,j\in J$.
\end{itemize} 
We write the standard $\Z$-basis of $\Z^{J}$ as $\{f_j\mid j\in J\}$. Then for $\ba = (a_j)_{j \in J}=\sum_{j\in J}a_jf_j \in \Z^{J}$, we set
\[
A^{\ba} := \qq^{-\frac{1}{2}\sum_{i < j}a_i a_j \Lambda_{ij}} \dprod_{j \in J} A_j^{a_j},
\]
where we fix an arbitrary total ordering $<$ of the set $J$. The element $A^{\ba} \in \cT(\Lambda)$ does not depend on the choice of total ordering $<$. The set $\{A^\ba \mid \ba \in \Z^{J}\}$ forms a free $\Bbbk$-basis of $\cT(\Lambda)$. For $\ba=\sum_{j\in J}a_jf_j, \ba'=\sum_{j\in J}a'_jf_j\in \Z^{ J}$, we set 
\[
\Lambda(\ba, \ba'):=\sum_{i,j\in J}\Lambda_{ij}a_i a'_j.
\]
Then we have 
\[
A^{\ba}A^{\ba'}=\qq^{\frac{1}{2}\Lambda(\ba, \ba')} A^{\ba+\ba'}.
\]
Since $\cT(\Lambda)$ is an Ore domain, it is embedded into the skew field of fractions $\cF(\cT(\Lambda))$ \cite[Appendix A]{BZ}.

\subsection{Quantum cluster algebra}
\label{app:qca}
Let $J^\fr \subset J$ be a subset and put $J^\uf := J \setminus J^\fr$.
Let $\cE = (\varepsilon_{ij})_{i \in J^{\uf}, j \in J}$ be a $\Z$-valued $J^{\uf} \times J$-matrix.  
The pair $(\Lambda, \cE)$ is said to be \emph{compatible} if 
\begin{align}
\sum_{k \in J} \varepsilon_{ik}\Lambda_{kj} = - d_i  \delta_{i,j}/d    \label{eq:compatible}
\end{align}
for $i\in J^\uf, j \in J$, where $d_i\in \Z_{>0}$ for $i\in J^\uf$. In this case, $\cE$ has full rank, and its principal part $(\varepsilon_{ij})_{i,j \in J^{\uf}}$  is skew-symmetrizable by the diagonal matrix $\mathrm{diag}(d_i \mid i \in J^{\uf})$, that is, $\varepsilon_{ij}d_j=-\varepsilon_{ji}d_i$ for $i,j \in J^{\uf}$ \cite[Proposition 3.3]{BZ}. The matrix $\cE$ is called an \emph{exchange matrix}.

For a compatible pair $(\Lambda, \cE)$ and an element $k \in J^{\uf}$, define a new pair 
\[
\mu_k (\Lambda, \cE) = (\mu_k\Lambda, \mu_k\cE) := ((\Lambda'_{ij})_{i, j\in J}, (\varepsilon'_{ij})_{i \in J^{\uf}, j \in J}) 
\]
by 
\begin{align*}
    &\Lambda'_{ij}:=\Lambda(f'_i, f'_j),\ 
    f'_l:=\begin{cases}
        -f_k+\sum_{t\in J}[\varepsilon_{kt} ]_+f_t&\text{if $l=k$,}\\
        f_l&\text{if $l\neq k$,}        
    \end{cases}\\[10pt]
    &\varepsilon'_{ij}:=\begin{cases}
        -\varepsilon_{ij}&\text{if $k\in \{i, j\}$,}\\
        \varepsilon_{ij}+[\varepsilon_{ik}]_+[\varepsilon_{kj}]_+-[-\varepsilon_{ik}]_+[-\varepsilon_{kj}]_+&\text{otherwise,}        
    \end{cases}
\end{align*}
where $[a]_+:=\max\{a, 0\}$ for $a\in \Z$. The pair $\mu_k(\Lambda, \cE)$ again forms a compatible pair \cite[Proposition 3.4]{BZ}. 
The operation $\mu_k$ is called the \emph{mutation} at $k$ and it is involutive, i.e., $\mu_k(\mu_k(\Lambda, \cE)) = (\Lambda, \cE)$. 

Let $\cF$ be a skew-field which is isomorphic to the skew field of fractions of some quantum torus $\cT(\Lambda_{\bullet})$. The triple $\bi=(\Lambda, \cE, (A_{\bi; j})_{j\in J})$ is called a \emph{quantum seed} in $\cF$ if 
\begin{itemize}
    \item $(\Lambda, \cE)$ is a compatible pair, 
    \item There exists an injective homomorphism of $\Bbbk$-algebras $ \cT(\Lambda)\to \cF$ determined by 
    \begin{align}
    A_{j}\mapsto A_{\bi; j}\ \text{for}\ j\in J, \label{eq:toric_frame}
    \end{align}
which extends to an isomorphism $\cF(\cT(\Lambda))\xrightarrow{\sim}\cF$. 
\end{itemize}
The image of $A^{\bm{a}}$ under the map \eqref{eq:toric_frame} is denoted by  $A_{\bi}^{\bm{a}}$ for $\bm{a}\in \Z^J$. Note that $\cF$ is determined from $\Lambda$ as a $\Bbbk$-algebra. 

For $k \in J^{\uf}$, we set 
\begin{equation}
A_{\mu_k\bi; j} :=
\begin{cases} \label{eq:exrel}
A_{\bi}^{-f_k+\sum_{t\in J}[\varepsilon_{kt} ]_+f_t} + A_{\bi}^{-f_k+\sum_{t\in J}[-\varepsilon_{kt} ]_+f_t} & \text{if $j =k$}, \\
A_{\bi; j} & \text{if $j \neq k$}.
\end{cases}
\end{equation}
Then $\mu_k\bi:=(\mu_k\Lambda, \mu_k\cE, (A_{\mu_k\bi; j})_{j\in J})$ is again a quantum seed in $\cF$ \cite[Proposition 4.7]{BZ}. We say that $\mu_k\bi$ is obtained from $\bi$ by the mutation at $k$. It is also involutive, i.e., $\mu_k\mu_k\bi = \bi$. 

Let $\sigma$ be a permutation of the index set $J$ satisfying $\sigma(J^\fr) \subset J^\fr$ and $\sigma(J^\fr_{\times}) \subset J^\fr_{\times}$. 
We call it an \emph{admissible} permutation of $J$. 
For a quantum seed $\bi=(\Lambda, \cE, (A_{\bi; j})_{j\in J})$ in $\cF$, we can define another quantum seed $\sigma \bi = (\sigma \Lambda, \sigma \cE, (A_{\sigma\bi; j})_{j\in J})$ by 
\[
\sigma \Lambda := (\Lambda_{\sigma^{-1}(i),\sigma^{-1}(j)})_{i,j \in J},\quad 
\sigma \cE := (\varepsilon_{\sigma^{-1}(i),\sigma^{-1}(j)})_{i \in J^{\uf}, j \in J},\quad 
A_{\sigma\bi; j}:=A_{\bi; \sigma^{-1}(j)}.
\]

Two quantum seeds $\bi$ and $\bi'$ in $\cF$ are said to be \emph{mutation equivalent}, and written as $\bi\sim \bi'$, if there exist $k_1,\dots, k_t\in J^\uf$ and an admissible permutation $\sigma$ of $J$ satisfying 
\[
\sigma\mu_{k_t}\cdots \mu_{k_1}\bi=\bi'.
\]
\begin{dfn}
Let $\bi=(\Lambda, \cE, (A_{\bi; j})_{j\in J})$ be a quantum seed in $\cF$. Fix $J^\fr_{\times}\subset J^\fr$. The \emph{quantum cluster algebra} $A_{\qq}(\bi; J^\fr_{\times})$ is defined to be the $\Bbbk$-subalgebra of $\cF(\cT(\Lambda))$ generated by $
A_{\bi'; j}$ for $\bi'\sim \bi$ and $j\in J$, and $A_j^{-1}$ for $j\in J^\fr_{\times}$. 

An element of $\cF$ of the form $ 
A_{\bi'; j}$ (resp.  $A_{\bi'; J}^\ba$) for $\bi'\sim \bi$ and $j \in J$ (resp.~$\ba \in \Z_{\geq 0}^{J\setminus J^\fr_{\times}}\oplus \Z^{J^\fr_{\times}}$)  is called a \emph{quantum cluster variable} (resp.~\emph{quantum cluster monomial}) of $A_{\qq}(\bi; J^\fr_{\times})$. Note that, if $j\in J^{\fr}$, 
\[
A_{\bi'; j}=A_{\bi; j}
\]
for any $\bi'\sim \bi$. Hence, $A_{\bi; j}, j\in J^\fr$ are called the \emph{frozen variables}. 

For $\bi'\sim \bi$, define $\bcT(\bi'; J^\fr_{\times})$ as a $\Bbbk$-subalgebra of $\cF$ generated by $A_{\bi'; j}, j\in J$ and $A_{\bi'; j}^{-1}, j\in J^\uf\cup J^\fr_{\times}$. Then we define the \emph{quantum upper cluster algebra} by 
\[
U_{\qq}(\bi; J^\fr_{\times})=\bigcap_{\bi'; \bi'\sim\bi}\bcT(\bi'; J^\fr_{\times}). 
\]
\end{dfn}
We have the inclusion $A_{\qq}(\bi; J^\fr_{\times})\subset U_{\qq}(\bi; J^\fr_{\times})$, which is known as the \emph{quantum Laurent phenomenon} \cite[Corollary 5.2]{BZ} (cf.~\cite[Theorem 2.15]{GY:Quantumnilp}).

For a quantum seed $\bi=(\Lambda, \cE, (A_{\bi; j})_{j\in J})$ in $\cF$ and $k\in J^\uf$, we set 
\begin{align}
X_{\bi; k}:=A_{\bi}^{\sum_{t\in J}\varepsilon_{kt}f_t}.    \label{eq:Xvariable}
\end{align}
Moreover, we write the standard $\Z$-basis of $\Z^{J^{\uf}}$ as $\{e_j\mid j\in J^{\uf}\}$, and for $\bn = (n_j)_{j \in J^{\uf}}=\sum_{j\in J}n_je_j \in \Z^{J^{\uf}}$, we set
\begin{align}
X_{\bi}^{\bn} := A_{\bi}^{\sum_{k\in J^{\uf}}\sum_{t\in J}n_k\varepsilon_{kt}f_t}.\label{eq:Xmonomial}  
\end{align}
Note that $X_{\bi}^{e_k}=X_{\bi; k}$, and the compatibility condition \eqref{eq:compatible} is equivalent to the condition 
\begin{align}
X_{\bi; i}A_{\bi; j}=\qq^{-d_i\delta_{i, j}/d} A_{\bi; j}X_{\bi; i}\label{eq:AXrel}    
\end{align}
for all $i\in J^\uf$ and $j\in J$.

The following theorem is a highly non-trivial theorem which was conjectured in \cite{CAIV} and proved in \cite{DWZ,GHKK} (see also \cite[Proposition 5.6]{CAIV}, \cite[Theorem 6.7]{Tran}, \cite[Theorem A.8]{FHOO2}).

\begin{thm}\label{thm:separation}
Let $\bi$ be a quantum seed in $\cF$. Then, for any $\bi'\sim\bi$ and $j\in J^{\uf}$, there uniquely exist $\bg^\uf_{\bi'; j}\in \Z^{J^\uf}(\subset \Z^{J})$ and $\bg^\fr_{\bi'; j}\in \Z^{J^\fr}(\subset \Z^{J})$  such that 
\begin{align}
A_{\bi'; j}=A_{\bi}^{\bg^\uf_{\bi'; j}+\bg^\fr_{\bi'; j}}(1+\sum_{\bn\in \Z_{\geq 0}^{J^{\uf}}}c_{\bn}X_{\bi}^{\bn})\quad (c_{\bn}\in \Bbbk).\label{eq:separation}
\end{align}
Moreover, $\bg^\uf_{\bi'; j}=\bg^\uf_{\bi''; j'}$ implies $A_{\bi'; j}=A_{\bi'; j'}$ for $\bi'\sim\bi\sim \bi''$ and $j, j'\in J^{\uf}$. 
\end{thm}
For example, by \eqref{eq:exrel} and \eqref{eq:AXrel},
\begin{align}
A_{\mu_k\bi; k} =A_{\bi}^{-f_k+\sum_{t\in J}[-\varepsilon_{kt} ]_+f_t}(1+\qq^{d_k/2d} X_{\bi; k})\label{eq:onestep}
\end{align}
for $k\in J^{\uf}$. 

The vector $\bg^\uf_{\bi'; j}$ in \cref{thm:separation} is called \emph{the $g$-vector} of $A_{\bi'; j}$. In this paper, we write such $\bg^\uf_{\bi'; j}$ as 
\begin{equation}
\bg^\uf_{\bi'; j}=:\bg_{\bi}(A_{\bi'; j}).   \label{eq:gvect_notation}  
\end{equation}

In the main body of the paper, the following proposition plays an important role. 
\begin{prop}\label{prop:weight0}
    Let $\bi=(\Lambda, \cE, (A_{\bi; j})_{j\in J})$ be a quantum seed in $\cF$ and $L$ a free $\Z$-module. Assume that the quantum cluster algebra $A_{\qq}(\bi; J^\fr_{\times})$ (resp.~the quantum upper cluster algebra $U_{\qq}(\bi; J^\fr_{\times})$) is an $L$-graded $\Bbbk$-algebra, and that $A_{\bi; j}$, $j\in J$ are homogeneous with respect to this grading. Then $A_{\bi'; k}$ is homogeneous for all $\bi'\sim \bi$ and $k\in J^\uf$. Moreover, if we extend the graded structure on $A_{\qq}(\bi; J^\fr_{\times})$ (resp.~$U_{\qq}(\bi; J^\fr_{\times})$) to that on $\cT(\Lambda)$, then the degree of $X_{\bi'; k}$ is equal to $0$ for all $\bi'\sim \bi$ and $k\in J^\uf$.
\end{prop}
\begin{proof}
    In this proof, we write $A_{\qq}(\bi; J^\fr_{\times})$ (resp.~$U_{\qq}(\bi; J^\fr_{\times})$) as $R$. Denote by $R_{\alpha}$ the graded component of $R$ of degree $\alpha\in L$. We extend the graded structure on $R$ to that on $\cT(\Lambda)$, and write as $\deg x$ the degree of the homogeneous element $x$ of $\cT(\Lambda)$. 

    Let $k\in J^{\uf}$. We shall show $\deg (X_{\bi; k})=0$ by contradiction. Assume that $\deg (X_{\bi; k})\neq 0$. Then, by our assumption, $A_{\bi}^{-f_k+\sum_{t\in J}[-\varepsilon_{kt} ]_+f_t}$ and $A_{\bi}^{-f_k+\sum_{t\in J}[-\varepsilon_{kt} ]_+f_t}X_{\bi; k}$ are homogeneous elements with different degrees (in $\cT(\Lambda)$). Therefore, \eqref{eq:onestep} implies that both $A_{\bi}^{-f_k+\sum_{t\in J}[-\varepsilon_{kt} ]_+f_t}$ and $A_{\bi}^{-f_k+\sum_{t\in J}[-\varepsilon_{kt} ]_+f_t}X_{\bi; k}$ lie in $R$. 
    However, we can easily show that $A_{\bi}^{-f_k+\sum_{t\in J}[-\varepsilon_{kt} ]_+f_t}\notin \bcT(\mu_k\bi; J^\fr_{\times})$. Here note that $X_{\mu_k\bi; k}\neq 1$ since the exchange matrix of a quantum seed has full rank. It contradicts the quantum Laurent phenomenon (resp.~the definition of quantum upper cluster algebra). Therefore, $\deg (X_{\bi; k})=0$. 

    The above result and the expression \eqref{eq:separation} imply that $A_{\bi'; k}$ is homogeneous for all $\bi'\sim \bi$ and $k\in J^\uf$. Moreover, since $A_{\qq}(\bi'; J^\fr_{\times})=A_{\qq}(\bi; J^\fr_{\times})$ (resp.~$U_{\qq}(\bi'; J^\fr_{\times})=U_{\qq}(\bi; J^\fr_{\times})$) for $\bi'\sim \bi$, we  deduce that $\deg X_{\bi'; k}=0$ for all $\bi'\sim \bi$ and $k\in J^\uf$ by regarding $\bi'$ as an initial quantum seed.
\end{proof}
\begin{dfn}\label{def:coeffmod}
Let $\bi=(\Lambda, \cE, (A_{\bi; j})_{j\in J})$ be a quantum seed in $\cF$, and fix $J^\fr_{\times}\subset J^\fr$. Take $m_{j; l}\in \Z$ for $j\in J\setminus J^\fr_{\times}$ and $l\in J^\fr_{\times}$. For $j\in J$, define $f_{\mrm; j}\in \Z^J$ by 
\[
f_{\mrm; j}:=\begin{cases}
    f_j+\sum_{l\in J^\fr_{\times}} m_{j; l}f_l&\text{if }j\in J\setminus J^\fr_{\times},\\
    f_j&\text{if }j\in J^\fr_{\times}.
\end{cases} 
\]
Define $(A_{\bi_{\mrm}; j})_{j\in J}$, $\Lambda_{\mrm}=(\Lambda_{\mrm; ij})_{i,j\in J}$ and $\cE_{\mrm}=(\varepsilon_{\mrm;ij})_{i\in J^{\uf},j\in J}$ by 
\begin{align}
    &A_{\bi_{\mrm}; j}:=A_{\bi}^{f_{\mrm; j}}\quad \text{for }j\in J,\label{eq:A_modified_def}\\
    &\Lambda_{\mrm;ij}:=\Lambda(f_{\mrm; i}, f_{\mrm; j})\quad \text{for }i, j\in J,\label{eq:Lambda_modified_def}\\
    &\sum_{j\in J}\varepsilon_{\mrm;ij}f_{\mrm; j}=\sum_{j\in J}\varepsilon_{ij}f_j
    \quad \text{for }i\in J^{\uf}.\label{eq:varepsilon_modified_def}
\end{align}
Note that the equalities \eqref{eq:varepsilon_modified_def} uniquely determine $\cE_{\mrm}$, and 
\begin{align}
&A_{\bi_{\mrm}; i}A_{\bi_{\mrm}; j}=\qq^{\Lambda_{\mrm;ij}}A_{\bi_{\mrm}; j}A_{\bi_{\mrm}; i}\quad \text{for }i, j\in J,\label{eq:qcomm_A_modified}\\
&    \varepsilon_{\mrm;ij}= \varepsilon_{ij}\quad \text{for }i\in J^{\uf}, j\in J\setminus J^\fr_{\times}.\label{eq:varepsilon_modified_coincidence}
\end{align}
We call the collection $\bi_{\mrm}:=(\Lambda_{\mrm}, \cE_{\mrm},(A_{\bi_{\mrm}; j})_{j\in J})$  \emph{a strict similarity transform of $\bi$ associated with the similarity datum $(m_{j; l})_{j\in J\setminus J^\fr_{\times}, l\in J^\fr_{\times}}$}.
\end{dfn}
\begin{prop}\label{prop:modified_seed}
In the notation of \cref{def:coeffmod}, $\bi_{\mrm}:=(\Lambda_{\mrm}, \cE_{\mrm},(A_{\bi_{\mrm}; j})_{j\in J})$ forms a quantum seed in $\cF$. Moreover, the following hold. 
\begin{itemize}
    \item[(1)]  $A_{\qq}(\bi_{\mrm}; J^\fr_{\times})=A_{\qq}(\bi; J^\fr_{\times})$ and $U_{\qq}(\bi_{\mrm}; J^\fr_{\times})=U_{\qq}(\bi; J^\fr_{\times})$.
    \item[(2)] $\mu_k\bi_{\mrm}$ is a strict similarity transform of $\mu_k\bi$ for $k\in J^\uf$. 
    \item[(3)] For any sequence $(k_1,\dots, k_t)$ of elements of $J^\uf$ and $j\in J^\uf$, $\bg_{\bi}(A_{\mu_{k_t}\cdots \mu_{k_1}\bi; j})=\bg_{\bi_{\mrm}}(A_{\mu_{k_t}\cdots \mu_{k_1}\bi_{\mrm}; j})$.
\end{itemize}
\end{prop}
\begin{proof}
The correspondence $f_j\mapsto f_{\mrm; j}$ is extended to a $\Z$-module automorphism on $\Z^J$. Hence, the $\Bbbk$-subalgebra of $\cF$ generated by $A_{\bi_{\mrm}; j}^{\pm 1}$, $j\in J$ is isomorphic to the quantum torus $\cT(\Lambda_{\mrm})$, and $\cF$ is generated by $A_{\bi_{\mrm}; j}$, $j\in J$ as a skew field. Moreover, \eqref{eq:varepsilon_modified_def} implies that 
    \begin{align}
X_{\bi_\mrm; i}
 &:=A_{\bi_\mrm}^{\sum_{j\in J}\varepsilon_{\mrm;ij}f_j}\notag\\
 &:=\qq^{-\frac{1}{2}\sum_{j_1 < j_2}\varepsilon_{\mrm;ij_1} \varepsilon_{\mrm;ij_2} \Lambda_{\mrm;j_1j_2}} \dprod_{j \in J} A_{\bi_\mrm; j}^{\varepsilon_{\mrm;ij}}\notag\\
  &=\qq^{-\frac{1}{2}\sum_{j_1 < j_2}\varepsilon_{\mrm;ij_1} \varepsilon_{\mrm;ij_2} \Lambda_{\mrm;j_1j_2}} \dprod_{j \in J} A_{\bi}^{\varepsilon_{\mrm;ij}f_{\mrm; j}}\notag\\
    &=A_{\bi}^{\sum_{j\in J}\varepsilon_{\mrm;ij}f_{\mrm; j}}
    =A_{\bi}^{\sum_{j\in J}\varepsilon_{ij}f_j}=X_{\bi; i}\label{eq:X_coincidence}
\end{align}
for $i\in J_{\triangle}(\bs)^{\uf}$, where we fix a total ordering on $J$. Therefore, 
\begin{align*}
X_{\bi_\mrm; i}A_{\bi_\mrm; j}=\qq^{-d_i\delta_{i,j}/d}A_{\bi_\mrm; j}X_{\bi_\mrm; i}
\end{align*}
for $i\in J^{\uf}$ and $j\in J$. Here $d_i$ is a positive integer satisfying $X_{\bi; i}A_{\bi; i}=\qq^{-d_i/d}A_{\bi; i}X_{\bi; i}$. Hence, by the argument in \eqref{eq:AXrel}, $(\Lambda_\mrm, \cE_\mrm)$ is a compatible pair, and $\bi_\mrm$ is  a quantum seed in $\cF$. By \eqref{eq:varepsilon_modified_coincidence}, \eqref{eq:X_coincidence}, and \cite[Theorem 5.3]{Tran}, the quantum cluster variables of $A_{\qq}(\bi_{\mrm}; J^\fr_{\times})$ are those of $A_{\qq}(\bi; J^\fr_{\times})$ up to multiplication by the powers of $\qq^{\pm \frac{1}{2d}}$ and $A_{\bi; j}^{\pm 1}, j\in J^\fr_{\times}$. Therefore, the statement (1) follows. The statement (2) straightforwardly follows from \eqref{eq:onestep}, \eqref{eq:X_coincidence}, and \cite[Lemma 5.4]{Tran}. The statement (3) follows from \eqref{eq:varepsilon_modified_coincidence} and \cite[Proposition 6.6]{CAIV}.
\end{proof}
Note that the identity map $A_{\qq}(\bi_{\mrm}; J^\fr_{\times})\xrightarrow{\sim}A_{\qq}(\bi; J^\fr_{\times})$ is regarded as a quantum quasi-homomorphism with respect to $\bi_{\mrm}$ and $\bi$ in \cite[Definition 2.1]{CHL}, and the identity map $U_{\qq}(\bi_{\mrm}; J^\fr_{\times})\xrightarrow{\sim}U_{\qq}(\bi; J^\fr_{\times})$ is regarded as a variation map from $\bi_{\mrm}$ to $\bi$ in \cite[Definition 3.3, Lemma 3.10]{KQW}.

\begin{dfn}\label{def:clusteremb}
Let $R_1$ and $R_2$ be Ore domains over $\Bbbk$ such that $\cF(R_1)$ and $\cF(R_2)$ are isomorphic to the skew field of fractions of quantum tori. Assume that $\cF(R_k)$ admits a quantum seed $\bi_k=(\Lambda_k=(\Lambda_{k; ij})_{i, j\in J_k}, \cE_k=(\varepsilon_{k; ij})_{i \in J_k^{\uf}, j \in J_k}, (A_{\bi_k; j})_{j\in J_k})$ for $k=1,2$. An injective $\Bbbk$-algebra homomorphism $\iota:R_1\to R_2$ is said to be \emph{an embedding of seeds from $\bi_1$ to $\bi_2$} if there exists an injective map $\sfJ\colon  J_1\to J_2$ such that 
\begin{itemize}
    \item[(1)] $\sfJ(J_1^\uf)\subset J_2^\uf$, and $j\in \sfJ(J_1)$ whenever $\varepsilon_{2; ij}\neq 0$ for some $i\in \sfJ(J_1^\uf)$.\footnote{We do not assume $\sfJ(J_1^\fr)\subset J_2^\fr$.}
    \item[(2)] $\iota(A_{\bi_1; j})=A_{\bi_2; \sfJ(j)}$ for $j\in J_1$.
    \item[(3)] $\varepsilon_{1; ij}=\varepsilon_{2; \sfJ(i)\sfJ(j)}$ for $i \in J_1^{\uf}$ and $j \in J_1$. 
\end{itemize}
In this case, the injective map $\sfJ$, which is uniquely determined by (2), is called \emph{an index correspondence associated with $\iota$}. Note that $\Lambda_{1; ij}=\Lambda_{2; \sfJ(i)\sfJ(j)}$ for $i, j\in J_1$ by (2). Moreover, if $\sfJ(J_{1\times}^\fr)\subset J_{1\times}^\fr$, $\iota$ induces 
\[
A_{\qq}(\bi_1; J_{1\times}^\fr)\to A_{\qq}(\bi_2; J_{2\times}^\fr).
\]
\end{dfn}
\section{Quivers associated with Coxeter words}\label{sec:quiver}
A \emph{weighted quiver} (or \emph{quiver} for short) consists of the following data $Q=(J,J^\fr,\sigma,\bm{d})$:
\begin{itemize}
    \item $J^\fr \subset J$ are finite sets. 
    \item $\sigma=(\sigma_{ij})_{i,j \in J}$ is a  skew-symmetric $\Z/2$-valued matrix such that $\sigma_{ij} \in \Z$ unless $(i,j) \in J^\fr \times J^\fr$.
    \item $\bm{d}=(d_i)_{i \in J} \in \Z^J_{>0}$ is a tuple of positive integers.
\end{itemize}
Diagrammatically, $J$ is the set of vertices of the quiver, $\bm{d}$ is the tuple of weights assigned to vertices, and the data of arrows are encoded in the matrix $\sigma$ as 
\begin{align*}
    \sigma_{ij}:= \#\{\text{arrows from $i$ to $j$}\}- \#\{\text{arrows from $j$ to $i$}\}.
\end{align*}
Here we have \lq\lq half" arrows when $\sigma_{ij} \in \Z/2$ (shown by dashed arrows in figures). The quiver has no loops nor 2-cycles by definition. The subset $J^\fr$ is called the frozen set, and mutations will be allowed only at the vertices in the complement $J^\uf:= J \setminus J^\fr$, called the unfrozen set. We define the \emph{exchange matrix} $\cE=(\ve_{ij})_{i\in J^\uf,j \in J}$ of $Q$ to be  
\begin{align*}
    \ve_{ij}:= d_i \sigma_{ij}\gcd(d_i,d_j)^{-1}.
\end{align*}
The \emph{chiral dual} of $Q$ is defined by $Q^\mathrm{op}:=(J,J^\fr,-\sigma,d)$. 

\smallskip
\paragraph{\textbf{Picture convention}}
In figures, we show unfrozen vertices by circular nodes and frozen vertices by square nodes. All the weights of quivers related to the Lie theory are either $1,2$ or $3$. The corresponding node styles are shown in the table below.

\begin{table}[ht]
    \centering
\begin{tabular}{c|ccc}
 & $d_i=1$ & $d_i=2$ & $d_i=3$ \\ \hline
unfrozen     & $\tikz{\draw(0,0) circle(2pt);}$ & $\tikz{\dnode{0,0}{black};}$ &  $\tikz{\tnode{0,0}{black};}$\\
frozen     & $\tikz{\fnode{0,0}{black};}$ & $\tikz{\fdnode{0,0}{black};}$ & $\tikz{\ftnode{0,0}{black};}$
\end{tabular}
\end{table}

Recall the notation from \cref{sec:Lie_theory}. Let $\lieg$ be a complex finite dimensional simple Lie algebra associated with the Cartan matrix $C=(c_{st})_{s,t \in S}$.  

For $s \in S$, we define a weighted quiver $\bJ_\mathrm{full}(s)=(J_\mathrm{full}(s),J_\mathrm{full}(s)^\fr,\sigma(s),\bm{d}(s))$ as follows.
\begin{itemize}
    \item $J_\mathrm{full}(s)=J_\mathrm{full}(s)^\fr:=(S \setminus \{s\})\cup\{s^l,s^r,s^e\}$, where $s^l,s^r,s^e$ are new elements (`left', `right', `extra').
    \item The skew-symmetric matrix $\sigma(s)=(\sigma_{tu})_{t,u \in J_\mathrm{full}(s)}$ is given by
    \begin{align*}
            &\sigma_{s^r,s^l}=\sigma_{s^l,s^e}=\sigma_{s^e,s^r}=1, \\
            &\sigma_{s^l,u}=\sigma_{u,s^r}= \begin{cases} 
            1/2 & \text{if $u \neq s$ and $c_{su} \neq 0$}, \\
            0 & \text{if $u \neq s$ and $c_{su}=0$}.
            \end{cases}
    \end{align*}
    Note that other entries are determined by the skew-symmetricity. 
    \item $\bm{d}(s)$ is given by 
    \begin{align*}
        \bm{d}(s)_t := \begin{cases}
            (\alpha_s, \alpha_s)/2 & \mbox{if $t =s^l,s^r,s^e$}, \\
            (\alpha_t, \alpha_t)/2 & \mbox{if $t \in S \setminus \{s\}$}.
        \end{cases}
    \end{align*}
\end{itemize}
Let $\bJ_\mathrm{full}(\overline{s}):=\bJ_\mathrm{full}(s)^{\mathrm{op}}$. We call $\bJ_\mathrm{full}(s)$, $\bJ_\mathrm{full}(\overline{s})$ the \emph{elementary quivers} associated with $\mathfrak{g}$. 

We define a function $\delta: J_\mathrm{full}(s) \to S$ on the set of vertices, which we call the \emph{Dynkin labeling}, by $\delta_{s^l}=\delta_{s^r}=\delta_{s^e}:=s$ and $\delta_t:=t$ for $t \in S \setminus \{s\}$.

\begin{ex}
Here are some examples of the elementary quivers.
\begin{enumerate}
\item Type $A_3$: $S=\{1,2,3\}$ and the Cartan matrix is given by
\begin{align*}
    C = \begin{pmatrix}
    2&-1&0\\
    -1&2&-1\\
    0&-1&2
    \end{pmatrix}.
\end{align*}
The elementary quivers $\bJ_\mathrm{full}(1)$, $\bJ_\mathrm{full}(2)$ and $\bJ_\mathrm{full}(3)$ are given as follows:
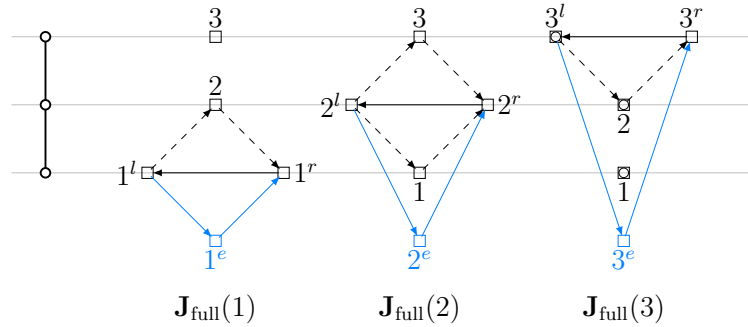
\[
\scalebox{0.9}{
\begin{tikzpicture}

\begin{scope}[>=latex]
\foreach \i in {0,1,2}
\draw[gray!50] (-1,\i) -- (10,\i);
\draw[thick] (-0.5,0)-- (-0.5,2);
\foreach \i in {0,1,2} \filldraw[draw=black,thick,fill=white](-0.5,\i) circle(2pt);
\fnode{3,0}{black} coordinate(B) node[right]{$1^r$};
\fnode{1,0}{black} coordinate(C) node[left]{$1^l$};
\fnode{2,1}{black} coordinate(D) node[above]{$2$};
\fnode{2,2}{black} node[above]{$3$};
\fnode{2,-1}{myblue} coordinate(X) node[myblue,below]{$1^e$};
\qarrow{B}{C};
\qdarrow{C}{D};
\qdarrow{D}{B};
{\color{myblue}
 \qarrow{C}{X};
 \qarrow{X}{B};
}
\draw (2,-2) node{$\bJ_{\mathrm{full}}(1)$};
\fnode{6,1}{black} coordinate(E) node[right]{$2^r$};
\fnode{4,1}{black} coordinate(F) node[left]{$2^l$};
\fnode{5,0}{black} coordinate(G) node[below]{$1$};
\fnode{5,2}{black} coordinate(H) node[above]{$3$};
\fnode{5,-1}{myblue} coordinate(Y) node[myblue,below]{$2^e$};
\qarrow{E}{F};
\qdarrow{F}{G};
\qdarrow{G}{E};
\qdarrow{F}{H};
\qdarrow{H}{E};
{\color{myblue}
 \qarrow{F}{Y};
 \qarrow{Y}{E};
}
\draw (5,-2) node{$\bJ_{\mathrm{full}}(2)$};
\fnode{9,2}{black} coordinate(B) node[above]{$3^r$};
\fnode{7,2}{black} circle(2pt) coordinate(C) node[above]{$3^l$};
\fnode{8,1}{black} circle(2pt) coordinate(D) node[below]{$2$};
\fnode{8,0}{black} circle(2pt) node[below]{$1$};
\fnode{8,-1}{myblue} coordinate(X) node[myblue,below]{$3^e$};
\qarrow{B}{C};
\qdarrow{C}{D};
\qdarrow{D}{B};
{\color{myblue}
 \qarrow{C}{X};
 \qarrow{X}{B};
}
\draw (8,-2) node{$\bJ_{\mathrm{full}}(3)$};
\end{scope}
\end{tikzpicture}}
\] 
\item Type $C_3$: $S=\{1,2,3\}$ and the Cartan matrix is given by
\begin{align*}
    C = \begin{pmatrix}
    2&-1&0\\
    -1&2&-2\\
    0&-1&2
    \end{pmatrix}.
\end{align*}
The elementary quivers $\bJ_\mathrm{full}(1)$, $\bJ_\mathrm{full}(2)$, $\bJ_\mathrm{full}(3)$ are given as follows.
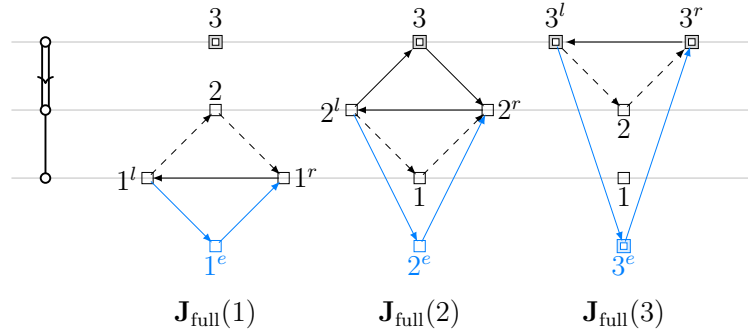
\[
\scalebox{0.9}{
\begin{tikzpicture}
\foreach \i in {0,1,2}
\draw[gray!50] (-1,\i) -- (10,\i);
\draw[thick] (-0.5,0)-- (-0.5,1);
\draw[thick,double distance=2pt,-<-](-0.5,1)-- (-0.5,2);
\foreach \i in {0,1,2} \filldraw[draw=black,thick,fill=white](-0.5,\i) circle(2pt);
\begin{scope}[>=latex]
\fnode{3,0}{black} coordinate(B) node[right]{$1^r$};
\fnode{1,0}{black} coordinate(C) node[left]{$1^l$};
\fnode{2,1}{black} coordinate(D) node[above]{$2$};
\fdnode{2,2}{black} node[above=0.2em]{$3$};
\fnode{2,-1}{myblue} coordinate(X) node[myblue,below]{$1^e$};
\qarrow{B}{C};
\qdarrow{C}{D};
\qdarrow{D}{B};
{\color{myblue}
 \qarrow{C}{X};
 \qarrow{X}{B};
}
\draw (2,-2) node{$\bJ_{\mathrm{full}}(1)$};
\fnode{6,1}{black} coordinate(E) node[right]{$2^r$};
\fnode{4,1}{black} coordinate(F) node[left]{$2^l$};
\fnode{5,0}{black} coordinate(G) node[below]{$1$};
\fdnode{5,2}{black} coordinate(H) node[above=0.2em]{$3$};
\fnode{5,-1}{myblue} coordinate(Y) node[myblue,below]{$2^e$};
\qarrow{E}{F};
\qdarrow{F}{G};
\qdarrow{G}{E};
\draw[qarrow,head] (F) -- (H);
\draw[qarrow,tail] (H)--(E);
{\color{myblue}
 \qarrow{F}{Y};
 \qarrow{Y}{E};
}
\draw (5,-2) node{$\bJ_{\mathrm{full}}(2)$};
\fdnode{9,2}{black} coordinate(B) node[above=0.2em]{$3^r$};
\fdnode{7,2}{black} coordinate(C) node[above=0.2em]{$3^l$};
\fnode{8,1}{black} coordinate(D) node[below]{$2$};
\fnode{8,0}{black} node[below]{$1$};
\fdnode{8,-1}{myblue} coordinate(X) node[myblue,below]{$3^e$};
\qsarrow{B}{C};
\draw[qarrow,dashed,tail](C)--(D);
\draw[qarrow,dashed,head] (D)--(B);
{\color{myblue}
 \qarrow{C}{X};
 \qarrow{X}{B};
}
\draw (8,-2) node{$\bJ_{\mathrm{full}}(3)$};
\end{scope}
\end{tikzpicture}}
\] 
\end{enumerate}
Note that the vertices with the same Dynkin label are drawn on the same level in the pictures.
\end{ex}

\paragraph{\textbf{The quiver $\bJ_\mathrm{full}(\bbs)$}}
Recall the amalgamation of quivers \cite{FG08}. See also \cite[Appendix B]{IO:Wilson}. 
A Coxeter word $\bbs=(\bbs_1, \ldots, \bbs_m)$ is a finite word of alphabets in $\overline{S} \sqcup S$, where $\overline{S}=\{\overline{s} \mid s \in S\}$. 
We construct an associated quiver
$\bJ_\mathrm{full}(\bbs)=\bJ_\mathrm{full}(\bbs_1)\ast \bJ_\mathrm{full}(\bbs_2)\ast \dots \ast\bJ_\mathrm{full}(\bbs_m)$ as follows. First, amalgamate the associated elementary quivers $\bJ_\mathrm{full}(\bbs_i)$ for $i=1,\dots,m-1$ by identifying the right-most vertex of $\bJ_\mathrm{full}(\bbs_i)$ with the left-most vertex of $\bJ_\mathrm{full}(\bbs_{i+1})$ having the same Dynkin label. The half-arrows with same direction combine to give a single solid arrow, and those with the opposite directions are canceled together. 
See \cref{ex:amalgamation}. 

Note that the Dynkin labelings are preserved under this amalgamation, hence it makes sense in the quiver $\bJ_\mathrm{full}(\bbs)$. 
For each $s \in S$, the left-most and right-most vertices of Dynkin index $s$ are declared to be frozen, as well as all the extra vertices. All the other vertices are promoted to be mutable. The full subquiver obtained from $\bJ_{\mathrm{full}}(\bbs)$ by deleting all the extra vertices is called the \emph{basic part} of $\bJ_{\mathrm{full}}(\bbs)$, and denoted by $\bJ(\bbs)$. The exchange matrix of $\bJ(\bbs)$ will be written as $\cE(\bbs)=(\varepsilon_{\bbs; ij})_{i\in J(\bbs)^\uf, j\in J(\bbs)}$. 

\begin{conv}[labeling of vertices in the basic part]\label{conv:basic_vertex}
Let $\bbs$ be any Coxeter word. 
For the basic part of $\bJ_{\mathrm{full}}(\bbs)$, let $v_m^s$ denote the $(m+1)$-st vertex with Dynkin label $s$ from the left, for $s \in S$ and $m=0,\dots,n^s(\bbs)$. Here $n^s(\bbs)$ is the total number of $s$ and $\overline{s}$ appearing in the word $\bbs$. The labeling of extra vertices will be given in special cases: see \cref{conv:extra_labeling}.
\end{conv}

\smallskip

\paragraph{\textbf{Inversion set and the quiver $\bJ_\std(\bbs)$}}
Given $w \in W$, let $\Phi^\vee(w):=\{ \alpha^\vee \in \Phi^\vee_+ \mid w\alpha^\vee\in -\Phi^\vee_+\}$, which we call the \emph{inversion set} for $w \in W$. We have $|\Phi^\vee(w)|=\ell(w)$. We also use the notation
\begin{align*}
    \Inv(w):= \{ s \in S \mid w \alpha_s^\vee \in -\Phi^\vee_+\}.
\end{align*}
Namely, $\Phi^\vee(w)\cap \Pi^\vee = \{\alpha_s^\vee \mid s \in \Inv(w)\}$. 
Given a reduced word $\bs=(s_1,\dots,s_m)$ of $w \in W$, let 
\begin{align*}
    \beta^\bs_k:= r_{s_m}\dots r_{s_{k+1}}\alpha^\vee_{s_k} \in \Phi^\vee
\end{align*}
for $k=1,\dots,m$. Then we have $\Phi^\vee(w)=\{\beta_1^\bs,\dots,\beta_m^\bs\}$. 

\begin{conv}\label{conv:double_reduced}
The Weyl group $W \times W$ is generated by the simple reflections $r_s$ with $s \in \overline{S} \sqcup S$, 
where the generators of the first component is indexed by $\overline{S}$. 
A reduced word of $(u,v) \in W \times W$ is represented by a sequence $\bbs=(\bbs_1,\dots,\bbs_m)$ with $\bbs_i \in \overline{S} \sqcup S$. The subsequence $\{i_1,\dots,i_{\overline{n}}\} $ (resp. $\{j_1,\dots,j_{n}\}$) of $\{1,\dots,m\}$ with $\bbs_{i_k} \in \overline{S}$ (resp. $\bbs_{j_k} \in S$) is called the $\overline{S}$-part (resp. $S$-part) of $\bbs$, which gives a reduced word of $u$ (resp. $v$). 
Then $\beta_k^\bbs=r_{\bbs_m}\dots r_{\bbs_{k+1}}\alpha_{\bbs_k}^\vee$ for $k=1,\dots,m$ parameterize $\Phi^\vee(u) \sqcup \Phi^\vee(v)$.
\end{conv}
Now we assume $\bbs=(\bbs_1,\dots,\bbs_m)$ is a reduced word of $(u,v) \in W \times W$. 
The extra vertex of the $k$-th elementary quiver $\bJ_{\mathrm{full}}(\bbs_k)$ is accompanied with the positive coroot $\beta_k^\bbs=r_{\bbs_m}\dots r_{\bbs_{k+1}}\alpha_{\bbs_k}^\vee \in \Phi^\vee(u) \sqcup \Phi^\vee(v)$. 
The extra vertices get distinct coroots in $\Phi^\vee(u) \sqcup \Phi^\vee(v)$. 

We then delete all the extra vertices except for those attached to the $k$-th elementary quiver such that $\beta_k^\bbs \in \Pi^\vee \sqcup \Pi^\vee$. For such $k$, we write $\beta_k^\bbs=\alpha_{s(k)}^{\vee}$. We add (half-)arrows among these remaining extra vertices as given by the exchange matrix
\begin{align*}
    \sigma_{\alpha_{s(j)}^{\vee}\alpha_{s(k)}^{\vee}}^{\mathrm{ex}}
    :=\begin{cases}
        \frac{\mathrm{sgn}(j-k)}{2} & \mbox{if $s(j), s(k) \in S, c_{s(j), s(k)}<0$}, \\
        \frac{\mathrm{sgn}(k-j)}{2} & \mbox{if $s(j), s(k) \in \overline{S}, c_{s(j), s(k)}<0$}, \\
        0 & \mbox{otherwise}.
    \end{cases}
\end{align*}
The resulting quiver is denoted by $\bJ_\std(\bbs)$. 

\begin{conv}[labeling of extra vertices]\label{conv:extra_labeling}
Let $\bbs$ be a reduced word of $(u,v) \in W \times W$. 
For the basic part of $\bJ_{\std}(\bbs)$, we use the vertex labeling as in \cref{conv:basic_vertex}. 
We denote by $v_{-\infty}^{\alpha^\vee}$ (resp. $v_{\infty}^{\alpha^\vee}$) the extra vertex associated with a positive coroot $\alpha^\vee \in \Phi^\vee(u)$ (resp. $\alpha^\vee \in \Phi^\vee(v)$). Moreover,
\begin{itemize}
    \item If $u=w_0$, then $\Phi^\vee(u)=\{\alpha_{\overline{s}} \mid \overline{s} \in \overline{S}\}$. In this case, we simply write $v_{-\infty}^s:=v_{-\infty}^{\alpha_{\overline{s}^\ast}^\vee}$. 
    \item If $v=w_0$, then $\Phi^\vee(v)=\{\alpha_s \mid s \in S\}$. In this case, we simply write $v_{\infty}^s:=v_{\infty}^{\alpha_s^\vee}$. 
\end{itemize}
With this notation, for a reduced word $\bbs$ of $(u,v)=(w_0,w_0)$, the
frozen vertices of $\bJ_\std(\bbs)$ are $v_0^s$, $v_{n^s(\bbs)}^s$, $v_\infty^s$ and $v_{-\infty}^s$ for $s \in S$.
\end{conv}

\begin{ex}\label{ex:amalgamation}
Here are some examples of the quivers $\bJ_\mathrm{full}(\bbs)$ and $\bJ_\std(\bbs)$: 
\begin{enumerate}
\item Type $A_3$, $\bbs=(1,2,3,1,2,1)$. The associated positive coroots are computed as 
    \begin{align*}
        &\beta_6^\bbs= \alpha^\vee_1, \quad \beta_5^\bbs=\alpha^\vee_1+\alpha^\vee_2, \quad \beta_4^\bbs=\alpha^\vee_2, \\
        &\beta_3^\bbs=\alpha^\vee_1+\alpha^\vee_2+\alpha^\vee_3, \quad \beta_2^\bbs=\alpha^\vee_2+\alpha^\vee_3, \quad \beta_1^\bbs=\alpha^\vee_3,
    \end{align*}
    and hence $1=s(6),~2=s(4),~3=s(1)$.
The picture follows, where we omit the arrows between extra vertices in $\bJ_\mathrm{full}(\bbs)$.
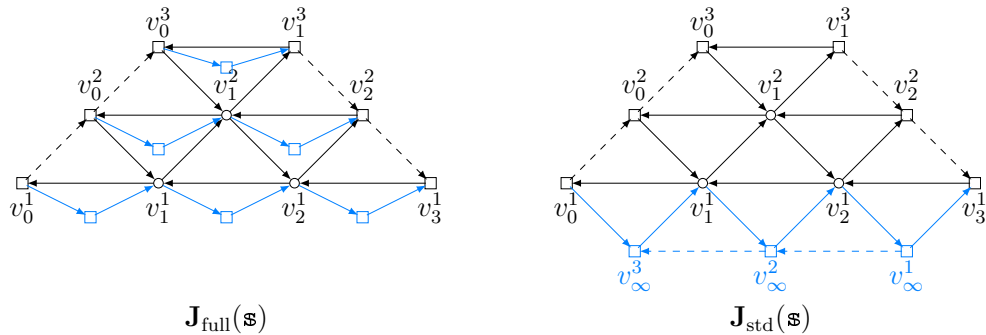
\[
\scalebox{0.9}{
\begin{tikzpicture}

\begin{scope}[>=latex]
\foreach \i in {1,2}
\draw (2*\i+1,0) circle(2pt) node[below]{$v_{\i}^1$};
\foreach \i in {1}
\draw (2*\i+2, 1) circle(2pt) node[above=0.2em]{$v_{\i}^2$};
\foreach \i in {0,3}
\fnode{2*\i+1,0}{black} node[below]{$v_{\i}^1$};
\foreach \i in {0,2}
\fnode{2*\i+2, 1}{black} node[above=0.2em]{$v_{\i}^2$};
\foreach \i in {0,1}
\fnode{2*\i+3, 2}{black} node[above]{$v_{\i}^3$};
\foreach \i in {1,2,3}
\qarrow{2*\i+1,0}{2*\i-1,0};
\foreach \i in {1,2}
\qarrow{2*\i+2,1}{2*\i,1};
\qarrow{5,2}{3,2};
\foreach \y in {1,2}
	{
	\qdarrow{\y,\y-1}{\y+1,\y};
	\qarrow{\y+2,\y-1}{\y+3,\y};
	}
\qarrow{5,0}{6,1};
\foreach \y in {1,2}
	{
	\qarrow{-\y+5,\y}{-\y+6,\y-1};
	\qdarrow{-\y+7,\y}{-\y+8,\y-1};
	}
\qarrow{2,1}{3,0};
{\color{myblue}
\foreach \j in {1,2,3}
    {\fnode{8-2*\j,-0.5}{myblue};
    \qarrow{7-2*\j,0}{8-2*\j,-0.5};
    \qarrow{8-2*\j,-0.5}{9-2*\j,0}; }
\foreach \j in {1,2}
    {\fnode{7-2*\j,0.5}{myblue}; 
    \qarrow{6-2*\j,1}{7-2*\j,0.5};
    \qarrow{7-2*\j,0.5}{8-2*\j,1}; }
\foreach \j in {1}
    {\fnode{6-2*\j,1.7}{myblue}; 
    \qarrow{5-2*\j,2}{6-2*\j,1.7};
    \qarrow{6-2*\j,1.7}{7-2*\j,2};
    }
}
\node at (4,-2) {$\bJ_\mathrm{full}(\bbs)$};

{\begin{scope}[xshift=8cm]
\foreach \i in {1,2}
\draw (2*\i+1,0) circle(2pt) node[below]{$v_{\i}^1$};
\foreach \i in {1}
\draw (2*\i+2, 1) circle(2pt) node[above=0.2em]{$v_{\i}^2$};
\foreach \i in {0,3}
\fnode{2*\i+1,0}{black} node[below]{$v_{\i}^1$};
\foreach \i in {0,2}
\fnode{2*\i+2, 1}{black} node[above=0.2em]{$v_{\i}^2$};
\foreach \i in {0,1}
\fnode{2*\i+3, 2}{black} node[above]{$v_{\i}^3$};
    {\color{myblue}
    \foreach \j in {1,2,3}
    \fnode{8-2*\j,-1}{myblue} node[below]{$v_\infty^{\j}$};
    }
\foreach \i in {1,2,3}
\qarrow{2*\i+1,0}{2*\i-1,0};
\foreach \i in {1,2}
\qarrow{2*\i+2,1}{2*\i,1};
\qarrow{5,2}{3,2};
\foreach \y in {1,2}
	{
	\qdarrow{\y,\y-1}{\y+1,\y};
	\qarrow{\y+2,\y-1}{\y+3,\y};
	}
\qarrow{5,0}{6,1};
\foreach \y in {1,2}
	{
	\qarrow{-\y+5,\y}{-\y+6,\y-1};
	\qdarrow{-\y+7,\y}{-\y+8,\y-1};
	}
\qarrow{2,1}{3,0};
{\color{myblue}
\qarrow{6,-1}{7,0};
\qarrow{5,0}{6,-1};
\qarrow{4,-1}{5,0};
\qarrow{3,0}{4,-1};
\qarrow{2,-1}{3,0};
\qarrow{1,0}{2,-1};
\qdarrow{4,-1}{2,-1};
\qdarrow{6,-1}{4,-1};
}
\node at (4,-2) {$\bJ_\std(\bbs)$};
\end{scope}}
\end{scope}
\end{tikzpicture}}
\] 
Here the vertices $v_1^1$, $v_2^1$, $v_1^2$ are mutable. The latter is the Fock--Goncharov quiver in \cite{FG06}.

\item Type $C_3$, $\bbs=(1,2,3,1,2,3,1,2,3)$. The associated positive coroots are computed as
    \begin{align*}
        &\beta_9^\bbs= \alpha^\vee_3, \quad \beta_8^\bbs=\alpha^\vee_2+2\alpha^\vee_3, \quad \beta_7^\bbs=\alpha^\vee_1+\alpha^\vee_2+2\alpha^\vee_3, \\
        &\beta_6^\bbs= \alpha^\vee_2+\alpha^\vee_3, \quad \beta_5^\bbs=\alpha^\vee_1+2\alpha^\vee_2+2\alpha^\vee_3, \quad \beta_4^\bbs=\alpha^\vee_2, \\
        &\beta_3^\bbs=\alpha^\vee_1+\alpha^\vee_2+\alpha^\vee_3, \quad \beta_2^\bbs=\alpha^\vee_1+\alpha^\vee_2, \quad \beta_1^\bbs=\alpha^\vee_1,
    \end{align*}
    and hence $1=s(1),~2=s(4),~3=s(9)$. 
The picture follows, where we omit the arrows between extra vertices in $\bJ_\mathrm{full}(\bbs)$.
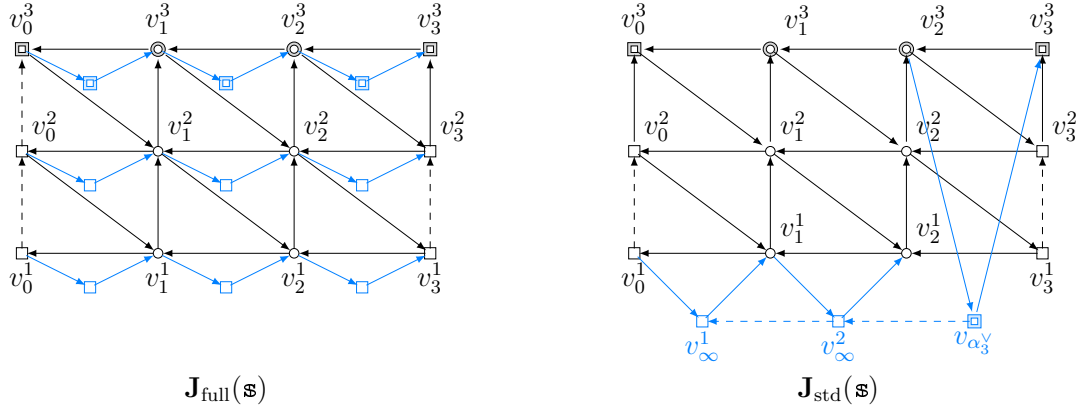
\[
\scalebox{0.9}{
\begin{tikzpicture}
\begin{scope}[>=latex]
\foreach \i in {1,2}
	\dnode{2*\i+1,3}{black} node[above=0.2em]{$v_{\i}^3$};
\foreach \i in {1,2}
{
    \draw(2*\i+1,0) circle(2pt) node[below]{$v_{\i}^1$};
    \draw(2*\i+1,1.5) circle(2pt) node[above right]{$v_{\i}^2$};    
}
\foreach \i in {0,3}
	\fdnode{2*\i+1,3}{black} node[above=0.2em]{$v_{\i}^3$};
\foreach \i in {0,3}
{
    \fnode{2*\i+1,0}{black} node[below]{$v_{\i}^1$};
    \fnode{2*\i+1,1.5}{black} node[above right]{$v_{\i}^2$};    
}
\foreach \i in {0,1,2}
    {\qarrow{2*\i+3,0}{2*\i+1,0};
     \qarrow{2*\i+3,1.5}{2*\i+1,1.5};
     \qsarrow{2*\i+3,3}{2*\i+1,3};
    }
\qdarrow{7,0}{7,1.5};
\qarrow{5,1.5}{7,0};
\qarrow{5,0}{5,1.5};
\qarrow{3,1.5}{5,0};
\qarrow{3,0}{3,1.5};
\qarrow{1,1.5}{3,0};
\qdarrow{1,0}{1,1.5};

\draw[qarrow,head](7,1.5)--(7,3);
\draw[qarrow,tail](5,3)--(7,1.5);
\draw[qarrow,head](5,1.5)--(5,3);
\draw[qarrow,tail](3,3)--(5,1.5);
\draw[qarrow,head](3,1.5)--(3,3);
\draw[qarrow,tail](1,3)--(3,1.5);
\draw[qarrow,dashed,head](1,1.5)--(1,3);

{\color{myblue}
\foreach \i in {1,2,3}
    {\fnode{2*\i,-0.5}{myblue};
    \qarrow{2*\i-1,0}{2*\i,-0.5}; 
    \qarrow{2*\i,-0.5}{2*\i+1,0}; 
    }
\foreach \i in {1,2,3}
    {\fnode{2*\i,1}{myblue};
    \qarrow{2*\i-1,1.5}{2*\i,1}; 
    \qarrow{2*\i,1}{2*\i+1,1.5}; 
    }
\foreach \i in {1,2,3}
    {\fdnode{2*\i,2.5}{myblue};
    \qarrow{2*\i-1,3}{2*\i,2.5}; 
    \qarrow{2*\i,2.5}{2*\i+1,3}; 
    }
}
\node at (4,-2) {$\bJ_\mathrm{full}(\bbs)$};

{\begin{scope}[xshift=9cm]
\foreach \i in {1,2}
	\dnode{2*\i+1,3}{black} node[above right=0.2em]{$v_{\i}^3$};
\foreach \i in {1,2}
{
    \draw(2*\i+1,0) circle(2pt) node[above right]{$v_{\i}^1$};
    \draw(2*\i+1,1.5) circle(2pt) node[above right]{$v_{\i}^2$};    
}
\foreach \i in {0,3}
	\fdnode{2*\i+1,3}{black} node[above=0.2em]{$v_{\i}^3$};
\foreach \i in {0,3}
{
    \fnode{2*\i+1,0}{black} node[below]{$v_{\i}^1$};
    \fnode{2*\i+1,1.5}{black} node[above right]{$v_{\i}^2$};    
}
{\color{myblue}
    \foreach \j in {1,2}
        \fnode{2*\j,-1}{myblue} node[below]{$v_\infty^{\j}$};
    \fdnode{6,-1}{myblue} node[below]{$v_{\alpha^\vee_3}$};
}
\foreach \i in {0,1,2}
    {\qarrow{2*\i+3,0}{2*\i+1,0};
     \qarrow{2*\i+3,1.5}{2*\i+1,1.5};
     \qsarrow{2*\i+3,3}{2*\i+1,3};
    }
\qdarrow{7,0}{7,1.5};
\qarrow{5,1.5}{7,0};
\qarrow{5,0}{5,1.5};
\qarrow{3,1.5}{5,0};
\qarrow{3,0}{3,1.5};
\qarrow{1,1.5}{3,0};
\qdarrow{1,0}{1,1.5};

\draw[qarrow,tail](7,1.5)--(7,3);
\draw[qarrow,head](5,3)--(7,1.5);
\draw[qarrow,tail](5,1.5)--(5,3);
\draw[qarrow,head](3,3)--(5,1.5);
\draw[qarrow,tail](3,1.5)--(3,3);
\draw[qarrow,head](1,3)--(3,1.5);
\draw[qarrow,tail](1,1.5)--(1,3);
{\color{myblue}
\foreach \j in {1,2}
    {
    \qarrow{2*\j-1,0}{2*\j,-1};
    \qarrow{2*\j,-1}{2*\j+1,0};
    }
\qdarrow{4,-1}{2,-1};
\qdarrow{6,-1}{4,-1};
\qsarrow{5,3}{6,-1};
\qsarrow{6,-1}{7,3};
}
\node at (4,-2) {$\bJ_\std(\bbs)$};
\end{scope}}

\end{scope}
\end{tikzpicture}}
\] 
Here the vertices $v_i^s$ for $s=1,2,3$, $i=1,2$ are mutable. 

\item Type $A_2$, $\bbs=(1,2,1,\overline{2},\overline{1},\overline{2})$. The picture shows $\bJ_\std(\bbs)$.
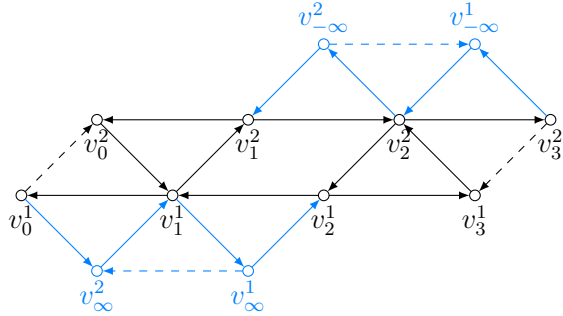
\[
\begin{tikzpicture}
\def\R{2}
\def\H{1}
\draw (0,0) circle(2pt) coordinate(V10) node[below,scale=0.9]{$v_0^1$};
\draw (\R,0) circle(2pt) coordinate(V11) node[below,scale=0.9]{$v_1^1$};
\draw (2*\R,0) circle(2pt) coordinate(V12) node[below,scale=0.9]{$v_2^1$};
\draw (3*\R,0) circle(2pt) coordinate(V13) node[below,scale=0.9]{$v_3^1$};
\draw (0.5*\R,\H) circle(2pt) coordinate(V20) node[below,scale=0.9]{$v_0^2$};
\draw (1.5*\R,\H) circle(2pt) coordinate(V21) node[below,scale=0.9]{$v_1^2$};
\draw (2.5*\R,\H) circle(2pt) coordinate(V22) node[below,scale=0.9]{$v_2^2$};
\draw (3.5*\R,\H) circle(2pt) coordinate(V23) node[below,scale=0.9]{$v_3^2$};
\qarrow{V11}{V10};
\qarrow{V12}{V11};
\qarrow{V21}{V20};
\qarrow{V21}{V22};
\qarrow{V22}{V23};
\qarrow{V12}{V13};
\qarrow{V20}{V11};
\qarrow{V11}{V21};
\qarrow{V13}{V22};
\qarrow{V22}{V12};
\qdarrow{V10}{V20};
\qdarrow{V23}{V13};
{\color{myblue}
\draw (0.5*\R,-\H) circle(2pt) coordinate(Y1) node[below,scale=0.9]{$v_\infty^2$};
\draw (1.5*\R,-\H) circle(2pt) coordinate(Y2) node[below,scale=0.9]{$v_\infty^1$};
\draw (2*\R,2*\H) circle(2pt) coordinate(Y3) node[above,scale=0.9]{$v_{-\infty}^2$};
\draw (3*\R,2*\H) circle(2pt) coordinate(Y4) node[above,scale=0.9]{$v_{-\infty}^1$};
\qarrow{V10}{Y1}; 
\qarrow{Y1}{V11};
\qarrow{V11}{Y2};
\qarrow{Y2}{V12};
\qdarrow{Y2}{Y1};
\qarrow{V22}{Y3};
\qarrow{Y3}{V21};
\qarrow{V23}{Y4};
\qarrow{Y4}{V22};
\qdarrow{Y3}{Y4};
}
\end{tikzpicture}
\]

\end{enumerate}
\end{ex}

\begin{conv}[Another labeling of vertices]\label{conv:labeling_corresp}
For a Coxeter word $\bbs=(\bbs_1,\dots, \bbs_m)$, the set of vertices $J(\bbs)$ of the basic part $\bJ(\bbs)$ of $\bJ_{\mathrm{full}}(\bbs)$ is in bijection with $S_0\sqcup [1, m]$, where $S_0:=\{s_0\mid s\in S\}$, as follows. 
\begin{itemize}
    \item $v_0^s$ corresponds to $s_0\in S_0$ for $s\in S$.
    \item The vertex arising from $|\bbs_k|^r$ in the $k$-th elementary quiver $\bJ(\bbs_k)$ in $\bJ(\bbs)=\bJ(\bbs_1)\ast \bJ(\bbs_2)\ast \dots \ast\bJ(\bbs_m)$ corresponds to $k$. Here, $|s|=|\ols|:=s$ for $s\in S$. 
\end{itemize}
In the main body of the paper, we focus especially on reduced words of $(w_0, w_0)$. Recall that the set of reduced words of $(w_0, w_0)$ is denoted by $S(w_0, w_0)$, and write $N:=\ell (w_0)$. For $\bbs\in S(w_0, w_0)$, we identify $J(\bbs)$ with $S_0\sqcup [1, 2N]$ by the above bijection. Then, we can directly check that $\cE(\bbs)=(\varepsilon_{\bbs; ij})_{i\in J(\bbs)^\uf, j\in J(\bbs)}$ coincides with the one defined in \eqref{eq:cE_bbs}. 
\end{conv}

\paragraph{\textbf{Quivers associated with decorated triangulations}}
Let $\boldsymbol{D}_{K+2}$ be a $(K+2)$-gon with $K \geq 1$. A \emph{$\mathfrak{g}$-decorated triangulation} $\bD$ of $\boldsymbol{D}_{K+2}$ consists of the following data:
\begin{enumerate}
    \item a triangulation of $\boldsymbol{D}_{K+2}$;
    \item a choice of a corner of each triangle $T$;
    \item a choice of a reduced word $\bs_T$ of $w_0$ for each triangle $T$.
\end{enumerate}
Given a $\mathfrak{g}$-decorated triangulation, we put the quiver $\bJ_\std(\bs_T)$ on each triangle $T$ so that 
\begin{itemize}
    \item the frozen sets $\{v_0^s\}_{s \in S}$, $\{v_{n^s(\bs_T)}^s\}_{s \in S}$, $\{v_{\alpha_s^\vee}\}_{s \in S}$ are located on the three edges of $T$, respectively, in this clockwise order;
    \item $\{v_{\alpha_s^\vee}\}_{s \in S}$ lies on the opposite edge of the chosen corner. 
\end{itemize}
Define the ``weights'' (cf. \cite[Lemma 4.9]{IO:Wilson}) of the frozen vertices by
\begin{align*}
    d_v:=\begin{cases}
        \alpha_s & \mbox{if $v=v_0^s$}, \\
        \alpha_{s^\ast} & \mbox{if $v=v_{n^s(\bs)}$}, \\
        \alpha_s & \mbox{if $v=v_{\alpha_s}$}.
    \end{cases}
\end{align*}
We define a quiver $Q_{\bD}$ on $\boldsymbol{D}_{K+2}$ by amalgamating these quivers along the interior edges of the triangulation so that the pairs of vertices with the same weight are identified, and then promoting the frozen vertices on the interior edges to be unfrozen.
See \cref{fig:quiver_polygon} for an example. It is known that the quivers $Q_{\bD}$ associated with any $\mathfrak{g}$-decorated triangulations $\bD$ of $\boldsymbol{D}_{K+2}$ are mutation-equivalent to each other \cite[Section 12.5]{GS:Quantum}. 

\begin{figure}[ht]
    \centering
\begin{tikzpicture}
\draw[blue] (0,0) -- (4,0) -- (4,4) -- (0,4) -- cycle;
\draw[blue] (4,0) -- (0,4);
\node[scale=0.7] at (0.1,4-0.3) {$\bullet$};
\node[scale=0.7] at (4-0.2,4-0.2
) {$\bullet$};
\node[scale=0.9] at (1,1.5) {$(1,2,1,2)$};
\node[scale=0.9,rotate=-45] at (2.5,2.5) {$(1,2,1,2)$};

\begin{scope}[xshift=7cm]
\draw[blue] (0,0) -- (4,0) -- (4,4) -- (0,4) -- cycle;
\draw[blue] (4,0) -- (0,4);
\node[scale=0.7] at (0.1,4-0.3) {$\bullet$};
\node[scale=0.7] at (4-0.2,4-0.2
) {$\bullet$};
\quiverCv{0,4}{0,0}{4,0}{2} 
\foreach \i in {0} \fnode{V\i 1}{black} coordinate(X\i 1);
\foreach \i in {1,2} \draw(V\i 1) circle(2pt) coordinate(X\i 1);
\foreach \i in {0} \fdnode{V\i 2}{black} coordinate(X\i 2);
\foreach \i in {1,2} \dnode{V\i 2}{black} coordinate(X\i 2);
\fnode{V10}{black} coordinate (X10);
\fdnode{V20}{black} coordinate (X20);
\draw[qarrow,dashed,head] (X01) to[bend left=20] (X02);
\draw[qarrow,dashed,tail] (X20) to[bend left=20] (X10);

\vertexC{4,4}{0,4}{4,0}{2} 
\foreach \i in {0,2} \fnode{V\i 1}{black} coordinate(Y\i 1);
\foreach \i in {1} \draw(V\i 1) circle(2pt) coordinate(Y\i 1);
\foreach \i in {0,2} \fdnode{V\i 2}{black} coordinate(Y\i 2);
\foreach \i in {1} \dnode{V\i 2}{black} coordinate(Y\i 2);
\draw(V10) coordinate (Y10);
\draw(V20) coordinate (Y20);
\qarrow{Y21}{Y11};
\qarrow{Y11}{Y01};
\qsarrow{Y22}{Y12};
\qsarrow{Y12}{Y02};
\draw[qarrow,head](Y11) -- (Y12);
\draw[qarrow,tail](Y02) -- (Y11);
\draw[qarrow,tail](Y12) -- (Y21);
\draw[qarrow,dashed,tail,head](Y01) to[bend left=20] (Y02);
\draw[qarrow,dashed,tail,head](Y21) to[bend right=20] (Y22);
{\color{myblue}
    \qarrow{Y01}{X21};
    \qarrow{X21}{Y11};
    \qsarrow{Y12}{X22};
    \draw[qarrow,tail,head] (X22) to[bend right=20] (Y22);
}
\end{scope}
\end{tikzpicture}
    \caption{An $\mathfrak{sp}_4$-decorated triangulation of $\boldsymbol{D}_4$ (Left) and the associated quiver $Q_{\bD}$ (Right). The chosen corner is shown by the symbol \scalebox{0.8}{$\bullet$}.}
    \label{fig:quiver_polygon}
\end{figure}
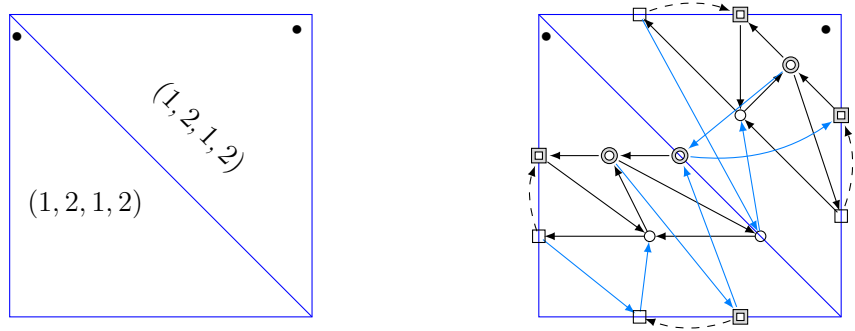

\begin{figure}[ht]
    \centering
\begin{tikzpicture}[scale=2.5]
\node[anchor=east] at (-0.2,0.5){$K$: even};
{\color{blue}
  \draw (0,1) -- (5,1);
  \draw (0,0) -- (5,0);
  \draw (0,0) -- (0,1);
  \draw (5,0) -- (5,1);

  \draw (1,0) -- (1,1);
  \draw (2,0) -- (2,1);
  \draw (4,0) -- (4,1);
  
  \draw (0,1) -- (1,0);
  \draw (1,1) -- (2,0);
  \draw (4,1) -- (5,0);}
  
  \node at (3,0.5) {$\cdots$};  

  \foreach \i in {0,1,4} \node[scale=0.7] at (0.05+\i,0.85){$\bullet$};
  \foreach \i in {1,2,5} \node[scale=0.7] at (-0.05+\i,0.15){$\bullet$};

  \node[scale=0.9] at (0.25,0.25){$\bs_1$};
  \node[scale=0.9] at (1.25,0.25){$\bs_3$};
  \node[scale=0.9] at (4.25,0.25){$\bs_{K-1}$};
  \node[scale=0.9,rotate=180] at (0.75,0.75){$\bs_2$};
  \node[scale=0.9,rotate=180] at (1.75,0.75){$\bs_4$};
  \node[scale=0.9,rotate=180] at (4.75,0.75){$\bs_K$};

\begin{scope}[yshift=-1.5cm]
\node[anchor=east] at (-0.2,0.5){$K$: odd};
{\color{blue}
  \draw (0,1) -- (5,1);
  \draw (0,0) -- (6,0);
  \draw (0,0) -- (0,1);
  \draw (5,0) -- (5,1);

  \draw (1,0) -- (1,1);
  \draw (2,0) -- (2,1);
  \draw (4,0) -- (4,1);
  
  \draw (0,1) -- (1,0);
  \draw (1,1) -- (2,0);
  \draw (4,1) -- (5,0);
  \draw (5,1) -- (6,0);}
  
  \node at (3,0.5) {$\cdots$};

  \foreach \i in {0,1,4,5} \node[scale=0.7] at (0.05+\i,0.85){$\bullet$};
  \foreach \i in {1,2,5} \node[scale=0.7] at (-0.05+\i,0.15){$\bullet$};

  \node[scale=0.9] at (0.25,0.25){$\bs_1$};
  \node[scale=0.9] at (1.25,0.25){$\bs_3$};
  \node[scale=0.9] at (4.25,0.25){$\bs_{K-2}$};
  \node[scale=0.9] at (5.25,0.25){$\bs_K$};
  \node[scale=0.9,rotate=180] at (0.75,0.75){$\bs_2$};
  \node[scale=0.9,rotate=180] at (1.75,0.75){$\bs_4$};
  \node[scale=0.9,rotate=180] at (4.75,0.75){$\bs_{K-1}$};

\end{scope}
\end{tikzpicture}
    \caption{A special example of $\mathfrak{g}$-decorated triangulation of $\boldsymbol{D}_{K+2}$.}
    \label{fig:polygon_std_triangulation}
\end{figure}
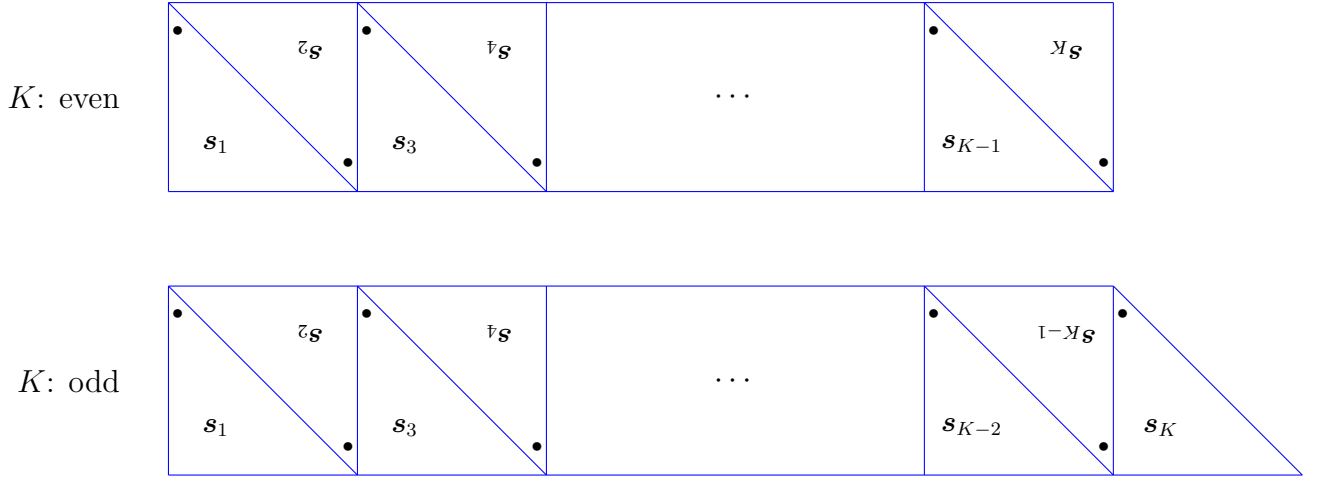

A special example of $\mathfrak{g}$-decorated triangulation of $\boldsymbol{D}_{K+2}$ is shown in \cref{fig:polygon_std_triangulation}. In this case, the quiver $Q_{\bD}$ consists of basic part $\bJ(\tbbs)$ of $\bJ_{\mathrm{full}}(\tbbs)$ with the following concatenation $\tbbs$ of Coxeter words:
\begin{align*}
    \tbbs&:=\bs_1\ast (\overline{\bs}_2^{\ast})^\op\ast \bs_3\ast (\overline{\bs}_4^\ast)^\op\ast \cdots\\
    &=
\begin{cases}
\bs_1\ast (\overline{\bs}_2^\ast)^\op\ast \bs_3\ast (\overline{\bs}_4^\ast)^\op\ast \cdots \ast (\overline{\bs}_K^\ast)^\op&\text{when $K$ is even},\\
\bs_1\ast (\overline{\bs}_2^\ast)^\op\ast \bs_3\ast (\overline{\bs}_4^\ast)^\op\ast \cdots \ast \bs_K&\text{when $K$ is odd},
\end{cases}
\end{align*}
together with the extra vertices coming from each piece: $\bJ_\std(\bs_{2l-1})$ or $\bJ_\std((\overline{\bs}_{2l}^\ast)^\op)$. Here for a reduced word $\bs=(s_1,\dots,s_N)$ of $w_0$, we write $\overline{\bs}:=(\ols_1,\dots,\ols_N)$, $\bs^\ast:=(s_1^\ast,\dots,s_N^\ast)$, and $\bs^\op:=(s_N,\dots,s_1)$. The extra vertex of $Q_{\bD}$ coming from $v^{s}_{\infty}$ of $\bJ_\std(\bs_{2\ell-1})$ (resp.~$\bJ_\std(\bs_{2\ell})$) is denoted by $v_{\infty_{2\ell-1}}^s$ (resp.~$v_{\infty_{2\ell}}^{s^{\ast}}$). 
For this special choice of $\lieg$-decorated triangulation of $\boldsymbol{D}_{K+2}$, we denote the resulting quiver described above by 
\begin{align}
\bJ_{\std}(\tbbs):=Q_{\bD}. \label{eq:std_tbbs}
\end{align}
See \cref{fig:quiver_std_triangulation} for type $A_2$ and $C_2$ examples.

\begin{figure}
    \centering
\begin{tikzpicture}
\def\R{3}
\draw[blue] (0,\R) -- (0,0) -- (2*\R,0) --(2*\R,\R) -- (0,\R) -- (\R,0) -- (\R,\R) -- (2*\R,0);
\quiverAv{0,\R}{0,0}{\R,0}{2}
\foreach \i in {1,2}{
    \fnode{a\i}{black};
    \fnode{V\i 0}{myblue};
    }
\draw(V11) circle(2pt);
\quiverAv{\R,0}{\R,\R}{0,\R}{2}
\foreach \i in {1,2}{
    \draw(a\i) circle(2pt);
    \draw(b\i) circle(2pt);
    \fnode{V\i 0}{myblue};
    }
\draw(V11) circle(2pt);
\quiverAv{\R,\R}{\R,0}{2*\R,0}{2}
\foreach \i in {1,2}{
    \fnode{V\i 0}{myblue};
    }
\draw(V11) circle(2pt);
\quiverAv{2*\R,0}{2*\R,\R}{\R,\R}{2}
\foreach \i in {1,2}{
    \fnode{a\i}{black};
    \draw(b\i) circle(2pt);
    \fnode{V\i 0}{myblue};
    }
\draw(V11) circle(2pt);
\node[scale=0.9,align=center] at (\R,-1){
Type $A_2$ \\ 
$\bs_1=\dots=\bs_4=(1,2,1)$\\ 
$\widetilde{\bbs}=(1,2,1,\overline{2},\overline{1},\overline{2},1,2,1,\overline{2},\overline{1},\overline{2})$};

\begin{scope}[xshift=3*\R cm]
\draw[blue] (0,\R) -- (0,0) -- (2*\R,0) --(2*\R,\R) -- (0,\R) -- (\R,0) -- (\R,\R) -- (2*\R,0);
\quiverCvMiddle{0,\R}{0,0}{\R,0}{2};
    \foreach \i in {1,2} \draw(V\i 1) circle(2pt);
    \foreach \i in {1,2} \dnode{V\i 2}{black};
    \draw[qarrow,tail] (V21) -- (V22);
    \draw[qarrow,head,dashed] (V01) to[bend left=20] (V02);
    \fnode{V01}{black};
    \fdnode{V02}{black};
    \fnode{V10}{myblue};
    \fdnode{V20}{myblue};
    \draw[qarrow,tail,dashed,myblue] (V20) to[bend left=20] (V10);
\vertexC{\R,0}{\R,\R}{0,\R}{2};
    \foreach \i in {0,1} \draw(V\i 2) circle(2pt);
    \foreach \i in {0,1} \dnode{V\i 1}{black};
    \qsarrow{V21}{V11};
    \qsarrow{V11}{V01};
    \qarrow{V22}{V12};
    \qarrow{V12}{V02};
    \draw[qarrow,head] (V12) -- (V11);
    \draw[qarrow,tail] (V21) -- (V12);
    \draw[qarrow,tail] (V11) -- (V02);
    \draw[qarrow,head] (V02) -- (V01);
    \fnode{V10}{myblue};
    \fdnode{V20}{myblue};
    {\color{myblue}
    \qarrow{V02}{V10};
    \qarrow{V10}{V12};
    \qsarrow{V11}{V20};
    \qsarrow{V20}{V21};
    \draw[qarrow,tail,dashed] (V20) to[bend left=20] (V10);}
\quiverCvMiddle{\R,\R}{\R,0}{2*\R,0}{2};
    \foreach \i in {0,1,2} \draw(V\i 1) circle(2pt);
    \foreach \i in {0,1,2} \dnode{V\i 2}{black};
    \draw[qarrow,tail](V21)--(V22);
    \fnode{V10}{myblue};
    \fdnode{V20}{myblue};
    \draw[qarrow,tail,myblue,dashed] (V20) to[bend left=20] (V10);
\vertexC{2*\R,0}{2*\R,\R}{\R,\R}{2};
    \foreach \i in {1} \draw(V\i 2) circle(2pt);
    \foreach \i in {1} \dnode{V\i 1}{black};
    \fnode{V02}{black};
    \fdnode{V01}{black};
    \qsarrow{V21}{V11};
    \qsarrow{V11}{V01};
    \qarrow{V22}{V12};
    \qarrow{V12}{V02};
    \draw[qarrow,head] (V12) -- (V11);
    \draw[qarrow,tail] (V21) -- (V12);
    \draw[qarrow,tail] (V11) -- (V02);
    \draw[qarrow,head,dashed] (V02) to[bend right=20] (V01);
    \fnode{V10}{myblue};
    \fdnode{V20}{myblue};
    {\color{myblue}
    \qarrow{V02}{V10};
    \qarrow{V10}{V12};
    \qsarrow{V11}{V20};
    \qsarrow{V20}{V21};
    \draw[qarrow,tail,dashed] (V20) to[bend left=20] (V10);}
\node[scale=0.9,align=center] at (\R,-1){
Type $C_2$\\
$\bs_1=\dots=\bs_4=(1,2,1,2)$\\
$\widetilde{\bbs}=(1,2,1,2,\overline{1},\overline{2},\overline{1},\overline{2},1,2,1,2,\overline{1},\overline{2},\overline{1},\overline{2})$};
\end{scope}
\end{tikzpicture}
    \caption{Quivers associated with the $\mathfrak{g}$-decorated triangulation of $\boldsymbol{D}_4$ in \cref{fig:polygon_std_triangulation}. The black part is the same as the basic part of  $\bJ_{\mathrm{full}}(\widetilde{\bbs})$.}
    \label{fig:quiver_std_triangulation}
\end{figure}

\section{Cluster structure on \texorpdfstring{$\Conf_4^\times \sA_G$}{Conf4}}\label{subsub:cluster_config}
We briefly recall the classical cluster structure on $\Conf_4^\times \sA_G$ following \cite[Section 11]{GS:Quantum}. See \cite{GS:Quantum} for all missing definitions. 

\subsection{Relative position of decorated flags}
Recall the notation in \cref{subsec:QCR} and \cref{subsec:config_classical}. 
In particular, for $K \in \Z_{\geq 2}$, the configuration space of $K$ decorated flags is defined to be the set 
\[
\Conf_K \sA_G := (\overbrace{\sA_G \times \dots \times \sA_G}^{K\text{ times}})/G,
\]
where the action of $G$ is the diagonal left action. The elements of $\Conf_K \sA_G$ are written as $[\sfA_1,\dots, \sfA_K]$ for $\sfA_1,\dots, \sfA_K\in \sA_G$. Write $H:=B^+\cap B^-$, which is called a Cartan subgroup. The weight lattice $P$ can be understood as $\Hom_{\text{alg. grp.}}(H,\C^{\times})$, and the coroot lattice $Q^{\vee}:=\sum_{s\in S}\Z\alpha_s^{\vee}$ can be understood as $\Hom_{\text{alg. grp.}}(\C^{\times}, H)$. For $w\in W$, we take a certain lift $\overline{w}$ of $w$ in $G$. See \cite[Section 2.1]{IOS} for their precise definitions. The Bruhat decomposition $G=\bigsqcup_{w \in W} U^+ H \overline{w}U^+$ implies the following:

\begin{lem}\label{lem:hw-distance}
The $G$-orbit of any pair $(\flA_1,\flA_2)$ contains a unique representative of the form
\begin{align*}
    (h.[U^+],\overline{w}.[U^+])
\end{align*}
for $h \in H$ and $w \in W$. We call $h=:h(\flA_1,\flA_2)$ and $w=:w(\flA_1,\flA_2)$ the \emph{$h$-} and \emph{$w$-distances}, respectively.
\end{lem}
A pair $(\flA_1,\flA_2)$ is said to be \emph{generic} if $w(\flA_1,\flA_2)=w_0$. We note that the $w$-distance only depends on the underlying flags in $G/B^+$.

\begin{lem}[Interpolation Lemma, {\cite[Lemma-Definition 5.3]{GS:Quantum}}]\label{lem:decorated_chain}
Let $(\flA_l,\flA_r)$ be a generic pair of decorated flags. 
Given a reduced word $\bs=(s_1,\dots,s_N)$ of $w_0$, 
there exists a unique chain $\flA_l=\flA_{(0)}\xleftarrow{s_1} \flA_{(1)}\xleftarrow{s_2}\dots\xleftarrow{s_N} \flA_{(N)}=\flA_r$ of decorated flags such that
\begin{itemize}
    \item $w(\flA_{(k)},\flA_{(k-1)})=r_{s_k}$,
    \item $h(\flA_{(k)},\flA_{(k-1)})=\alpha_{s_k}^\vee(c_k)$ 
\end{itemize} 
for $k=1,\dots,N$.
Here, $c_k \in \C^\ast$ is given by
\begin{align*}
    c_k:=\begin{cases}
    h(\flA_r,\flA_l)^{\varpi_t} & \mbox{if $\beta_k^\bs=\alpha_t^\vee$ is simple}, \\
    1 & \mbox{otherwise}.
    \end{cases}
\end{align*}
\end{lem}

When we write $(\flA_r,\flA_l)=g(h.[U^+],\overline{v}.[U^+])$, the intermediate flags are explicitly given by
\begin{align}\label{eq:interpolation_explicit}
    \flA_{(k)}:=g\overline{r}_{s_N}\dots \overline{r}_{s_{k+1}} h_k.[U^+],
\end{align}
where $h_k:=\prod_{j=1}^k r_{s_k}\dots r_{s_{j+1}}\alpha_{s_j}^\vee(c_j) \in H$. It is easy to verify that the flags \eqref{eq:interpolation_explicit} satisfy the required conditions.

\subsection{Cluster $K_2$-structure on $\Conf_4^\times \sA_G$}
By the Peter--Weyl theorem, we have $\cO(\sA_G) \cong \bigoplus_{\lambda \in P_+} V(\lambda)^\ast$, and hence
\begin{align*}
    \cO(\Conf_2 \sA_G) = (\cO(\sA_G)\otimes \cO(\sA_G))^{G} \cong \bigoplus_{\lambda \in P_+}(V(\lambda)^\ast \otimes V(\lambda^\ast)^\ast)^G.
\end{align*}
For $s \in S$, let $\Delta_s \in \cO(\Conf_2 \sA_G)$ denote the unique regular function such that
\begin{align*}
    \Delta_s \in (V(\varpi_s)^\ast \otimes V(\varpi_s^\ast)^\ast)^G, \quad \Delta_s([U^+],\vw.[U^+])=1.
\end{align*}
For a generic pair $(\flA_1,\flA_2)$ of decorated flags, we have $\Delta_s([\flA_1,\flA_2])=h(\flA_1,\flA_2)^{\varpi_s}$ for $s \in S$. These functions are the basic ingredients for the cluster variables described below.

\begin{figure}[ht]
    \centering
\begin{tikzpicture}
\path (3,3) coordinate(A);
\path (0,3) coordinate(B);
\path (0,0) coordinate(C);
\path (3,0) coordinate(D);
\bline{-0.5,0}{3.5,0}{0.15};
\tline{-0.5,3}{3.5,3}{0.15};
\draw[red,->-,thick] (1.5,3) --node[midway,right]{$g$} (1.5,0);
\filldraw(A) circle(1.5pt) node[above=0.3em,anchor=south west]{$\flA_2=\flA^\sfR=[U^+]$}; 
\filldraw(B) circle(1.5pt) node[above=0.3em,anchor=south east]{$\vw^{-1}{h'}^\ast.[U^+]=\flA^\sfL=\flA_1$};
\filldraw(D) circle(1.5pt) node[below=0.3em,anchor=north west]{$\flA_3=\flA_\sfR=gh.[U^+]$};
\filldraw(C) circle(1.5pt) node[below=0.3em,anchor=north east]{$g\vw.[U^+]=\flA_\sfL=\flA_4$};
\begin{pgfonlayer}{bg}  
\filldraw[fill=myblue!15,draw=blue,dashed,thick] (B) -- (C) -- (D) -- (A) --cycle;
\end{pgfonlayer}
\end{tikzpicture}
    \caption{A quadruple of decorated flags.}
    \label{fig:band_config_appendix}
\end{figure}
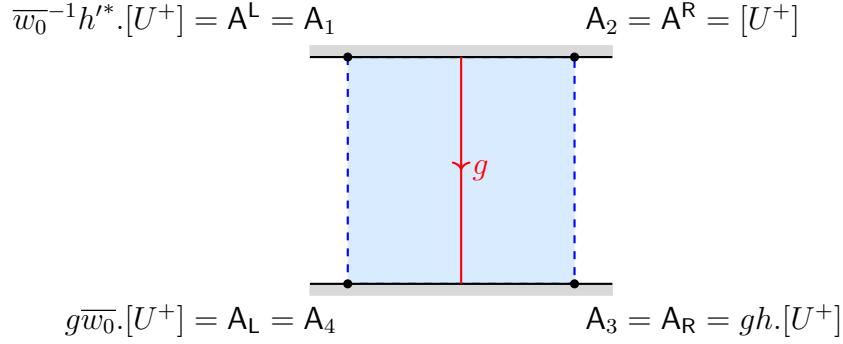

Fix a double reduced word $\bbs\in S(w_0,w_0)$. The subword of $\bbs$ consisting of letters in $S$ gives a reduced word $\bs_{\bullet}=(s_1,\dots,s_N)$ of $w_0$, and the $\overline{S}$-part gives another reduced word $\bs^{\bullet}=(s^1,\dots,s^N)$ of $w_0$. 
These subwords define order-preserving maps 
\begin{align*}
    \iota_\bullet, \iota^\bullet: \{1,\dots,N\} \to \{1,\dots,2N\}
\end{align*}
such that $\mathrm{Im}(\iota_\bullet) \sqcup \mathrm{Im}(\iota^\bullet)=\{1,\dots,2N\}$. 
Let $(\flA^\sfL,\flA^\sfR,\flA_\sfR,\flA_\sfL)$ be a quadruple of decorated flags such that the pairs $(\flA^\sfL,\flA^\sfR)$ and $(\flA_\sfL,\flA_\sfR)$ are generic. We apply \cref{lem:decorated_chain} to the pair $(\flA^\sfR,\flA^\sfL)$ with the word $(\bs^\bullet)^\op$, and to the pair $(\flA_\sfL,\flA_\sfR)$ with the word $\bs_\bullet^\ast$. Then we get the following chains of decorated flags:
\begin{align}
    &\flA^\sfL=\flA^0\xrightarrow{s^1} \flA^1\xrightarrow{s^2}\dots\xrightarrow{s^{N}} \flA^N=\flA^\sfR, \qquad w(\flA^{k-1},\flA^{k})=r_{s^k}, \quad h(\flA^{k-1},\flA^{k})=\alpha_{s^k}^\vee(c^k),\label{eq:chain_upper} \\
    &\flA_\sfL=\flA_0\xleftarrow{s_1^\ast} \flA_1\xleftarrow{s_2^\ast} \dots\xleftarrow{s_N^\ast} \flA_N=\flA_\sfR, \qquad w(\flA_k,\flA_{k-1})=r_{s_k^\ast}, \quad h(\flA_k,\flA_{k-1})=\alpha_{s_k^\ast}^\vee(c_k),\label{eq:chain_lower}
\end{align}
where $c^k$ and $c_k$ are given by
\begin{align*}
    c^k:=\begin{cases}
    h(\flA^\sfL,\flA^\sfR)^{\varpi_t} & \mbox{if $\beta_{k}^{\bs^\bullet}=\alpha_{t^\ast}^\vee$ is simple\footnotemark}, \\
    1 & \mbox{otherwise},
    \end{cases}
\end{align*}
and 
\begin{align*}
    c_k:=\begin{cases}
    h(\flA_\sfR,\flA_\sfL)^{\varpi_u^\ast} & \mbox{if $\beta_k^{\bs_\bullet^\ast}=\alpha_{u^\ast}^\vee$ is simple}, \\
    1 & \mbox{otherwise},
    \end{cases}
\end{align*}
respectively. \footnotetext{Here, note that $\beta^{\bs^\op}_{N+1-k}=\alpha_t^\vee$ if and only if $\beta^{\bs}_k=\alpha_{t^\ast}^\vee$.}
We call the following triples the \emph{elementary triples}:
\begin{itemize}
    \item The triple $(\flA^{{k}},\flA_{l+1},\flA_{l})$ such that $\iota^\bullet(k)< \iota_\bullet(l+1)<\iota^\bullet(k+1)$. 
    \item The triple $(\flA^{{k}},\flA^{{k+1}},\flA_{l})$ such that $\iota_\bullet(l)<\iota^\bullet(k+1)<\iota_\bullet(l+1)$. 
\end{itemize}
Here we set $\iota_\bullet(0)=\iota^\bullet(0):=0$ and $\iota_\bullet(N+1)=\iota^\bullet(N+1):=2N+1$. 
\begin{dfn}[Cluster $K_2$-coordinates on $\Conf_4^\times \sA_G$]
For a double reduced word $\bbs\in S(w_0,w_0)$, let $\mathbf{A}(\bbs)$ denote the collection of the following regular functions on $\Conf_4^\times \sA_G$:
\begin{itemize}
    \item $\Delta_s(\flA^{{k}},\flA_{l})$ for $s \in S$, where $(\flA^{{k}},\flA_{l+1},\flA_{l})$ or $(\flA^{{k}},\flA^{{k+1}},\flA_{l})$ is an elementary triple.
    \item $\Delta_s(\flA^R,\flA_R)$ for $s \in S$.
    \item $h(\flA_{l},\flA_{l-1})^{\varpi_{t^\ast}}$ such that $\beta_l^{\bs_\bullet^\ast}=\alpha_{t^\ast}^\vee$ is simple.
    \item $h(\flA^{k-1},\flA^{k})^{\varpi_{t}}$ such that $\beta_k^{\bs^\bullet}=\alpha_{t^\ast}^\vee$ is simple.
\end{itemize}
Define the frozen subset to be
\begin{align*}
    \mathbf{A}^\fr(\bbs):=& \{\Delta_s(\flA^L,\flA_L) \mid s \in S\}\cup\{\Delta_s(\flA^R,\flA_R) \mid s \in S\}\\
    &\cup\{h(\flA_{l},\flA_{l-1})^{\varpi_{t^\ast}} \mid \text{$\beta_l^{\bs_\bullet^\ast}=\alpha_{t^\ast}^\vee$ is simple}\} \cup \{h(\flA^{k-1},\flA^{k})^{\varpi_{t}} \mid \text{$\beta_k^{\bs^\bullet}=\alpha_{t^\ast}^\vee$ is simple}\}.
\end{align*} 
\end{dfn}
Here, some of functions in the collection $\mathbf{A}(\bbs)$ are identical to each other. 
The independent functions are well-assigned to the vertices of the quiver $\bJ_\std(\bbs)$ (see \cref{sec:quiver} for the construction), as follows.

For each elementary configuration of the form $(\flA^k,\flA_{l+1},\flA_l)$, we have $\Delta_t(\flA^k,\flA_l)=\Delta_t(\flA^k,\flA_{l+1})$ for $t \neq s_l$. Similarly, $\Delta_t(\flA^k,\flA_l)=\Delta_t(\flA^{k+1},\flA_l)$ for $t \neq s^k$ for each elementary configuration $(\flA^k,\flA^{k+1},\flA_l)$. These properties agree with the amalgamation pattern for the quiver $\bJ_\std(\bbs)$, and hence the collection $\bA(\bbs)$ can be assigned to the vertices of $\bJ_\std(\bbs)$. The set of vertices of $\bJ_\std(\bbs)$ is denoted by $J_{\square}(\bbs)$. 
For a vertex $j \in J_{\square}(\bbs)$, denote by $A_{\bbs; j}^\GS \in \bA(\bbs)$ the corresponding function. By using the labeling in \cref{conv:labeling_corresp} for $J(\bbs)\subset J_{\square}(\bbs)$, we can describe $A_{\bbs; j}^\GS$ explicitly as follows. 
\begin{align}
    A_{\bbs; j}^\GS
    &=\begin{cases}
        \Delta_{s}(\flA^\sfL,\flA_\sfL)&\text{if $j=s_0\in S_0$,}\\
        \Delta_{|\bbs_j|}(\flA^k,\flA_l)&\text{if $j\in [1, 2N]$ and $\iota^\bullet(k)\leq j <\iota^\bullet(k+1), \iota_\bullet(l)\leq j <\iota_\bullet(l+1)$,}\\
          h(\flA^{k-1},\flA^{k})^{\varpi_{s}}& \text{if $j=s_{-\infty}:=v^s_{-\infty}$ and  $\beta_k^{\bs^\bullet}=\alpha_{s^{\ast}}^\vee$,}\\
        h(\flA_{l},\flA_{l-1})^{\varpi_{s^\ast}}& \text{if $j=s_{\infty}:=v^s_{\infty}$ and  $\beta_l^{\bs_\bullet}=\alpha_{s}^\vee$.}
    \end{cases}\label{eq:GS_variable}
\end{align}
Denote by $\cE^\GS(\bbs):=(\ve_{\bbs; ij}^\GS)_{i\in J_{\square}(\bbs)^\uf, j\in J_{\square}(\bbs)}$ the exchange matrix of $\bJ_\std(\bbs)$. Then we get a cluster $K_2$-seed $\bi^\GS(\bbs)=(\cE^\GS(\bbs), (A_{\bbs; j}^\GS )_{j\in J_{\square}(\bbs)})$ in the field of rational functions $\mathcal{K}(\Conf_4^\times \sA_G)$ on $\Conf_4^\times \sA_G$. 

\begin{thm}[{\cite{GS:Quantum}}]\label{thm:indep_GS}
The cluster $K_2$-seeds $\bi^\GS(\bbs)$ associated with any double reduced words $\bbs$ of $(w_0,w_0)$ are mutation-equivalent to each other. 
\end{thm}

\subsection{Weights of cluster variables}\label{subsec:weights}
The right action of $H$ on $\sA_G$ induces a right action of $H^4$ on $\Conf_4 \sA_G$, given by 
\begin{align*}
    [\flA_1,\flA_2,\flA_3,\flA_4].(h_1,h_2,h_3,h_4) := [\flA_1.h_1,\flA_2.h_2,\flA_3.h_3,\flA_4.h_4]
\end{align*}
for $[\flA_1,\flA_2,\flA_3,\flA_4] \in \Conf_4 \sA_G$ and $(h_1,h_2,h_3,h_4) \in H^4$. 
For any homogeneous function $F \in \cO(\Conf_4 \sA_G)$ and $k=1,2,3,4$, let $\deg_{\flA_k} F \in P$ denote the weight for the $H$-action on $\flA_k$. 
When we view the configuration as in \cref{fig:band_config_appendix}, we write 
\begin{align*}
    \deg F:=\begin{bmatrix}
    \deg_{\flA^\sfL} F & \deg_{\flA^\sfR} F \\ 
    \deg_{\flA_\sfL} F & \deg_{\flA_\sfR} F
\end{bmatrix}
\end{align*}
We are going to compute $\deg A_v^{\bbs}$ for $v \in J_{\square}(\bbs)$.

\begin{conv}
For any subset $A \subset S$, let $H_A:=\{ \prod_{s \in A} \alpha_s^\vee(h^{\varpi_s}) \mid h \in H\}$. Denote the corresponding projection by
\begin{align*}
    \pi_A: H \to H_A, \quad h \mapsto \prod_{s \in A} \alpha_s^\vee(h^{\varpi_s}).
\end{align*}
\end{conv}
Recall the intermediate flags \eqref{eq:interpolation_explicit} and 
$h_k:=\prod_{j=1}^k r_{s_k}\dots r_{s_{j+1}}\alpha_{s_j}^\vee(c_j) \in H$. If we write $u_k:=r_{s_{k+1}}\dots r_{s_m}$ for $k=1,\dots,m$ with $u_m=e$, we get 
\begin{align*}
    u_k^{-1}(h_k)=\prod_{j=1}^k \beta_j^{\bs}(c_j)=\prod_{t \in \Inv(w_0 u_k)} \alpha_t^\vee(h^{\varpi_t})=\pi_{\Inv(w_0u_k)}(h).
\end{align*}
Hence
\begin{align}
    h_k =u_k(\pi_{\Inv(w_0u_k)}(h(\flA_r,\flA_l))).
\end{align}

\begin{lem}\label{lem:weight_interpolating_flag}
For $h \in H$ and  $u \in W$, let  
\begin{align*}
    \vtr{u}{h}:=u(\pi_{\Inv(w_0 u)}(h)),\quad \btr{u}{h}:=uw_0(\pi_{\Inv(u)}(h)) \in H
\end{align*}
Then we have 
\begin{align*}
    (\vtr{u}{h})^{\lambda} = h^{[u^{-1}\lambda]_+}, \quad (\btr{u}{h})^{\lambda} = h^{[u^{-1}\lambda]_-^\ast}
\end{align*}
for any $\lambda \in  P_+$.
\end{lem}

\begin{proof}
Write $u^{-1}\lambda=\sum_{s \in S} c_s \varpi_s$ with $c_s \in \Z$. The coefficients are given by $c_s=\langle \alpha_s^\vee, u^{-1}\lambda \rangle=\langle u.\alpha_s^\vee, \lambda \rangle$, and hence $u.\alpha_s^\vee > 0$ if and only if $c_s \geq 0$. 
Then we have
\begin{align*}
    (\vtr{u}{h})^\lambda &=u(\pi_{\Inv(w_0 u)}(h))^{\lambda} = \prod_{s \in \Inv(w_0u)} (h^{\varpi_s})^{\langle \alpha_s^\vee, u^{-1}\lambda\rangle} = h^{[u^{-1}\lambda]_+}, \\
    (\btr{u}{h})^\lambda &=uw_0(\pi_{\Inv(u)}(h))^{\lambda} = \prod_{s \in \Inv(u)} (h^{\varpi_s})^{-\langle \alpha_{s^\ast}^\vee, u^{-1}\lambda\rangle} = h^{[u^{-1}\lambda]_-^\ast}.
\end{align*}
\end{proof}

\begin{lem}\label{lem:scaling_interpolating_flag}
Let $(\flA_l,\flA_r)$ be a generic pair of decorated flags, and $\bs=(s_1,\dots,s_N)$ a reduced word of $w_0$. Then by the $H^2$-action $(\flA_l,\flA_r) \mapsto (\flA_l.h_l, \flA_r.h_r)$ for $(h_l,h_r) \in H^2$, the intermediate flags $(\flA_{(k)})_{k=0,\dots,N}$ obtained by \cref{lem:decorated_chain} are rescaled as
\begin{align*}
    \flA_{(k)} \mapsto \flA_{(k)}.(\btr{u_k}{h_l}) (\vtr{u_k}{h_r}).
\end{align*}
\end{lem}

\begin{proof}
Assume $(\flA_r,\flA_l)=(h.[U^+], \vw.[U^+])$ with $h=h(\flA_r,\flA_l)$. Then 
\begin{align*}
    \flA_{(k)}=\overline{u_k^{-1}} (\vtr{u_k}{h}).[U^+]
\end{align*}
with $u_k=r_{s_{k+1}}\dots r_{s_N}$. The $H^2$-action is computed as 
\begin{align*}
    (\flA_r.h_r,\flA_l.h_l) = (hh_r.[U^+], \vw h_l.[U^+]) = w_0(h_l).( h_l^\ast h h_r.[U^+], \vw.[U^+]),
\end{align*}
and hence
\begin{align*}
    \flA_{(k)} \mapsto w_0(h_l) \overline{u_k^{-1}} (\vtr{u_k}{h_l^\ast h h_r}).[U^+] = \flA_{(k)}. u_kw_0(h_l)\cdot (\vtr{u_k}{h_l^\ast}) \cdot (\vtr{u_k}{h_r}).
\end{align*}
The dependence on $h_l$ is further computed as follows:
\begin{align*}
    u_k w_0(h_l) \cdot(\vtr{u_k}{h_l^\ast}) &= u_k \bigg(\prod_{s \in S} \alpha_s^\vee(h_l^{-\varpi_s^\ast}) \cdot \prod_{s \in \Inv(w_0u_k)} \alpha_s^\vee(h_l^{\varpi_s^\ast}) \bigg) \\
    &= u_k \bigg(\prod_{s \in \Inv(u_k)} \alpha_s^\vee(h_l^{-\varpi_s^\ast})  \bigg)\\
    &= u_k w_0 (\pi_{\Inv(u_k)}(h_l)) = \btr{u_k}{h_l}.
\end{align*}
Thus the assertion is proved. 
\end{proof}

\begin{prop}\label{prop:weight}
For any double reduced word $\bbs$ of $(w_0,w_0)$, the weights of the associated cluster variables are given by
\begin{align*}
    \deg \Delta_s(\flA^k,\flA_l) &=\begin{bmatrix}
        [u_k^{-1}\varpi_s]_+ & [u_k^{-1}\varpi_s]_-^\ast \vspace{2mm}\\ 
        [v_l^{-1}\varpi_s]_- & [v_l^{-1}\varpi_s]_+^\ast
    \end{bmatrix} \\
    \deg h(\flA_{l},\flA_{l-1})^{\varpi_t^\ast} &= \begin{bmatrix}
        0 & 0 \\
        \varpi_t & \varpi_t^\ast
    \end{bmatrix} & \mbox{if $\beta_l^{\bs_\bullet^\ast}=\alpha_{t^\ast}^\vee$ is simple}, \\
    \deg h(\flA^{k-1},\flA^{k})^{\varpi_t} &= \begin{bmatrix}
        \varpi_t & \varpi_t^\ast \\
        0 & 0
    \end{bmatrix} & \mbox{if $\beta_k^{\bs^\bullet}=\alpha_{t^\ast}^\vee$ is simple}.
\end{align*}
Here, $u_k:=r_{s^{k}}\dots r_{s^1}$ and $v_l:=r_{s_{l+1}}\dots r_{s_N}$. By using the labeling in \eqref{eq:GS_variable} and the notation in \cref{subsec:Qseed_square}, 
\begin{align*}
    \deg A_{\bbs; j}^\GS &=
    \begin{cases}
    \begin{bmatrix}
        \varpi_s & 0\\ 
        \varpi_s^\ast & 0
    \end{bmatrix}&  \mbox{if $j=s_0\in S$}, \vspace{5pt}\\
    \begin{bmatrix}
        [\gamma_{\bbs; j}]_+ & [\gamma_{\bbs; j}]_-^\ast \\ 
        [\delta_{\bbs; j}]_- & [\delta_{\bbs; j}]_+^\ast
    \end{bmatrix}&  \mbox{if $j\in [1, 2N]$}, \vspace{5pt}\\
        \begin{bmatrix}
        \varpi_s & \varpi_s^\ast \\
        0 & 0
    \end{bmatrix} & \mbox{if $j=s_{-\infty}\in S_{-\infty}$}, \vspace{5pt}\\ 
    \begin{bmatrix}
        0 & 0 \\
        \varpi_s & \varpi_s^\ast
    \end{bmatrix}  & \mbox{if $j=s_{\infty}\in S_{\infty}$}.
        \end{cases}
\end{align*}
\end{prop}

\begin{proof}
The rescalings of the cluster variables in $\mathbf{A}_\uf(\bbs)$ are computed as follows:
\begin{align*}
    \Delta_s(\flA^k,\flA_l) \mapsto& \Delta_s(\flA^k.(\btr{u_k}{h^\sfR})(\vtr{u_k}{h^\sfL}), \flA_l.(\btr{v_l}{h_\sfL})(\vtr{v_l}{h_\sfR})) \\
    =& \Delta_s(\flA^k,\flA_l) \cdot (\btr{u_k}{h^\sfR})^{\varpi_s}(\vtr{u_k}{h^\sfL})^{\varpi_s}(\btr{v_l}{h_\sfL})^{\varpi_s^\ast}(\vtr{v_l}{h_\sfR})^{\varpi_s^\ast}.
\end{align*}
Therefore we get the degrees as asserted by \cref{lem:scaling_interpolating_flag}. 
For the variables in $\mathbf{A}^\fr(\bbs)$, note that 
\begin{align*}
    h(\flA_{l},\flA_{l-1})^{\varpi_{t^\ast}} = h(\flA_\sfR,\flA_\sfL)^{\varpi_{t^\ast}}, \quad h(\flA^{k-1},\flA^{k})^{\varpi_{t}} = h(\flA^\sfL,\flA^\sfR)^{\varpi_t}
\end{align*}
by construction, from which we immediately get the degrees as asserted. 
\end{proof}

\subsection{Cluster variables and generalized minors}
The following is essentially \cite[Proposition 4.12]{IOS}, where some conventions are modified:

\begin{prop}\label{prop:GS_variable_minor}
Let $(g, h, h') \in G\times H\times H$ be the parameter in the standard configuration in $\Conf_4^\times \sA_G$, see \cref{fig:band_config_appendix}. Then we have
\begin{align}\label{eq:GS_variable_minor}
    \Delta_s(\flA^{k},\flA_l)=\Delta_{u_k^{-1}\varpi_s, v_l^{-1}\varpi_s}(g^\ast) h^{[v_l^{-1}\varpi_s]_+^\ast} {h'}^{[u_k^{-1}\varpi_s]^\ast_+}
\end{align}
for $s \in S$ and $k,l=1,\dots,N$. Here, $u_k:=r_{s^{k}}\dots r_{s^1}$ and $v_l:=r_{s_{l+1}}\dots r_{s_N}$. By using the labeling in \eqref{eq:GS_variable} and the notation in \cref{subsec:Qseed_square}, 
\begin{align*}
    A_{\bbs; j}^\GS (g, h, h')&=
    \begin{cases}
\Delta_{\varpi_s^{\ast}, w_0\varpi_s^{\ast}}(g) {h'}^{\varpi_s^\ast}&  \mbox{if $j=s_0\in S$}, \\
 \Delta_{\gamma_{\bbs; j}^{\ast}, \delta_{\bbs; j}^{\ast}}(g) h^{[\delta_{\bbs; j}]_+^\ast} {h'}^{[\gamma_{\bbs; j}]^\ast_+}&  \mbox{if $j\in [1, 2N]$}, \\
    {h'}^{\varpi_s^\ast} & \mbox{if $j=s_{-\infty}\in S_{-\infty}$}, \\
    h^{\varpi_s^\ast}& \mbox{if $j=s_{\infty}\in S_{\infty}$}.
        \end{cases}
\end{align*}
\end{prop}

\begin{proof}
In the standard configuration of \cref{fig:band_config}, the intermediate flags are explicitly given by
\begin{align*}
    \flA^{k} &= \vw^{-1}\overline{r}_{s^1}\dots \overline{r}_{s^{k}} (\vtr{u_k^{-1}}{{h'}^\ast}).[U^+], \\
    \flA_l &= g\overline{r}_{s_N}^\ast\dots \overline{r}_{s_{l+1}}^\ast (\vtr{v_l^{-1}}{h}).[U^+].
\end{align*}
Then we compute
\begin{align*}
    \Delta_s(\flA^k,\flA_l) &= \Delta_s(\vw^{-1}\overline{r}_{s^1}\dots \overline{r}_{s^{k}}.[U^+],\ g\overline{r}_{s_N}^\ast\dots \overline{r}_{s_{l+1}}^\ast.[U^+]) (\vtr{u_k^{-1}}{{h'}^\ast})^{\varpi_s} (\vtr{v_l^{-1}}{h})^{\varpi_s^\ast} \\
    &= \Delta_s(\overline{r}_{s^1}\dots \overline{r}_{s^{k}}.[U^+],\ \vw g\overline{r}_{s_N}^\ast\dots \overline{r}_{s_{l+1}}^\ast \vw^{-1}.[U^-]) {h'}^{[u_k^{-1}.\varpi_s]^\ast_+} h^{[v_l^{-1}.\varpi_s]_+^\ast} \\
    &= \Delta_s(\overline{r}_{s^1}\dots \overline{r}_{s^{k}}.[U^+],\  ((g^\ast)^\sfT)^{-1} \overline{r}_{s_N}\dots \overline{r}_{s_{l+1}}.[U^-]) {h'}^{[u_k^{-1}.\varpi_s]^\ast_+} h^{[v_l^{-1}.\varpi_s]_+^\ast} \\
    &= \Delta_s((g^\ast)^\sfT\overline{r}_{s^1}\dots \overline{r}_{s^{k}}.[U^+],\   \overline{r}_{s_N}\dots \overline{r}_{s_{l+1}}.[U^-]) {h'}^{[u_k^{-1}.\varpi_s]^\ast_+} h^{[v_l^{-1}.\varpi_s]_+^\ast} \\
    &= \Delta_{u_k^{-1}\varpi_s, v_l^{-1}\varpi_s}(g^\ast){h'}^{[u_k^{-1}.\varpi_s]^\ast_+} h^{[v_l^{-1}.\varpi_s]_+^\ast}.
\end{align*}
Here we used \cref{lem:weight_interpolating_flag} and $[U^-]=\vw.[U^+]$ in the second line, and $g^\ast = \vw (g^{-1})^\sfT \vw^{-1}$ in the third line.
\end{proof}

\subsection{Cluster Poisson variables}
Let us quickly recall the cluster Poisson variables on $\Conf_4^\times \mathscr{B}_{G'}$, following \cite{GS:Quantum}. 

Let $G'$ be the adjoint group of $G$ with the canonical projection $\pi_G: G \to G'$. 
Consider the flag variety $\mathscr{B}_{G'}:=G'/{B'}^+$, where ${B'}^+:=\pi_G(B^+)$. We have the configuration space
\begin{align*}
    \Conf_4 \mathscr{B}_{G'}:= [(\mathscr{B}_{G'}\times \mathscr{B}_{G'}\times \mathscr{B}_{G'}\times \mathscr{B}_{G'})/G],
\end{align*}
and the open subspace $\Conf_4^\times \mathscr{B}_{G'} \subset \Conf_4 \mathscr{B}_{G'}$ of generic configurations. We have a projection
\begin{align}\label{eq:Conf_ensemble}
    p: \Conf_4 \sA_G \to \Conf_4 (G/B^+) \to \Conf_4 \mathscr{B}_{G'}
\end{align}
defined by the projections $\pi: G/U^+ \to G/B^+$ and $\pi_G: G\to G'$. For any double reduced word $\bbs$ of $(w_0,w_0)$, there exists an associated cluster Poisson seed $(\cE^{\GS}(\bbs),(X_{\bbs; j}^\GS)_{j\in J_{\square}(\bbs)^{\uf}})$ in the field of rational functions $\mathcal{K}(\Conf_4 \mathscr{B}_{G'})$ on $\Conf_4 \mathscr{B}_{G'}$. 
The map \eqref{eq:Conf_ensemble} has the coordinate expression
\begin{align*}
    p^\ast X_{\bbs; j}^\GS = \prod_{j' \in J_{\square}(\bbs)} (A_{\bbs; j'}^\GS) ^{\ve_{\bbs; j,j'}^\GS}
\end{align*}
for all $j\in J_{\square}(\bbs)^{\uf}$. 
Since the first map in \eqref{eq:Conf_ensemble} is a principal $H$-bundle, we have 
\begin{align}\label{eq:weight_Poisson}
    \deg_{\flA} p^\ast X_{\bbs; j}^\GS =0
\end{align}
for all $j\in J_{\square}(\bbs)^{\uf}$ and $\flA \in \{\flA^\sfL,\flA^\sfR,\flA_\sfR,\flA_\sfL\}$.

\printbibliography

\end{document}